%% file: phdThesis_Pedro.tex
\documentclass[phd,american]{ThesisPUC_uk_mod}

\usepackage{amsmath,amsfonts,amssymb,a4,enumitem,epsfig,array,psfrag}
\usepackage{url}
\usepackage{amsthm}
\usepackage{subfig}
\usepackage{chngcntr} %To deal with counter-related problems
\usepackage{color}
\usepackage[section]{placeins} %To make figures stay within their sections
\usepackage{algorithm}
\usepackage{algpseudocode}
\newcommand{\num}{}
\graphicspath{{figures/}}

\newtheorem{definition}{Definition}[chapter]
\newtheorem{theo}{Theorem}
\newtheorem{prop}[definition]{Proposition}
\newtheorem{coro}[definition]{Corollary}
\newtheorem{lemma}[definition]{Lemma}

\newtheorem{example}[definition]{Example}
\newtheorem{case}{Case}
\counterwithin*{case}{definition}
\newtheorem{rem}[definition]{Remark}

\AfterBegin{itemize}{\addtolength{\itemsep}{-0.5em}}

\BeforeBegin{equation}{\vspace{-0.5em}}
\BeforeEnd{equation}{\vspace{-0.5em}}
\BeforeBegin{theo}{\vspace{-0.4em}}
\BeforeEnd{theo}{\vspace{-0.2em}}
\BeforeBegin{definition}{\vspace{-0.4em}}
\BeforeEnd{definition}{\vspace{-0.2em}}
\BeforeBegin{coro}{\vspace{-0.4em}}
\BeforeEnd{coro}{\vspace{-0.2em}}
\BeforeBegin{prop}{\vspace{-0.4em}}
\BeforeEnd{prop}{\vspace{-0.2em}}

\newcommand{\half}{{\frac{1}{2}}}
\newcommand{\baralpha}{\bar{\alpha}}
\newcommand{\barbeta}{\bar{\beta}}
\newcommand{\bargamma}{\bar{\gamma}}
\newcommand{\NN}{{\mathbb{N}}}
\newcommand{\ZZ}{{\mathbb{Z}}}

\newcommand{\RR}{{\mathbb{R}}}

\newcommand{\dist}{\operatorname{dist}}
\newcommand{\interior}{\operatorname{int}}
\newcommand{\cR}{\mathcal{R}}
\newcommand{\cD}{\mathcal{D}}
\newcommand{\cS}{\mathcal{S}}
\newcommand{\clS}{\bar{\cS}}
\newcommand{\cA}{\mathcal{A}}
\newcommand{\cB}{\mathcal{B}}

\newcommand{\cP}{\mathcal{P}}
\newcommand{\Tw}{\operatorname{Tw}}
\newcommand{\TW}{\operatorname{TW}}
\newcommand{\sgn}{\operatorname{sgn}}
\newcommand{\ccol}{\operatorname{color}} %cube color

\newcommand{\charge}{\operatorname{charge}}
\newcommand{\wind}{\operatorname{wind}}
\newcommand{\Aout}{\cA^{\vu}}
\newcommand{\maxtw}{\Tw_{\max}}
\newcommand{\metricw}[2]{\operatorname{w}_{\operatorname{metric}}(#1,#2)}
\newcommand{\topw}[2]{\operatorname{w}_{\operatorname{top}}(#1,#2)}
\newcommand{\Link}{\operatorname{Link}}
\newcommand{\Wr}{\operatorname{Wr}}
\newcommand{\ST}{T_{\operatorname{sym}}}
\newcommand{\bB}{\mathbf{B}}
\newcommand{\tv}{\vec{v}}
\newcommand{\vu}{\vec{u}}
\newcommand{\vw}{\vec{w}}
\newcommand{\vbeta}{\vec{\beta}}

\newcommand{\ex}{\vec{\mathbf{i}}}
\newcommand{\ey}{\vec{\mathbf{j}}}
\newcommand{\ez}{\vec{\mathbf{k}}}

\newcommand{\neighbor}{\mathcal{N}}

\newcommand{\myScaleVar}{1} %used for scaling pdf_tex texts when needed
\newcommand{\floor}[1]{\left\lfloor #1 \right\rfloor}
\newcommand{\ceil}[1]{\left\lceil #1 \right\rceil}
\newcommand{\plshalf}[1]{{#1}^\sharp}
\newcommand{\base}{\vw}
\newcommand{\tbase}{t_{\base}}
\newcommand{\tbasez}{t_{\ez}}
\newcommand{\tbasew}{t_{\vw}}

\newcommand{\subsumtext}[1]{\scalebox{0.7}{\mbox{#1}}}
\newcommand{\puc}{PUC-Rio}

\author{Pedro Henrique Milet Pinheiro Pereira}
\authorR{Pereira, Pedro Henrique Milet Pinheiro}
\advisor{Nicolau Cor\c{c}\~ao Saldanha}
\advisorR{Saldanha, Nicolau Cor\c{c}\~ao}
\title{Domino tilings of three-dimensional regions}
\runninghead{Domino tilings of three-dimensional regions}

\titlebr{Coberturas de regi\~oes tridimensionais por domin\'os}
\day{05}
\dayOrd{5th}
\month{February} \monthN{02} \year{2015}

\city{Rio de Janeiro}
\CDD{510}
\department{Matem\'atica}
\program{Matem\'atica}
\programbr{Matem\'atica}
\school{Centro T\'ecnico Cient\'ifico}
\university{Pontif\'icia \mbox{Universidade} Cat\'olica do Rio de Janeiro}
\uni{\puc}

\jury{
\jurymember{Carlos Gustavo Tamm de Ara\'ujo Moreira}{Instituto Nacional de Matem\'atica Pura e Aplicada -- IMPA}
\jurymember{Carlos Tomei}{Departamento de Matem\'atica -- \puc}
\jurymember{Robert David Morris}{Instituto Nacional de Matem\'atica Pura e Aplicada -- IMPA}
\jurymember{Roberto Imbuzeiro Moraes Felinto de Oliveira}{Instituto Nacional de Matem\'atica Pura e Aplicada -- IMPA}
\jurymember{Thomas Maurice Lewiner}{Departamento de Matem\'atica -- \puc}
\jurymember{Juliana Abrantes Freire}{Departamento de Matem\'atica -- \puc}
\schoolhead{Jos\'e Eugenio Leal}
}

\resume{ 
 The author completed his undergraduate studies in Mathematics at \puc~in 2009, and obtained the degree of Mestre from \puc~in 2011.
}
\acknowledgment{ 
To CNPq, FAPERJ and PUC-Rio, whose support made this research possible.

To my advisor, Prof. Nicolau Cor\c{c}\~ao Saldanha, for his patience, support and willingness to share his knowledge.

To the Department of Mathematics, K\'atia and Creuza in particular, for their availability and efficiency in helping in various situations.

To my family, for their constant love and support. 
}
\keywords
{
\key{Three-dimensional tilings} 
\key{Dominoes}
\key{Dimers}
\key{Flip accessibility}
\key{Connectivity by local moves}
\key{Writhe}
\key{Knot theory}
}
\keywordsbr
{
\key{Coberturas \mbox{tridimensionais}} 
\key{Domin\'os}
\key{D\'imeros}
\key{Acessibilidade por flips}
\key{\mbox{Conectividade} por movimentos locais}
\key{Writhe}
\key{Teoria dos n\'os}
}

\abstractbr
{
Nessa tese, consideramos coberturas de regi\~oes tridimensionais por domin\'os, especialmente as da forma $\cD \times [0,N]$. Em particular, n\'os investigamos as componentes conexas do espa\c{c}o de coberturas desse tipo de regi\~ao por flips, o movimento local que consiste em remover dois domin\'os paralelos adjacentes e coloc\'a-los de volta na \'unica outra posi\c{c}\~ao poss\'ivel. Para regi\~oes da forma $\cD \times [0,2]$, n\'os definimos um invariante polinomial $P_t(q)$ que caracteriza coberturas que est\~ao ``quase na mesma componente conexa'', num sentido discutido na tese. Tamb\'em provamos que o espa\c{c}o de coberturas desse tipo de regi\~ao \'e conexo por flips e trits, um movimento local que consiste em remover tr\^es domin\'os adjacentes e ortogonais entre si e coloc\'a-los de volta na \'unica outra posi\c{c}\~ao poss\'ivel. No caso geral, o invariante \'e um inteiro, o twist, para o qual damos uma f\'ormula combinat\'oria simples, bem como uma interpreta\c{c}\~ao via teoria dos n\'os; tamb\'em provamos que o twist tem propriedades aditivas para decomposi\c{c}\~oes adequadas de uma regi\~ao. Por fim, investigamos tamb\'em o conjunto de valores que s\~ao twists de coberturas de uma caixa $L \times M \times N$.
}
\abstract
{
	In this thesis, we consider domino tilings of three-dimensional regions, especially those of the form $\cD \times [0,N]$. In particular, we investigate the connected components of the space of tilings of such regions by flips, the local move performed by removing two adjacent dominoes and placing them back in the only other possible position. For regions of the form $\cD \times [0,2]$, we define a polynomial invariant $P_t(q)$ that characterizes tilings that are ``almost in the same connected component'', in a sense discussed in the thesis. We also prove that the space of domino tilings of such a region is connected by flips and trits, a local move performed by removing three adjacent dominoes, no two of them parallel, and placing them back in the only other possible position. For the general case, the invariant is an integer, the twist, to which we give a simple combinatorial formula and an interpretation via knot theory; we also prove that the twist has additive properties for suitable decompositions of a region. Finally, we investigate the range of possible values for the twist of tilings of an $L \times M \times N$ box.
}

\tablesmode{noshow}
\begin{document}
%There are \totfig\ figures and \tottab\ tables in this document.
\input{introdchap}

\input{notationchap}
\input{twofloorschap}
\input{multiplexchap}
%\input{homologychap}
\input{possiblevalschap}
\input{glossary_new}

%\glsaddall
%\printnoidxglossaries

\bibliography{biblio}{}
\bibliographystyle{plain}
\end{document}

%% file: introdchap.tex
\chapter{Introduction}
Towards the end of the twentieth century, a lot has been said about tilings of two-dimensional regions by a number of different pieces: in particular, the so-called \emph{domino} and \emph{lozenge} tilings have received a lot of attention.

Kasteleyn \cite{Kasteleyn19611209} showed that the number of domino tilings of a plane region can be calculated via the determinant of a matrix. Conway \cite{conway1990tiling} discovered a technique using groups, that in a number of interesting cases can be used to decide whether a given region can be tesselated by a set of given tiles. Thurston \cite{thurston1990} introduced height functions, and proposed a linear time algorithm for solving the problem of tileability of simply connected plane regions by dominoes. 
In a more probabilistic direction, Jockusch, Propp and Shor \cite{jockusch1998random,cohn1996local} studied random tilings of the so-called Aztec Diamond (introduced in \cite{elkies1992alternating}), and showed the Arctic Circle Theorem. Richard Kenyon and Andrei Okounkov also studied random tilings and, in particular, their relation to Harnack curves \cite{kenyonokounkov2006dimers,kenyonokounkov2006planar}.
The concept of a flip is important in the context of dominoes as well as in that of lozenges. In both cases, two tilings of a simply connected region can always be joined by a sequence of flips (see \cite{saldanhatomei1998overview} for an overview). Also, see \cite{saldanhatomei1995spaces} for considerations on flip connectivity in more general two-dimensional regions.     

However, in comparison, much less is known about tilings of three-dimensional regions. Hammersley \cite{hammersley1966limit} proved results concerning the asymptotic behavior of the number of brick tilings of a $d$-dimensional box when all dimensions go to infinity. In particular, his results imply that the limit $\ell_3 = \lim_{n \to \infty} \frac{\log f(2n)}{(2n)^3}$, where $f(n)$ is the number of tilings of an $n \times n \times n$ box, exists and is finite; as far as we know, its exact value is not yet known, but several upper and lower bounds have been established for $\ell_3$ (see \cite{ciucu1998improved} and \cite{friedland2005theory} for more information on this topic). 
Randall and Yngve \cite{randall2000random} considered tilings of ``Aztec'' octahedral and tetrahedral regions with triangular prisms, which generalize domino tilings to three dimensions; they were able to carry over to this setting many of the interesting properties from two dimensions, e.g., height functions and connectivity by local moves. Linde, Moore and Nordahl \cite{linde2001rhombus} considered families of tilings that generalize rhombus (or lozenge) tilings to $n$ dimensions, for any $n \geq 3$. 
Bodini \cite{bodini2007tiling} considered tileability problems of pyramidal polycubes. Pak and Yang \cite{pak2013complexity} studied the complexity of the problems of tileability and counting for domino tilings in three and higher dimensions, and proved some hardness results in this respect.

Others have considered difficulties with connectivity by local moves in dimension higher than two (see, e.g., \cite{randall2000random}). We propose a few related algebraic invariants that could help understand the structure of connected components by flips in dimension three.

\begin{figure}[ht]
	
	\centering
    \includegraphics[width=0.4\textwidth]{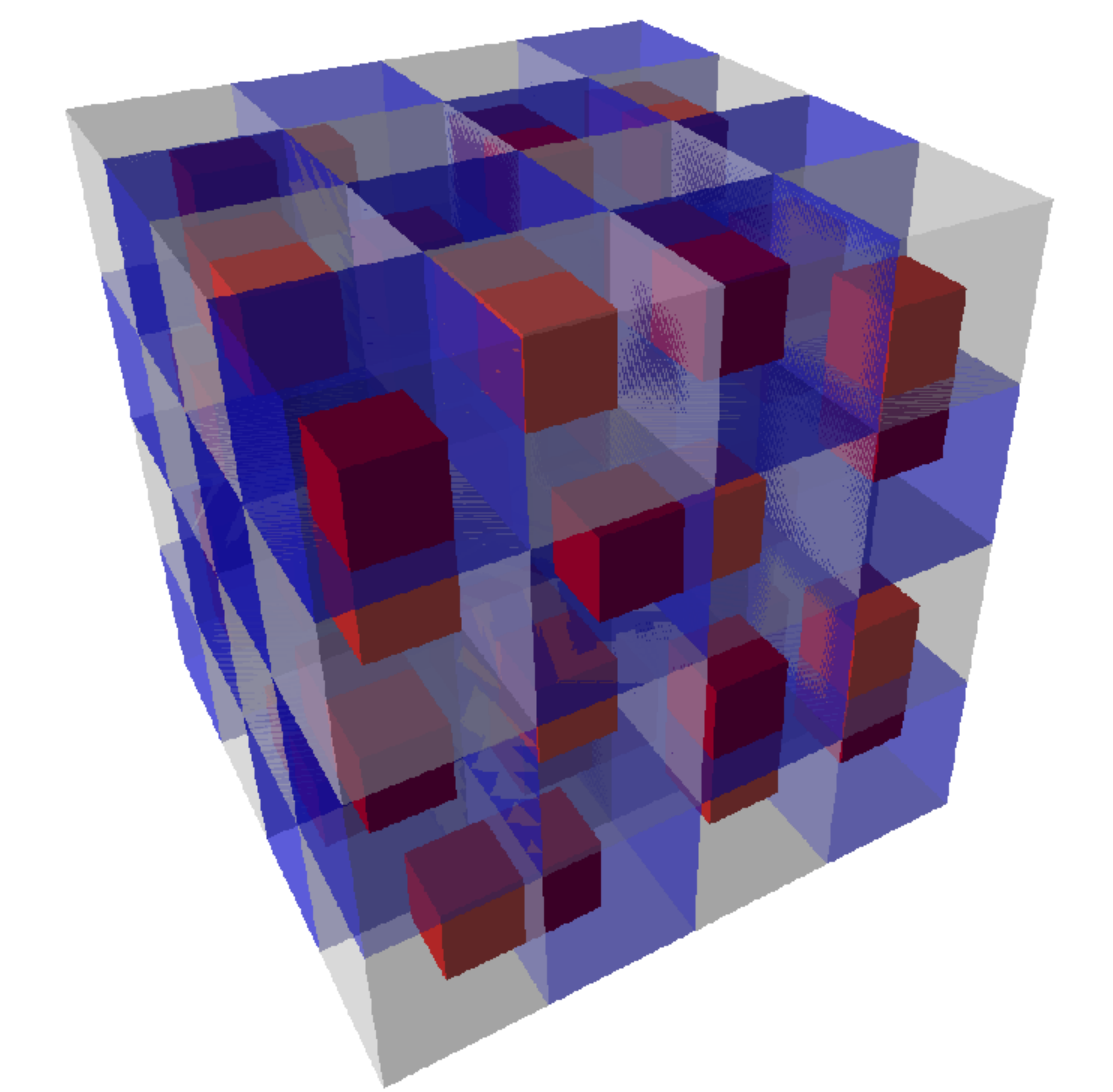}

		\caption{A tiling of a $4\times 4 \times 4$ box.}
			\label{fig:tiling444box}	

\end{figure}

This thesis contains material from the preprints \cite{primeiroartigo} and \cite{segundoartigo}.
In this thesis, we investigate tilings of contractible cubiculated regions (that is, a finite union of unit cubes with vertices in $\RR^2$) by \emph{domino brick} pieces, or \emph{dominoes}, which are $2 \times 1 \times 1$ rectangular cuboids, in one of the three possible rotations. An example of such a tiling is shown in Figure \ref{fig:tiling444box}. While this 3D representation of tilings may be attractive, it is also somewhat difficult to work with. Hence, we prefer to work with a 2D representation of tilings, which is shown in Figure \ref{fig:notation2Dexample}.

A key element in our study is the concept of a \emph{flip}, which is a straightforward generalization of the two-dimensional one. We perform a flip on a tiling by removing two (adjacent and parallel) domino bricks and placing them back in the only possible different position. The removed pieces form a $2 \times 2 \times 1$ slab, in one of three possible directions (see Figure \ref{fig:flipExample}). 

We also study the \emph{trit}, which is a move that happens within a $2 \times 2 \times 2$ cube with two opposite ``holes'', and which has an orientation (positive or negative: see Chapter \ref{chap:notation}). More precisely, we remove three dominoes, no two of them parallel, and place them back in the only other possible configuration (see Figure \ref{fig:negtrit_example}).   

\begin{figure}[hpt]
\centering
\def\svgwidth{0.8\columnwidth}
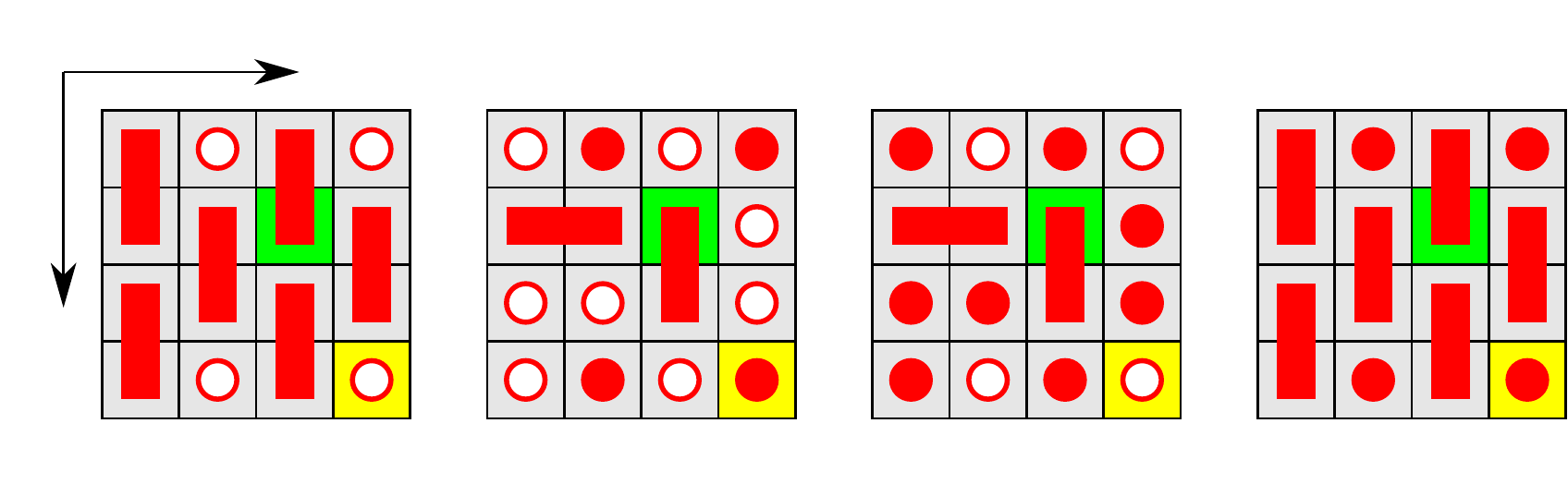
\caption{A tiling of the box $\cB = [0,4] \times [0,4] \times [0,4]$ box in our notation. The $x$ and $y$ axis are drawn, and $z$ points towards the paper, so that floors to the right have higher $z$ coordinates. Dominoes that are parallel to the $x$ or $y$ axis are represented as 2D dominoes, since they are contained in a single floor. Dominoes parallel to the $z$ axis are represented as circles, with the following convention: if the corresponding domino connects a floor with the floor to the left of it, the circle is painted red; otherwise, it is painted white. Thus, for example, in Figure \ref{fig:notation2Dexample}, each of the four white circles on the leftmost floor represents the same domino as the red circles on the floor directly to the right of it.
The squares highlighted in yellow represent cubes whose centers have the same $x$ and $y$ coordinates. Notice the top two yellow cubes are connected by a domino parallel to the $z$ axis, as well as the bottom two. The squares highlighted in green also represent cubes whose center have the same $x$ and $y$ coordinates, but the dominoes involving these cubes are not parallel to the $z$ axis.}
\label{fig:notation2Dexample}
\vspace{8pt}

\includegraphics[width=0.9\columnwidth]{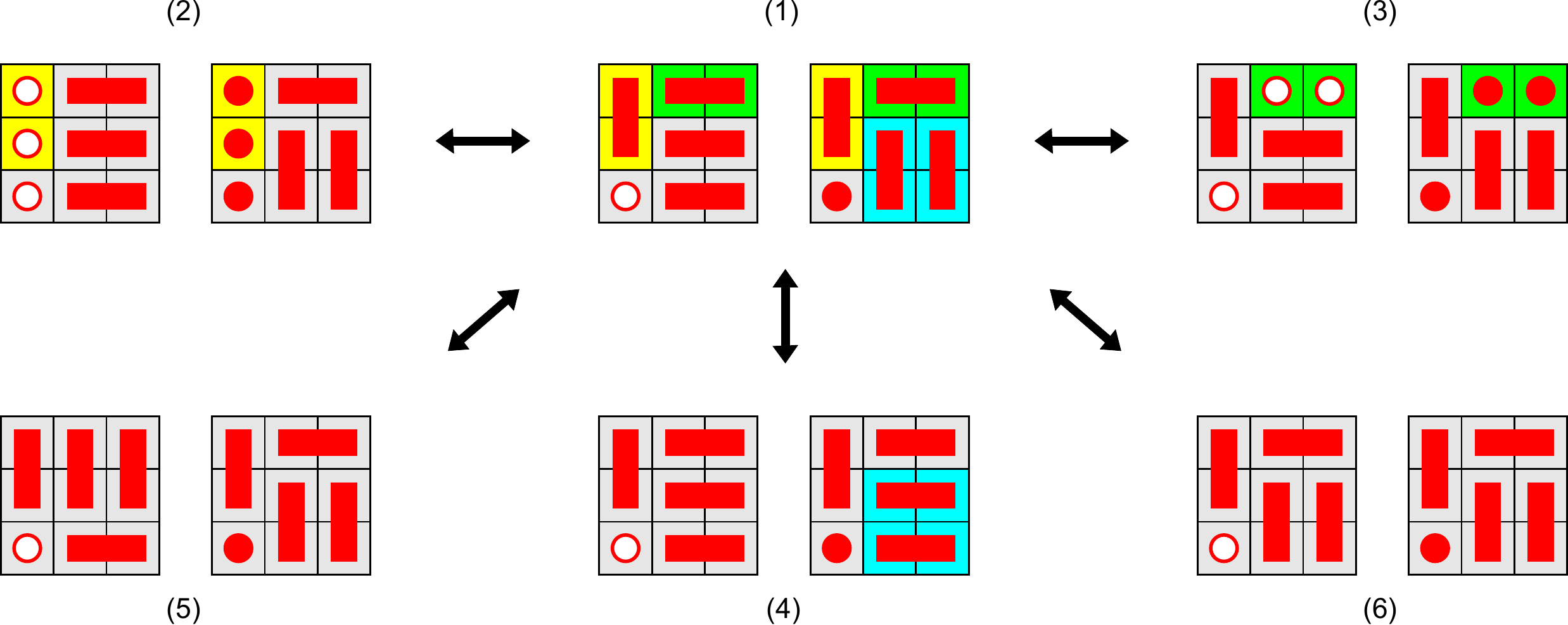}%
\caption{All flips available in tiling (1). The $2 \times 2 \times 1$ slabs involved in the flips taking (1) to (2), (3) and (4) are highlighted: they illustrate the three possible relative positions of dominoes in a flip.}%
\label{fig:flipExample}%

\vspace{10pt}

\includegraphics[width=0.6\columnwidth]{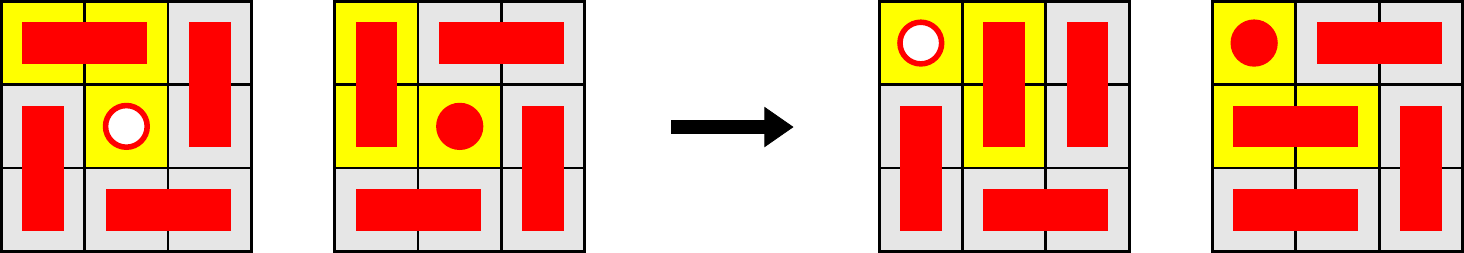}%
\caption{An example of a negative trit. The affected cubes are highlighted in yellow.}%
\label{fig:negtrit_example}%

\end{figure}

A \emph{(cubiculated) cylinder} or \emph{multiplex} is a region of the form $\cD \times [0,N]$ (possibly rotated), where $\cD \subset \RR^2$ is a simply connected planar region with connected interior. %In this paper, we introduce an algebraic invariant, the twist $\Tw(t)$, defined in Section \ref{sec:combTwistBoxes} for tilings of a cylinder.
%In \cite{primeiroartigo}, we study multiplexes with $n = 2$, called duplex regions. Although they are related to the general theory, tilings of these regions have some interesting characteristics of their own; in particular, we can define a polynomial $P_t(q)$ for tilings of duplex regions which is invariant by flips and which is finer than the twist. However, this construction breaks down when the duplex region is embedded in a region with more floors (see \cite{primeiroartigo} for details). 
When $N = 2$, the region is called a \emph{duplex region}. A cubiculated region $\cR$ is a \emph{two-story region} if it is contained in a set of the form $\RR^2 \times [0,2]$, possibly rotated, and such that the sets $\cR \cap (\RR^2 \times [0,1])$ and $\cR \cap (\RR^2 \times [1,2])$ are both connected and simply connected.
%If $R$ is a cubiculated region in $\RR^3$, a \emph{floor} of $R$ is $R \cap (\RR^2 \times [n,n+1])$, for some $n \in \ZZ$. We say that $R$ is a \emph{two-story region} when it has exactly two non-empty floors, $R \cap (\RR^2 \times [0,1])$ and $R \cap (\RR^2 \times [1,2])$, and both these floors are connected and simply connected. A two-story region is called a \emph{duplex region} if both floors are identical. 

%In this paper, we study tilings of two-story regions by \emph{domino brick} pieces (we often call these pieces \emph{dimers}), which are simply $2 \times 1 \times 1$ rectangular cuboids. An example of such a tiling is shown in Figure \ref{fig:tilingTwoFloors_example}. The general case (with an arbitrary number of floors) is discussed in \cite{segundoartigo}. 

In view of the situation with plane tilings, one might expect that the space of domino brick tilings of a simple three-dimensional region, say, an $L \times M \times N$ box, would be connected by flips. This turns out not to be the case, as the spaces of domino brick tilings of even relatively small boxes are already not flip-connected. In Chapter \ref{chap:twofloors}, we introduce an algebraic invariant for tilings of two-story regions. This invariant is a polynomial $P_t \in \ZZ[q,q^{-1}]$ associated with each tiling $t$ such that if two tilings $t_1,t_2$ are in the same flip connected component, then $P_{t_1} = P_{t_2}$. The converse is not necessarily true; however, it is almost true, in a sense that we shall discuss later. 
%-- Comment:
%Aqui talvez devamos dar uma definição informal de embedding ou coisas do gênero.
%---
We shall prove the following theorems in Chapter \ref{chap:twofloors}:

\begin{theo}
\label{theo:invariantProperties}
The polynomial $P_t(q)$, defined in Sections \ref{sec:twoIdenticalFloors} and \ref{sec:generalTwoStory}, has the following properties:
	%\item If $t_0$ and $t_1$ are tilings of a region with two identical floors that lie in the same flip connected component, then $P_{t_0} = P_{t_1}$.
\begin{enumerate}[label=(\roman*),topsep=0pc]
	%\item \label{item:changeGhostCurves} $P_t$ is uniquely defined, up to multiplication for a tiling $t$ of a two-story region $R$ up to a choice of ghost curves. Moreover, a different choice of ghost curves changes the polynomial $P_t$ for all tilings $t$ in the same way: by multiplying them by the same integer power of $q$.	
	\item \label{item:flipInvariant} If $t_0$ and $t_1$ are tilings of a two-story region that lie in the same flip connected component, then $P_{t_0} = P_{t_1}$.
	
	\item \label{item:moreSpace} If $\cR$ is a duplex region and $t_0$ and $t_1$ are two tilings of $\cR$, then $P_{t_0} = P_{t_1}$ if and only if there exists a box with two floors $\cB$ such that the embeddings of $t_0$ and $t_1$ in $\cB$ lie in the same flip connected component.
	
	\item \label{item:fourFloursEmbedding} If $\cR, t_0, t_1$ are as above and $P_{t_0}'(1) = P_{t_1}'(1)$, then there exists a box with four floors $\cB$ such that the embeddings of $t_0$ and $t_1$ in $\cB$ lie in the same flip connected component.
	
\end{enumerate}

\end{theo}

\begin{theo}
\label{theo:tritProperties}
The trit has the following properties in two-story regions:
\begin{enumerate}[label=(\roman*)]
	\item \label{item:posTrit} If $t_0$ and $t_1$ are two tilings of a two-story region and $t_1$ can be reached from $t_0$ after a single positive trit, then $P_{t_1}(q) - P_{t_0}(q) = q^k(q-1)$ for some $k \in \ZZ$; as a consequence, $P_{t_1}'(1) - P_{t_0}'(1) = 1$.
	\item \label{item:connectivityFlipsTrits} The space of domino brick tilings of a duplex region is connected by flips and trits. In other words, if $t_0$ is any tiling of such a region and $t_1$ is any other tiling of the same region, then there exists a sequence of flips and trits that take $t_0$ to $t_1$.
\end{enumerate}
\end{theo}

%For the moment, we shall restrict ourselves mostly to cylinders.
%, but we will in future works discuss much more general regions (see~\cite{terceiroartigo}).
%which is initially defined on multiplexes but can be extended to a much broader class of simply connected regions (see Section \ref{sec:additiveProperties}).
%
However, the construction for $P_t(q)$ no longer makes sense for more general regions.
Chapter \ref{chap:multiplex} discusses an invariant, the twist, defined for cylinders. 
The definitions of twist are somewhat technical and involve a relatively lengthy discussion. We shall give two different but equivalent definitions: the first one, given in Section \ref{sec:combTwistBoxes}, is a sum over pairs of dominoes.  At first sight, this formula gives a number in $\frac{1}{4} \ZZ$ and depends on a choice of axis. However, it turns out that, for cylinders, this number is an integer, and different choices of axis yield the same result. The proof of this claim will be completed in Section \ref{sec:differentDirections}, and it relies on the second definition, which uses the concepts of writhe and linking number from knot theory (see, e.g., \cite{knotbook}). In particular, we prove the following:

\begin{theo}
\label{theo:main}
Let $\cR$ be a cylinder, and $t$ a tiling of $\cR$. The twist $\Tw(t)$ is an integer with the following properties:
\begin{enumerate}[label=\upshape(\roman*)]
	\item \label{item:flipInvariance} If a tiling $t_1$ is reached from $t_0$ after a flip, then $\Tw(t_1) = \Tw(t_0)$.
	\item \label{item:tritDifference} If a tiling $t_1$ is reached from $t_0$ after a single positive trit, then $\Tw(t_1) - \Tw(t_0) = 1$.
	\item \label{item:duplexes} If $\cR$ is a duplex region, then $\Tw(t) = P_t'(1)$ for any tiling $t$ of $\cR$.   
	\item \label{item:multiplexUnion} Suppose a cylinder $\cR = \bigcup_{1 \leq i \leq m} \cR_i$, where each $\cR_i$ is a cylinder (they need not have the same axis) and such that $i \neq j \Rightarrow \interior(\cR_i) \cap \interior(\cR_j) \neq \emptyset$. Then there exists a constant $K \in \ZZ$ such that, for any family $(t_i)_{1 \leq i \leq n}$, $t_i$ a tiling of $\cR_i$,
	$$\Tw\left(\bigsqcup_{1 \leq i \leq m} t_{i}\right) = K + \sum_{1 \leq i \leq m} \Tw(t_{i}).$$
	%For each $i$, let $t_{i,0}$ and $t_{i,1}$ be two tilings of $\cR_i$. Then:
%$$\Tw\left(\bigsqcup_{1 \leq i \leq m} t_{i,0}\right) - \Tw\left(\bigsqcup_{1 \leq i \leq m} t_{i,1}\right) = \sum_{1 \leq i \leq m} (\Tw(t_{i,0}) - \Tw(t_{i,1})).$$ 
	\end{enumerate}
	\end{theo}

Finally, Chapter \ref{chap:possiblevals} discusses the possible range values for the twist of tilings of a given region. The main result of the chapter is the following theorem:
	
	\begin{theo}
	\label{theo:possible}
	For a box $\cB$, set $\Tw(\cB) = \{\Tw(t) |\allowbreak t \text{ tiling of } \cB\}$, and let $\maxtw(L,M,N) = \max \Tw([0,L] \times [0,M] \times [0,N])$. If $L,M,N \geq 2$, with at least two of the three dimensions strictly larger than $2$ (and at least one of them even), then
$$
C_0 LMN\min(L,M,N) \leq \maxtw(L,M,N) \leq  C_1 LMN\min(L,M,N),
$$
where $C_0 = 1/162$ and $C_1 = 1/16$.
Moreover, 
$$\limsup_{L,M,N \to \infty} \frac{\maxtw(L,M,N)}{LMN\min(L,M,N)} = \frac{1}{16} \text { but } \liminf_{L,M,N \to \infty} \frac{\maxtw(L,M,N)}{LMN\min(L,M,N)} \leq \frac{1}{24}.$$	
 \end{theo}

%% file: figures/notation2Dexample_axis_text.pdf_tex
%% Creator: Inkscape 0.48.3.1, www.inkscape.org
%% Accompanies image file 'notation2Dexample_axis_text.pdf' (pdf, eps, ps)
\begingroup%
  \makeatletter%
  \providecommand\color[2][]{%
    \errmessage{(Inkscape) Color is used for the text in Inkscape, but the package 'color.sty' is not loaded}%
    \renewcommand\color[2][]{}%
  }%
  \providecommand\transparent[1]{%
    \errmessage{(Inkscape) Transparency is used (non-zero) for the text in Inkscape, but the package 'transparent.sty' is not loaded}%
    \renewcommand\transparent[1]{}%
  }%
  \providecommand\rotatebox[2]{#2}%
  \ifx\svgwidth\undefined%
    \setlength{\unitlength}{488.18125bp}%
    \ifx\svgscale\undefined%
      \relax%
    \else%
      \setlength{\unitlength}{\unitlength * \real{\svgscale}}%
    \fi%
  \else%
    \setlength{\unitlength}{\svgwidth}%
  \fi%
  \global\let\svgwidth\undefined%
  \global\let\svgscale\undefined%
  \makeatother%
  \begin{picture}(1,0.31674967)%
    \put(0,0){\includegraphics[width=\unitlength]{notation2Dexample_axis_text.pdf}}%
    \put(0.15703811,0.29533088){\color[rgb]{0,0,0}\makebox(0,0)[lb]{\smash{$x$}}}%
    \put(-0.00044809,0.13784468){\color[rgb]{0,0,0}\makebox(0,0)[lb]{\smash{$y$}}}%
    \put(0.1042634,0.00035847){\color[rgb]{0,0,0}\makebox(0,0)[lb]{\smash{$z \in [0,1]$}}}%
    \put(0.35007374,0.00035847){\color[rgb]{0,0,0}\makebox(0,0)[lb]{\smash{$z \in [1,2]$}}}%
    \put(0.59588409,0.00035847){\color[rgb]{0,0,0}\makebox(0,0)[lb]{\smash{$z \in [2,3]$}}}%
    \put(0.84169443,0.00035847){\color[rgb]{0,0,0}\makebox(0,0)[lb]{\smash{$z \in [3,4]$}}}%
  \end{picture}%
\endgroup%

%% file: notationchap.tex
\chapter{Definitions and Notation}
\label{chap:notation}

This short chapter contains general notations and conventions that are used throughout the thesis, although we might postpone definitions that involve a lengthy discussion or are intrinsic of a given chapter or section.  

If $n$ is an integer, $\plshalf{n}$ will denote $n + \half$ (in music theory, $D\sharp$ is a half tone higher than $D$ in pitch). We also define $\plshalf{\ZZ}$ to be the set $\{\plshalf{n} | n \in \ZZ\}$.

Given $\tv_1,\tv_2,\tv_3 \in \RR^3$, $\det(\tv_1,\tv_2,\tv_3) = \tv_1 \cdot (\tv_2 \times \tv_3)$ denotes the determinant of the $3 \times 3 \times 3$ matrix whose $i$-th line is $\tv_i$, $i=1,2,3$. If $\beta = (\vbeta_1, \vbeta_2, \vbeta_3)$ is a basis, write $\det(\beta) = \det(\vbeta_1, \vbeta_2, \vbeta_3)$.

We denote the three canonical basis vectors as $\ex = (1,0,0)$, $\ey = (0,1,0)$ and $\ez = (0,0,1).$ %Since we used some graphical programs to visualize some domino configurations, we eventually got used to drawing the axes as most graphical programs do. Therefore, in this paper, we usually draw $\ex$ pointing from left to right, $\ey$ pointing downwards and $\ez$ pointing towards the paper. 
We denote by $\Delta = \{\ex, \ey, \ez\}$ the set of canonical basis vectors, and $\Phi = \{\pm \ex, \pm \ey, \pm \ez\}$. 
Let $\bB = \{\beta = (\vbeta_1,\vbeta_2,\vbeta_3) | \vbeta_i \in \Phi, \det(\beta) = 1\}$ be the set of positively oriented bases with vectors in $\Phi$. 

A \emph{basic cube} is a closed unit cube  in $\RR^3$ whose vertices lie in $\ZZ^3$. For $(x,y,z) \in \ZZ^3$, the notation $C\left(\plshalf{x} , \plshalf{y}, \plshalf{z}\right)$ denotes the basic cube $(x,y,z) + [0,1]^3$, i.e., the closed unit cube whose center is $\left(\plshalf{x},  \plshalf{y}, \plshalf{z}\right)$; it is \emph{white} (resp. \emph{black}) if $x + y + z$ is even (resp. odd). If $C = C\left(\plshalf{x} , \plshalf{y}, \plshalf{z}\right)$, define $\ccol(C) = (-1)^{x+y+z+1}$, or, in other words, $1$ if $C$ is black and $-1$ if $C$ is white. A \emph{region} is a finite union of basic cubes. A \emph{domino brick} or \emph{domino} is the union of two basic cubes that share a face. A \emph{tiling} of a region is a covering of this region by dominoes with pairwise disjoint interiors.

 %A basic square whose center is $(\plshalf{x}, \plshalf{y})$ has the same color as $C\left(\plshalf{x} , \plshalf{y}, \plshalf{0}\right)$. When the context is sufficiently clear, we may sometimes omit the word basic and simply say cube or square.

We sometimes need to refer to planar objects. Let $\pi$ denote either $\RR^2$ or a \emph{basic plane} contained in $\RR^3$, i.e., a plane with equation $x = k$, $y = k$ or $z = k$ for some $k \in \ZZ$. 
A \emph{basic square in} $\pi$ is a unit square $Q \subset \pi$ with vertices in $\ZZ^2$ (if $\pi = \RR^2$) or $\ZZ^3$.
 A \emph{planar region} $D \subset \pi$ is a finite union of basic squares. 
%$\cD \subset \RR^2$: we often identify $\RR^2$ with the $\xy$, $\yz$ or $\xz$ planes in $\RR^3$, as in the following paragraph. \margem{Maybe name these planes? Canonical planes?}

%Let $\tv \in \Delta$, and let $\pi \perp \tv$ be a basic plane. Let $n \geq 1$ be a positive integer. 
A region $\cR$ is a \emph{cubiculated cylinder} or \emph{multiplex region} if there exist a basic plane $\pi$ with normal vector $\tv \in \Delta$,
a simply connected planar region $\cD \subset \pi$ with connected interior 
and a positive integer $N$ 
such that 
$$\cR = \cD + [0,N]\tv = \{p + s\tv | p \in \cD, s \in [0,N]\};$$ 
we usually call $\cR$ a \emph{cylinder} or \emph{multiplex} for brevity. The cylinder $\cR$ above has \emph{base} $\cD$, \emph{axis} $\tv$ and \emph{depth} $N$. 
For instance, a cylinder with axis $\ez$ and depth $N$ can be written as $\cD \times [k,k+N]$, where $\cD \subset \RR^2$. 
A $\tv$-cylinder means a cylinder with axis $\tv$. A \emph{duplex region} or \emph{duplex} is a cylinder with depth $2$ (and hence the alternative name of multiplex for cylinders).

We sometimes want to point out that the hypothesis of simple connectivity (of a cylinder) is not being used: therefore, a \emph{pseudocylinder} (or \emph{pseudomultiplex}) with base $\cD$, axis $\tv$ and depth $N$ has the same definition as above, except that the planar region $\cD \subset \pi$ is only assumed to have connected interior (and is not necessarily simply connected).

A \emph{box} is a region of the form $\cB = [L_0, L_1] \times [M_0, M_1] \times [N_0, N_1]$, where $L_i, M_i, N_i \in \ZZ$. Boxes are special cylinders, in the sense that we can take any vector $\tv \in \Delta$ as the axis. In fact, boxes are the only regions that satisfy the definition of cylinder for more than one axis.

Regarding notation, Figures \ref{fig:notation2Dexample}, \ref{fig:flipExample} and \ref{fig:negtrit_example} were drawn with $\beta = (\ex,\ey,\ez)$ in mind. However, any $\beta \in \bB$ allows for such representations, as follows: we draw $\vbeta_3$ as perpendicular to the paper (pointing towards the paper). If $\pi = \vbeta_3^{\perp}$, we then draw each floor $\cR \cap (\pi + [n,n+1]\vbeta_3)$ as if it were a plane region. Floors are drawn from left to right, in increasing order of $n$.

%We say that a region is \emph{flip connected} if the space of domino brick tilings of the region, thought of as a graph where two tilings are joined by an edge if they differ by a single flip, is connected. In other words, a region is flip connected if, given two tilings of that region, there exists a sequence of flips that takes one tiling to the other. 
The \emph{flip connected component} of a tiling $t$ of a region $\cR$ is the set of all tilings of $\cR$ that can be reached from $t$ after a sequence of flips. %hence, a region $\cR$ is flip connected if and only all tilings of $\cR$ lie in the same flip connected component. %We characterize all the flip connected boxes in Section \ref{sec:combTwistBoxes} (see Remark \ref{rem:flipConnectedBoxes}).

%Since plane domino tilings may be seen as a special case of domino brick tilings with only one floor, it follows, for instance, that $L \times M \times 1$ boxes are always flip connected. Additionally, it is very easy to show, by induction, that $L \times 2 \times 2$ boxes are flip connected (the reader may want to prove this as an exercise).

%However, no other box is flip connected, as we see in Section \ref{sec:combTwistBoxes}. For instance, the $3 \times 3 \times 2$ box has 229 tilings, two of which have no flip positions. These are shown in Figures \ref{fig:tilings332_noflips1} and \ref{fig:tilings332_noflips2}. This box has, in fact, three flip connected components (two of which contain just one tiling).

%\begin{figure}[ht]
%\centering
%\subfloat[][]{\includegraphics[scale=0.70]{tiling332_noflips1.pdf}\label{fig:tilings332_noflips1}} \qquad \qquad
%\subfloat[][]{\includegraphics[scale=0.70]{tiling332_noflips2.pdf}\label{fig:tilings332_noflips2}}
%
%\caption{The two tilings of a $3 \times 3 \times 2$ box that have no flip positions}
%\label{fig:tilings332_noflips}
%
%\end{figure}

Suppose $t$ is a tiling of a region $\cR$, and let $\cB = [l,l+2] \times [m,m+2] \times [n,n+2]$, with $l,m,n \in \NN$. Suppose $\cB \cap \cR$ contains exactly three dominoes of $t$, no two of them parallel: notice that this intersection can contain six, seven or eight basic cubes of $\cR$. 
Also, a rotation (it can even be a rotation, say, in the $XY$ plane), can take us either to the left drawing or to the right drawing in Figure \ref{fig:postrit}.

\begin{figure}[ht]
\centering
\includegraphics[width=0.6\columnwidth]{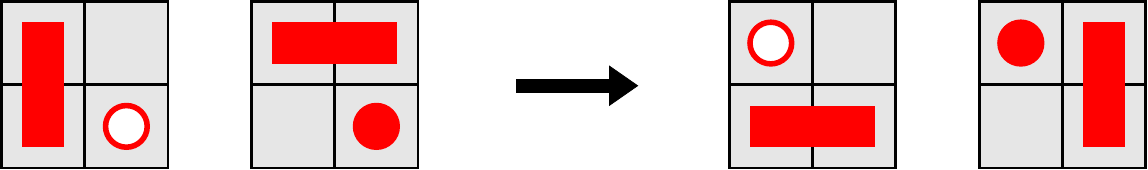}%
\caption{The anatomy of a positive trit (from left to right). The trit that takes the right drawing to the left one is a negative trit. The squares with no dominoes represent basic cubes that may or may not be in $\cR$ (see Figure \ref{fig:negtrit_example} for an example).}%
\label{fig:postrit}%
\end{figure}

If we remove the three dominoes of $t$ contained in $\cB \cap \cR$, there is only one other possible way we can place them back. This defines a move that takes $t$ to a different tiling $t'$ by only changing dominoes in $\cB \cap \cR$: this move is called a \emph{trit}. If the dominoes of $t$ contained in $\cB \cap \cR$ form a plane rotation of the left drawing in Figure \ref{fig:postrit}, then the trit is \emph{positive}; otherwise, it's \emph{negative}. Notice that the sign of the trit is unaffected by translations (colors of cubes don't matter) and rotations in $\RR^3$ (provided that these transformations take $\ZZ^3$ to $\ZZ^3$). A reflection, on the other hand, switches the sign (the drawing on the right can be obtained from the one on the left by a suitable reflection).

%% file: twofloorschap.tex
% Another way to input figures via pdf_tex files is
% \def\svgwidth{0.5\columnwidth}
%\input{flip_case1_diff.pdf_tex}
\chapter{The two-floored case}
\label{chap:twofloors}

A \emph{two-story region} is a connected and simply connected region $\cR \subset \pi + [0,2]\vu$, where $\pi$ is a basic plane with normal $\vu \in \Delta$, such that both sets $\cR \cap (\pi + [0,1] \vu)$ and $\cR \cap (\pi + [1,2] \vu)$ are simply connected.
We often think of these two sets as the two ``floors'' of $\cR$.
In this chapter (unlike the next), rotations will not play an important role, so we shall assume that $\pi = \RR^2 \times \{0\}$ and $\vu = \ez$. Likewise, duplex regions will be assumed to have axis $\ez$, and will often be denoted $\cD \times [0,2]$ where $\cD \subset \RR^2$ has connected interior and is simply connected. 

Most of the material in this chapter is also covered in \cite{primeiroartigo}.
 
\section{Duplex regions}
\label{sec:twoIdenticalFloors}
%Let $D \subset \RR^2$ be a quadriculated simply connected plane region, and let $\tilde{D} = D \times [0,1] \subset \RR^3$ be the three-dimensional region with only one floor that is the natural inclusion of $D$ as a cubiculated region. Lastly, consider $R = \tilde{D} \cup (\tilde{D} + (0,0,1)) = D \times [0,2]$, a region with two identical simply connected floors. We are interested in studying the flip connected components of $R$. More specifically, we will describe a simple algebraic invariant that will separate some of the flip connected components in $R$.
Let $\cD \subset \RR^2$ be a quadriculated simply connected plane region, and let $\cR = \cD \times [0,2] \subset \RR^3$ be a duplex region. Our goal throughout this section is to associate to each tiling $t$ of $\cR$ a polynomial $P_t \in \ZZ[q,q^{-1}]$ which always coincides for tilings in the same flip connected component.

%Consider the two floors of a tiling $t$ of $\cR$. Clearly, dominoes parallel to $\ez$ occur in the same positions in both floors. Hence, if we project all the non-$\ez$ dominoes of the bottom floor into the top floor, we will see two plane tilings of $\cD$ with ``stones'' (which occupy exactly one square) happening precisely in $\ez$ domino positions, as shown in Figure \ref{fig:tiling742_invariant}: we will call these ``stones'' \emph{jewels}. A white (resp. black) jewel is a jewel that happens in a white (resp. black) square; we write $\ccol(j) = 1$ if $j$ is black, and $-1$ if it is white.

Consider the two floors of a tiling $t$ of $\cR$, each dimer of $t$ being oriented from the white cube that it contains to the black one, as illustrated in the left of Figure \ref{fig:tiling742_invariant}. If we project the dimers of $t$ to the plane $z = 0$,    
we will see two plane tilings of $\cD$ with \emph{jewels} (which occupy exactly one square), which are the projections of dimers parallel to $\ez$, as shown in Figure \ref{fig:tiling742_invariant}. 
 A white (resp. black) jewel is a jewel that happens in a white (resp. black) square; we write $\ccol(j) = 1$ if $j$ is black, and $-1$ if it is white.

\begin{figure}[ht]
\centering
\includegraphics[width=0.75\columnwidth]{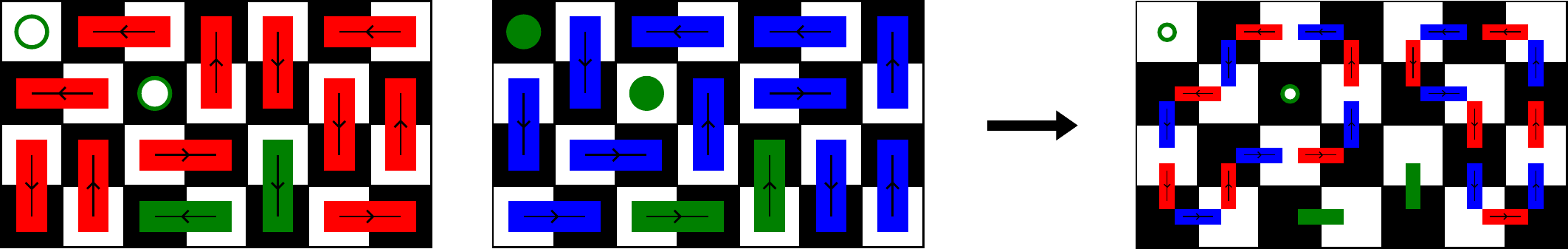}%
\caption{A tiling of a $7 \times 4 \times 2$ box and its associated drawing. The associated drawing has two jewels and four cycles, two of which are trivial ones. The jewels have opposite color, and both cycles spin counterclockwise. }%
\label{fig:tiling742_invariant}%
\end{figure} 

Thus, the resulting drawing can be seen as a set of disjoint plane cycles (projections of dimers that are not parallel to $\ez$ maintain their orientations) together with jewels. A cycle is called \textit{trivial} if it has length $2$. In order to construct $P_t$, consider a jewel $j$, and let $k_t(j) = \sum_{\gamma} \wind(\gamma,j)$, where the sum is taken over all the cycles $\gamma$ in the associated drawing, and $\wind(\gamma,j)$ denotes the winding number of $\gamma$ around $j$. Then 
$$P_t(q) = \sum_{j}\ccol(j)q^{k_t(j)}.\label{def:ptduplex}$$
For instance, for the tiling $t$ in Figure \ref{fig:tiling742_invariant}, $P_t(q) = q - 1$. For now, define the twist of a tiling to be $\Tw(t) = P_t'(1).$ (see Chapter \ref{chap:multiplex} for a general definition).

We now show that $P_t$ is, in fact, flip invariant.

\begin{prop}
\label{prop:twoFloorFlipInvariant}
Let $\cR$ be a duplex region, and let $t_0$ be a tiling of $\cR$. If $t_1$ is obtained from $t_0$ by performing a single flip on $t_0$, then $P_{t_1} = P_{t_0}$. 
\end{prop}
\begin{proof}
Let us consider the tiling $t_0$ and its associated drawing as a plane tiling with jewels; we want to see how this drawing is altered by a single flip. We will split the proof into cases, and the reader may find it easier to follow by looking at Figure \ref{fig:flipCasesTwoFloors}.

\begin{figure}[ht]
\centering
\subfloat[Case \ref{case:flipzdimer}: $P_t(q) = 1 - q^{-2}$ in both tilings]{\includegraphics[scale=0.4]{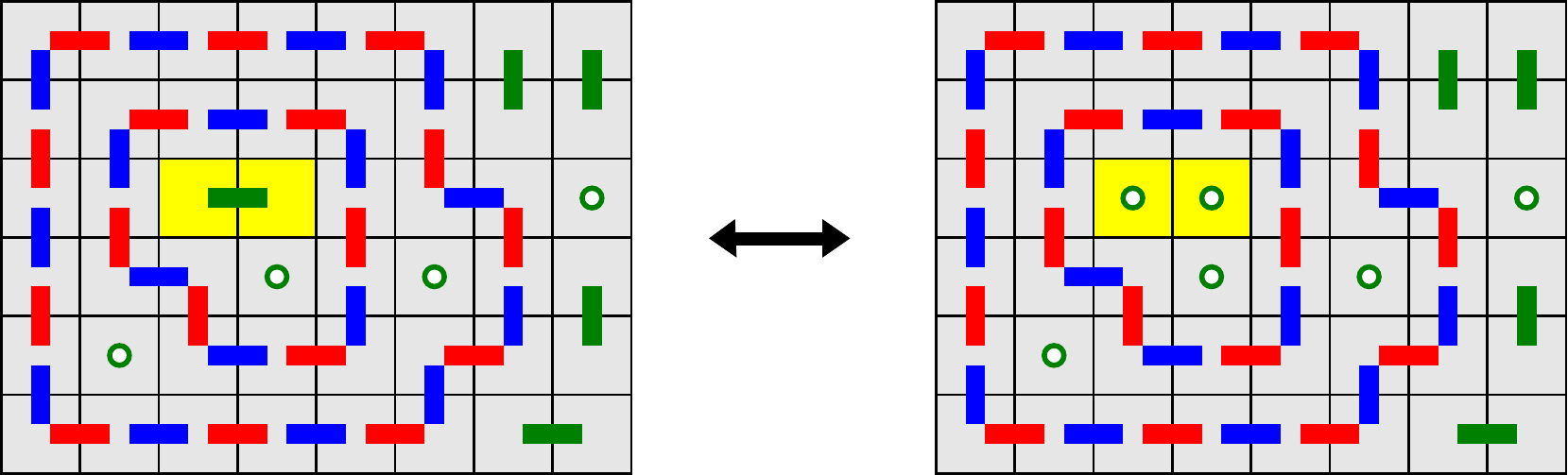}\label{fig:flipcase1}} \\
\subfloat[Case \ref{enum:sameorient}: $P_t(q) = q - 1$ in both tilings. Both cycles must have the same orientation to allow a flip.]{\includegraphics[scale=0.4]{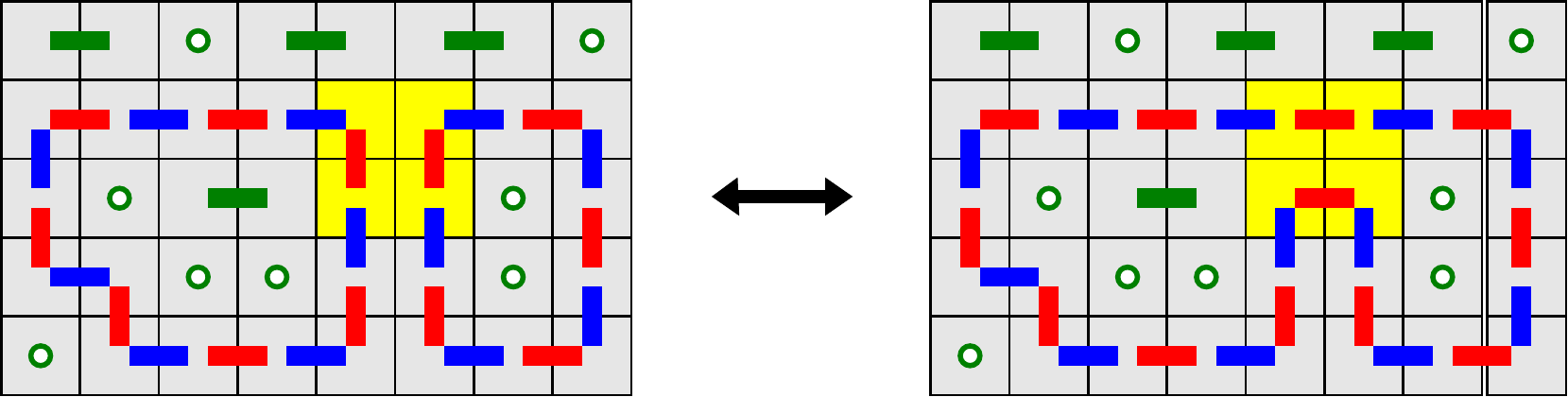}\label{fig:flipcase2_1}} \\
\subfloat[Case \ref{enum:differentorient}: $P_t(q) = q - 1$ in both tilings. Notice that the two nested cycles must have opposite orientation.]{\includegraphics[scale=0.4]{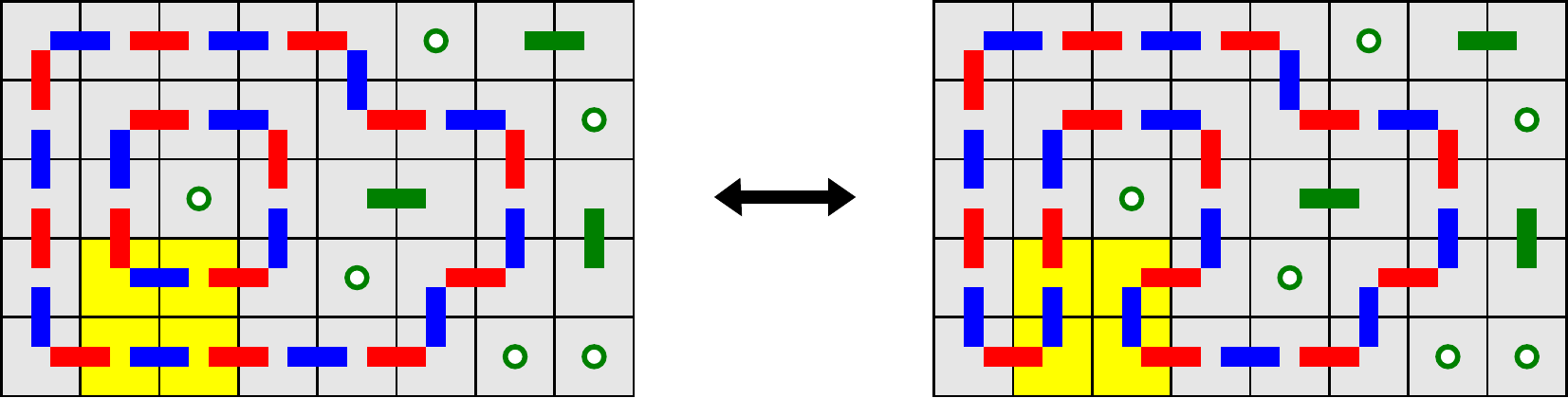}\label{fig:flipcase2_2}}
\caption{Examples illustrating the effects of flips in each of the cases. The flip positions are highlighted in yellow.}
\label{fig:flipCasesTwoFloors}
\end{figure}

\begin{case} 
\label{case:flipzdimer}
A flip that takes two non-$\ez$ dominoes which are in the same position in both floors to two adjacent $\ez$ dominoes (or the reverse of this flip)
\end{case}
The non-$\ez$ dominoes have no contribution to $P_{t_0}$. On the other hand, the two $\ez$ dominoes that appear after the flip have opposite colors (since they are adjacent), and are enclosed by exactly the same cycles. Hence, their contributions to $P_{t_1}$ cancel out, and thus $P_{t_0} = P_{t_1}$ in this case.

\begin{case}
\label{case:fliponefloor}
A flip that is completely contained in one of the floors.
\end{case}
If we look at the effect of such a flip in the associated drawing, two things can happen:
\begin{enumerate}[label=2.\arabic*., ref=2.\arabic*]
	\item\label{enum:sameorient} It connects two cycles that are not enclosed in one another and have the same orientation, and creates one larger cycle with the same orientation as the original ones, or it is the reverse of such a move;
	\item\label{enum:differentorient} It connects two cycles of opposite orientation such that one cycle is enclosed by the other (or it is the reverse of such a flip). The new cycle has the same orientation as the outer cycle.
	%\item It is the reverse flip of either (\ref{enum:sameorient}) or (\ref{enum:differentorient}).
\end{enumerate} 

In case \ref{enum:sameorient}, a jewel is enclosed by the new larger cycle if and only if it is enclosed by exactly one of the two original cycles. Hence, its contribution is the same in both $P_{t_0}$ and $P_{t_1}$.

In case \ref{enum:differentorient}, a jewel is enclosed by the new cycle if and only if it is enclosed by the outer cycle and not enclosed by the inner one. If it is enclosed by the new cycle, its contribution is the same in $P_{t_0}$ and $P_{t_1}$, because the new cycle has the same orientation as the outer one. On the other hand, if a jewel $j$ is enclosed by both cycles, their contributions to $k_t(j)$ cancel out, hence the jewel's contribution is also the same in $P_{t_0}$ and $P_{t_1}$. 

Hence, $P_{t_0} = P_{t_1}$ whenever $t_0$ and $t_1$ differ by a single flip.
\end{proof}
Notice that in cases \ref{enum:sameorient} and \ref{enum:differentorient} in the proof, one or both of the cycles involved may be trivial cycles. However, this does not change the analysis.

%We now describe the invariant in general.
\section{General two-story regions}
\label{sec:generalTwoStory}

Let $\cR$ be a two-story region, and let $\cR'$ be the smallest duplex region containing $\cR$. The cubes in $\cR' \setminus \cR$ are called \emph{holes}. If we project the dominoes of a tiling $t$ into the plane $z = 0$, we end up with a set of disjoint (simple) paths, some of which may be cycles, while others have loose ends: an example is showed in the left of Figure \ref{fig:unequalFloors} (in that case, all nontrivial paths have loose ends).

%Figure \ref{fig:unequalFloors} shows an example of a region with two unequal simply connected floors. If we try to look at the associated drawing, in the same sense as in Section \ref{sec:twoIdenticalFloors}, we will end up with a set of disjoint (simple) paths. Some of them are cycles, while others have loose ends: an example is showed in the left of Figure \ref{fig:unequalFloors} (in that case, all paths have loose ends).

\begin{figure}[ht]%
\centering
\includegraphics[width=0.5\columnwidth]{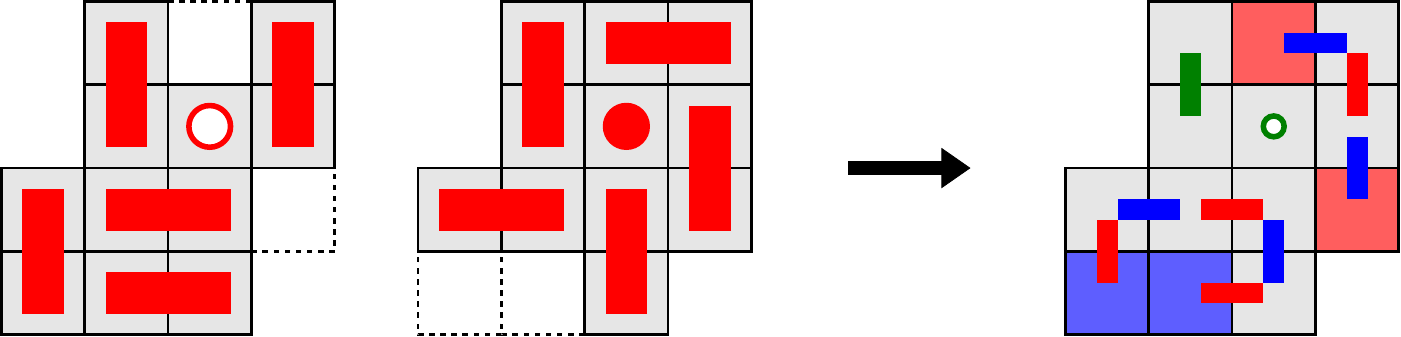}%
\caption{An example of tiled region with two simply connected but unequal floors. Holes in the leftmost floor are painted red in the associated drawing, while holes in the rightmost floor are painted blue.}%
\label{fig:unequalFloors}%
\end{figure}

%In Figure \ref{fig:unequalFloors}, there are four highlighted squares, which represent cubes in the symmetric difference of the floors. These highlighted squares are called \emph{holes}. The holes highlighted in blue in Figure \ref{fig:unequalFloors} stem from cubes $C(x^\sharp, y^\sharp, 0^\sharp)$ in the top floor such that $C(x^\sharp, y^\sharp, 1^\sharp) \notin R$, whereas the ones highlighted in red stem from cubes $C(x^\sharp, y^\sharp, 1^\sharp)$ in the bottom floor such that $C(x^\sharp, y^\sharp, 0^\sharp) \notin R$. 
Notice also that every white (resp. black) hole (regardless of which floor it is in) creates a loose end in the associated drawing where a domino is oriented in such a way that it is entering (resp. leaving) the square, which we call a \emph{source} (resp \emph{sink}): the names source and sink do not refer to the dominoes, but instead refer to the ghost curves, which are defined below.  Also, sources and sinks do not depend on the specific tiling and, since the number of black holes must equal the number of white holes, the number of sources always equals the number of sinks in a tileable region.

A \emph{ghost curve}\label{def:ghostCurve} in the associated drawing of $\cR$ is a curve that connects a source to a sink and which never touches the closure of a square that is common to both floors. Since the floors are simply connected, we can always connect any source to any sink via a ghost curve: an associated drawing of a tiling $t$ is then the ``usual'' associated drawing (from Section \ref{sec:twoIdenticalFloors}) together with a set of ghost curves such that each source and each sink is in exactly one ghost curve: this is shown in Figure \ref{fig:unequalFloorsJoined}.   
%If we connect pairs of loose ends in such a way that a source is always connected to a sink by a curve oriented from the source to the sink, it follows that we now actually see a set of (not necessarily disjoint) cycles, as shown in Figure \ref{fig:unequalFloorsJoined}. Notice that, since the floors are simply connected, we can always connect any source to any sink in the associated drawing via a path that never touches a closed square that is common to both floors, i.e., which may only cross holes (notice that a hole may never contain a jewel). Hence, a \textit{ghost curve} is defined as a curve that connects a source to a sink and which never touches the closure of a square that is common to both floors. 

%Hence, for a tiling of a two-story region with unequal floors, the associated drawing is the ``usual'' associated drawing (from Section \ref{sec:twoIdenticalFloors}) together with a set of ghost curves such that each source and each sink is in exactly one ghost curve.   

\begin{figure}[ht]%
\centering
\includegraphics[width=0.7\columnwidth]{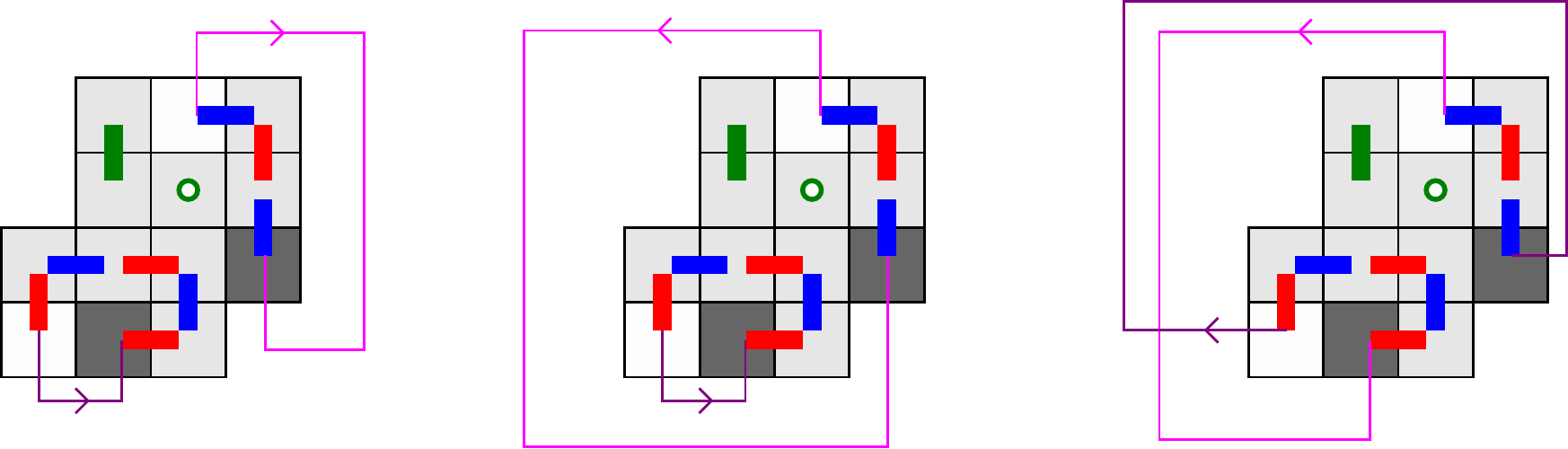}%
\caption{Three different ways to join sources and sinks. The sources are highlighted in very light grey (almost white), while the sinks are highlighted in dark grey. The invariants $P_t(q)$ for each case, from left to right, are $1$, $q$ and $1$.}%
\label{fig:unequalFloorsJoined}%
\end{figure}

%Therefore, the invariant in the general case is as follows. %Given a region $\cR$ with sinks and sources, connect the sinks and the sources via ghost curves that never cross a square that is common to both floors. 
%For a two-story region $\cR$, fix a set of ghost curves such that every source and every sink is in exactly one ghost curve. Therefore, since for each tiling $t$ the associated drawing is a set of cycles and jewels, we may define $P_t(q)$ in the same way as we did for duplex regions: if for a jewel $j$ we let $k_t(j)$ denote the sum of the winding numbers of the newly formed cycles with respect to that jewel, we set $P_t(q) = \sum_{j}\ccol(j)q^{k_t(j)}$.\label{def:ptTwostory} The only difference with respect to Section \ref{sec:twoIdenticalFloors} is that the winding number of a cycle with respect to a jewel is no longer necessarily $1$ or $-1$, but can be any integer (see Figure \ref{fig:unequalTwoFloor2}).
For a two-story region $\cR$, fix a set of ghost curves such that every source and every sink intersects exactly one ghost curve. Therefore, the associated drawing of each tiling $t$ is a set of cycles and jewels, so that we may define $P_t(q)$ in the same way as we did for duplex regions. Namely, for a jewel $j$, let $k_t(j)$ denote the sum of the winding numbers of the newly formed cycles with respect to that jewel, and set $P_t(q) = \sum_{j}\ccol(j)q^{k_t(j)}$\label{def:ptTwoStory}. One important difference with respect to Section \ref{sec:twoIdenticalFloors} is that the winding number of a cycle with respect to a jewel is no longer necessarily $1$ or $-1$, but can be any integer (see Figure \ref{fig:unequalTwoFloor2}).

\begin{figure}[ht]%
\centering
\includegraphics[width=0.9\columnwidth]{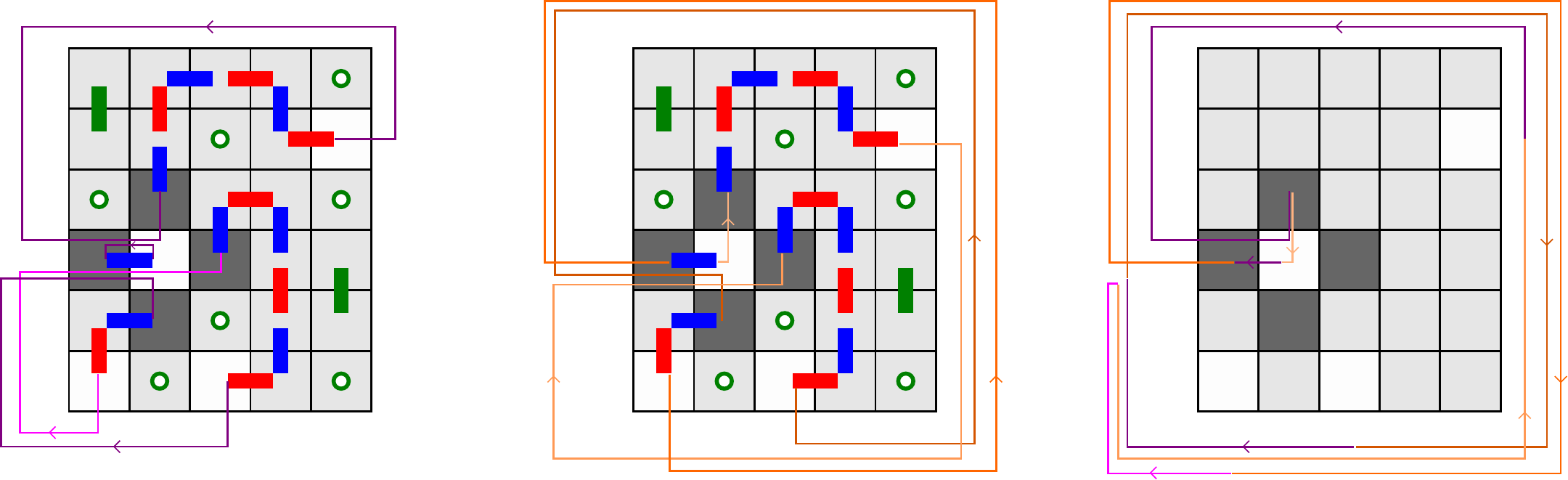}%
\caption{Two different ways to connect sources and sinks. The last diagram illustrates that the difference (first minus second) between these two connections forms a set of cycles, each of which has the same winding number with respect to each square where it is possible to have a jewel. In this case, this set has two cycles: one with winding number $0$ and the other with winding number $-1$ (with respect to every square that is common to both floors). Notice that the invariant is $-2q + 1 - 2q^{-1}$ in the first diagram and $-2q^2 + q - 2$ in the second, that is, the first invariant is indeed $q^{-1}$ times the second.}%
\label{fig:unequalTwoFloor2}%
\end{figure}

Strictly speaking, we should have made an explicit reference to the choice of ghost curves instead of simply writing $P_t$. However, it turns out that the following holds: if $P_{t,1}$ is the invariant associated with one choice of connection and $P_{t,2}$ is associated with another, then there exists $k \in \ZZ$ such that for every tiling $t$, $P_{t,1}(q) = q^k P_{t,2}(q)$. 

%To see this, fix a tiling $t$. We want to look at the contributions of a jewel to $P_t$ for two given choices of source-sink connections. Since the exponent of the contribution of a jewel is the sum of the winding numbers of all the cycles with respect to it, it follows that the difference of exponents between two choices of connections is the sum of winding numbers of the cycles formed by putting the ghost curves from both source-sink connections together in the same picture, as illustrated in Figure \ref{fig:unequalTwoFloor2}. Now this sum of winding numbers is the same for every jewel: since the set of ghost curves never touches the closure of a square common to both floors, each of the newly formed cycles must enclose every jewel in the same way. Hence, the effect of changing the connection is multiplying the contribution of each jewel by the same integer power of $q$, and we shall simply write $P_t$, assuming that the choice of ghost curves is given (and fixed).

To see this, fix a tiling $t$. We want to look at the contributions of a jewel to $P_t$ for two given choices of source-sink connections. Since the exponent of the contribution of a jewel is the sum of the winding numbers of all the cycles with respect to it, it follows that the difference of exponents between two choices of connections is the sum of winding numbers of the cycles formed by putting the ghost curves from both source-sink connections together in the same picture, as illustrated in Figure \ref{fig:unequalTwoFloor2}. Now this sum of winding numbers is the same for every jewel: since the set of ghost curves never touches the closure of a square common to both floors (and the intersection of the two floors is connected and simply connected), each of the newly formed cycles must enclose every jewel in the same way. 
Hence, the effect of changing the connection is multiplying the contribution of each jewel by the same integer power of $q$, and we shall simply write $P_t$, assuming that the choice of ghost curves is given (and fixed).

\begin{prop} \label{prop:twoStoryFlipInvariant} Let $\cR$ be a two-story region, and suppose $t_1$ is obtained from a tiling $t_0$ of $\cR$ after a single flip. Then $P_{t_0} = P_{t_1}$.
\end{prop}
\begin{proof}
The proof is basically the same as that of Proposition \ref{prop:twoFloorFlipInvariant}. In fact, for Case \ref{case:fliponefloor} (flips contained in one floor) in that proof, literally nothing changes, whereas Case \ref{case:flipzdimer} (flips involving jewels) can only happen in a pair of squares that are common to both floors. Since the ghost curves must never touch such squares, it follows that the pair of adjacent jewels that form a flip position have the same $k_t(j)$, hence their contributions cancel out.
\end{proof}

%Therefore, the invariant for the general case with two simply connected floors is uniquely defined up to multiplication by an integer power of $q$, $q^k$. As a consequence, the twist $\Tw(t) = P_t'(1)$ is defined up to the additive constant $k$.

\section{The effect of trits on $P_t$}     
\label{sec:effectOfTritsOnPt}

\begin{figure}[ht]
\centering
\includegraphics[width=0.4\columnwidth]{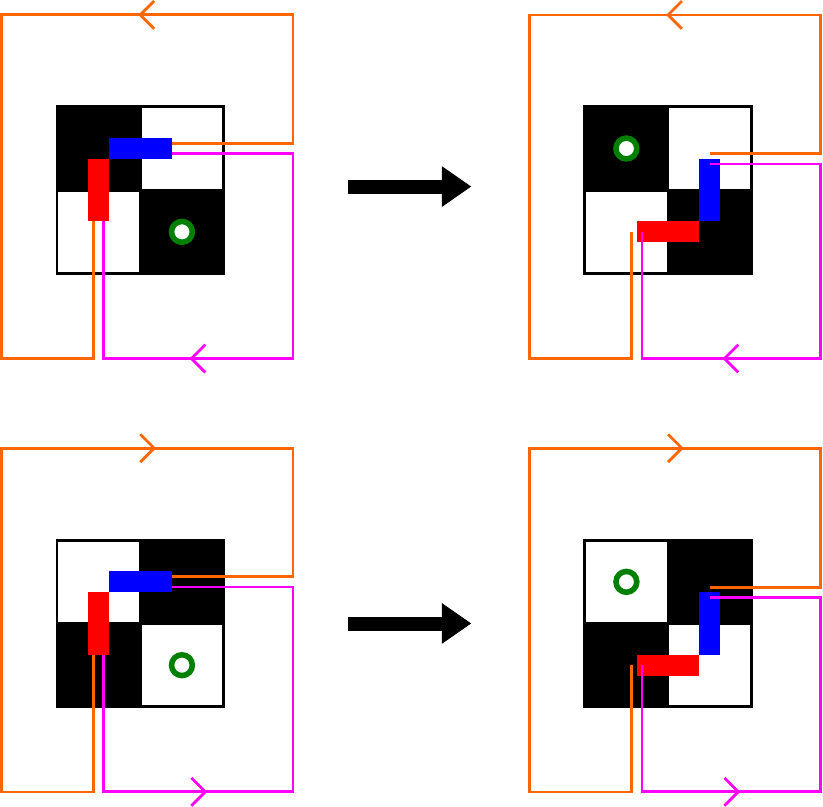}%
\caption{Schematic drawing of the effect of positive trits, with the magenta (shorter bottom right line) and orange lines (longer line) indicating (schematically) the two possible relative positions of the cycle $\gamma$ altered by the trit (the magenta and orange segments represent cycles that may have ghostly parts or not). It is clear that, in either case, the contribution of the portrayed jewel changes from $q^k$ to $q^{k+1}$ if it is black, and from $-q^k$ to $-q^{k-1}$ if it is white.}%
\label{fig:postrit_contrib}%
\end{figure}

We now turn our attention to the effect of a trit on $P_t$. By looking at Figure \ref{fig:postrit_contrib}, we readily observe that a trit affects only the contribution of the jewel $j$ that takes part in the trit (a trit always changes the position of precisely one jewel, since it always involves exactly one domino parallel to $\ez$). Furthermore, it either pulls a jewel out of a cycle, or pushes it into a cycle, but the jewel maintains its color. Hence, if $\pm q^k$ is the contribution of $j$ to $P_t$, then after a trit involving $j$ its contribution becomes $\pm q^{k + 1}$ or $\pm q^{k - 1}$.

A more careful analysis, however, as portrayed in Figure \ref{fig:postrit_contrib}, shows that a positive trit involving a black jewel always changes its contribution from $q^k$ to $q^{k+1}$, and a positive trit involving a white jewel changes $-q^k$ to $-q^{k-1}$. Hence, we have proven the following:

\begin{prop} \label{prop:posTrit} Let $t_0$ and $t_1$ be two tilings of a region $\cR$ which has two simply connected floors, and suppose $t_1$ is reached from $t_0$ after a single positive trit. Then, for some $k \in \ZZ$, 
\begin{equation}
P_{t_1}(q) - P_{t_0}(q) =  q^k(q - 1).
\label{eq:effectOnPt}
\end{equation}
\end{prop}
A closer look at Equation \eqref{eq:effectOnPt} shows that $\Tw(t_1) - \Tw(t_0) = P_{t_1}'(1) - P_{t_0}'(1) = 1$ whenever $t_1$ is reached from $t_0$ after a single positive trit. This gives the following easy corollary: 

\begin{coro} \label{coro:tritstwofloors} Let $t_0$ and $t_1$ be two tilings of a region $\cR$ with two simply connected floors, and suppose we can reach $t_1$ from $t_0$ after a sequence $S$ of flips and trits. Then $$\#(\text{positive trits in } S) - \#(\text{negative trits in } S) = \Tw(t_1) - \Tw(t_0).$$ 
\end{coro}
%It may not yet be clear why we are looking at $P_t'(1)$, especially because $P_t(q)$ gives a much finer piece of information about the flip connected components of $\cR$. However, it turns out that, for regions with more than two floors, the polynomial $P_t(q)$ will no longer be defined, whereas $P_t'(1)$ will. In fact, $P_t'(1)$ is the \textit{twist} of $t$, and can be generalized for regions with an arbitrary number of floors (see \cite{segundoartigo}).  
\section{Examples}
\label{sec:twofloorExamples}
For the examples below, we wrote programs in the C$^\sharp$ language.
\begin{example}[The $3 \times 3 \times 2$ box]
\label{example:332box}
\begin{table}[ht]
\centering
\begin{tabular}{|c|c|c|c|c|}
\hline
 \parbox[c]{2cm}{\centering Connected \\Component} & \parbox[c]{2cm}{\centering Number of \\tilings} & Tiling &$P_t(q)$&$\Tw(t)$\\ \hline
0 & 227 & \includegraphics[scale=0.3]{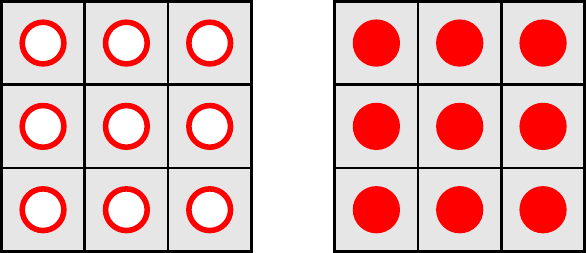} & $-1$ & $0$  \\ \hline
 1 & 1 & \includegraphics[scale=0.3]{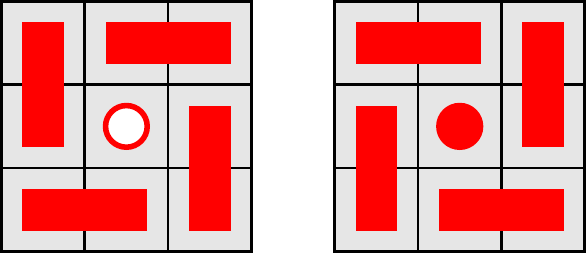} & $-q$ & $-1$  \\ \hline
 2 & 1 & \includegraphics[scale=0.3]{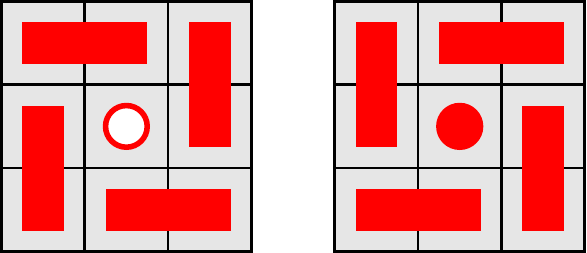} & $-q^{-1}$ & $1$  \\ \hline
\end{tabular}
\caption{Flip connected components of a $3 \times 3 \times 2$ box.}
\label{table:CC332}
\end{table}

The $3 \times 3 \times 2$ box is the smallest box whose space of domino tilings is not connected by flips (this can proved using techniques from Chapter \ref{chap:multiplex}), and it has $229$ tilings: the tiling $t$ shown in Figure \ref{fig:negtrit_example} has twist $-1$ and is alone in its flip connected component; the tiling obtained by reflecting $t$ on the plane $z = 1$ has twist $1$ (and is also alone in its flip connected component). The third component contains the remaining $227$ tilings, with twist $0$. This information is summarized in Table \ref{table:CC332}. Finally, the space of tilings is connected by flips and trits (via the trit shown in \ref{fig:negtrit_example} and its reflection).
% and Figure \ref{}.

\end{example}

%\begin{figure}[ht]%
%\centering
%\includegraphics[width=0.3\columnwidth]{filename}%
%\caption{}%
%\label{}%
%\end{figure} 

\begin{example}[The $7 \times 3 \times 2$ box] \label{example:732box} The $7 \times 3 \times 2$ box has a total of $880163$ tilings, and thirteen flip connected components. Table \ref{table:CC732} contains information about the invariants of these connected components. 

\begin{table}[h]
\begin{tabular}{|c|c|c|c|c|}
\hline
 \parbox[c]{2cm}{\centering Connected \\Component} & \parbox[c]{2cm}{\centering Number of \\tilings} & Tiling &$P_t(q)$&$\Tw(t)$\\ \hline
 0 & 856617 & \includegraphics[scale=0.3]{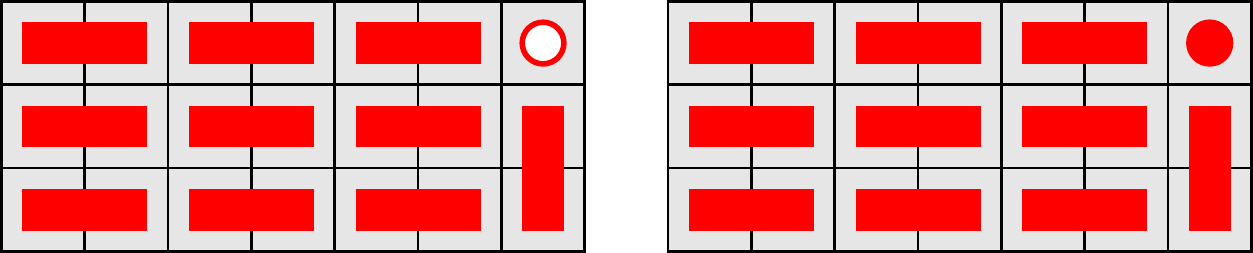} & $-1$ & $0$  \\ \hline
 1 & 8182 & \includegraphics[scale=0.3]{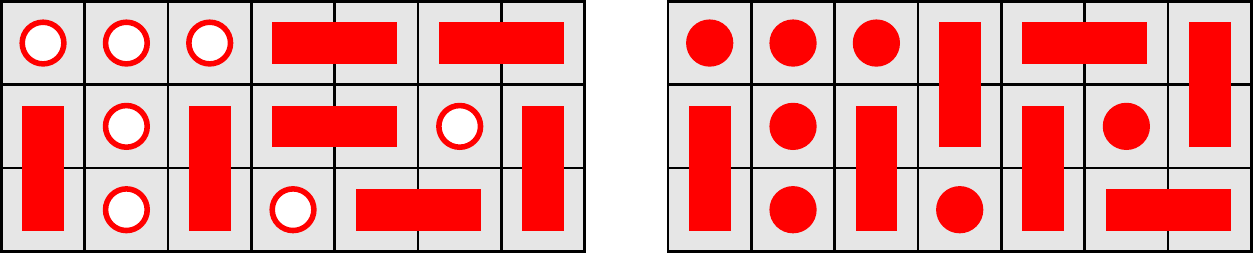} & $-q$ & $-1$  \\ \hline
 2 & 8182 & \includegraphics[scale=0.3]{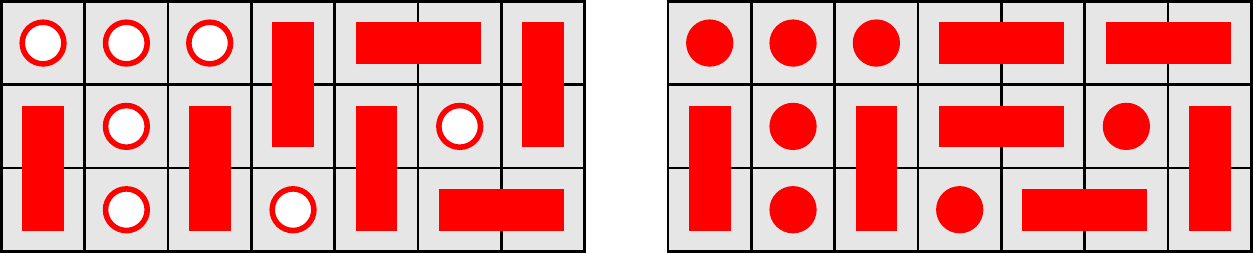} & $-q^{-1}$ & $1$  \\ \hline
 3 & 3565 & \includegraphics[scale=0.3]{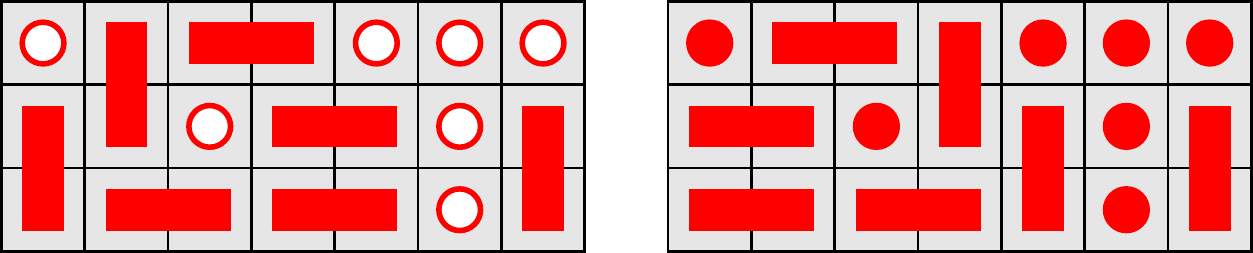} & $-2 + q^{-1}$ & $-1$  \\ \hline
 4 & 3565 & \includegraphics[scale=0.3]{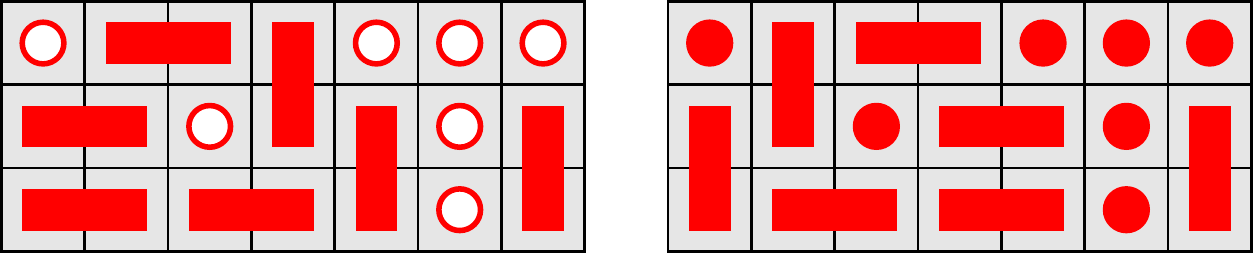} & $q - 2$ & $1$  \\ \hline
 5 & 9 & \includegraphics[scale=0.3]{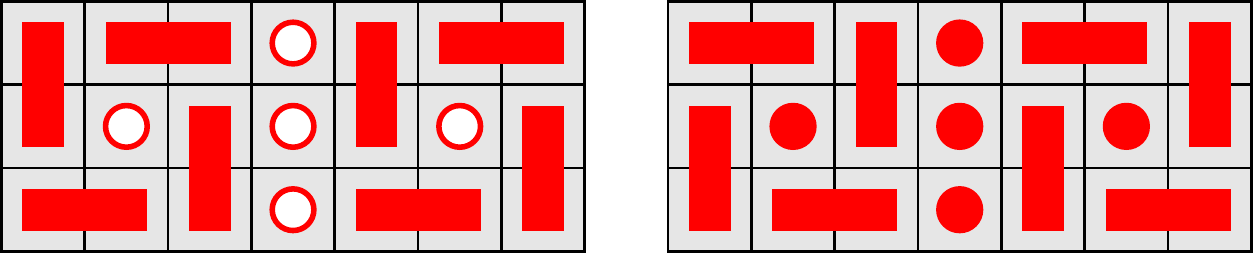} & $-2q + 1$ & $-2$  \\ \hline
 6 & 9 & \includegraphics[scale=0.3]{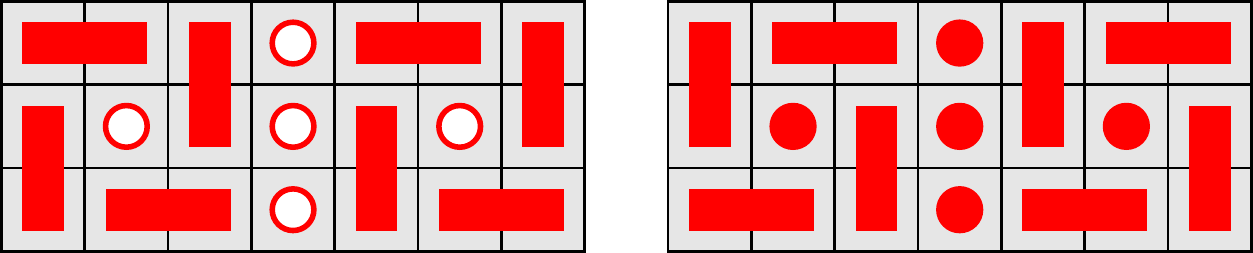} & $1 - 2q^{-1}$ & $2$   \\ \hline %This is CC 7 in the file
 7 & 7 & \includegraphics[scale=0.3]{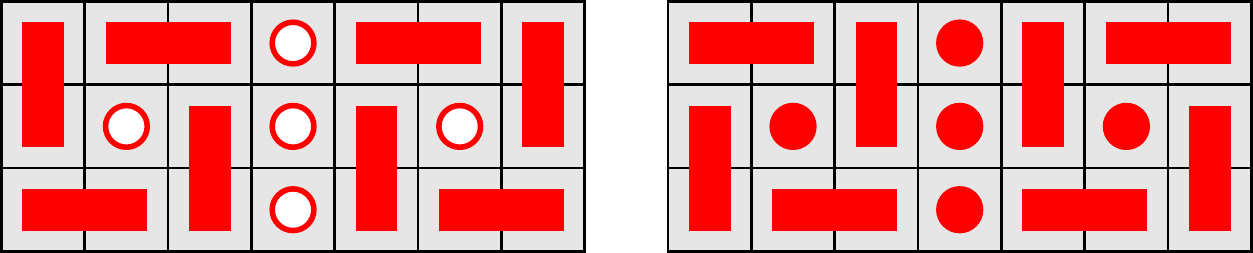} & $-q + 1 - q^{-1}$ & $0$  \\ \hline %This is CC 6 in the file
 8 & 7 & \includegraphics[scale=0.3]{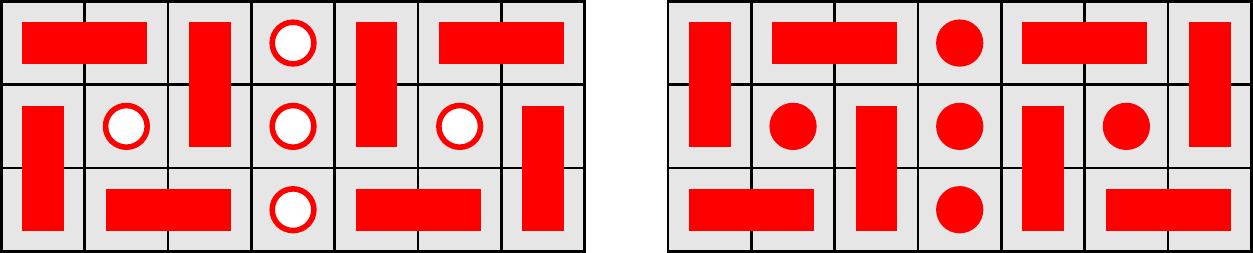} & $-q + 1 - q^{-1}$ & $0$  \\ \hline
 9 & 5 & \includegraphics[scale=0.3]{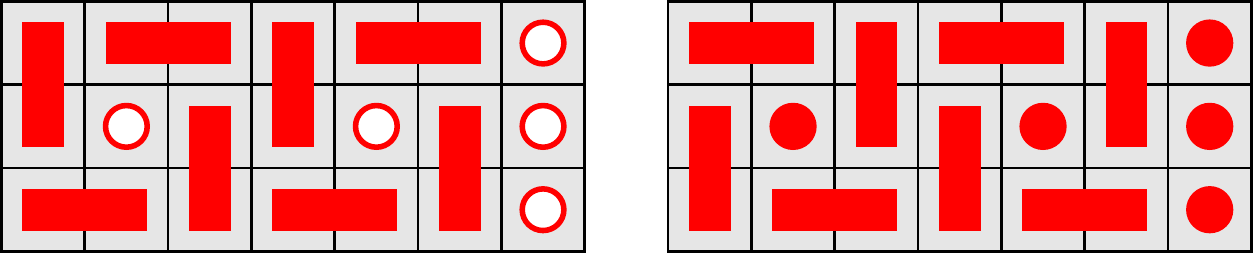} & $-q - 1 + q^{-1}$ & $-2$  \\ \hline
 10 & 5 & \includegraphics[scale=0.3]{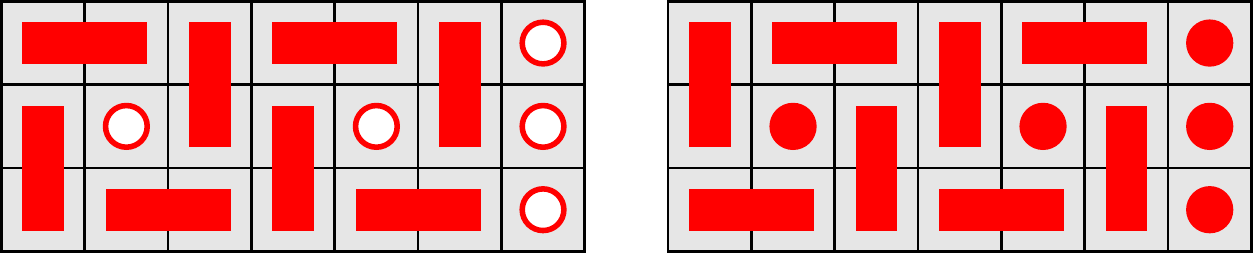} & $q - 1 - q^{-1}$ & $2$  \\ \hline
 11 & 5 & \includegraphics[scale=0.3]{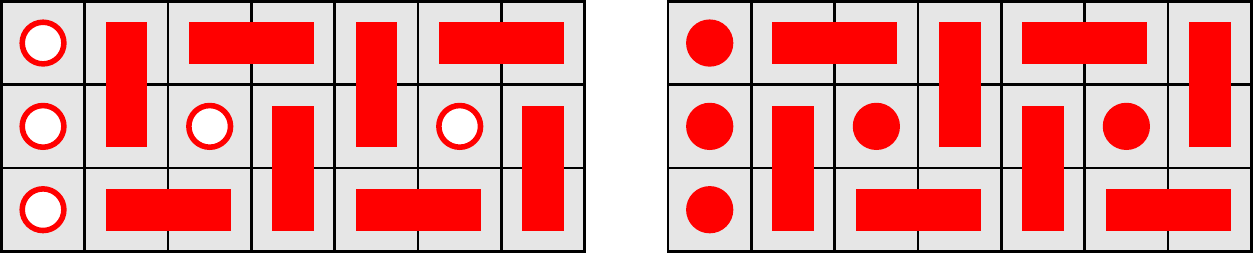} & $-q - 1 + q^{-1}$ & $-2$  \\ \hline
 12 & 5 & \includegraphics[scale=0.3]{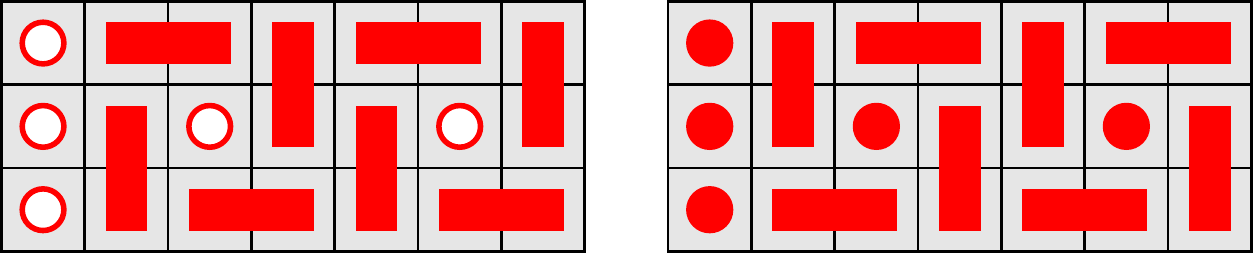} & $q - 1 - q^{-1}$ & $2$   \\ \hline

\end{tabular}
\caption{Flip connected components of a $7 \times 3 \times 2$ box.}
\label{table:CC732}
\end{table}

We readily notice that the invariant does a good job of separating flip connected components, albeit not a perfect one: the pairs of connected components 7 and 8, 9 and 11, and 10 and 12 have the same $P_t(q)$. %$P_t'(1)$ is, of course, not as refined as $P_t(q)$. %Nevertheless, it is important not only because of the statement in Proposition \ref{prop:tritstwofloors} but also because a similar invariant (although with a completely different formulation) works in a much more general setting (the twist).

Figure \ref{fig:box732_CCdiagram} shows a diagram of the flip connected components, arranged by their twists. We also notice that we can always reach a tiling from any other via a sequence of flips and trits.

\begin{figure}[ht]
\centering
\includegraphics[width=0.5\columnwidth]{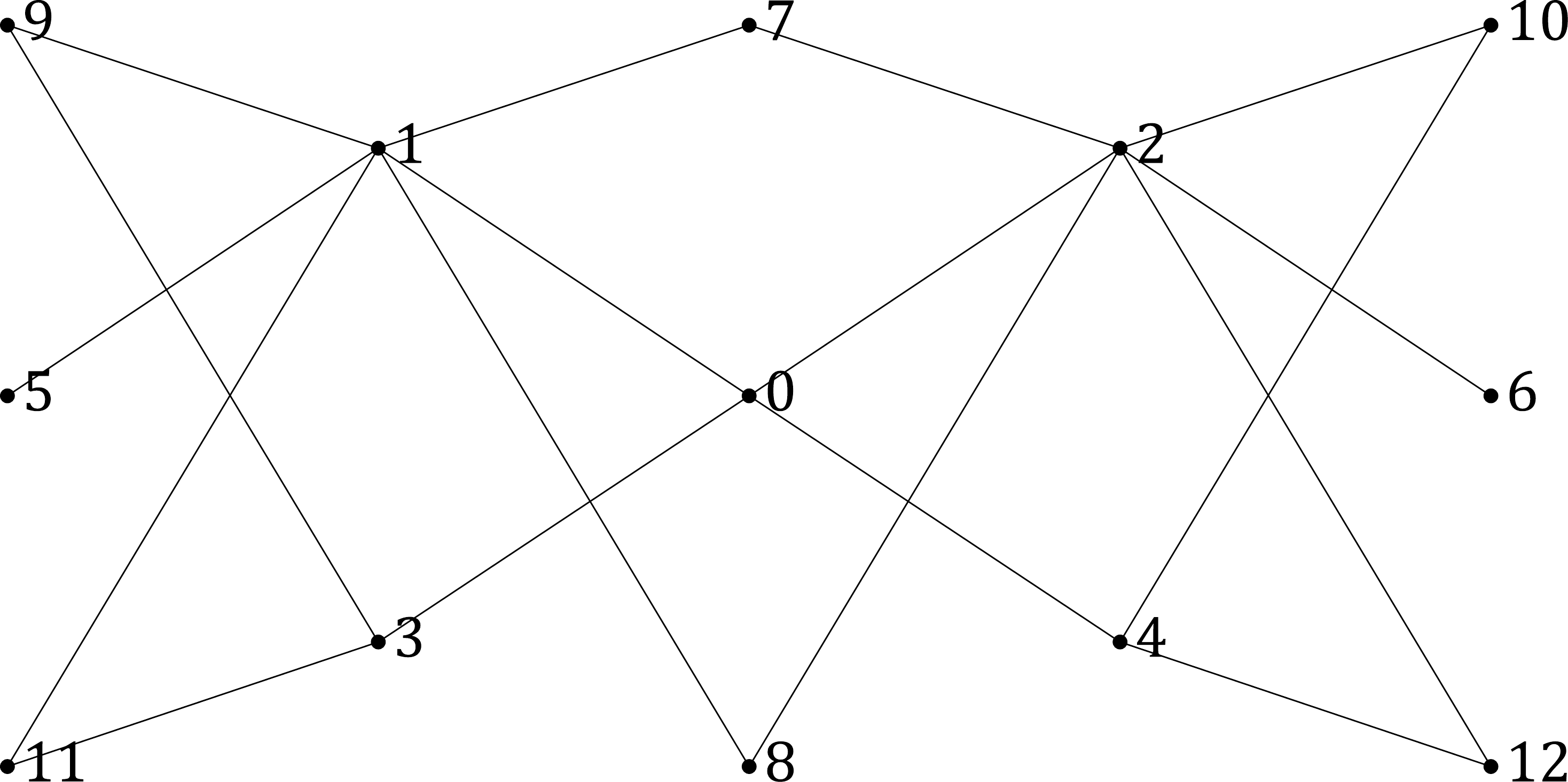}%
\caption{Flip connected components of the $7 \times 3 \times 2$ box, arranged by twist in increasing order. The numbering is the same as in table \ref{table:CC732}, and two dots $A$ and $B$ are connected if there exists a trit that takes a tiling in connected component $A$ to a tiling in connected component $B$. As we proved earlier, a trit from left to right in the diagram is always a positive trit; and a trit from right to left is always a negative one. }%
\label{fig:box732_CCdiagram}%
\end{figure}
\end{example}

\begin{example}[A region with two unequal floors] \label{example:unequalFloors}

\begin{figure}[ht]%
\centering
\includegraphics[width=0.6\columnwidth]{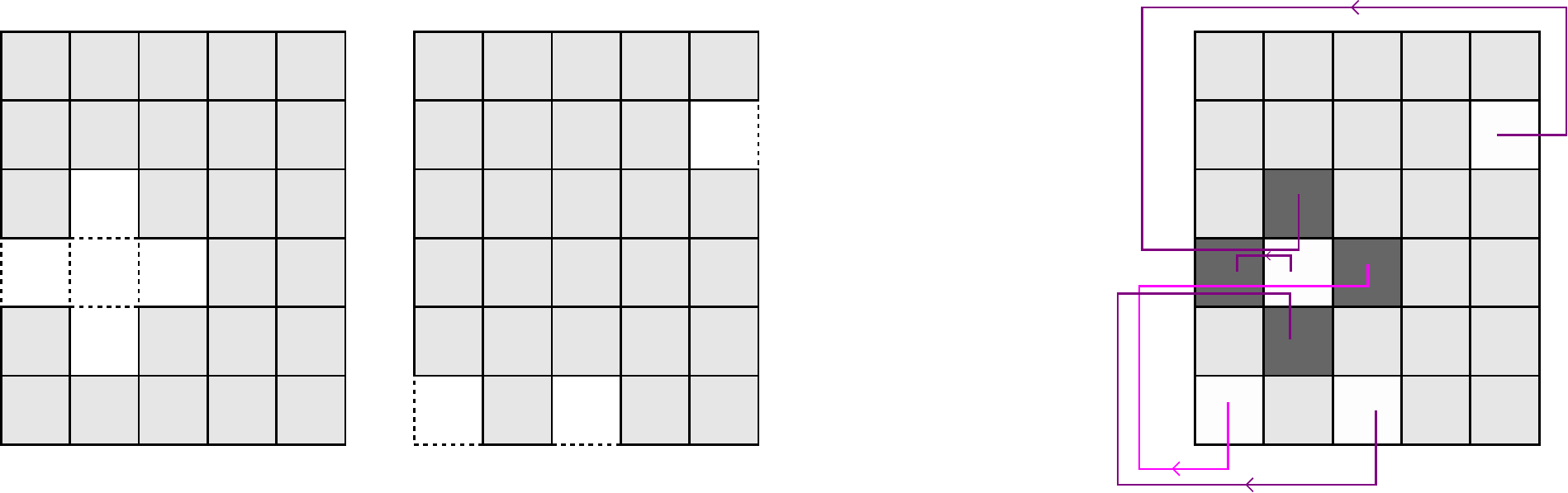}%
\caption{A region with two unequal floors, together with a choice of connections between the sources and the sinks. This choice of connections is the one used for the calculations in Table \ref{table:biggerRegion_CC}.}
\label{fig:bigger_region_sourcesAndSinks}%
\end{figure}

Figure \ref{fig:bigger_region_sourcesAndSinks} shows a region with two unequal floors, together with a choice of how to join sources and sinks. It has $642220$ tilings, and $30$ connected components. 

Table \ref{table:biggerRegion_CC} shows some information about these components and Figure \ref{fig:bigGraph_CC_unequal} shows a diagram of the flip connected components, arranged by their twists. This graph is not as symmetric as the one in the previous example; nevertheless, the space of tilings is also connected by flips and trits in this case.

\begin{table}[hpt]
\centering
\footnotesize
\begin{tabular}{|c|c|c|c|}
\hline
 \parbox[c]{2cm}{\centering Connected \\ Component} & \parbox[c]{2cm}{\centering Number of \\tilings} &$P_t(q)$ & $\Tw(t)$\\ \hline
 
0 & \num{165914} & $-2q -q^{-1}$ & $-1$ \\ \hline
1 & \num{153860} & $-q -1 -q^{-1}$ & $0$ \\ \hline
2 & \num{92123} & $-2q -1$ & $-2$ \\ \hline
3 & \num{56936} & $-q -1 -q^{-2}$ & $1$ \\ \hline
4 & \num{50681} & $-q -2$ & $-1$ \\ \hline
5 & \num{41236} & $-2q -q^{-2}$ & $0$ \\ \hline
6 & \num{17996} & $-2 -q^{-1}$ & $1$ \\ \hline
7 & \num{13448} & $-q -2q^{-1}$ & $1$ \\ \hline
8 & \num{11220} & $-3q$ & $-3$ \\ \hline
9 & \num{8786} & $-2 -q^{-2}$ & $2$ \\ \hline
10 & \num{7609} & $-q -q^{-1} -q^{-2}$ & $2$ \\ \hline
11 & \num{6423} & $-2q + 1 -2q^{-1}$ & $0$ \\ \hline
12 & \num{4560} & $-3q + 1 -q^{-1}$ & $-2$ \\ \hline
13 & \num{4070} & $-3$ & $0$ \\ \hline
14 & \num{3299} & $-2q + 1 -q^{-1} -q^{-2}$ & $1$ \\ \hline
15 & \num{2097} & $-1 -2q^{-1}$ & $2$ \\ \hline
16 & \num{1382} & $-1 -q^{-1} -q^{-2}$ & $3$ \\ \hline
17 & \num{221} & $-q + 1 -3q^{-1}$ & $2$ \\ \hline
18 & \num{137} & $-q + 1 -2q^{-1} -q^{-2}$ & $3$ \\ \hline
19 & \num{51} & $-3q^{-1}$ & $3$ \\ \hline
20 & \num{48} & $-2q -1$ & $-2$ \\ \hline
21 & \num{36} & $-2q^{-1} -q^{-2}$ & $4$ \\ \hline
22 & \num{17} & $-3q^{-1}$ & $3$ \\ \hline
23 & \num{17} & $-2q^{-1} -q^{-2}$ & $4$ \\ \hline
24 & \num{16} & $-1 -2q^{-1}$ & $2$ \\ \hline
25 & \num{16} & $-1 -q^{-1} -q^{-2}$ & $3$ \\ \hline
26 & \num{12} & $-2q + 2 -3q^{-1}$ & $1$ \\ \hline
27 & \num{7} & $-2q + 2 -2q^{-1} -q^{-2}$ & $2$ \\ \hline
28 & \num{1} & $1 -4q^{-1}$ & $4$ \\ \hline
29 & \num{1} & $1 -3q^{-1} -q^{-2}$ & $5$ \\ \hline

\end{tabular}
\caption{Information about the flip connected components of the region $\cR$ from Figure \ref{fig:bigger_region_sourcesAndSinks}.}
\label{table:biggerRegion_CC}
\end{table}

\begin{figure}[ht]%
\centering
\includegraphics[width=0.6\columnwidth]{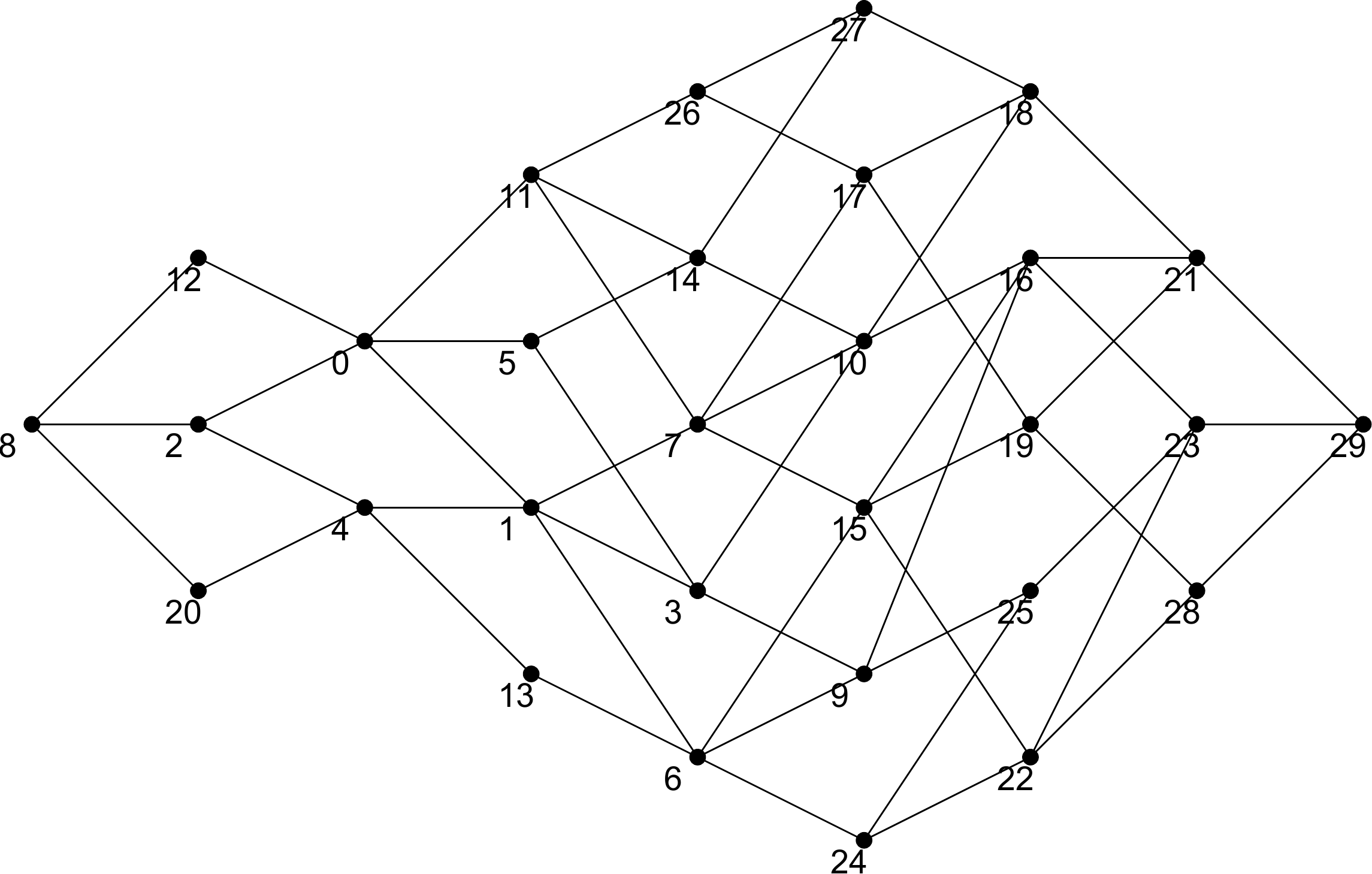}%
\caption{Graph with $30$ vertices, each one representing a connected component of the region. As in Figure \ref{fig:box732_CCdiagram}, two vertices are connected if there exists a trit taking a tiling in one component to a tiling in the other; a trit from left to right is always positive.}%
\label{fig:bigGraph_CC_unequal}%
\end{figure}

\end{example}

\section{The invariant in more space}
\label{sec:twoFloorsMoreSpace}
We already know that tilings that are not in the same flip connected component may have the same polynomial invariant, even in the case with two equal simply connected floors (see, for instance, Example \ref{example:732box}). However, as we will see in this section, this is rather a symptom of lack of space than anything else. 

More precisely, suppose $\cR$ is a duplex region and $t$ is a tiling of $\cR$. If $\cB$ is an $L \times M \times 2$ box (or a two-floored box) containing $\cR$, then $\cB \setminus \cR$ can be tiled in an obvious way (using only dominoes parallel to $\ez$). Thus $t$ induces a tiling $\hat{t}$ of $\cB$ that contains $t$; we call this tiling $\hat{t}$ the \emph{embedding} of $t$ in $\cB$.

\begin{figure}[ht]%
\centering
\includegraphics[width=0.5\columnwidth]{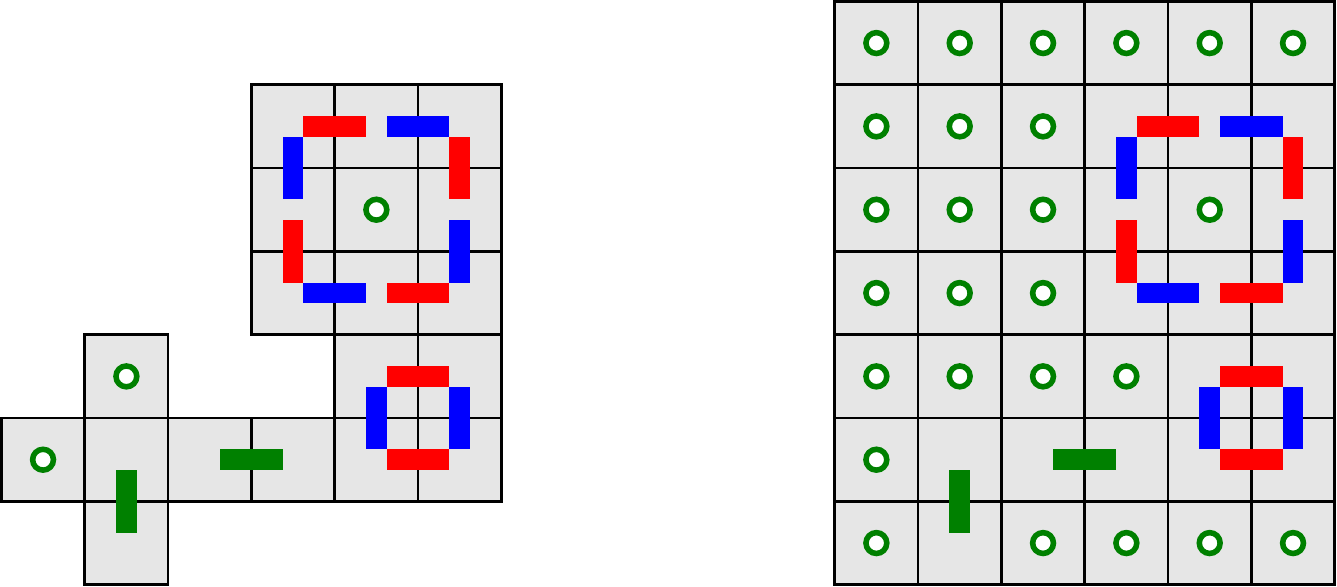}%
\caption{The associated drawing of a tiling, and its embedding in a $6 \times 7 \times 2$ box. }%
\label{fig:embeddingTwoFloors}%
\end{figure}

Another way to look at the embedding $\hat{t}$ of a tiling $t$ is the following: start with the associated drawing of $t$ (which is a plane region), add empty squares until you get an $L \times M$ rectangle, and place a jewel in every empty square. Since the newly added jewels are outside of any cycle in $\hat{t}$, it follows that $P_{\hat{t}}(q) - P_t(q) = k \in \ZZ$, where $k$ is the number of new black jewels minus the number of new white jewels (which depends only on the choice of box $\cB$ and not on the tiling $t$). Our goal for the section is to prove the following:

\begin{prop}
\label{prop:moreSpace}
Let $\cR$ be a duplex region, and let $t_0, t_1$ be two tilings that have the same invariant, i.e., $P_{t_0} = P_{t_1}$. Then there exists a two-floored box containing $\cR$ such that the embeddings $\hat{t_0}$ and $\hat{t_1}$ of $t_0$ and $t_1$ lie in the same flip connected component.  
\end{prop}

If $t_0$ and $t_1$ already lie in the same flip connected component in $\cR$, then their embeddings $\hat{t_0}$ and $\hat{t_1}$ in any two-floored box will also lie in the same connected component, because you can reach $\hat{t_1}$ from $\hat{t_0}$ using only flips already available in $\cR$.

Also, notice that $P_{\hat{t_0}} = P_{\hat{t_1}}$ if and only if $P_{t_0} = P_{t_1}$, because  $P_{\hat{t_1}}(q) - P_{t_1}(q) = P_{\hat{t_0}}(q) - P_{t_0}(q)$. Therefore, Proposition \ref{prop:moreSpace} states that two tilings have the same invariant if and only if there exists a two-floored box where their embeddings lie in the same flip connected component.

The reader might be wondering why we're restricting ourselves to duplex regions. One reason is that for general regions with two simply connected floors, it is not always clear that you can embed them in a large box in a way that their complement is tileable, let alone tileable in a natural way. 

Although it is technically possible to prove Proposition \ref{prop:moreSpace} only by looking at associated drawings, it will be useful to introduce an alternative formulation for the problem. 

Let $G = G(\cR)$ be the undirected plane graph whose vertices are the centers of the squares in the associated drawing of $\cR$, and where two vertices are joined by an edge if their Euclidean distance is exactly $1$. A \emph{system of cycles}, or \emph{sock}, in $G$ is a (finite) directed subgraph of $G$ consisting only of disjoint oriented (simple) cycles. An \emph{edge} of a sock is an (oriented) edge of one of the cycles, whereas a \emph{jewel} is a vertex of $G$ that is not contained in the system of cycles.

\begin{figure}%
\centering
\includegraphics[width=0.6\columnwidth]{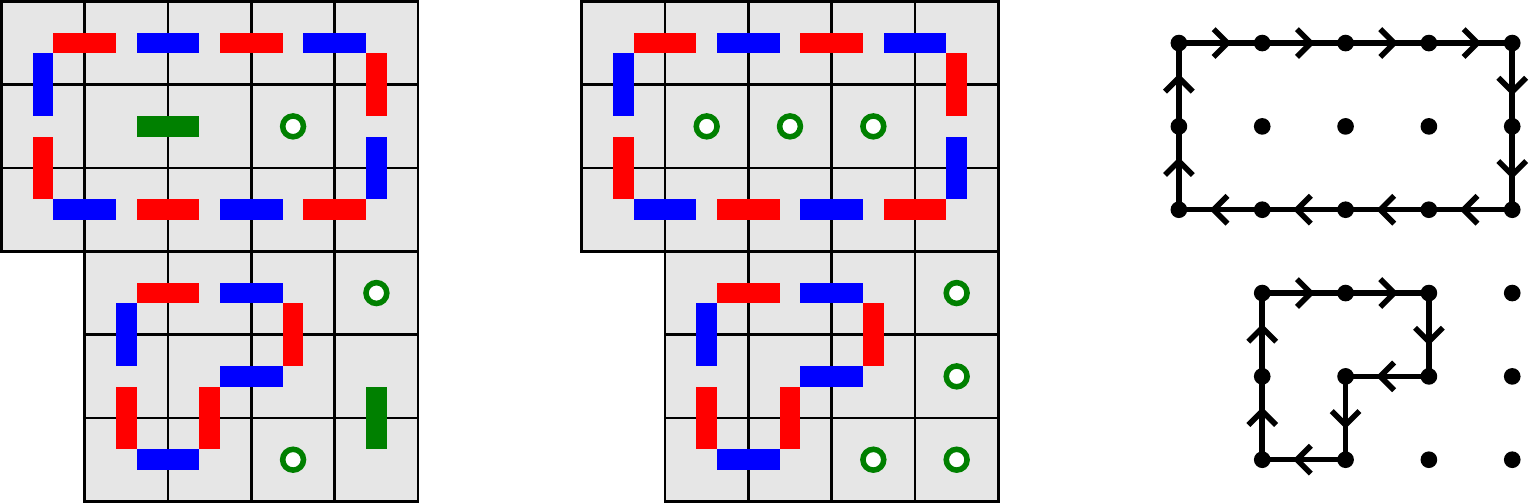}%
\caption{A tiling with two trivial cycles; the same tiling with the two trivial cycles flipped into jewels; and the system of cycles that corresponds to both of them.}%
\label{fig:tilingTwoFloors_withSOC}%
\end{figure}
There is an ``almost'' one-to-one correspondence between the systems of cycles of $G$ and the tilings of $\cR$, which is illustrated in Figure \ref{fig:tilingTwoFloors_withSOC}. In fact, tilings with trivial cycles have no direct interpretation as a system of cycles; but since all trivial cycles can be flipped into a pair of adjacent jewels, we can think that every sock represents a set of tilings, all in the same flip connected component.

We would now like to capture the notion of a flip from the world of tilings to the world of socks. This turns out to be rather simple: a \emph{flip move} on a sock is one of three types of moves that take one sock into another, shown in Figure \ref{fig:flipsteps}. Notice that performing a flip move on a sock corresponds to performing one or more flips on its corresponding tiling. 

\begin{figure}[ht]%
\centering
\def\svgwidth{0.3\columnwidth}
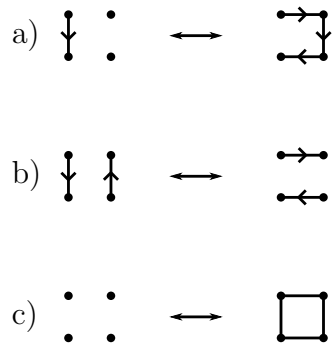
\caption{The three types of flip moves. The square in type c) is not oriented because it can have either one of the two possible orientations.}%
\label{fig:flipsteps}%
\end{figure}

A \emph{flip homotopy} in $G$ between two socks $s_1$ and $s_2$ is a finite sequence of flip moves taking $s_1$ into $s_2$. If there exists a flip homotopy between two socks, they are said to be \emph{flip homotopic} in $G$. Notice that two tilings are in the same flip connected component of $\cR$ if and only if their corresponding socks are flip homotopic in $G$, because every flip can be represented as one of the flip moves (and the flip that takes a trivial cycle into two jewels does not alter the corresponding sock). Figure \ref{fig:flipStepsExample} shows examples of flips and their corresponding flip moves.

\begin{figure}[ht]%
\centering
\includegraphics[width=0.7\columnwidth]{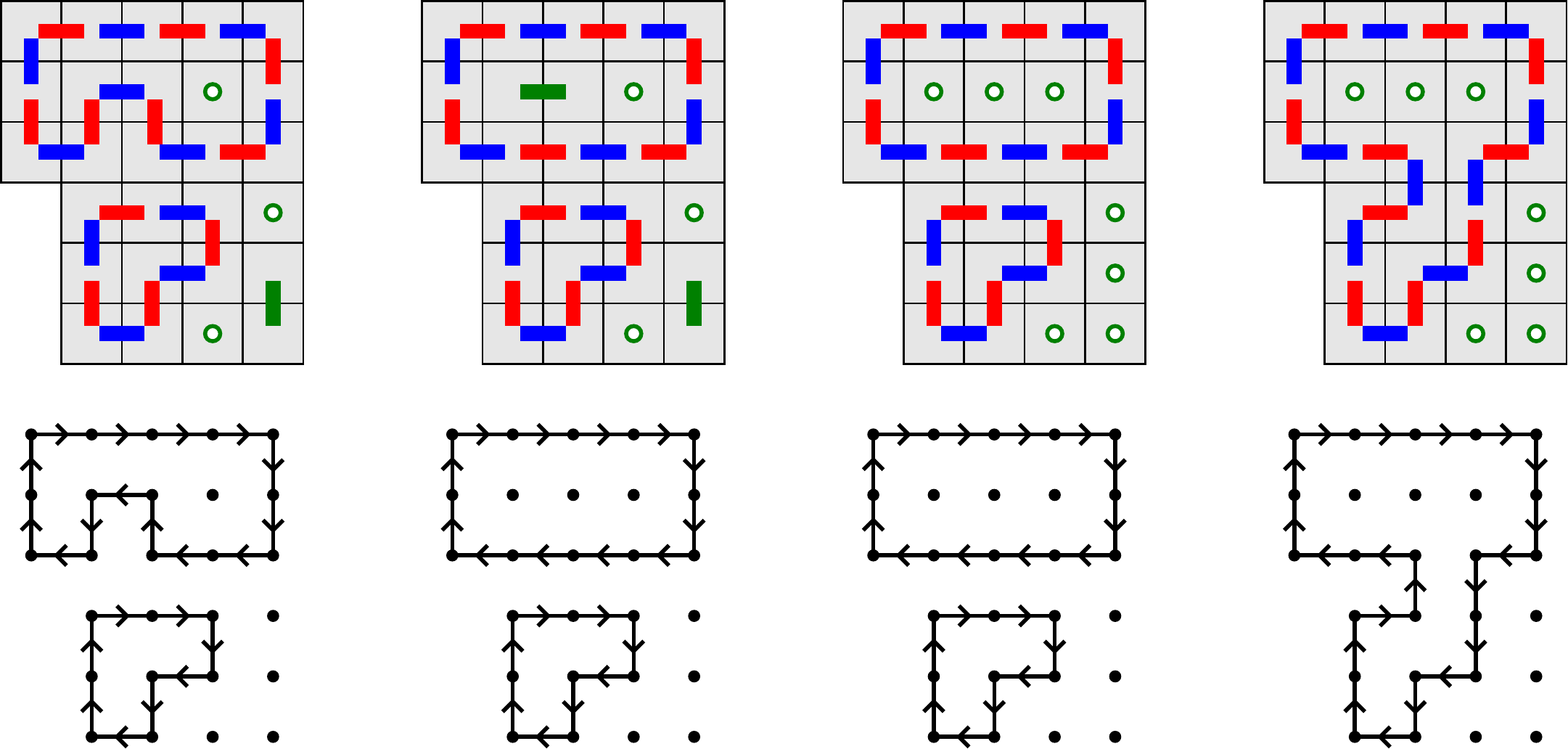}%
\caption{Examples of how flips affect the corresponding sock. The first flip induces a flip move of type (a), The second one does not alter the corresponding sock, and the third one induces a flip move of type (b). }%
\label{fig:flipStepsExample}%
\end{figure}

One advantage of this new interpretation is that we can easily add as much space as we need without an explicit reference to a box. In fact, notice that $G$ is a subgraph of the infinite graph $\ZZ^2$, so that a sock in $G$ is also a finite subgraph of $\ZZ^2$, so that we may see it as a system of cycles in $\ZZ^2$.

\begin{lemma}
\label{lemma:equivEmbeddingHomotopy}
For two tilings $t_0$ and $t_1$ of a region $\cR$, the following assertions are equivalent:
\begin{enumerate}[label=(\roman*)]
	\item \label{item:embeddedBox} There exists a two-floored box $\cB$ containing $\cR$ such that the embeddings of $t_0$ and $t_1$ in $\cB$ lie in the same flip connected component.
	\item \label{item:flipHomotopicSOC} The corresponding systems of cycles of $t_0$ and $t_1$ are flip homotopic in $\ZZ^2$.
\end{enumerate}  
\end{lemma}
\begin{proof}
To see that \ref{item:flipHomotopicSOC} implies \ref{item:embeddedBox}, notice the following:
if $s_0, s_1, \ldots, s_n$ are the socks involved in the flip homotopy between the socks of $t_0$ and $t_1$ in $\ZZ^2$, let $\cB$ be a sufficiently large two-floored box such that $(\ZZ^2 \setminus G(\cB))$ contains only vertices of $\ZZ^2$ that are jewels in all the $n+1$ socks $s_0, s_1, \ldots, s_n$. Then the socks of $t_0$ and $t_1$ are flip homotopic in $G(\cB)$, so the embeddings of $t_0$ and $t_1$ lie in the same flip connected component.

The converse is obvious, since if the socks of $t_0$ and $t_1$ are flip homotopic in $G(\cB)$ for some two-floored box $\cB$, then they are flip homotopic in $\ZZ^2$.
\end{proof} 

A vertex $v \in \ZZ^2$ is said to be white (resp. black) if the sum of its coordinates is even (resp. odd), and we write $\ccol(v) = -1$ (resp. $1$). If $s$ is a system of cycles in $\ZZ^2$, we define the graph invariant of $s$ as $$P_s(q) = \sum_{j : k_s(j) \neq 0}\ccol(j)q^{k_s(j)},$$
where $k_s(j)$ is the sum of the winding numbers of all the cycles in $s$ (as curves) with respect to $j$; this is a finite sum because the number of jewels enclosed by cycles of $s$ is finite. Notice that if $t$ is a tiling of $\cR$ and $s$ is its corresponding sock in $\ZZ^2$, $P_s(q) = P_t(q) - P_t(0)$, so that $P_s$ is completely determined by $P_t$. Conversely,
$P_t(q) - P_s(q) = P_t(1) - P_s(1) = P_t(0)$: %, which equals the $q^0$ term in $P_t$. 
since $P_t(1)$ equals the number of black squares minus the number of white squares in $\cR$ (thus does not depend on $t$), it follows that $P_t(q)$ is also completely determined by $\cR$ and $P_s(q)$.

A corollary of Proposition \ref{prop:twoFloorFlipInvariant} is that if two systems of cycles $s_0$ and $s_1$ are flip homotopic in $\ZZ^2$, then $P_{s_0} = P_{s_1}$. We now set out to prove that the converse also holds, which will establish Proposition \ref{prop:moreSpace}.

A \emph{boxed jewel} is a subgraph of $\ZZ^2$ formed by a single jewel enclosed by a number of square cycles (cycles that are squares when thought of as plane curves), all with the same orientation. Figure \ref{fig:boxedJewelsExample} shows examples of boxed jewels. Working with boxed jewels is easier, for if they have ``free space'' in one direction (for instance, if there are no cycles to the right of it), they can move an arbitrary even distance in that direction; the simplest case is illustrated in Figure \ref{fig:boxedJewelsMove}. More complicated boxed jewels move just as easily: we first turn the outer squares into rectangles, then we move the inner boxed jewel, and finally we close the outer squares again.

\begin{figure}[ht]%
\centering
\subfloat[]{\includegraphics[width=0.06\columnwidth]{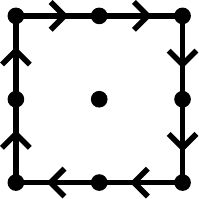}}\qquad
\subfloat[]{\includegraphics[width=0.24\columnwidth]{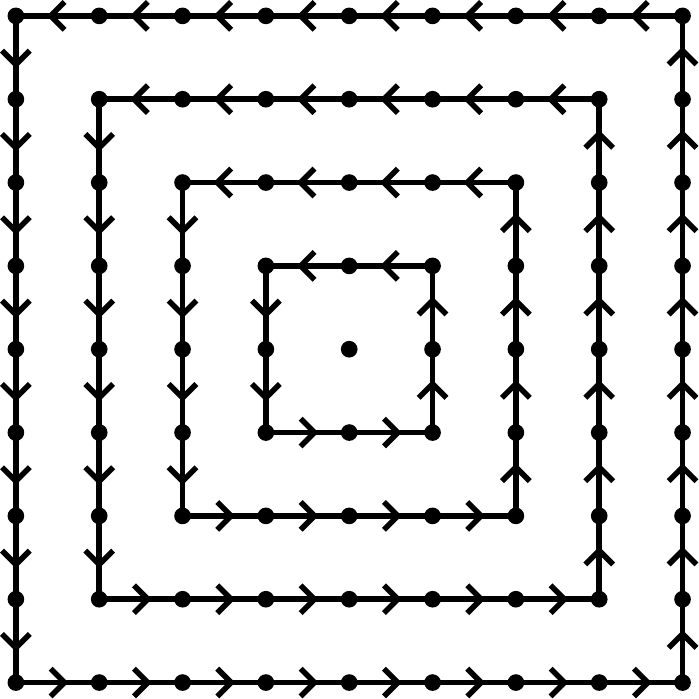}}
\caption{Two examples of boxed jewels.}%
\label{fig:boxedJewelsExample}%
\end{figure} 

\begin{figure}%
\centering
\includegraphics[width=0.7\columnwidth]{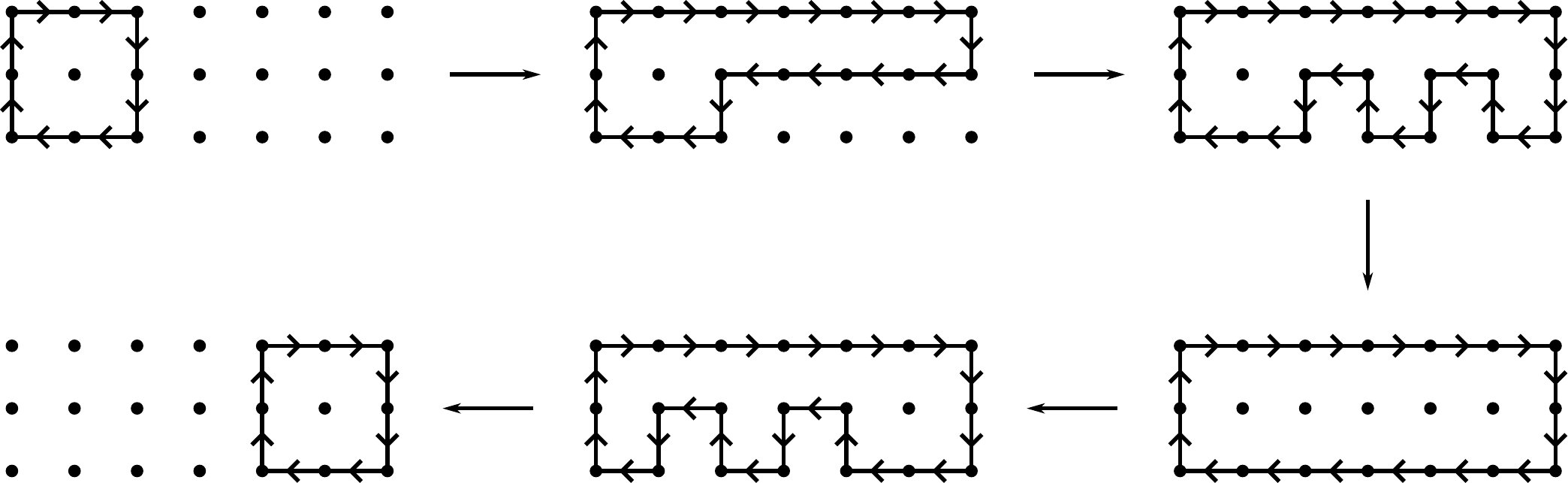}%
\caption{A boxed jewel with ``free space'' to the right. Starting from the first sock, we perform: four flip moves of type (a); two flip moves of type (a); two flip moves of type (a) to create a rectangle; two flip moves of type (a); and finally, the last sock is obtained by performing six flip moves of type (a) on the penultimate sock.}%
\label{fig:boxedJewelsMove}%
\end{figure}

An \emph{untangled} sock is a sock that contains only boxed jewels, and such that the center of each boxed jewel is of the form $(n, 0)$ for some $n \in \ZZ$ (that is, all the enclosed jewels lie in the $x$ axis), as illustrated in Figure \ref{fig:untangledSockExample}. Therefore, each boxed jewel in an untangled sock moves very easily: it has free space both downwards and upwards.

\begin{figure}%
\centering
\includegraphics[width=0.6\columnwidth]{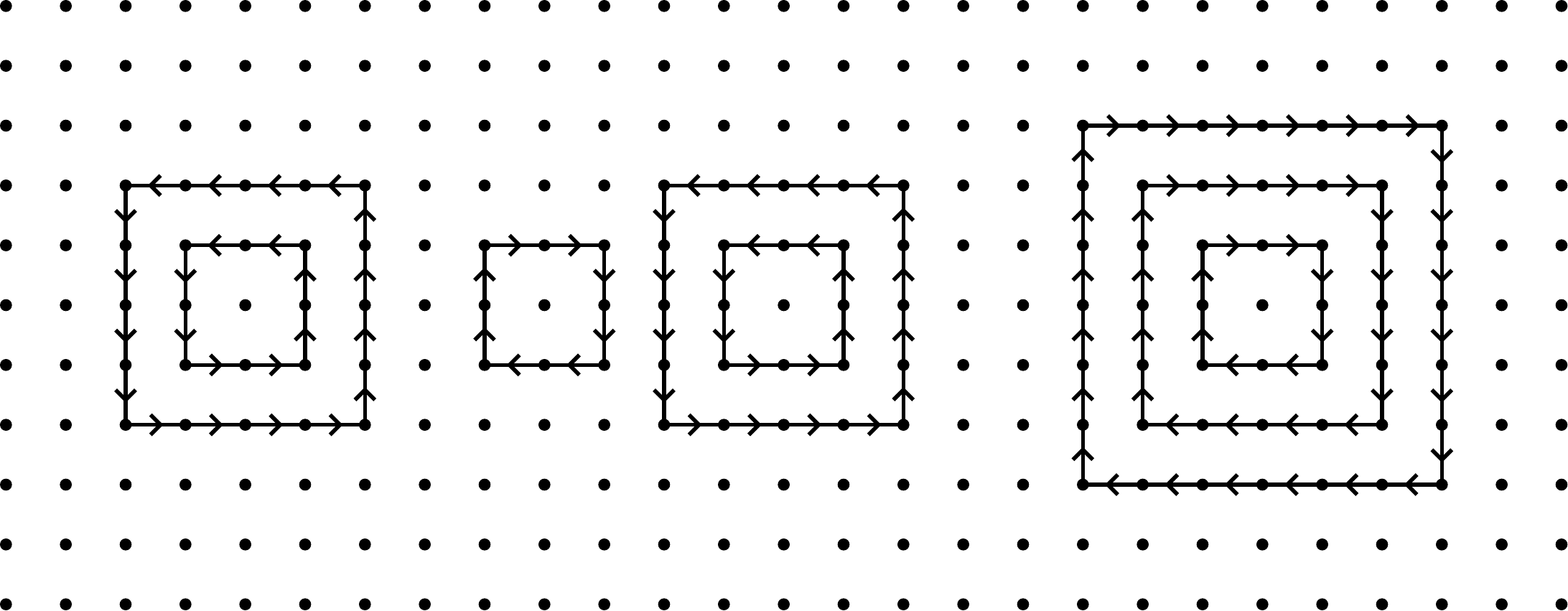}%
\caption{Example of an untangled sock.}%
\label{fig:untangledSockExample}%
\end{figure}

%Calculating the invariant for a canonical form SOC is particularly easy: each boxed jewel corresponds to exactly one term

\begin{lemma}
\label{lemma:untangledSocksInvariant}
Two untangled socks that have the same invariant are flip homotopic in $\ZZ^2$.
\end{lemma}
In this and in other proofs, we omit easy (but potentially tedious) details when we think the picture is sufficiently clear. %and that more text would only be tedious and unhelpful.
\begin{proof}
If the two socks consist of precisely the same boxed jewels but in a different order, then they are clearly flip homotopic, since we can easily move the jewels around and switch their positions as needed. We only need to check that boxed jewels that cancel out (that is, they refer to terms with the same exponent but opposite signs) can be ``dissolved'' by flip moves.

\begin{figure}%
\centering
\includegraphics[width=0.6\columnwidth]{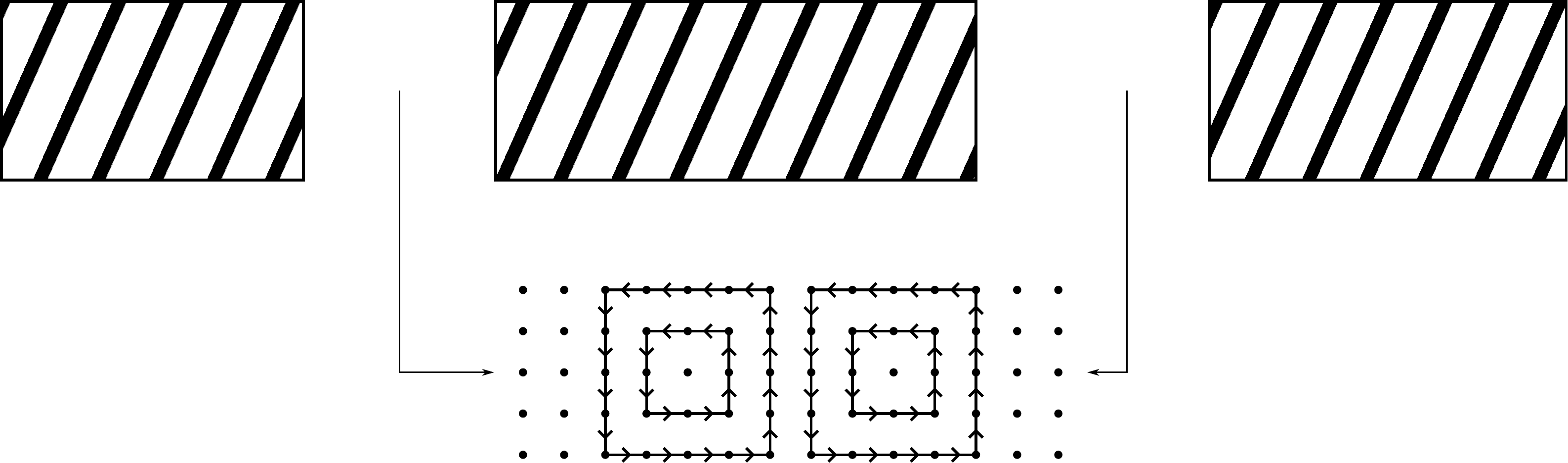}%
\caption{Illustration of how boxed jewels that cancel out can be brought close. The striped rectangles indicate areas where there may be other boxed jewels.}%
\label{fig:boxedJewelsCancellation}%
\end{figure}
\begin{figure}[ht]%
\centering
\includegraphics[width=0.6\columnwidth]{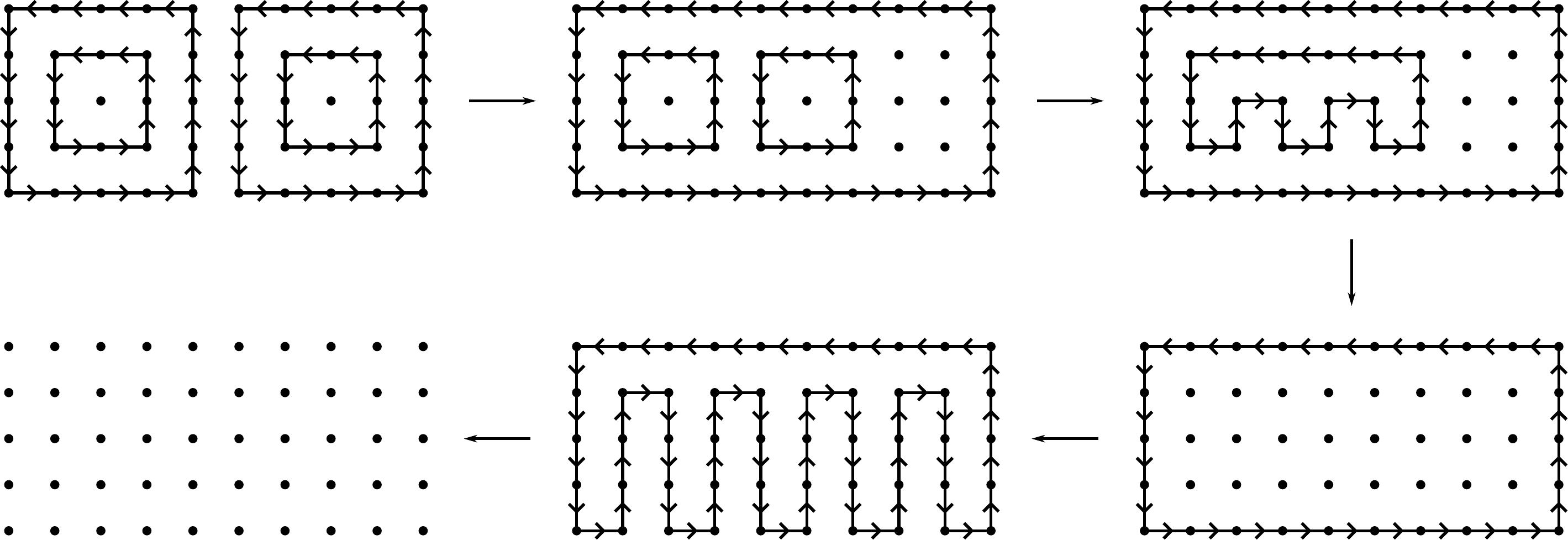}%
\caption{Some intermediate steps in the flip homotopy that ``dissolves'' a pair of cancelling boxed jewels}%
\label{fig:boxedJewelsCancellationSteps}%
\end{figure}

In order to see this, start by moving the boxed jewels that cancel out downwards, then toward each other until they are next to each other, as illustrated in Figure \ref{fig:boxedJewelsCancellation}. Once they are next to each other, they can be elliminated by a sequence of flip moves. Figure \ref{fig:boxedJewelsCancellationSteps} shows some of the steps involved in the flip homotopy that eliminates this pair of cancelling jewels. 
 
\end{proof}

If $s$ is a sock, we define the \emph{area} of $s$ to be the sum of the areas enclosed by each cycle of $s$, thought of as a plane curve. The areas count as positive regardless of the orientation of the cycles. As an example, the boxed jewels shown in Figure \ref{fig:boxedJewelsExample} have areas $4$ and $120,$ respectively, and the untangled sock in Figure \ref{fig:untangledSockExample} has area $20 + 4 + 20 + 56 = 100$. Naturally, the only sock with zero area is the empty sock (the sock with no cycles).

The following Lemma is the key step in the proof of Proposition \ref{prop:moreSpace}:

\begin{lemma}
\label{lemma:homotopicToUntangled}
Every sock is flip homotopic to an untangled sock in $\ZZ^2$.
\end{lemma}
\begin{proof}
Suppose, by contradiction, that there exists a sock which is not flip homotopic to an untangled sock. Of all the examples of such socks, pick one, $s_0$, that has minimal area (which is greater than zero, because the empty sock is already untangled).

Among all the cycles in $s_0$, consider the ones who have vertices that are furthest bottom. Among all these vertices, pick the rightmost one, which we will call $v$. In other words, assuming the axis are as in Figure \ref{fig:notation2Dexample}: if $m = \max \{n : (k,n) \mbox{ is a vertex of } s_0 \mbox{ for some } k \in \ZZ\}$ and $l = \max \{k : (k,m) \mbox{ is a vertex of } s_0\}$, then $v = (l,m)$. 

Clearly $v$ is the right end of a horizontal edge, and the bottom end of a vertical edge, as portrayed in Figure \ref{fig:rightmostSouthmostPoint}. We may assume without loss of generality that these edges are oriented as in the aforementioned Figure.    
\begin{figure}[ht]%
\centering
\def\svgwidth{0.25\columnwidth}
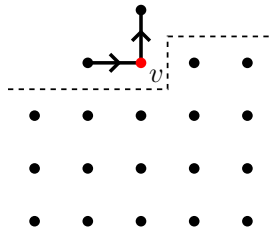
\caption{The rightmost bottommost vertex ($v$) in the nonempty sock $s_0$. By definition, there can be no cycle parts in $s_0$ below the dotted line.}%
\label{fig:rightmostSouthmostPoint}%
\end{figure}

Consider the diagonal of the form $v - (n,n), n \geq 0, n \in \ZZ$, starting from $v$ and pointing northwest, and let $w$ be the first point (that is, the one with the smallest $n$) in this diagonal that is not the right end of a horizontal edge pointing to the right and the bottom end of a vertical edge pointing up (see Figure \ref{fig:firstInTheDiagonal}).

\begin{figure}[ht]%
\centering
%\includegraphics[width=0.4\columnwidth]{firstInTheDiagonalNew_curved.pdf}%
%-------------------------- Inkscape pdf_tex
\setlength{\unitlength}{0.4\columnwidth}
  \begin{picture}(1,0.95832508)%
    \put(0,0){\includegraphics[width=\unitlength]{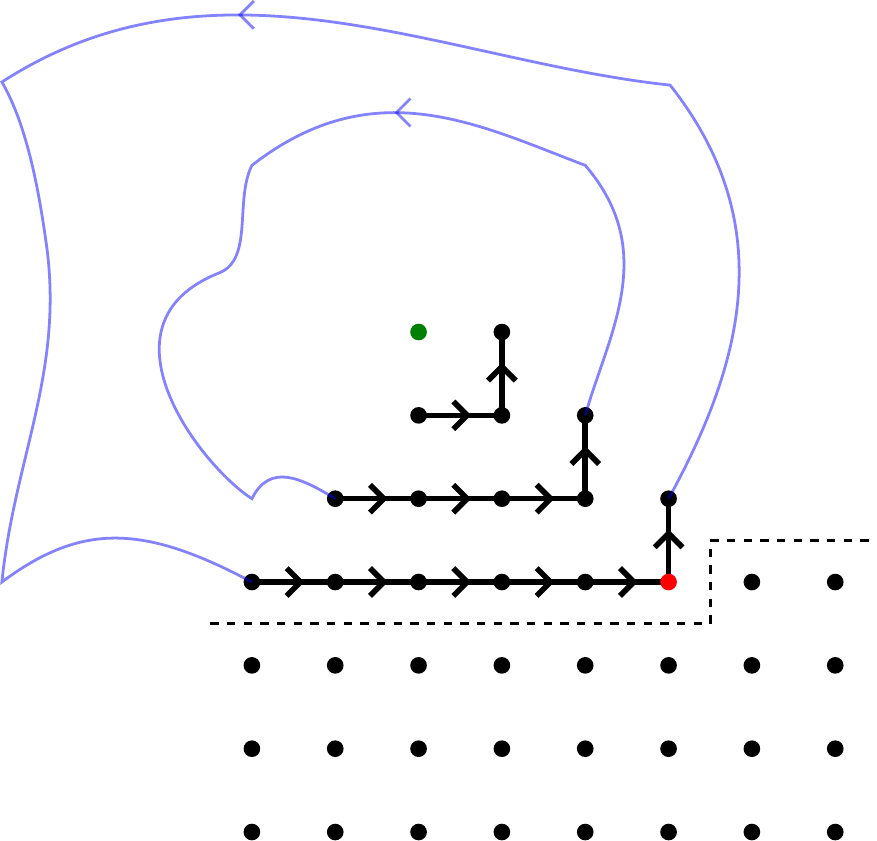}}%
    \put(0.77214768,0.25517476){\color[rgb]{0,0,0}\scalebox{0.9}{\makebox(0,0)[lb]{\smash{$v$}}}}%
    \put(0.4352189,0.59879322){\color[rgb]{0,0,0}\scalebox{0.9}{\makebox(0,0)[lb]{\smash{$w$}}}}%
  \end{picture}%
%-------------------------------------------
\caption{The vertices $v$ and $w$; $w$ is the first vertex in the diagonal that does not follow the pattern (``edgewise'') of the other three. The curved blue segments represent (schematically) the relative positions of two of the cycles, which must be in this way because there can be no cycle parts below the dashed line. }%
\label{fig:firstInTheDiagonal}%
\end{figure}

We will now see that if $w$ is not a jewel, then we can immediately reduce the area of $s_0$ with a single flip move. This leads to a contradiction, because the sock obtained after this flip move cannot be flip homotopic to an untangled sock, but has smaller area than $s_0$.

Suppose $w$ is not a jewel. With Figure \ref{fig:firstInTheDiagonal} in mind, consider the vertex directly below $w$. It is either the bottom end of a vertical edge pointing downward (case 1), or the right end of a horizontal edge pointing to the right. In case 1, a single flip move of type (a) will immediately reduce the area, as shown in the first drawing in Figure \ref{fig:firstInTheDiagonal_notJewel}. The other three drawings handle all the possibilities for the other case (again, we omit the easy details): notice that in each case there is a flip move that reduces the area.

\begin{figure}[ht]%
\centering
\includegraphics[width=0.7\columnwidth]{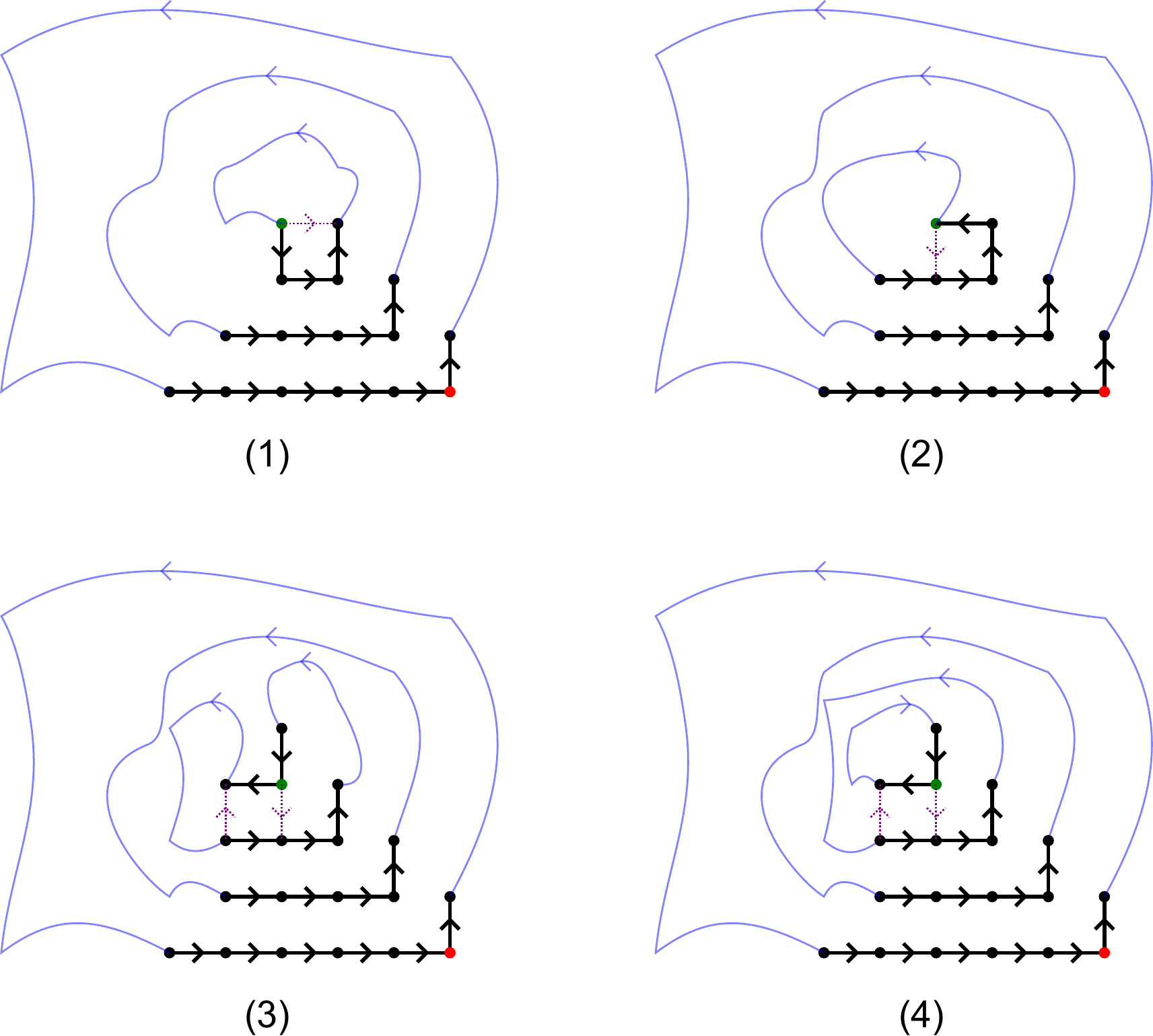}%
\caption{The four possibilities when $w$ is not a jewel. Notice that, in each case, there is a flip move that reduces the area, which are indicated by the dotted purple lines.}%
\label{fig:firstInTheDiagonal_notJewel}%
\end{figure} 

The most interesting case is when $w$ is a jewel. The key observation here is that the jewel may be then ``extracted'' from all the cycles as a boxed jewel, and what remains has smaller area. Figure \ref{fig:jewelExtraction} illustrates the steps involved in extracting a jewel.

\begin{figure}[ht]%
\centering
\includegraphics[width=\columnwidth]{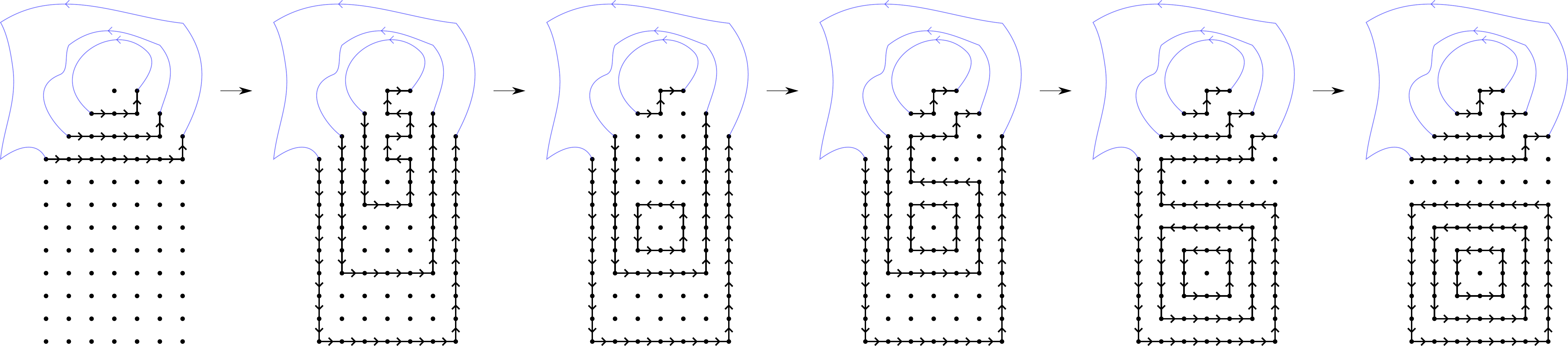}%
\caption{Some intermediate steps in the extraction of a jewel. Notice that this flip homotopy reduced the area of the cycles that previously enclosed the jewel.}%
\label{fig:jewelExtraction}%
\end{figure}

Suppose $w$ is a jewel. As in Figure \ref{fig:jewelExtraction}, the jewel $w$ can be extracted via flip moves. We first pull the cycles downward with flip moves of type $(a)$ and then ``cut'' each cycle (innermost cycles first) with flip moves of type $(b)$, so that the jewel leaves the cycles as a boxed jewel (easy details are left to the reader). Let $s_1$ be the sock obtained at the end of this flip homotopy, and let $\tilde{s_1} = \tilde{s_2}$ be the sock consisting of all the cycles of $s_1$ except the extracted boxed jewel. Clearly the area of $\tilde{s_2}$ is less than the area of $s_0$, and since $s_0$ is a sock that has minimal area among those that are not flip homotopic to an untangled sock, it follows that there exists a flip homotopy, say  $\tilde{s_2},\tilde{s_3},\ldots, \tilde{s_n}$ such that $\tilde{s_n}$ is an untangled sock.

Let $M = \max \{k : (l,k) \mbox{ is a vertex of } \tilde{s_i} \mbox{ for some } k \in \ZZ, 2 \leq i \leq n\}$. Recall that boxed jewels can move an arbitrary even distance if they have free space in some direction; by the definition of $v$, there can be no cycles or cycle parts below $v$ in $s_0$, so we can define $s_2$ (homotopic to $s_1$) by pulling the extracted jewel in $s_1$ down as much as we need, namely so that all vertices of the boxed jewel have $y$ coordinate strictly larger than $M$. Notice that $\tilde{s_2}$ is obtained from $s_2$ by removing this boxed jewel. Clearly we can perform on $s_2$ all the flip moves that took $\tilde{s_2}$ to $\tilde{s_n}$, so that we obtain a flip homotopy $s_2, s_3, \ldots, s_n$, where $s_n$ consists of all the cycles in $\tilde{s_n}$ plus a single boxed jewel down below. The other cycles in $s_n$ are also boxed jewels whose centers have $y$ coordinate equal to zero, because $\tilde{s_n}$ is untangled; therefore, the boxed jewel down below can be brought up and sideways as needed, thus obtaining a flip homotopy from $s_0$ to an untangled sock $s_{n+1}$. This contradicts the initial hypothesis, and thus the proof is complete.     
\end{proof}

This proof also yields an algorithm for finding the flip homotopy from an arbitrary sock to an untangled sock, although probably not a very efficient one. Start with the initial sock, and find $v$ and $w$ as in the proof. If $w$ is not a jewel, perform the flip move that reduces the area, and recursively untangle this new sock. If $w$ is a jewel, extract the boxed jewel, recursively untangle the sock without the boxed jewel (which has smaller area), and calculate how much space you needed to solve it: this will tell how far down the boxed jewel needs to be pulled. The algorithm stops recursing when we reach a sock with zero area: all that's left to do then is to ``organize'' the boxed jewels.

%Together, Lemmas \ref{lemma:equivEmbeddingHomotopy}, \ref{lemma:untangledSocksInvariant} and \ref{lemma:homotopicToUntangled} establish Proposition \ref{prop:moreSpace}, which was the goal for this section. As a conclusion, the polynomial invariant here presented is, in a sense, complete: if two tilings have the same invariant and sufficient space is added to the region, then there is a sequence of flips taking one to the other. 

\begin{proof}[Proof of Proposition \ref{prop:moreSpace}]
Let $t_0$ and $t_1$ be two tilings of a duplex region, and suppose $P_{t_0} = P_{t_1}$. Let $s_0$ and $s_1$ denote their corresponding socks in $\ZZ^2$. Since $P_{s_0}(q) = P_{t_0}(q) - P_{t_0}(0) = P_{t_1}(q) - P_{t_1}(0) = P_{s_1}(q)$, it follows by Lemma \ref{lemma:homotopicToUntangled} that $s_0$ and $s_1$ are flip homotopic to untangled socks $\tilde{s}_0$ and $\tilde{s}_1$. By \ref{lemma:untangledSocksInvariant}, $\tilde{s}_0$ and $\tilde s_1$ are flip homotopic in $\ZZ^2$, and therefore so are $s_0$ and $s_1$. Finally, by Lemma \ref{lemma:equivEmbeddingHomotopy}, there exists a two-floored box such that the embeddings $\hat{t}_0$ and $\hat{t}_1$ of $t_0$ and $t_1$ lie in the same flip connected component.
\end{proof}

As a conclusion, the polynomial invariant here presented is, in a sense, complete: if two tilings have the same invariant and sufficient space is added to the region, then there is a sequence of flips taking one to the other. 
  
\section{Connectivity by flips and trits}
\label{sec:connectedFlipsTrits}
%We introduced in this paper a new move, which we called a trit, that we defined in  \ref{sec:defsAndNotations} (see Figure \ref{fig:postrit}). Also, in Section \ref{sec:effectOfTritsOnPt} we studied the effect of a trit in our invariant $P_t$.
In Section \ref{sec:twofloorExamples}, we pointed out that the graph of flip connected components for Examples \ref{example:732box} and \ref{example:unequalFloors}, shown in Figures \ref{fig:box732_CCdiagram} and \ref{fig:bigGraph_CC_unequal}, is connected in both cases. There, two components are joined if there exists a trit taking a tiling in one component to a tiling in the other. Hence, this graph is connected if and only if for any two tilings of the given region, we can reach one from the other using flips and trits.

A natural question is, therefore: for two-story regions, is it true that one can always reach a tiling from any other via flips and trits? The answer is, in general, no, as Figure \ref{fig:unequal_twofloors_notConnectedFlipsTrits} shows. 

\begin{figure}%
\centering
\subfloat[]{\includegraphics[width=0.21\columnwidth]{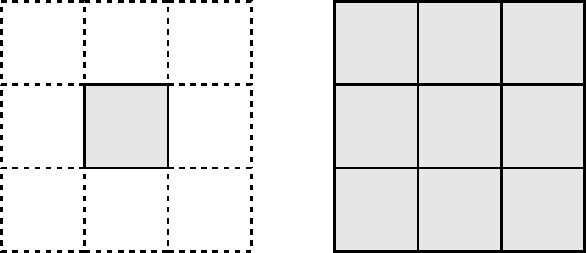}}\qquad
\subfloat[]{\includegraphics[width=0.33\columnwidth]{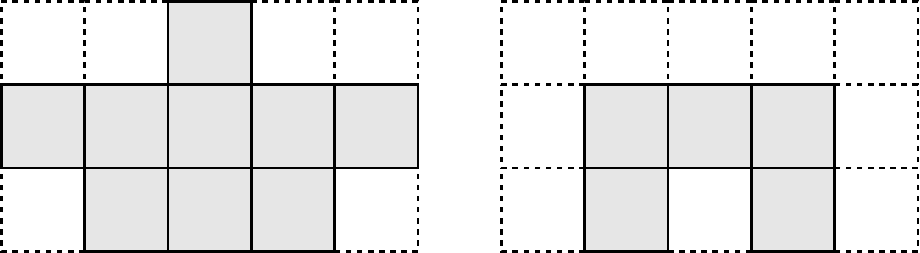}}
\caption{Two regions, each with two tilings where neither a flip nor a trit is possible.}%
\label{fig:unequal_twofloors_notConnectedFlipsTrits}%
\end{figure} 

Nevertheless, this is true for duplex regions:
\begin{prop}
\label{prop:connectivityFlipsTrits}
If $\cR$ is a duplex region and $t_0,t_1$ are two tilings of $\cR$, there exists a sequence of flips and trits taking $t_0$ to $t_1$.
\end{prop} 

In order to prove this result, we'll once again make use of the concept of systems of cycles, or socks, that were introduced in Section \ref{sec:twoFloorsMoreSpace}.

Understanding the effect of a trit on a sock turns out to be quite easy: in fact, the effect of a trit can be captured to the world of socks via the insertion of a new move, which we will call the \emph{trit move} (in addition to the three flip moves shown in Figure \ref{fig:flipsteps}). The trit move is shown in Figure \ref{fig:tritstep}.

\begin{figure}[ht]%
\centering
\includegraphics[width=0.3\columnwidth]{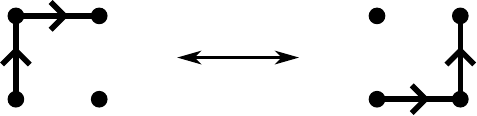}%
\caption{The trit move.}%
\label{fig:tritstep}%
\end{figure}

Another way to look at trit moves is that it either pulls a jewel out of a cycle or pushes one into a cycle. Two socks $s_1$ and $s_2$ are \emph{flip and trit homotopic} in a graph $G$ (which contains both $s_1$ and $s_2$) if there is a finite sequence of flip and/or trit moves taking $s_1$ to $s_2$. Notice that, unlike in Section \ref{sec:twoFloorsMoreSpace}, where we were mainly interested in flip homotopies in $\ZZ^2$, we are now interested in flip and trit homotopies in the finite graphs $G(\cR)$. Recall from Section \ref{sec:twoFloorsMoreSpace} that if $\cR$ is a duplex region, $G(\cR)$ is the planar graph whose vertices are the centers of the squares in the associated drawing of $\cR$, and where two vertices are joined by an edge if their Euclidean distance is exactly $1$.

Again, for two tilings $t_0, t_1$, there exists a sequence of flips and trits taking $t_0$ to $t_1$ if and only if their corresponding socks are flip and trit homotopic in $G(\cR)$.

\begin{lemma} \label{lemma:flipAndTritHomotopy} If $\cR$ is duplex region, and $s$ is a sock in $G(\cR)$, then $s$ is flip and trit homotopic to the empty sock (the sock with no cycles) in $G(\cR)$.
\end{lemma}
\begin{proof}
Suppose, by contradiction, that there exists a sock contained in $G(\cR)$ that is not flip and trit homotopic to the empty sock in $G(\cR)$. Of all the examples of such socks, let $s$ be one with minimal area (the concept of area of a sock was defined in Section \ref{sec:twoFloorsMoreSpace}). We will show that there exists either a flip move or a trit move that reduces the area of $s$, which is a contradiction.

Let $\gamma$ be a cycle of $s$ such that there is no other cycle inside $\gamma$: hence, if there is any vertex enclosed by $\gamma$, it must be a jewel. Similar to the proof of Lemma \ref{lemma:homotopicToUntangled}, let $v$ be the rightmost among the bottommost vertices of $\gamma$ (notice that we are only considering the vertices of $\gamma$, and not all the vertices in $s$), or, in other words: if $m = \max \{n : (k,n) \mbox{ is a vertex of } \gamma \mbox{ for some } k \in \ZZ\}$ and $l = \max \{k : (k,m) \mbox{ is a vertex of } \gamma\}$, let $v = (l,m)$. Notice that $v$ is the right end of a horizontal edge and the bottom end of a vertical edge: we may assume without loss of generality that this horizontal edge points to the right. 

\begin{figure}[ht]%
\centering
\includegraphics[width=0.45\columnwidth]{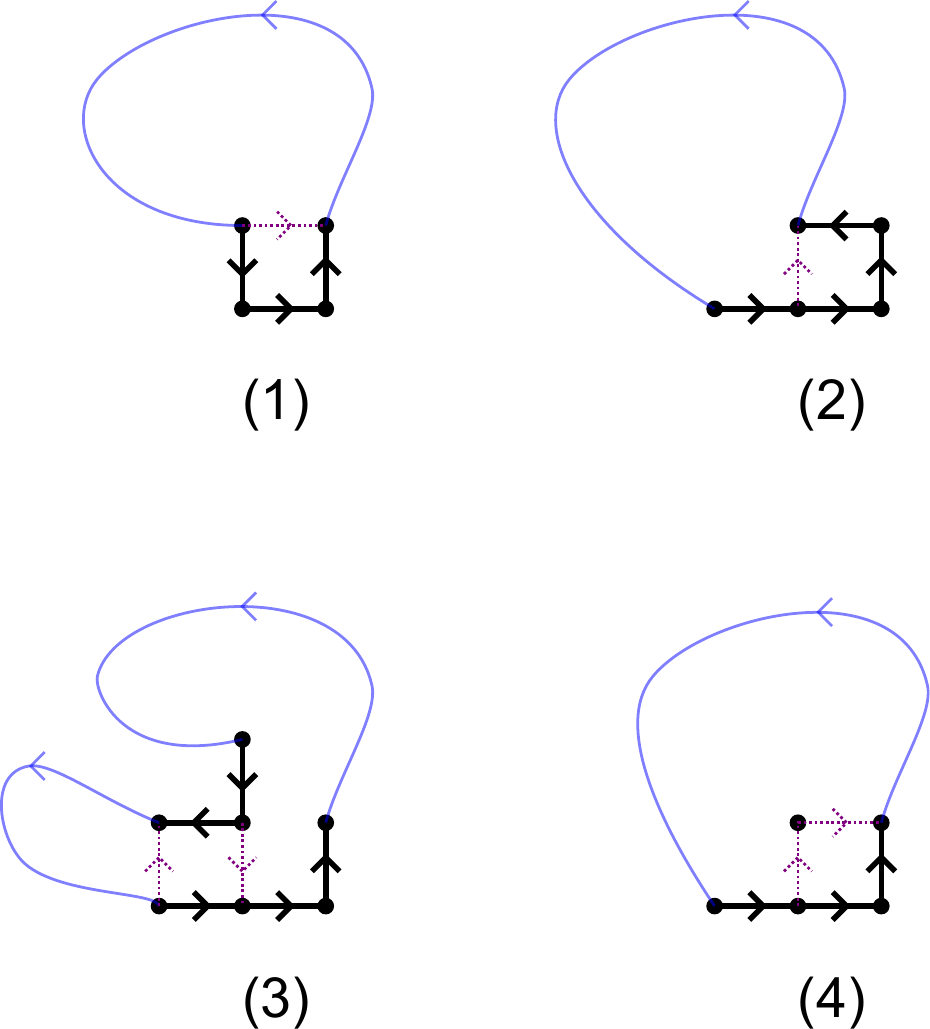}%
\caption{The four cases in the proof of Lemma \ref{lemma:flipAndTritHomotopy}, where the curved blue lines indicate schematically the position of the edges in $\gamma$ that are not directly portrayed. The move that reduces the area is indicated by the dashed purple line: in the first three cases, it is a flip move; the last one is a trit move.}%
\label{fig:tritAndFlipHomotopy_cases}%
\end{figure}

Notice that the vertex $w = v - (1,1)$ is necessarily in the graph $G(\cR)$, because otherwise $\gamma$ would have a hole inside, which contradicts the hypothesis that the (identical) floors of $\cR$ are simply connected. Suppose first that $w$ is not a jewel. Since there are no cycles inside $\gamma$, it follows that $w$ must be a vertex of $\gamma$. If $w$ is the topmost end of a vertical edge poiting downward, we have the first case in Figure \ref{fig:tritAndFlipHomotopy_cases}, where we clearly have a flip move that reduces the area. If this is not the case, it follows, in particular, that $u = v - (2,0)$ must be in the graph $G(\cR)$: Cases (2) and (3) of Figure \ref{fig:tritAndFlipHomotopy_cases} show the two possibilities for the edges that are incident to $w$, and it is clear that there exists a flip move which reduces the area of $s$.

Finally, the case where $w$ is a jewel is shown in case (4) of Figure \ref{fig:tritAndFlipHomotopy_cases}: the available trit move clearly reduces the area. Hence, there is always a flip or trit move that reduces the area of the sock, which contradicts the minimality of the area of $s$.    
\end{proof}
%Notice that, in the above proof, all that matters is what happens inside $\gamma$, which is consistent with the fact that we are not adding more space to the region $\cR$. In other words, here we may not assume that $\gamma$ can easily be moved (which was the case in the proof of Lemma \ref{lemma:homotopicToUntangled}).

Therefore, we have established Proposition \ref{prop:connectivityFlipsTrits}. Another observation is that the above proof also yields an algorithm for finding the flip and trit homotopy from a sock to the empty sock: while there is still some cycle in the sock, find one cycle that contains no other cycle. Then find $v$, as in the proof, and do the corresponding flip or trit move, depending on the case. Since each move reduces the area, it follows that we'll eventually reach the only sock with zero area, which is the empty sock.
 
\section{The invariant when more floors are added}
\label{sec:embedFourFloors}
In this section, we'll discuss the fact that the invariant $P_t(q)$ is not preserved when the tilings are embedded in ``big'' regions with more than two floors.

As in Section \ref{sec:twoFloorsMoreSpace}, we'll consider duplex regions. Recall from that section that we defined the embedding of such a tiling in a two-floored box $\cB$. Here, we'll extend this notion to embeddings in boxes with four floors in a rather straightforward manner.

If $t$ is a tiling of a duplex region $\cR$, and $\cB$ is a box with four floors such that $\cR$ is contained in the top two floors of $\cB$, then the embedding $\hat{t}$ of $t$ in $\cB$ is the tiling obtained by first embedding $t$ in the top two floors of $\cB$ (as in Section \ref{sec:twoFloorsMoreSpace}), and then filling the bottom two floors with dominoes parallel to $\ez$. This is illustrated in Figure \ref{fig:embeddingFourFloorsExample}.

\begin{figure}[ht]%
\centering
\includegraphics[width=0.5\columnwidth]{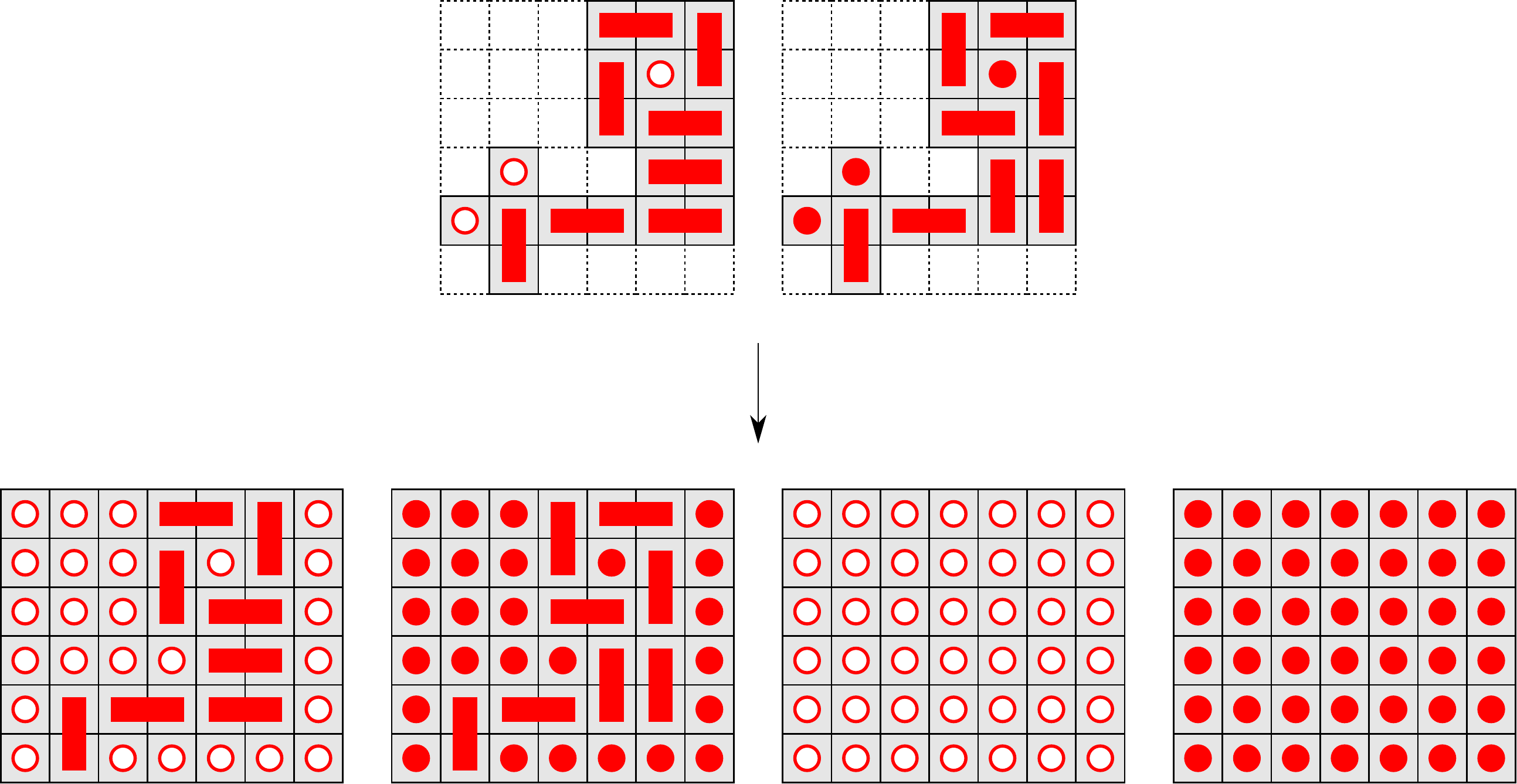}%
\caption{A tiling of a two-floored region and its embedding in a $7 \times 6 \times 4$ box.}%
\label{fig:embeddingFourFloorsExample}%
\end{figure}

\begin{prop}
\label{prop:embedFourFloors}
If $t_0$ and $t_1$ are two tilings of a duplex region and $P_{t_0}'(1) = P_{t_1}'(1)$, then there exists a box with four floors $\cB$ such that the embeddings of $t_0$ and $t_1$ in $\cB$ lie in the same flip connected component.
\end{prop}

%The converse also holds, but the proof requires some technology from Chapter \ref{chap:multiplex} (see Proposition \ref{prop:flipsAndTrits}). 

%Proposition \ref{prop:embedFourFloors} is essentially stating that the invariant $P_t(q)$ no longer survives when more floors are added; the twist $\Tw(t)$, however, still does, as we shall see in Chapter \ref{chap:multiplex}.
Proposition \ref{prop:embedFourFloors} is essentially stating that the invariant $P_t(q)$ no longer survives when more floors are added, although the twist $\Tw(t)$ still does, as we will see in Chapter \ref{chap:multiplex}. 
The converse of Proposition \ref{prop:embedFourFloors}, i.e., the fact that the twist is invariant by flips, follows from item \ref{item:flips} of Proposition \ref{prop:flipsAndTrits}.

Recall from Section \ref{sec:twoFloorsMoreSpace} the definitions of boxed jewel and sock. The key fact in the proof of Proposition \ref{prop:embedFourFloors} is that a boxed jewel associated to a $q^n$ term in $P_t(q)$ can be transformed into a number of smaller boxed jewels, their terms adding up to $nq$. In what follows, the \emph{sign} of a boxed jewel is the sign of its contribution to $P_t$ (i.e., $1$ if the jewel is black, and $-1$ if it is white) and its \emph{degree} is the number of cycles it contains, if all the cycles spin counterclockwise; it is minus this number if all the cycles are clockwise.
In other words, if $t$ is a tiling of a two-floored region whose sock is untangled and $\{b_i\}_i$ is the set of boxed jewels in this sock, then 
$$P_t(q) = a_0 + \sum_i \sgn(b_i) q^{\deg(b_i)},$$
where $a_0 = P_t(1) - \sum_i \sgn(b_i)$, $\sgn(b_i)$ is the sign of $b_i$ and $\deg(b_i)$, its degree. 
%In what follows, if a boxed jewel is associated to a term $\pm q^n$ in $P_t(q)$, its \emph{sign} is $\pm 1$ and its \emph{degree} is $n \in \ZZ$. 

\begin{lemma}
\label{lemma:jewelGoesDown}
If $t_0$ and $t_1$ are tilings of a two-floored region such that their associated socks are both untangled (consist only of boxed jewels) and:
\begin{enumerate}[label=(\roman*)]
	\item The associated sock of $t_0$ consists of a single boxed jewel of degree $n > 0$.
	\item The associated sock of $t_1$ consists of $n$ boxed jewels of degree $1$ and same sign as the boxed jewel in $t_0$.
\end{enumerate} 
Then there exists a box with four floors $\cB$ where their embeddings lie in the same connected components. 
\end{lemma} 
\begin{proof}
\begin{figure}[ht]%
\centering
\includegraphics[width=0.4\columnwidth]{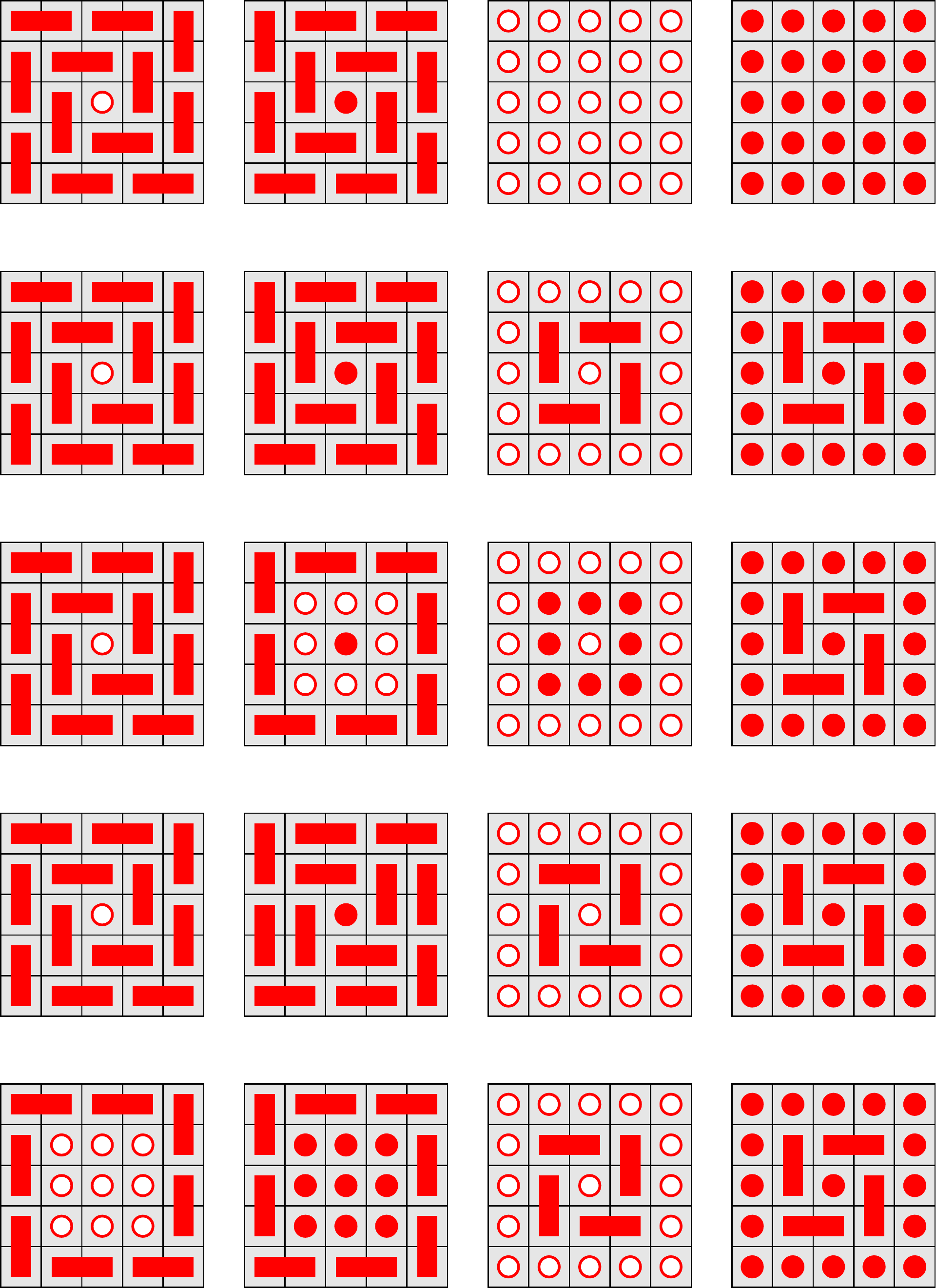}%
\caption{Five steps that bring a small boxed jewel from the top two floors to the bottom two floors, each consisting of four flips. In this case, a boxed jewel with degree $2$ in the top two floors is transformed into a boxed jewel with degree $1$ in the bottom two floors plus a cycle that can be easily flipped to a boxed jewel with degree $1$ in the top two floors. Since the bottom jewel can move freely in the bottom floors, it can be brought back up in a different position.}%
\label{fig:jewelGoesDown}%
\end{figure}
We omit some easy details, but 
Figure \ref{fig:jewelGoesDown} illustrates the key step in the proof: that the innermost boxed jewel can be transported to the bottom two floors via flips. If the box $\cB$ is big enough, the innermost boxed jewel can freely move in the bottom two floors, and can eventually be brought back up outside of any other cycles. Since after this maneuver the bottom two floors are back as they originally were, the result of this maneuver is the embedding of a tiling whose sock is flip homotopic to an untangled sock with two boxed jewels, one with degree $n-1$ and another with degree $1$ (but both have the same sign as the original boxed jewel). Proceeding by induction and using Proposition \ref{prop:moreSpace}, we obtain the result.
\end{proof}

Therefore, boxed jewels with degree $n > 0$ (resp. degree $ -n < 0$) can be flipped into $n$ boxed jewels with degree $1$ (resp. $-1$). It only remains to see that boxed jewels with degrees $1$ and $-1$ and same sign cancel out.

\begin{lemma}
\label{lemma:jewelsCancel}
 If $t$ is a tiling of a duplex region whose sock is untangled and consists of two boxed jewels with degrees $1$ and $-1$ but same sign, then there exists a box with four floors where the embedding of $t$ is in the same flip connected component as the tiling consisting only of ``jewels'', that is, the tiling containing only dominoes parallel to $\ez$.
\end{lemma}
\begin{proof}
\begin{figure}[ht]%
\centering
\includegraphics[width=0.55\columnwidth]{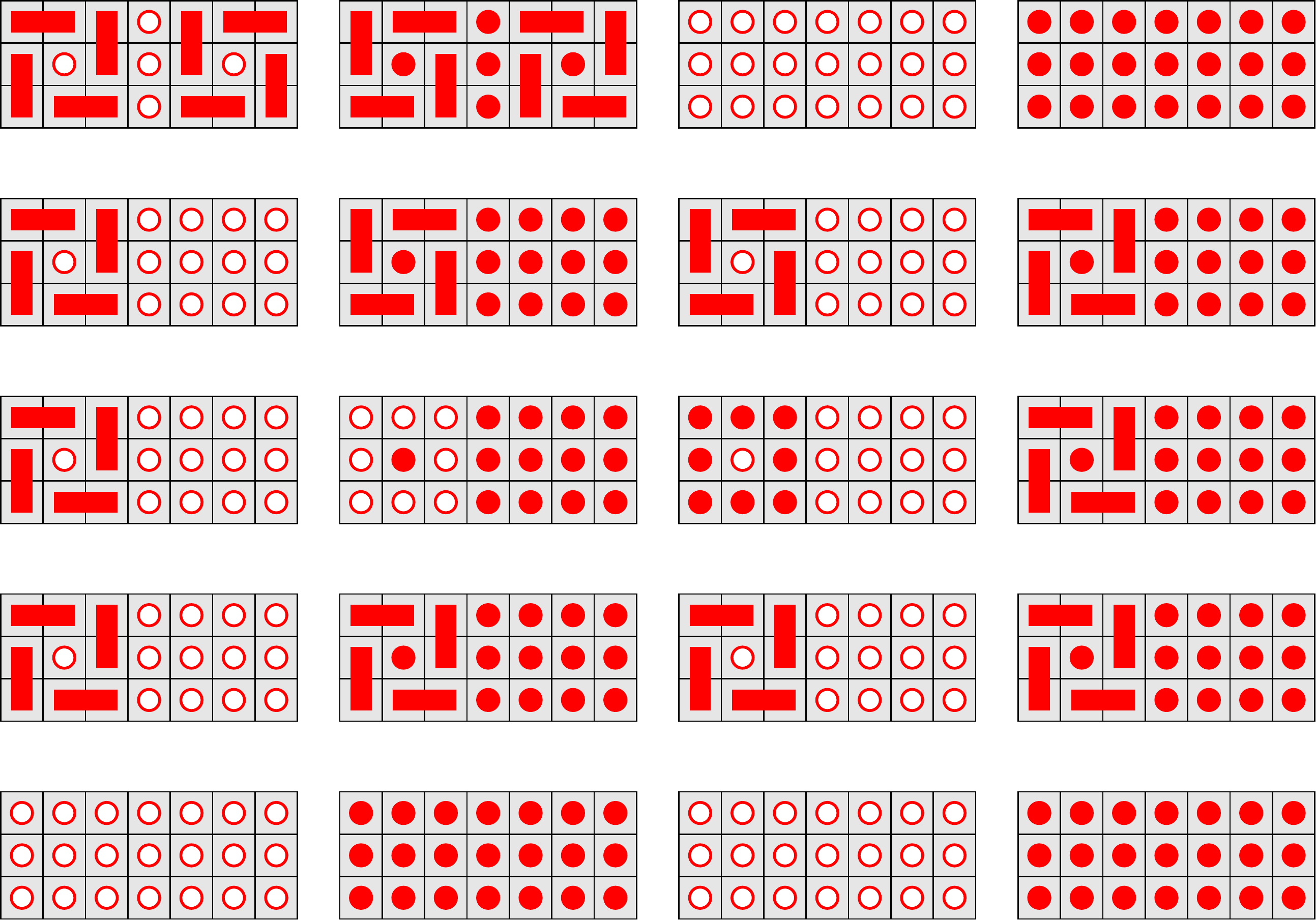}%
\caption{Illustration of how two boxed jewels with same sign but opposite degree cancel. One boxed jewel is transported to the bottom floor and there it moves so that it is exactly under the other boxed jewel (this is the first step). From then on, it is a relatively straightforward sequence of flips, and the three bottom drawings show some of the intermediate steps.}%
\label{fig:jewelsCancel}%
\end{figure}
The basic procedure is illustrated in Figure \ref{fig:jewelsCancel}. First, one jewel can be transported to the bottom floor using the procedure in Figure \ref{fig:jewelGoesDown}. There it can be moved so that it is right under the other boxed jewel, when they can easily be flipped into a tiling with all dominoes parallel to $\ez$.
\end{proof}

In short, what Lemmas \ref{lemma:jewelGoesDown} and \ref{lemma:jewelsCancel} imply is that a tiling $t$ whose sock is untangled and whose invariant is $P_t(q) = a_0 + \sum_{n \neq 0} a_n q^n$ can be embedded in a box with four floors in such a way that this embedding is in the same flip connected component as the embedding of another tiling $\tilde{t}$ (of a duplex region) whose sock is untangled and whose invariant is $P_{\tilde{t}}(q) = \tilde{a_0} + \left(\sum_{n \neq 0} na_n \right) q = \tilde{a_0} + P_t'(1)q$.

%Using these last results, we can prove Proposition \ref{prop:embedFourFloors} as follows: 
\begin{proof}[Proof of Proposition \ref{prop:embedFourFloors}]
Let $t_0$ and $t_1$ be two tilings of a duplex region such that $P_{t_0}'(1) = P_{t_1}'(1)$. By Lemmas \ref{lemma:homotopicToUntangled} and \ref{lemma:equivEmbeddingHomotopy}, there exists a two-floored box $\tilde{\cB}$ where the embeddings of $t_0$ and $t_1$ are in the same connected component, respectively, as $\tilde{t_0}$ and $\tilde{t_1}$, two tilings whose sock is untangled. By Lemmas \ref{lemma:jewelGoesDown} and \ref{lemma:jewelsCancel} and the previous paragraph, there exists a box with four floors $\cB$ where the embeddings of $\tilde{t_0}$ and $\tilde{t_1}$ lie in the same connected component, respectively, as the embeddings of $\hat{t_0}$ and $\hat{t_1}$ (of the same duplex region), with invariants $P_{\hat{t_0}}(q) = a_0 + P_{t_0}'(1)q$ and $P_{\hat{t_1}}(q) = a_1 + P_{t_1}'(1) q$. Since $P_{t_0}'(1) = P_{t_1}'(1)$, it follows that $a_0 = a_1$ and so $P_{\hat{t_0}}(q) = P_{\hat{t_1}}(q)$. By Proposition \ref{prop:moreSpace}, the box $\cB$ can be chosen such that the embeddings of $\hat{t_0}$ and $\hat{t_1}$ lie in the same flip connected component; this concludes the proof.  
\end{proof}

%% file: figures/flipsteps.pdf_tex
%% Creator: Inkscape 0.48.3.1, www.inkscape.org
%% Accompanies image file 'flipsteps.pdf' (pdf, eps, ps)
\begingroup%
  \makeatletter%
  \providecommand\color[2][]{%
    \errmessage{(Inkscape) Color is used for the text in Inkscape, but the package 'color.sty' is not loaded}%
    \renewcommand\color[2][]{}%
  }%
  \providecommand\transparent[1]{%
    \errmessage{(Inkscape) Transparency is used (non-zero) for the text in Inkscape, but the package 'transparent.sty' is not loaded}%
    \renewcommand\transparent[1]{}%
  }%
  \providecommand\rotatebox[2]{#2}%
  \ifx\svgwidth\undefined%
    \setlength{\unitlength}{180.228125bp}%
    \ifx\svgscale\undefined%
      \relax%
    \else%
      \setlength{\unitlength}{\unitlength * \real{\svgscale}}%
    \fi%
  \else%
    \setlength{\unitlength}{\svgwidth}%
  \fi%
  \global\let\svgwidth\undefined%
  \global\let\svgscale\undefined%
  \makeatother%
  \begin{picture}(1,1.05962929)%
    \put(0,0){\includegraphics[width=\unitlength]{flipsteps.pdf}}%
    \put(-0.00192464,0.94546842){\makebox(0,0)[lb]{\smash{a)}}}%
    \put(-0.00192464,0.50158653){\makebox(0,0)[lb]{\smash{b)}}}%
    \put(-0.00192464,0.05770465){\makebox(0,0)[lb]{\smash{c)}}}%
  \end{picture}%
\endgroup%

%% file: figures/rightmostSouthmostPoint.pdf_tex
\begingroup%
  \makeatletter%
  \providecommand\color[2][]{%
    \errmessage{(Inkscape) Color is used for the text in Inkscape, but the package 'color.sty' is not loaded}%
    \renewcommand\color[2][]{}%
  }%
  \providecommand\transparent[1]{%
    \errmessage{(Inkscape) Transparency is used (non-zero) for the text in Inkscape, but the package 'transparent.sty' is not loaded}%
    \renewcommand\transparent[1]{}%
  }%
  \providecommand\rotatebox[2]{#2}%
  \ifx\svgwidth\undefined%
    \setlength{\unitlength}{120bp}%
    \ifx\svgscale\undefined%
      \relax%
    \else%
      \setlength{\unitlength}{\unitlength * \real{\svgscale}}%
    \fi%
  \else%
    \setlength{\unitlength}{\svgwidth}%
  \fi%
  \global\let\svgwidth\undefined%
  \global\let\svgscale\undefined%
  \makeatother%
  \begin{picture}(1,0.84)%
    \put(0,0){\includegraphics[width=\unitlength]{rightmostSouthmostPoint.pdf}}%
    \put(0.5305542,0.54685445){\color[rgb]{0,0,0}\scalebox{0.9}{\makebox(0,0)[lb]{\smash{$v$}}}}%
  \end{picture}%
\endgroup%

%% file: multiplexchap.tex
\chapter{The general case}
\label{chap:multiplex}

Most of the material in this chapter is also covered in \cite{segundoartigo}.

\section{The twist for cylinders}
\label{sec:combTwistBoxes}
For a domino $d$, define $\tv(d) \in \Phi$ to be the center of the black cube contained in $d$ minus the center of the white one. We sometimes draw $\tv(d)$ as an arrow pointing from the center of the white cube to the center of the black one.

For a set $X \subset \RR^3$ and $\vu \in \Phi$, we define the (\emph{open}) $\vu$-\emph{shade} of $X$ as
$$\cS^{\vu}(X) = \interior((X + [0,\infty)\vu) \setminus X) = \interior\left(\{ x + s \vu \in \RR^3 | x \in X, s \in [0,\infty)\} \setminus X\right),$$
where $\interior(Y)$ denotes the interior of $Y$.  
The \emph{closed} $\vu$-\emph{shade} $\clS^{\vu}(X)$ is the closure of $\cS^{\vu}(X)$.
We shall only refer to $\vu$-shades of unions of basic cubes or basic squares, such as dominoes.
%We most often refer to $\vu$-shades of dominoes (which are sets of $\RR^3$), but shades of unions  may also appear. 

%Let $\cR$ be a region, let $t$ be a tiling of $\cR$ and consider a domino $d$ of $t$. For $\vu \in \{\pm \ex, \pm \ey, \pm \ez\}$, we define the $\vu$-\emph{shade} of $d$ as $\cS^{\vu}(d) = \{p + n \vu | p \in d, n \in \ZZ, n > 0\}.$   

\begin{figure}[ht]
\centering
\includegraphics[width=0.7\columnwidth]{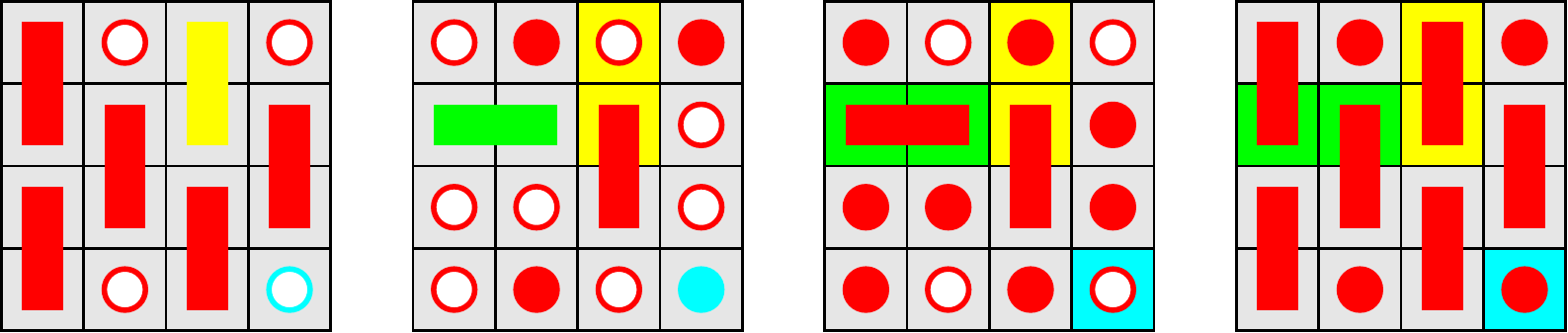}%
\caption{Tiling of a $4 \times 4 \times 4$ box, with three distinguished dominoes (painted yellow, green and cyan), whose $\ez$-shades are highlighted in the same color as they are. Notice that the yellow shade intersects four dominoes, the green shade intersects three, and the cyan shade, only one.}%
\label{fig:shadowExample}%
\end{figure} 

Given two dominoes $d_0$ and $d_1$ of $t$, we define the \emph{effect of} $d_0$ \emph{on} $d_1$ \emph{along} $\vu$, as:
$$\tau^{\vu}(d_0,d_1) = \begin{cases} \frac{1}{4} \det(\tv(d_1),\tv(d_0),\vu), &d_1 \cap \cS^{\vu}(d_0) \neq \emptyset \\ 0, &\mbox{otherwise} \end{cases} \label{def:effect}$$

In other words, $\tau^{\vu}(d_0, d_1)$ is zero unless the following three things happen: $d_1$ intersects the $\vu$-shade of $d_0$; neither $d_0$ nor $d_1$ are parallel to $\vu$; and $d_0$ is not parallel to $d_1$. When $\tau^{\vu}(d_0,d_1)$ is not zero, it's either $1/4$ or $-1/4$ depending on the orientations of $\tv(d_0)$ and $\tv(d_1)$.

For example, in Figure \ref{fig:shadowExample}, for $\vu = \ez$, the yellow domino $d_Y$ has no effect on any other domino: $\tau^{\ez}(d_Y,d) = 0$ for every domino $d$ in the tiling. The green domino $d_G$, however, affects the two dominoes in the rightmost floor which intersect its $\ez$-shade, and $\tau^{\ex}(d_G,d) = 1/4$ for both these dominoes.    

If $t$ is a tiling, we define the $\vu$-\emph{pretwist} as 
$$
T^{\vu}(t) = \sum_{d_0,d_1 \in t} \tau^{\vu}(d_0,d_1).\label{def:pretwist}
$$

For example, the tiling on the left of Figure \ref{fig:negtrit_example} has $\ez$-pretwist equal to $1$. To see this, notice that each of the four dominoes of the leftmost floor that are not parallel to $\ez$ has nonzero effect along $\ez$ on exactly one domino of the rightmost floor, and this effect is $1/4$ in each case. The reader may also check that the $\ez$-pretwist of the tiling in Figure \ref{fig:shadowExample} is $0$.

%A natural question at this point concerns how the choice of $\vu$ affects $T^{\vu}$.

\begin{lemma}
\label{lemma:twistNegatingDirection}
For any pair of dominoes $d_0$ and $d_1$ and any $\vu \in \Phi$,
$\tau^{\vu}(d_0,d_1) = \tau^{-\vu}(d_1,d_0)$. In particular, for a tiling $t$ of a region we have $T^{-\vu}(t) = T^{\vu}(t)$. 
\end{lemma}
\begin{proof}
Just notice that
$d_1 \cap \cS^{\vu}(d_0) \neq\emptyset$ if and only if $d_0 \cap \cS^{-\vu}(d_1) \neq \emptyset,$ 
and $\det(\tv(d_1),\tv(d_0),\vu) = \det(\tv(d_0),\tv(d_1),-\vu)$.
\end{proof}
%First, notice that $\tau^{\vu}(d_0,d_1) = \tau^{-\vu}(d_1,d_0)$: in fact, $d_1 \cap \interior(\cS^{\vu}(d_0)) \neq \emptyset \Leftrightarrow d_0 \cap \interior(\cS^{-\vu}(d_1)) \neq \emptyset$, and $\det(\tv(d_1),\tv(d_0),\vu) = \det(\tv(d_0),\tv(d_1),-\vu)$. Therefore, $T^{-\vu}(t) = T^{\vu}(t)$. 
%In general, $T^{\ex}(t)$, $T^{\ey}(t)$ and $T^{\ez}(t)$ need not be equal; however, we will see in Section \ref{sec:differentDirections} that the three are equal if $t$ is a tiling of a multiplex.
Translating both dominoes by a vector with integer coordinates clearly does not affect $\tau^{\vu}(d_0,d_1)$, as $\det(\tv(d_1),\tv(d_0),\vu) = \det(-\tv(d_1),-\tv(d_0),\vu)$. Therefore, if $t$ is a tiling and $f(p) = p + b$, where $b \in \ZZ^3$, then $T^{\vu}(f(t)) = T^{\vu}(t)$.

\begin{lemma}
\label{lemma:twistReflection}
Let $\cR$ be a region, and let $\vw \in \Delta$. Consider the reflection $r = r_{\vw}: \RR^3 \to \RR^3: p \mapsto p - 2(p\cdot \vw)\vw$; notice that $r(\cR)$ is a region.
If $t$ is a tiling of $\cR$ and $\vu \in \Phi$, then the tiling $r(t) = \{r(d), d \in t\}$ of $r(\cR)$ satisfies $T^{\vu}(r(t)) = - T^{\vu}(t)$. 
\end{lemma}
\begin{proof}
Given a domino $d$ of $t$, notice that $\tv(r(d)) = - r(\tv(d))$ and that $\cS^{\vu}(r(d)) = r(\cS^{r(\vu)}(d))$. Therefore, $r(d_1) \cap \cS^{\vu}(r(d_0)) \neq \emptyset \Leftrightarrow d_1 \cap S^{r(\vu)}(d_0) \neq \emptyset$ and
\begin{align*}
&\det(\tv(r(d_1)),\tv(r(d_0)),\vu) = \det(-r(\tv(d_1)), -r(\tv(d_0)), \vu)\\ 
&= \det(r(\tv(d_1)), r(\tv(d_0)), r(r(\vu))) = - \det(\tv(d_1), \tv(d_0), r(\vu)).  
\end{align*}
Therefore, $\tau^{\vu}(r(d_0),r(d_1)) = - \tau^{r(\vu)}(d_0,d_1)$ and thus $T^{\vu}(r(t)) = - T^{r(\vu)}(t)$. Since $r(\vu) = \pm \vu$, Lemma \ref{lemma:twistNegatingDirection} implies that $T^{\vu}(r(t)) = - T^{\vu}(t)$, completing the proof.
\end{proof}

 A natural question at this point concerns how the choice of $\vu$ affects $T^{\vu}$. It turns out that it will take us some preparation before we can tackle this question.

\begin{prop}
\label{prop:equalTwistsMultiplex}
If $\cR$ is a cylinder and $t$ is a tiling of $\cR$,
$$T^{\ex}(t) = T^{\ey}(t) = T^{\ez}(t) \in \ZZ.$$
\end{prop}
\begin{proof}
Follows directly from Propositions \ref{prop:equalTwistsIntNEven} and \ref{prop:equalTwistsIntNOdd} below.
\end{proof}
This result doesn't hold in pseudocylinders or in more general simply connected regions; see Section \ref{sec:examples} for counterexamples.

%Even though we postpone the proof to Section \ref{sec:differentDirections}, we use this result (only) for the following definition.
%When there is no possibility of confusion, the term \emph{effect of} $d_0$ \emph{on} $d_1$ will be used to the denote the effect along $\ez$; also, we shall write $\tau$ to denote $\tau^{\ez}$.

%As a last comment regarding shades, it follows from the definition of $\tau$ that, given a dimer $d$, the dominoes that may be affected by $d$ are the ones that intersect the $\ez$-shade of $d$. However, it is just as easy to determine the dominoes that may have an effect on $d$: they are the ones that intersect the $-z$-\emph{shade} $S_{-z}(d)$, that is, the set of all cubes that have same $x$ and $y$ coordinates as one of the cubes that form $d$, but smaller $\ez$ coordinate. Moreover, if the effects had been defined in terms of $-z$-shades with an analogous formula, we would have $\tau_{-z}(d_1,d_0) = \tau(d_0,d_1)$ for every $d_0,d_1 \in t$, since $d_1 \cap \cS^{\ez}(d_0) \neq \emptyset \Leftrightarrow d_0 \cap S_{-z}(d_1) \neq \emptyset$ and $(0,0,1) \cdot (\tv(d_1) \times \tv(d_0)) = (0,0,-1) \cdot (\tv(d_0) \times \tv(d_1))$.

%{ \color[rgb]{0,0,1} WHICH CLASS OF REGIONS RECEIVES THIS DEFINITION OF TWIST? }

\begin{definition}
\label{def:twist}
For a tiling $t$ of a cylinder $\cR$, we define the \emph{twist} $\Tw(t)$ as
\begin{equation*}
\Tw(t) = T^{\ex}(t) = T^{\ey}(t) = T^{\ez}(t).
%\label{eq:twistDef}
\end{equation*} 
\end{definition}

Until Section \ref{sec:differentDirections}, we will not use Proposition \ref{prop:equalTwistsMultiplex}, and will only refer to pretwists.

Let $\vu \in \Delta$, and let $\beta = (\vbeta_1,\vbeta_2,\vbeta_3) \in \bB$ be such that $\vbeta_3 = \vu$.
A region $\cR$ is said to be \emph{fully balanced with respect to} $\vu$ if for each square $Q = p + [0,2]\vbeta_1 + [0,2] \vbeta_2$, where $p \in \ZZ^3$ and $Q \subset \cR$, each of the two sets $\cA^{\vu} = \cR \cap \clS^{\vu}(Q)$ and $\cA^{-\vu} = \cR \cap \clS^{-\vu}(Q)$  contains as many black cubes as white ones. In other words,
$$ \sum_{C \subset \cA^{\vu}} \ccol(C) =  \sum_{C \subset \cA^{-\vu}} \ccol(C) = 0.$$ 
$\cR$ is \emph{fully balanced} if it is fully balanced with respect to each $\vu \in \Delta$.
% ---- Old text:
%Let $\vu \in \Delta$, and consider squares $Q \subset \pi$ of side $2$ and vertices in $\ZZ^3$, where $\pi \perp \vu$ is a basic plane. A region $\cR$ is said to be \emph{fully balanced with respect to} $\vu$ if for each such square $Q$ that is contained in $\cR$, each of the two sets $\cA^{\vu} = \cR \cap \clS^{\vu}(Q)$ and $\cA^{-\vu} = \cR \cap \clS^{-\vu}(Q)$  contains as many black cubes as white ones. 

\begin{lemma}
\label{lemma:fullyBalancedMultiplex}
Every pseudocylinder (in particular, every cylinder) is fully balanced.
\end{lemma}
\begin{proof}
Let $\cR$ be a pseudocylinder with base $\cD$ and depth $n$, let $\vu \in \Delta$ and let $Q = p_0 + [0,2]\vbeta_1 + [0,2] \vbeta_2 \subset \cR$, where $\beta \in \bB$ is such that $\vbeta_3 = \vu$ and $p_0 \in \ZZ^3$.
Consider $\cA^{\pm\vu} = \cR \cap \clS^{\pm\vu}(Q)$.
%$\pi \perp \vu$ a basic plane and $Q \subset \pi \cap \cR$ a square of side $2$ and vertices in $\ZZ^3$. 

If $\vu$ is the axis of the pseudocylinder, then $Q = Q' + k\vu$, for some square $Q' \subset \cD$ and some $0 \leq k \leq n$. Now $\cA^{\vu} = Q' + [k,n]\vu$, which clearly contains $2(n-k)$ black cubes and $2(n-k)$ white ones; similarly, $\cA^{-\vu} = Q' + [0,k]\vu$ contains $2k$ black cubes and $2k$ white ones.

If $\vu$ is perpendicular to the axis of the pseudocylinder, assume without loss of generality that $\vbeta_1$ is the axis. Let $\Pi$ denote the orthogonal projection on $\cD$, and let $\cD^{\pm} = \Pi\left(\clS^{\pm\vu}(Q)\right) \cap \cD$, which are planar regions, since they are unions of squares of $\cD$. If $p_0 - \Pi(p_0) = k \vbeta_1$, we have $\cA^{\pm\vu} = \cD^{\pm} + [k,k+2]\vbeta_1$, which clearly has the same number of black squares as white ones. 
%
%\{p \in \cD |(p - p_0) \cdot \vbeta_2 \in [0,2]\}$ which is a planar region,since it is a union of squares of $\cD$. Finally, set $k = p_0 \cdot \vbeta_1$, $\cD_0^{+} = \{p \in \cD_0 | (p - p_0) \cdot \vu \geq 0\}$, and $\cD_0^{-} = \{p \in \cD_0 | (p - p_0) \cdot \vu \leq 0\}$, so that $\cA^{\pm\vu} = \cD_0^{\pm} + [k,k+2]\vbeta_1$, which clearly has the same number of black squares as white ones.  
%
 %which we may assume without loss of generality to have axis $\ez$, so that $\cR = \cD \times [0,n]$ for some $\cD$. Take $\vu \in \Phi$ and a square $Q$ of side $2$ with vertices in $\ZZ^3$ and normal vector $\vu$.
%
%If $\vu = \ez$, then $Q \subset \cD \times \{k\}$ for some $0 \leq k \leq n$, so that $\cR \cap \overline{\cS^{\vu}(Q)} = Q \times [k,n]$, which clearly contains $2(n-k)$ black cubes and $2(n-k)$ white ones; a similar argument shows that $\cR$ is also fully balanced with respect to $-\ez$.
%
%Othewise, suppose $\vu = \ex$. It is easy to see that $\cR \cap \overline{\cS^{\vu}(Q)} = \cD_0 \times [k,k+2]$ for some $\cD_0 \subset \cD$ and some $0 \leq k \leq n - 2$; therefore, $\cR$ is fully balanced with respect to $\ex$. By symmetry, this holds for $\vu \in \{\pm \ex, \pm \ey\}$, and thus we're done.
\end{proof}
%We now need to show that the twist is a flip invariant. For what follows, we will fix $\ez$ as our main direction along which we calculate effects. We also write $\tau = \tau^{\ez}$ and effect will mean effect along $\ez$.

%To simplify the notation in the following two proofs, consider the following objects: given a domino $d^\ast$ of a tiling $t$, we set $\Tout(d^\ast) = \sum_{d \in t} \tau(d^\ast,d)$, the sum of the effects (along $\ez$) of $d^\ast$ on all the other dominoes of $t$  ($\tau(d,d) = 0$ for any domino), $\Tin(d^\ast) = \sum_{d \in t} \tau(d,d^\ast) = \sum_{d \in t} \tau^{-\ez}(d^\ast,d)$, the sum of the effects of all the other dominoes on $d^\ast$, and $T(d^\ast) = \Tout(d^\ast) + \Tin(d^\ast)$, the total effect of $d^\ast$. Notice that $$\Tw(t) = \sum_{d \in t} \Tin(d) = \sum_{d \in t} \Tout(d) = \half \sum_{d \in t} T(d)  $$

%{\color[rgb]{0,0,1} I believe this notation is conflicting with other notations, and might actually not be that helpful. Maybe it's better to remove this paragraph and use the $E_{\text{out}}$ notation from the proof of Proposition \ref{prop:tritDifference}.}

\begin{prop}
\label{prop:flipsAndTrits}
Let $\cR$ be a region that is fully balanced with respect to $\vu \in \Phi$.
\begin{enumerate}[label=\upshape(\roman*),topsep = 0.1px]
	\item \label{item:flips} If a tiling $t_1$ of $\cR$ is reached from $t_0$ after a flip, then $T^{\vu}(t_0) = T^{\vu}(t_1)$
	\item \label{item:trits} If a tiling $t_1$ of $\cR$ is reached from $t_0$ after a single positive trit, then $T^{\vu}(t_1) = T^{\vu}(t_0) + 1$.
\end{enumerate}
\end{prop}
\begin{proof}
In this proof, $\vu$ points towards the paper in all the drawings.
We begin by proving \ref{item:flips}.
Suppose a flip takes the dominoes $d_0$ and $\tilde{d_0}$ in $t_0$ to $d_1$ and $\tilde{d_1}$ in $t_1$. 
Notice that $\tv(d_0) = -\tv(\tilde{d_0})$ and $\tv(d_1) = - \tv(\tilde{d_1})$. 
For each domino $d \in t_0 \cap t_1$, define 
$$E^{\pm\vu}(d) = \tau^{\pm\vu}(d,d_1) + \tau^{\pm\vu}(d,\tilde{d_1}) - \tau^{\pm\vu}(d,d_0) - \tau^{\pm\vu}(d,\tilde{d_0}).$$ %and $E^{-\vu}(d) = \tau^{-\vu}(d,d_1) + \tau^{-\vu}(d,\tilde{d_1}) - \tau^{-\vu}(d,d_0) - \tau^{-\vu}(d,\tilde{d_0})$. 
Notice that
$$T^{\vu}(t_1) - T^{\vu}(t_0) = \sum_{d \in t_0 \cap t_1} {E^{\vu}(d) + E^{-\vu}(d)}.$$

%We may assume without loss of generality (by switching the roles of $t_0$ and $t_1$ if necessary) that $d_0$ is not parallel to $\ez$. 

\begin{case}
\label{case:d1parz}
Either $d_0$ or $d_1$ is parallel to $\vu$.
\end{case}
\begin{figure}[ht]%
\centering
\includegraphics[width=\columnwidth]{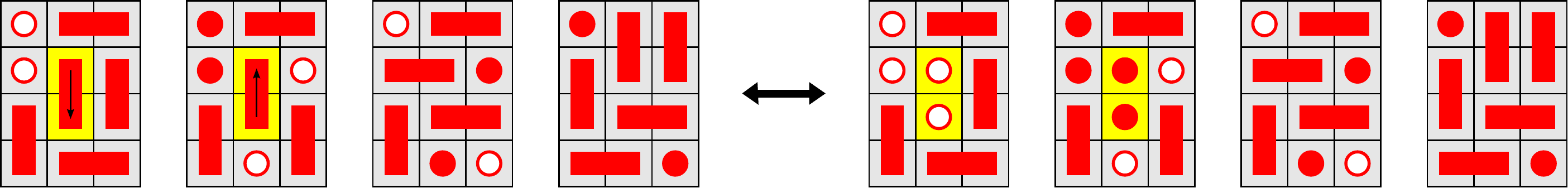}%
\caption{An example of Case \ref{case:d1parz}, where the black arrows represent $\tv(d_0)$ and $\tv(\tilde{d_0})$. It is clear that the effects of $d_0$ and $\tilde{d_0}$ cancel out.}%
\label{fig:flip_shadow_c1}%
\end{figure}
Assume, without loss of generality, that $d_1$ (and thus also $\tilde{d_1}$) is parallel to $\vu$.
By definition, $\tau^{\pm\vu}(d,d_1) = \tau^{\pm\vu}(d,\tilde{d_1}) = 0$ for each domino $d$. Now notice that $d_0$ and $\tilde{d_0}$ are parallel and in adjacent floors (see Figure \ref{fig:flip_shadow_c1}) : since $\tv(d_0) = -\tv(\tilde{d_0})$, it follows that $\tau^{\pm\vu}(d,d_0) + \tau^{\pm\vu}(d,\tilde{d_0}) = 0$ for each domino $d$, so that $E^{\pm\vu}(d) = 0$ and thus $T^{\vu}(t_1) = T^{\vu}(t_0)$. 
%
%; say $d_0$ is in the floor with smallest $z$ coordinate. Clearly (see Figure \ref{fig:flip_shadow_c1}) $\cS^{\ez}(d_0) = \cS^{\ez}(\tilde{d_0}) \cup \tilde{d_0}$, so that $\tau(d,d_0) = -\tau(d,\tilde{d_0})$ for any $d \in t_0$, and $\Tout(d_0) + \Tout(\tilde{d_0}) = 0$. A similar observation shows that $\Tin(d_0) + \Tin(\tilde{d_0}) = 0$, and hence $T(d_0) + T(\tilde{d_0}) = 0 = T(d_1) + T(\tilde{d_1})$. 
%
%In this case, clearly $d_1$ and $\tilde{d_1}$ have no effect on any domino and are not affected by any domino in $t_1$, hence we need to show that $T(d_0) + T(\tilde{d_0}) = 0$.
%
%It is clear that $d_0$ and $\tilde{d_0}$ are parallel and in adjacent floors (in the same positions in both floors). Also, $\tau(d_0,\tilde{d_0}) = \tau(\tilde{d_0}, d_0) = 0$ because they are parallel. Since they intersect the $\ez$-shadows of the exact same dominoes and they point in opposite directions, it follows that $\tau(d,d_0) = -\tau(d,\tilde{d_0})$ for any $d \in t_0$. Also, since their $\ez$-shadows intersect the exact same dominoes (other than $d_0$ or $\tilde{d_0}$), $\tau(d_0,d) = -\tau(\tilde{d_0},d)$ for any $d \in t_0$, hence it follows that the total effects of $d_0$ and $\tilde{d_0}$ cancel out, and $\Tw(t_0) = \Tw(t_1)$.

\begin{case}
\label{case:d1notparz}
Neither $d_0$ nor $d_1$ is parallel to $\vu$.
\end{case}
\begin{figure}[ht]%
\centering
\subfloat[The flip position is highlighted in yellow in both tilings, and $\cA^{\vu}$ is highlighted in green. The vectors $\tv(d)$ have been drawn for the most relevant dominoes.]{\includegraphics[width=0.4\columnwidth]{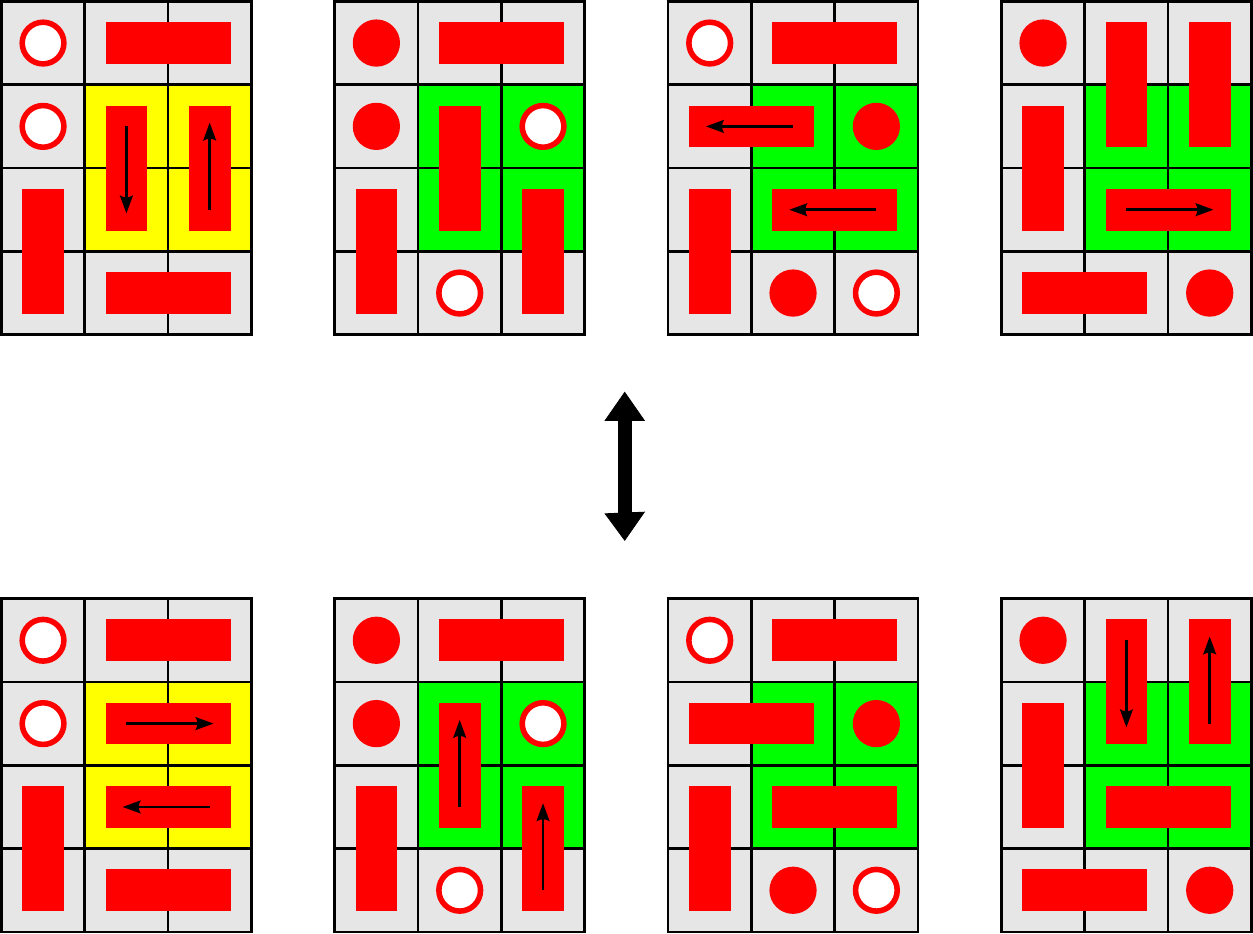}\label{fig:flip_shadow_c2}} \qquad \qquad
\subfloat[This refers to the tilings in \protect\subref{fig:flip_shadow_c2}, but only the arrows are drawn (not the dominoes). Notice that we have drawn $-\tv(d_0)$ and $-\tv(\tilde{d_0})$.]{\def\svgwidth{0.4\columnwidth}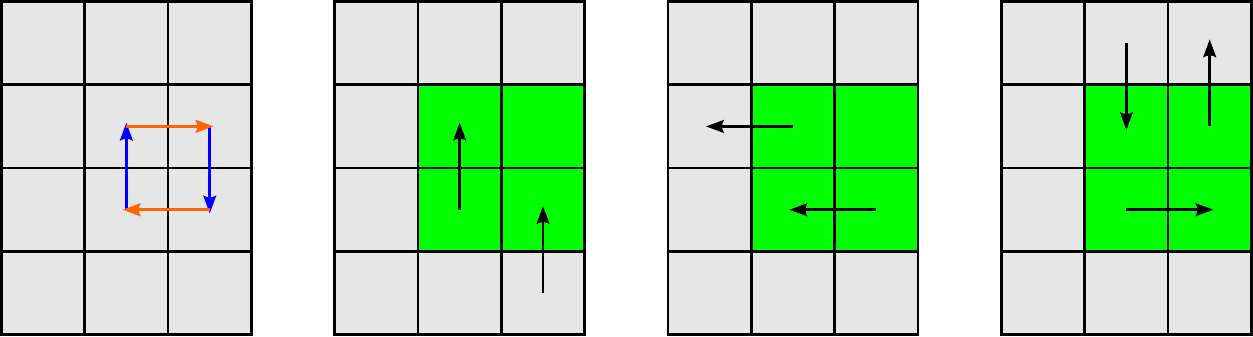\label{fig:flip_shadow_c2_scheme}}
\caption{Example of a flip in Case \ref{case:d1notparz}, together with a schematic drawing portraying $\tv(d)$ for the relevant dominoes.}
\label{fig:flip_shadow_c2_complete}
\end{figure}

In this case, $d_0 \cup \tilde{d_0} = d_1 \cup \tilde{d_1} = Q + [0,1]\vu \subset \cR$ for some square $Q$ of side $2$ and normal vector $\vu$. 

Notice that $\clS^{\vu}(d_0) \cup \clS^{\vu}(\tilde{d_0}) = \clS^{\vu}(d_1) \cup \clS^{\vu}(\tilde{d_1}) = \clS^{\vu}(Q + \vu)$; let $\cA^{\vu} = \cR \cap \clS^{\vu}(Q + \vu)$. 
%Since $\cR$ is fully balanced, the number of black cubes in $\Aout$ equals the number of white ones.% = (\cS^{\ez}(d_0) \cup \cS^{\ez}(\tilde{d_0})) \cap \cR$.

Let $d$ be a domino that is completely contained in $\Aout$: we claim that $\tau(d_0,d) + \tau(\tilde{d_0},d) = 0 = \tau(d_1,d) + \tau(\tilde{d_1},d)$. This is obvious if $d$ is parallel to $\vu$; if not, we can switch the roles of $t_0$ and $t_1$ if necessary and assume that $d$ is parallel to $d_0$, which implies that $\tau(d_0,d) = \tau(\tilde{d_0},d) = 0$. Now notice that $d$ is in the $\vu$-shades of both $d_1$ and $\tilde{d_1}$, so that $\tau(d_1,d) = - \tau(\tilde{d_1},d)$. Hence, if $d \subset \cA^{\vu}$ (or if $d \cap \cA^{\vu} = \emptyset$), $E^{-\vu}(d) = 0$.

For dominoes $d$ that intersect $\Aout$ but are not contained in it, first observe that by switching the roles of $t_0$ and $t_1$ and switching the colors of the cubes (i.e., translating) if necessary, we may assume that the vectors are as shown in Figure \ref{fig:flip_shadow_c2}. By looking at Figure \ref{fig:flip_shadow_c2_scheme} and working out the possible cases, we see that
$$
E^{-\vu}(d) = 
\begin{cases}
 -\frac{1}{4}, &\text{if } \tv(d) \text{ points into } \Aout; \\
\frac{1}{4}, &\text{if } \tv(d) \text{ points away from } \Aout.\end{cases}
 $$  
%is $\pm 1/4$ if $\tv(d)$ points away from $\Aout$, and $\mp 1/4$ if it points into $\Aout$ (in other words, they have opposite signs), as it's clear from Figure \ref{fig:flip_shadow_c2_complete}(\protect\subref{fig:flip_shadow_c2} and \protect\subref{fig:flip_shadow_c2_scheme}). 

Now for such dominoes, $\tv(d)$ points away from the region if and only if $d$ intersects a white cube of $\Aout$, and points into the region if and only if $d$ intersects a black cube in $\Aout$: hence,
$$\sum_{d \in t_0 \cap t_1} E^{-\vu}(d) = \sum_{C \subset \cA^{\vu}} (-\ccol(C)) = 0,$$ 
because $\cR$ is fully balanced with respect to $\vu$.
A completely symmetrical argument shows that $\sum_{d \in t_0 \cap t_1} E^{\vu}(d) = 0$, so we are done.

We now prove \ref{item:trits}. Suppose $t_1$ is reached from $t_0$ after a single positive trit. By rotating $t_0$ and $t_1$ in the plane $\vu^{\perp} = \{ \vw | \vw \cdot \vu = 0\}$ (notice that this does not change $T^{\vu}$), we may assume without loss of generality that the dominoes involved in the positive trit are as shown in %Figure \ref{fig:postrit}. 
Figure \ref{subfig:postrit_diff_dominoes}. Moreover, by translating if necessary, we may assume that the vectors $\tv(d)$ are as shown in Figure \ref{fig:postrit_diff}.

A trit involves three dominoes, no two of them parallel. Since dominoes parallel to $\vu$ have no effect along $\vu$, we consider only the four dominoes involved in the trit that are not parallel to $\vu$: $d_0, \tilde{d_0} \in t_0$, and $d_1, \tilde{d_1} \in t_1$. Define $E^{\pm\vu}$ with the same formulas as before.  
%Let $d_0$ and $\tilde{d_0}$ be the two non-$\ez$ dominoes of $t_0$ involved in the trit, and $d_1$ and $\tilde{d_1}$ the ones in $t_1$. We wish to show that $\Tw(t_1) - \Tw(t_0) = T(d_1) + T(\tilde{d_1}) - T(d_0) - T(\tilde{d_0}) = 1$.

By looking at Figure \ref{fig:postrit}, the reader will see that $\tau(d_0,\tilde{d_0}) + \tau(\tilde{d_0},d_0) = -1/4$ and $\tau(d_1,\tilde{d_1}) + \tau(\tilde{d_1},d_1) = 1/4$. %We now wish to consider the effect of these four dominoes on the other dominoes.

Let $D = d_0 \cup \tilde{d_0} \cup d_1 \cup \tilde{d_1}$: $\clS^{\vu}(D)$ is shown in Figure \ref{fig:postrit_diff}. %: notice that the cube marked $C$ is black.  
$D$ contains a single square $Q$ of side $2$ and normal vector $\vu$.
Define $\Aout = \clS^{\vu}(D) \cap \cR$, and notice that (see Figure \ref{fig:postrit_diff}) $\clS^{\vu}(Q) \cap \cR = \Aout \cup C_1 \cup C_2 \cup C_3$, where $C_i$ are three basic cubes: if we look at the arrows in Figure \ref{fig:postrit_diff}, we see that two of them are white and one is black. Since $\cR$ is fully balanced with respect to $\vu$, 
$$\sum_{C \subset \Aout} \ccol(C) = \sum_{C \subset \clS^{\vu}(Q) \cap \cR}{\ccol(C)} - \sum_{1 \leq i \leq 3}\ccol(C_i) = 1.$$  

% Also $\cS^{\ez}(Q) \setminus \cS^{\ez}(D)$ consists of three cubes (

%let $\Aout$ be the intersection of $\cR$ with the union of the shades of these four dominoes, minus the union of the four dominoes themselves. Hence, $\Aout = \cB \cup C$, where $\cB$ is the intersection of a $2 \times 2 \times \infty$ grid consisting of all the cubes below the position of the trit, and $C$ is a single cube in the bottom floor of the trit position (the one that is not involved in the trit itself). We can assume, without loss of generality, that $C$ is black (otherwise, translate $t_0$ and $t_1$ by, say, $\ex$): this situation is illustrated in Figure \ref{fig:postrit_diff}.

%\begin{figure}[ht]%
%\centering
%\def\svgwidth{0.5\columnwidth}%
%\input{figures/postrit_diff.pdf_tex}
%\caption{Illustration of a positive trit position: the portrayed dominoes belong to $t_0$, and the green cubes represent $S^{\vu}(D)$. The vectors $-\tv(d_0)$, $-\tv(\tilde{d_0})$, $\tv(d_1)$ and $\tv(\tilde{d_1})$ are shown.}%
%\label{fig:postrit_diff}%
%\end{figure} 

\begin{figure}[ht]%
\centering
\subfloat[]{\includegraphics[width=0.48\columnwidth]{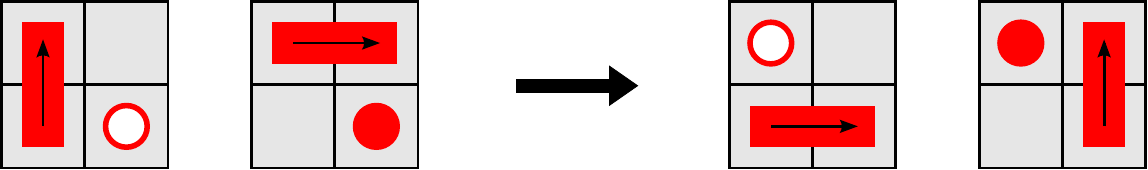}\label{subfig:postrit_diff_dominoes}}
\qquad
\subfloat[]{
\def\svgwidth{0.45\columnwidth}%
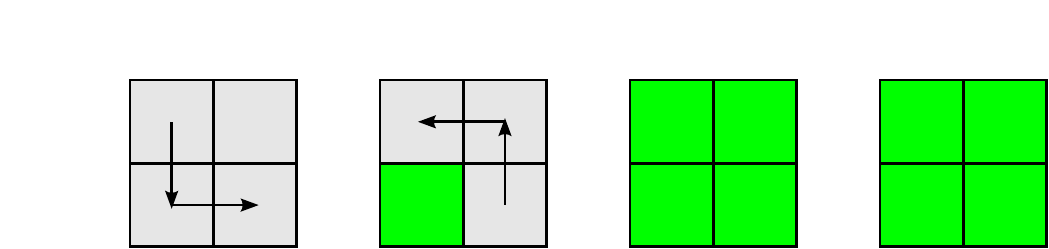\label{subfig:postrit_diff_effects}}
\caption{Illustration of a positive trit position. In \protect\subref{subfig:postrit_diff_dominoes}, the dominoes and the corresponding vectors $\tv(d)$ are shown, while in \protect\subref{subfig:postrit_diff_effects}, the highlighted cubes represent $\clS^{\vu}(D)$, and the vectors $-\tv(d_0)$, $-\tv(\tilde{d_0})$, $\tv(d_1)$ and $\tv(\tilde{d_1})$ are shown.}%
\label{fig:postrit_diff}%
\end{figure} 

By looking at Figure \ref{fig:postrit_diff}, we see that we have a situation that is very similar to Figure \ref{fig:flip_shadow_c2_scheme}; for each $d \in t_0 \cap t_1$, we have 
$$E^{-\vu}(d) = \begin{cases} 0, &\text{if } d \subset \Aout \text{ or } d \cap \Aout = \emptyset;\\
 \frac{1}{4}, &\text{if } \tv(d) \text{ points into } \Aout; \\
-\frac{1}{4}, &\text{if } \tv(d) \text{ points away from } \Aout\end{cases} $$  
(when we say that $\tv(d)$ points into or away from $\Aout$, we are assuming that $d$ intersects one cube of $\Aout$).
Hence, $$\sum_{d \in t_0 \cap t_1}E^{-\vu}(d) = \frac{1}{4} \sum_{C \subset \cA^{\vu}}\ccol(C) = \frac{1}{4}.$$

A completely symmetrical argument shows that $\sum_{d \in t_0 \cap t_1}E^{\vu}(d) = 1/4$, and hence
\begin{align*}
&T^{\vu}(t_1) - T^{\vu}(t_0) = (\tau(d_1,\tilde{d_1}) + \tau(\tilde{d_1},d_1)) - (\tau(d_0,\tilde{d_0}) + \tau(\tilde{d_0},d_0))\\ &+ \sum_{d \in t_0 \cap t_1}E^{-\vu}(d) + \sum_{d \in t_0 \cap t_1}E^{\vu}(d)
= \frac{1}{4} + \frac{1}{4} + \frac{1}{4} + \frac{1}{4}  = 1,
\end{align*}
which completes the proof.   
\end{proof}

\section{Topological groundwork for the twist}
\label{sec:topologicalGroundwork}

%In this section, we introduce a dual viewpoint on objects previously defined in Section \ref{sec:combTwistBoxes} 
In this section, we develop a topological interpretation of tilings and twists. Dominoes are (temporarily) replaced by dimers, which, although formally different objects, are really just a different way of looking at dominoes. Although we will tend to work with dimers in this and the following section, we may in later sections switch back and forth between these two viewpoints.

We will now give some useful (but somewhat technical) definitions. The reader that is familiar with perfect matchings in graphs may find it easier to follow by recalling that domino tilings of a region can be seen as perfect matchings of a dual bipartite graph, and that the symmetric difference of two perfect matchings of the same graph is a disjoint set of cycles: we want to look at these cycles as oriented curves in $\RR^3$.    

Let $\cR$ be a region. A \emph{segment} $\ell$ of $\cR$ is a straight line of unit length connecting the centers of two cubes of $\cR$; in other words, $\ell:[0,1] \to \RR^3$ with $\ell(s) = p_0 + (p_1 - p_0)s$, where $p_0$ and $p_1$ are the centers of two cubes that share a face: this segment is a \emph{dimer} if $p_0 = \ell(0)$ is the center of a white cube. We define $\tv(\ell) = \ell(1) - \ell(0)$ (compare this with the definition of $\tv(d)$ for a domino $d$). If $\ell$ is a segment, $(-\ell)$ denotes the segment $s \mapsto \ell(1-s)$: notice that either $\ell$ or $-\ell$ is a dimer.

Two segments $\ell_0$ and $\ell_1$ are \emph{adjacent} if $\ell_0 \cap \ell_1 \neq \emptyset$ (here we make the usual abuse of notation of identifying a curve with its image in $\RR^3$); nonadjacent segments are \emph{disjoint}. In particular, a segment is always adjacent to itself. 

A \emph{tiling} of $\cR$ by dimers is a set of pairwise disjoint dimers such that the center of each cube of $\cR$ belongs to exactly one dimer of $t$. If $t$ is a tiling, $(-t)$ denotes the set of segments $\{-\ell | \ell \in t\}$.   

Given a map $\gamma:[m,n] \to \RR^3$, a segment $\ell$ and an integer $k \in [m,n-1]$, we abuse notation by making the identification $\gamma|_{[k,k+1]} = \ell$ if $\gamma(s) = \ell(s-k)$ for each $s \in [k,k+1]$. 
A \emph{curve} of $\cR$ is a map $\gamma:[0,n] \to \RR^3$ such that $\gamma|_{[k,k+1]}$ is (identified with) a segment of $\cR$ for $k=0,1,\ldots,n-1$.
We make yet another abuse of notation by also thinking of $\gamma$ as a sequence or set of segments of $\cR$, and we shall write $\ell \in \gamma$ to denote that $\ell = \gamma|_{[k,k+1]}$ for some $k$. 

A curve $\gamma:[0,n] \to \RR^3$ of $\cR$ is \emph{closed} if $\gamma(0) = \gamma(n)$; it is \emph{simple} if $\gamma$ is injective in $[0,n)$. A closed curve $\gamma:[0,2] \to \RR^3$ of $\cR$ is called \emph{trivial}: notice that, in this case, $\gamma|_{[0,1]} = -(\gamma|_{[1,2]})$ (when identified with their respective segments of $\cR$). A \emph{discrete rotation} on $[0,n]$ is a function $\rho:[0,n] \to [0,n]$ with $\rho(s) = (s + k) \bmod n$, for a fixed $k \in \ZZ$. If $\gamma_0:[0,n] \to \RR^3$ and $\gamma_1:[0,m] \to \RR^3$ are two closed curves, we say $\gamma_0 = \gamma_1$ if $n = m$ and $\gamma_1 = \gamma_0 \circ \rho$ for some discrete rotation $\rho$ on $[0,n]$.

Given two tilings $t_0$ and $t_1$, there exists a unique (up to discrete rotations) finite set of disjoint closed curves $\Gamma(t_0,t_1) = \{\gamma_i | 1 \leq i \leq m\}$ such that $t_0 \cup (-t_1) = \{ \ell | \ell \in \gamma_i \text{ for some } i\}$ and such that every nontrivial $\gamma_i$ is simple. Figure \ref{fig:auxiliaryLinesExample} shows an example. We define $\Gamma^*(t_0,t_1) := \{\gamma \in \Gamma(t_0,t_1) | \gamma \text{ nontrivial } \}$. 

Translating effects from the world of dominoes to the world of dimers is relatively straightforward. For $\vu \in \Phi$, $\Pi^{\vu}$ will denote the orthogonal projection on the plane $\pi^{\vu} = \vu^{\perp} = \{\vw \in \RR^3 | \vw \cdot \vu = 0\}$. Given two segments $\ell_0$ and $\ell_1$, we set:

$$
\tau^{\vu}(\ell_0, \ell_1) = 
\begin{cases}
\frac{1}{4} \det(\tv(\ell_1), \tv(\ell_0), \vu), &\Pi^{\vu}(\ell_0) \cap \Pi^{\vu}(\ell_1) \neq \emptyset, \ell_0(0) \cdot \vu < \ell_1(0) \cdot \vu; \\
0, &\text{otherwise.} 
\end{cases}
\label{def:effectSegments}
$$
Notice that this definition is analogous to the one given in Section \ref{sec:combTwistBoxes} for dominoes.

%For what follows, we need the concept of linking number $\Link(\gamma_0,\gamma_1)$ of two disjoint closed simple curves of $\RR^3$ (see, e.g., \cite[pp. 18--19]{knotbook}). Also, recall that each crossing in a projection has a sign (see, e.g., \cite[p. 19]{knotbook}); this sign can be calculated via a modified right-hand rule, as follows: let $\vec{v_0}$ and $\vec{v_1}$, respectively, denote the tangent vectors of the segments that are above and below in the crossing. The sign of the crossing is given by $\det(\vec{v_1}, \vec{v_0}, \ez)$ (recall that $\ez$ points towards the paper).  

The definition of $\tau^{\vu}$ is given in terms of the orthogonal projection $\Pi^{\vu}$. From a topological viewpoint, however, this projection is not ideal, because it gives rise to nontransversal intersections between projections of segments. In order to solve this problem, we consider small perturbations of these projections.

%the perturbations $\Pi_{a,b}$, for small $(a,b) \neq (0,0).$ Before we can comfortably work with these objects, however, we need a few technical Lemmas. %This extra work will pay off, however, as some nontrivial consequences will follow.

Recall that $\bB$ is the set of positively oriented basis $\beta = (\vbeta_1,\vbeta_2,\vbeta_3)$ with vectors in $\Phi$. If $\beta \in \bB$ and $a,b \in \RR$, $\Pi^{\beta}_{a,b}$ will be used to denote the projection on the plane $\pi^{\vbeta_3} = \vbeta_3^{\perp} = \{\vu \in \RR^3 | \vu \cdot \vbeta_3 = 0\}$ whose kernel is the subspace (line) generated by the vector $\vbeta_3 + a \vbeta_1 + b \vbeta_2$. For instance, if $\beta = (\ex, \ey, \ez)$ is the canonical basis, $\Pi^{\beta}_{a,b}(x,y,z) = (x - az, y - bz, 0)$. 
Notice that $\Pi^{\beta}_{0,0} = \Pi^{\vbeta_3}$ is the orthogonal projection on the plane $\pi^{\vbeta_3}$, and, for small $(a,b) \neq (0,0)$, $\Pi^{\beta}_{a,b}$ is a nonorthogonal projection on $\pi^{\vbeta_3}$ which is a slight perturbation of $\Pi^{\vbeta_3}$.

Given $\beta \in \bB$, $\vu = \vbeta_3$ and small nonzero $a,b \in \RR$, set the \emph{slanted effect}
$$
\tau^{\beta}_{a,b}(\ell_0,\ell_1) =
\begin{cases}
\det(\tv(\ell_1), \tv(\ell_0), \vu), &\Pi^{\beta}_{a,b}(\ell_0) \cap \Pi^{\beta}_{a,b}(\ell_1) \neq \emptyset, \vu \cdot \ell_0(0)  < \vu \cdot \ell_1(0);\\
0, &\text{otherwise.} 
\end{cases}
\label{def:slantedEffect}
$$ 

Recall from knot theory the concept of crossing (see, e.g., \cite[p.18]{knotbook}). Namely, if $\gamma_0: I_0 \to \RR^3$, $\gamma_1: I_1 \to \RR^3$ are two continuous curves, $s_j \in \interior(I_j)$ and $\Pi$ is a projection from $\RR^3$ to a plane, then $(\Pi, \gamma_0, s_0, \gamma_1, s_1)$ is a \emph{crossing} if $\gamma_0(s_0) \neq \gamma_1(s_1)$ but $\Pi(\gamma_0(s_0)) = \Pi(\gamma_1(s_1))$. If, furthermore, $\gamma_j$ is of class $C^1$ in $s_j$ and the vectors $\gamma_1'(s_1)$, $\gamma_0'(s_0)$ and $\gamma_1(s_1) - \gamma_0(s_0)$ are linearly independent, then the crossing is \emph{transversal}; its \emph{sign} is the sign of $\det(\gamma_1'(s_1), \gamma_0'(s_0),\gamma_1(s_1) - \gamma_0(s_0))$. We are particularly interested in the case where the curves are segments of a region $\cR$.

%If $\ell_0$ and $\ell_1$ are two segments of a region $\cR$ and $s_0, s_1 \in [0,1]$,
%we say that $(\Pi^{\beta}_{a,b}, \ell_0, s_0, \ell_1, s_1)$ is a \emph{crossing} if 
%$\ell_0(s_0) \neq \ell_1(s_1)$ but $\Pi^{\beta}_{a,b}(\ell_0(s_0)) = \Pi^{\beta}_{a,b}(\ell_1(s_1))$. If, in addition, $\tv(\ell_0) \perp \tv(\ell_1)$ and $s_0,s_1 \in (0,1)$, we say that the crossing is \emph{transversal}. 
%
%If $\ell_0$ and $\ell_1$ are involved in a transversal crossing, the pair $(s_0,s_1)$ is unique, and the point $p = \Pi^{\beta}_{a,b}(\ell_0(s_0)) = \Pi^{\beta}_{a,b}(\ell_1(s_1))$ is then a crossing in the usual topological sense.
%The sign of this crossing is given by a modified right-hand rule (see, e.g., \cite[\S 3]{writhingNumber}), as follows: let $\vw = \vbeta_3 + a \vbeta_1 + b \vbeta_2$. Notice that $\vw \cdot \ell_0(s_0) \neq \vw \cdot \ell_1(s_1)$: if $\vw \cdot \ell_0(s_0) < \vw \cdot \ell_1(s_1)$, the sign of the crossing is the sign of $\det(\tv(\ell_1), \tv(\ell_0), \vw)$ and therefore equals $\pm 1$. 

For a region $\cR$ and $\vu \in \Phi$, we define the $\vu$-\emph{length} of $\cR$ as
 $$N = \max_{p_0, p_1 \in R} |\vu \cdot (p_0 - p_1)|.$$

\begin{lemma}
\label{lemma:transversalCrossings}
Let $\cR$ be a region, and fix $\beta \in \bB$. Let $N$ be the $\vbeta_3$-length of $\cR$, and let $a,b \in \RR$ with $0 < |a|,|b| < 1/N$.
Then $\tau^{\beta}_{a,b}(\ell_0,\ell_1) + \tau^{\beta}_{a,b}(\ell_1,\ell_0) \neq 0$ if and only if there exist $s_0, s_1 \in [0,1]$ such that $\ell_0(s_0) \neq \ell_1(s_1)$ but $\Pi^{\beta}_{a,b}(\ell_0(s_0)) = \Pi^{\beta}_{a,b}(\ell_1(s_1))$. 

Moreover, if the latter condition holds for $s_0, s_1$, then $(\Pi^{\beta}_{a,b},\ell_0, s_0, \ell_1, s_1)$ is a transversal crossing whose sign is given by $\tau^{\beta}_{a,b}(\ell_0,\ell_1) + \tau^{\beta}_{a,b}(\ell_1,\ell_0)$.

%If $(\Pi^{\beta}_{a,b},\ell_0, s_0, \ell_1, s_1)$ is a crossing, then it is transversal.
%Moreover, a transversal crossing occurs if and only if $\tau^{\beta}_{a,b}(\ell_0,\ell_1) + \tau^{\beta}_{a,b}(\ell_1,\ell_0) \neq 0$, and the sign of the crossing is given by $\tau^{\beta}_{a,b}(\ell_0,\ell_1) + \tau^{\beta}_{a,b}(\ell_1,\ell_0)$.
%Moreover, the sign of this crossing is given by $\tau^{\beta}_{a,b}(\ell_0,\ell_1) + \tau^{\beta}_{a,b}(\ell_1,\ell_0)$.
%
%
 %
 %Then, $\ell_0$ and $\ell_1$ cross transversally in $\Pi^{\beta}_{a,b}$ if and only if $\tau^{\beta}_{a,b}(\ell_0,\ell_1) + \tau^{\beta}_{a,b}(\ell_1,\ell_0) \neq 0$.
%Moreover, if $\ell_0$ and $\ell_1$ cross transversally, the sign of the crossing is given by $\tau^{\beta}_{a,b}(\ell_0,\ell_1) + \tau^{\beta}_{a,b}(\ell_1,\ell_0)$.
%\begin{enumerate}[label=\upshape(\roman*)]
	%\item $\tau^{\beta}_{a,b}(\ell_0,\ell_1) + \tau^{\beta}_{a,b}(\ell_1,\ell_0) \neq 0$;
	%\item $\Pi(\ell_0)$ and $\Pi(\ell_1)$ intersect transversally.%, and the sign of the crossing is given by $\tau^{\beta}_{a,b}(\ell_0,\ell_1) + \tau^{\beta}_{a,b}(\ell_1,\ell_0)$.
%\end{enumerate}
\end{lemma}
\begin{proof}
%$(\Pi^{\beta}_{a,b},\ell_0, s_0, \ell_1, s_1)$ being a crossing can be rephrased as
Suppose $\tau^{\beta}_{a,b}(\ell_0,\ell_1) + \tau^{\beta}_{a,b}(\ell_1,\ell_0) \neq 0$. We may without loss of generality assume $\tau^{\beta}_{a,b}(\ell_0,\ell_1) \neq 0$. By definition, we have $\Pi^{\beta}_{a,b}(\ell_0(s_0)) = \Pi^{\beta}_{a,b}(\ell_1(s_1))$ for some $s_0, s_1 \in [0,1]$ and $\vbeta_3 \cdot \ell_0(0) < \vbeta_3 \cdot \ell_1(0)$. Since $\det(\tv(\ell_1), \tv(\ell_0), \vbeta_3) \neq 0$, we have
$$\vbeta_3 \cdot \ell_0(s_0) = \vbeta_3 \cdot (\ell_0(0) + s_0 \tv(\ell_0)) = \vbeta_3 \cdot \ell_0(0) <\vbeta_3 \cdot \ell_1(0) = \vbeta_3 \cdot \ell_1(s_1),$$
and thus $\ell_0(s_0) \neq \ell_1(s_1)$. 

Conversely, suppose $\ell_0(s_0) \neq \ell_1(s_1)$ but $\Pi^{\beta}_{a,b}(\ell_0(s_0)) = \Pi^{\beta}_{a,b}(\ell_1(s_1))$: this can be rephrased as
\begin{equation}
\ell_1(s_1) - \ell_0(s_0) = c (\vbeta_3 + a \vbeta_1 + b \vbeta_2)
\label{eq:projectionKernel}
\end{equation}
for some $c \neq 0$. Notice that $c = \vbeta_3 \cdot (\ell_1(s_1) - \ell_0(s_0))$, so that $|c| \leq N$.

We now observe that $\det(\tv(\ell_1), \tv(\ell_0), \vbeta_3) \neq 0$. Suppose, by contradiction, that $\det(\tv(\ell_1), \tv(\ell_0), \vbeta_3) = 0$. Then, at least one of the following statements must be true: $\vbeta_1 \cdot \tv(\ell_0) = \vbeta_1 \cdot \tv(\ell_1) = 0$; or $\vbeta_2 \cdot \tv(\ell_0) = \vbeta_2 \cdot \tv(\ell_1) = 0$. Assume that the first statement holds (i.e., $\vbeta_1 \cdot \tv(\ell_i) = 0$). 
By definition of segment, 
$\ell_i(s_i) = \ell_i(0) + s_i \tv(\ell_i)$. By taking the inner product with $\vbeta_1$ on both sides of \eqref{eq:projectionKernel}, $a c = \vbeta_1 \cdot (\ell_1(s_1) - \ell_0(s_0)) = \vbeta_1 \cdot (\ell_1(0) - \ell_0(0))$. Now $\ell_0(0), \ell_1(0) \in (\plshalf{\ZZ})^3$, so that $a c = \vbeta_1 \cdot (\ell_1(0) - \ell_0(0)) \in \ZZ$. Since $|a| < 1/N$, $|ac| < 1$ and thus $c = 0$, which is a contradiction.  

Finally, since  $\vbeta_3 \cdot \tv(\ell_0) = \vbeta_3 \cdot \tv(\ell_1) = 0$, we have $\vbeta_3 \cdot (\ell_1(0) - \ell_0(0)) = \vbeta_3 \cdot (\ell_1(s_0) - \ell_0(s_1)) = c \neq 0$. From the definition of $\tau^{\beta}_{a,b}$, we see that $\tau^{\beta}_{a,b}(\ell_0, \ell_1) + \tau^{\beta}_{a,b}(\ell_1, \ell_0) \neq 0$. 

To see the last claim, we first note that $s_i \in (0,1)$: since $\tv(\ell_i) \in \{\pm\vbeta_1, \pm \vbeta_2\}$, we may take the inner product with $\tv(\ell_i)$ on both sides of \eqref{eq:projectionKernel} to get that $s_i$ equals  either $|a c|$ or $|bc|$, and hence $s_i \in (0,1)$. Since $\tv(\ell_0) \perp \tv(\ell_1)$, this proves that $(\Pi^{\beta}_{a,b},\ell_0, s_0, \ell_1, s_1)$ is a transversal crossing. If $\vw = \vbeta_3 + a\vbeta_1 + b\vbeta_2$, the sign of this crossing is given by the sign of $\det(\tv(\ell_1), \tv(\ell_0), c\vw)$. By switching the roles of $\ell_0$ and $\ell_1$ if necessary, we may assume that $c > 0$, so that this sign equals $\det(\tv(\ell_1), \tv(\ell_0), \vw) = \det(\tv(\ell_1), \tv(\ell_0), \vbeta_3) = \tau_{a,b}(\ell_0,\ell_1)$, completing the proof. 
\end{proof}

\begin{lemma}
\label{lemma:crossings}
Let $\cR$ be a region, and let $\beta \in \bB$. Let $N$ denote the $\vbeta_3$-length of $\cR$, and suppose $0 < \epsilon < 1/N$. Given two segments $\ell_0$ and $\ell_1$,
$$\tau^{\vbeta_3}(\ell_0,\ell_1) = \frac{1}{4}\sum_{i,j \in \{-1,1\}} \tau^{\beta}_{i\epsilon, j\epsilon}(\ell_0,\ell_1) .$$
\end{lemma}  
 
\begin{proof}
We may assume that $\vbeta_3 \cdot \ell_0(0) < \vbeta_3 \cdot \ell_1(0)$ and that $\det(\tv(\ell_1), \tv(\ell_0), \vbeta_3) \neq 0$ (otherwise both sides would be zero). Since rotations in the $\vbeta_3^{\perp}$ plane leave both sides unchanged, we may assume that $\tv(\ell_1) = \pm\vbeta_1$, $\tv(\ell_0) = \pm\vbeta_2$ (see Figure \ref{fig:projectionDimerCrossings}).

\begin{figure}[ht]%
\centering
\includegraphics[width=0.9\columnwidth]{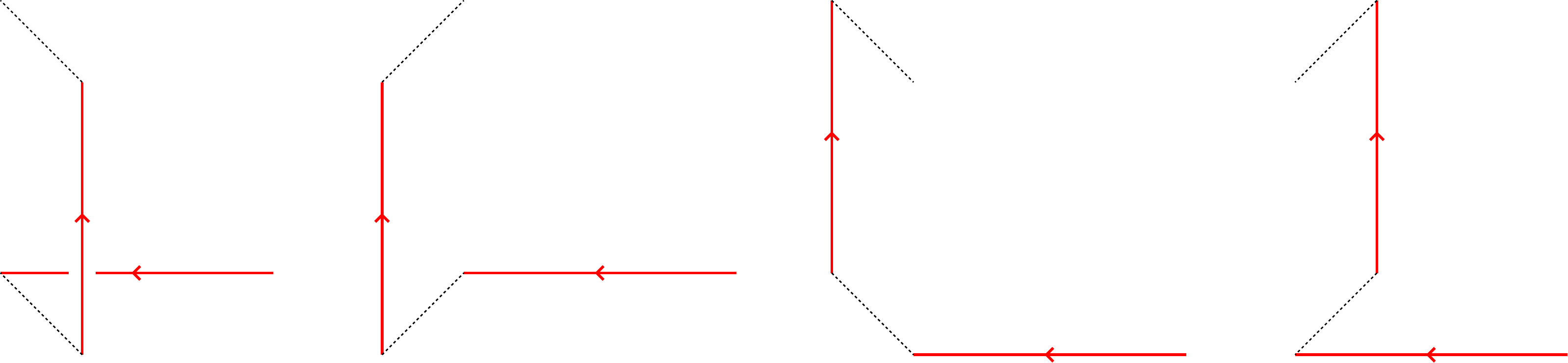}%
\caption{Illustrations of the four different projections $\Pi^{\beta}_{\pm\epsilon, \pm \epsilon}$ of two segments $\ell_0, \ell_1$ with $\tau^{\vbeta_3}(\ell_0,\ell_1) = 1/4$. The dotted lines represent the projection of lines which are parallel to $\vbeta_3$, in each of the four cases. Notice that the segments are involved in a crossing for exactly one of the projections, and this crossing is positive.}%
\label{fig:projectionDimerCrossings}%
\end{figure}

Our strategy is to show these two facts:
\begin{enumerate}[label=(\roman*)]
	\item \label{item:epsImpliesZero} If $\tau^{\beta}_{i\epsilon, j\epsilon}(\ell_0,\ell_1) \neq 0$ for some $(i,j) \in \{-1,1\}^2$, then $\tau^{\vbeta_3}(\ell_0,\ell_1) \neq 0$.
	\item \label{item:zeroImpliesEps} If $\tau^{\vbeta_3}(\ell_0,\ell_1) \neq 0$, then there exists a unique $(i,j) \in \{-1,1\}^2$ such that  $\tau^{\beta}_{i\epsilon, j\epsilon}(\ell_0,\ell_1) \neq 0$. 
\end{enumerate}
Once we prove \ref{item:epsImpliesZero} and \ref{item:zeroImpliesEps}, we get the result. 

Let $c = \vbeta_3 \cdot (\ell_1(0) - \ell_0(0))$, and consider the closed sets 
$$
A_{ij} = \left\{\delta \in [0,\epsilon] | \exists s_0,s_1 \in [0,1],  \ell_1(s_1) - \ell_0(s_0) = c(\vbeta_3 + i \delta \vbeta_1 + j \delta \vbeta_2)\right\}.
$$
Notice that $\epsilon \in A_{ij}$ if and only if $\tau^{\beta}_{i\epsilon, j\epsilon}(\ell_0,\ell_1) \neq 0$, and $0 \in A_{ij}$ if and only if $\tau^{\vbeta_3}(\ell_0,\ell_1) \neq 0$.

Suppose $\epsilon \in A_{ij}$ for some $(i,j) \in \{-1,1\}^2$, and let $\delta = \min A_{ij}$. If $\delta > 0$, $\ell_1(s_1) - \ell_0(s_0) = c(\vbeta_3 + i \delta \vbeta_1 + j \delta \vbeta_2)$ implies, by Lemma \ref{lemma:transversalCrossings},  that $s_0,s_1 \in (0,1)$. Hence, there must exist $\delta' < \delta$ such that $\delta' \in A$, a contradiction. Therefore, we must have $\delta = 0$, so that $0 \in A_{ij}$. We have proved \ref{item:epsImpliesZero}.

Now suppose $\tau^{\vbeta_3}(\ell_0,\ell_1) \neq 0$, that is, $\ell_1(k_1) - \ell_0(k_0) = c\vbeta_3$ for some $k_0,k_1 \in [0,1]$. Clearly $k_0, k_1 \in \{0,1\}$; for simplicity, assume that $k_0 = k_1 = 0$ (the other cases are analogous). Now for any $s_0, s_1 \in [0,1]$,
$\ell_1(s_1) - \ell_0(s_0) = c\vbeta_3 - s_0 \tv(\ell_0) + s_1 \tv(\ell_1)$. Thus, given $(i,j) \in \{-1,1\}^2$,
$$\epsilon \in A_{ij} \Leftrightarrow \exists s_0,s_1 \in [0,1]: \quad s_1 (\tv(\ell_1) \cdot \vbeta_1) = i \epsilon c ,\quad -s_0 (\tv(\ell_0) \cdot \vbeta_2) = j \epsilon c,$$
which occurs if and only if $i \epsilon c (\tv(\ell_1) \cdot \vbeta_1) > 0$ and $j \epsilon c (\tv(\ell_0) \cdot \vbeta_2) < 0$: this determines a unique $(i,j) \in \{-1,1\}^2$, so we have proved \ref{item:zeroImpliesEps}.
\end{proof}

%We now finally begin to shift our focus from the combinatorial formula of the twist to its topological interpretation:
If $A_0$ and $A_1$ are two sets of segments (curves are also seen as sets of segments), $\vu \in \Phi, \beta \in \bB$, define
$$T^{\vu}(A_0,A_1) = \sum_{\substack{\ell_0 \in A_0 \\ \ell_1 \in A_1}} \tau^{\vu}(\ell_0,\ell_1), \quad T^{\beta}_{a,b}(A_0,A_1) = \sum_{\substack{\ell_0 \in A_0 \\ \ell_1 \in A_1}} \tau^{\beta}_{a,b}(\ell_0,\ell_1).\label{def:effectsTwoParams}$$
For shortness, $T^{\vu}(A) = T^{\vu}(A,A)$ and $T^{\beta}_{a,b}(A) = T^{\beta}_{a,b}(A,A)$.

Consider two disjoint simple closed curves $\gamma_0, \gamma_1$ and a projection $\Pi$ from $\RR^3$ to some plane. Assume there exists finitely many crossings $(\Pi,\gamma_0, s_0, \gamma_1, s_1)$, all transversal.
Recall from knot theory (see, e.g., \cite[pp. 18--19]{knotbook}) that the \emph{linking number} $\Link(\gamma_0,\gamma_1)$ equals half the sum of the signs of all these crossings.

\begin{lemma}
\label{lemma:linkingNumber}
Let $\gamma_0$ and $\gamma_1$ be two disjoint simple closed curves of a region $\cR$. Fix $\beta \in \bB$, and let $N$ denote the $\vbeta_3$-length of $\cR$. Then
\begin{enumerate}[label=\upshape(\roman*),topsep=0.1em]
\item \label{item:linking_ab} If $0 < |a|,|b| < 1/N$, $T^{\vbeta}_{a,b}(\gamma_0, \gamma_1) + T^{\vbeta}_{a,b}(\gamma_1, \gamma_0) = 2\Link(\gamma_0,\gamma_1)$. 
\item \label{item:linking_00} $T^{\vbeta_3}(\gamma_0,\gamma_1) + T^{\vbeta_3}(\gamma_1,\gamma_0) = 2\Link(\gamma_0,\gamma_1).$
\end{enumerate}
\end{lemma}

\begin{proof}
By Lemma \ref{lemma:transversalCrossings}, the sum of signs of the crossings is given by $T^{\beta}_{a,b}(\gamma_0,\gamma_1) + T^{\beta}_{a,b}(\gamma_1,\gamma_0)$, which establishes \ref{item:linking_ab}. Also, \ref{item:linking_00} follows from \ref{item:linking_ab} and Lemma \ref{lemma:crossings}. 
%$$T^{\vbeta_3}(\gamma_0,\gamma_1) + T^{\vbeta_3}(\gamma_1,\gamma_0) = \frac{1}{4} \sum_{i,j \in \{-1,1\}} (T^{\beta}_{i\epsilon,j\epsilon}(\gamma_0,\gamma_1) + T^{\beta}_{i\epsilon,j\epsilon}(\gamma_1,\gamma_0)) = 2 \Link(\gamma_0,\gamma_1).$$
\end{proof}

%Lemma \ref{lemma:linkingNumber} has a number of important consequences, as it shall become clear in the following sections. 
%in Section \ref{sec:writheInterpretation}, it will help us obtain a new topological formula for the twist of cylinders, which, in particular, allows us to show that it is always as integer. In Section \ref{sec:additiveProperties}, it will be used to show some additive properties of the twist, which allow us to extend it to a much broader class of regions. 
%---------------------\input{tentative_new_facts}
%For what follows, we will need the concept of directional writhing number, or writhe, of a smooth simple closed space curve $\gamma$ in the direction of a nonzero vector $\vu$ (see \cite[\S 3]{writhingNumber}), which is given by $\Link(\gamma, \gamma + \delta \vu)$ for a sufficiently small $\delta > 0$ (provided $\gamma'(s)$ is never parallel to $\vu$): this is equal to the sum of the signs of the crossings in the orthogonal projection of $\gamma$ in the direction of $\vu$. 

\begin{lemma}
\label{lemma:intersectionOfSegments}
Let $\ell_0$ and $\ell_1$ be two segments of $\cR$, and let $\vu \in \RR^3$ be a vector such that $\|\vu\| < 1$. Then these two statements are equivalent:
\begin{enumerate}[label=\upshape(\roman*)]
	\item \label{item:intersection} There exist $s_0,s_1 \in [0,1]$ such that $\ell_0(s_0) - \ell_1(s_1) = \vu$.
	\item \label{item:conditions} There exist $(i,j) \in \{0,1\}^2$ and $a_0,a_1 \in (-1,1)$ such that
	$\ell_0(i) = \ell_1(j)$ and $\vu = a_0 \tv(\ell_0) + a_1 \tv(\ell_1)$ with $(-1)^i a_0 \geq 0$ and $(-1)^j a_1 \leq 0$.
\end{enumerate}
\end{lemma}
\begin{proof}
First, suppose \ref{item:intersection} holds. If $\ell_0$ and $\ell_1$ are not adjacent, then $\dist(\ell_0,\ell_1) \geq 1 > \|\vu\|$, which is a contradiction. Thus, $\ell_0$ and $\ell_1$ are adjacent, and thus $\ell_0(i) = \ell_1(j)$ for some $(i,j) \in \{0,1\}^2$: then 
\begin{align*}
\vu = \ell_0(s_0) - \ell_1(s_1) &= [\ell_0(i) + (s_0 - i) \tv(\ell_0)] - [\ell_1(j) + (s_1 - j)\tv(\ell_1)]\\
 &= (s_0 - i) \tv(\ell_0) + (j-s_1)\tv(\ell_1),
\end{align*}
that is, $\vu = a_0 \tv(\ell_0) + a_1 \tv(\ell_1)$ with $(i + a_0), (j - a_1) \in [0,1]$, which implies that $(-1)^i a_0 \geq 0$ and $(-1)^j a_1 \leq 0$. Also, since $\|\vu\| < 1$, we can take $a_0,a_1 \in (-1,1)$.

For the other direction, suppose \ref{item:conditions} holds, so that $\ell_0(i) = \ell_1(j)$ for some $(i,j) \in \{0,1\}^2$. Then setting  $s_0 = (i + a_0)$ and $s_1 = (j - a_1)$, we have $s_0,s_1 \in [0,1]$ and 
$\ell_0(i + a_0) - \ell_1(j - a_1) = [\ell_0(i) + a_0\tv(\ell_0)] - [\ell_1(j) - a_1\tv(\ell_1)] = \vu. $   
\end{proof}

For a map $\gamma:[0,n] \to \RR^3$ and a vector $\vu \in \RR^3$, let $(\gamma + \vu):[0,1] \to \RR^3: s \mapsto \gamma(s) + \vu$ denote the translation of $\gamma$ by $\vu$. 
%also, if $\ell$ is a segment of $\gamma$, $\ell + (a,b,c)$ will denote the translation of $\ell$ by $(a,b,c)$ (which is not necessarily a segment of $\cR$).

\begin{lemma}
\label{lemma:translationIntersection}
Let $\gamma$ be a curve of $\cR$, let $\beta \in \bB$, and let $\vu = a \vbeta_1 + b \vbeta_2 + c \vbeta_3 \in \RR^3$. If $\|\vu\| < 1$ and $abc \neq 0$, then the curves $\gamma$ and $\gamma + \vu$ are disjoint.  
\end{lemma}
Notice that $\gamma + \vu$ is not a curve of $\cR$.
\begin{proof}
Suppose, by contradiction, that there exist $s_0, s_1 \in [0,n]$ (the domain of $\gamma$) such that $\gamma(s_0) = \gamma(s_1) + \vu.$ Let $k_0, k_1 \in \ZZ$ be such that $k_i \leq s_i \leq k_i + 1 \leq n$, and set $\tilde{s}_i = s_i - k_i$. Since $\gamma$ is a curve of $\cR$, $\ell_i = \gamma|_{[k_i,k_i+1]}$ are segments of $\cR$ such that $\ell_0(\tilde{s}_0) - \ell_1(\tilde{s}_1) = \gamma(s_0) - \gamma(s_1) = \vu$. By Lemma \ref{lemma:intersectionOfSegments}, $\vu = a_0 \tv(\ell_0) + a_1 \tv(\ell_1)$, which means that at least one of the three coordinates of $\vu$ is zero: this contradicts the fact that $abc \neq 0$. 
\end{proof}

Consider a simple closed curve $\gamma: I \to \RR^3$ and a vector $\vu \in \RR^3$, $\vu \neq 0$. Assume that there exists $\delta > 0$ such that for each $s \in (0,\delta]$, the curves $\gamma$ and $\gamma + s\vu$ are disjoint. Then define the \emph{directional writhing number}\label{def:directionalWrithe} in the direction $\vu$ by $\Wr(\gamma, \vu) = \Link(\gamma, \gamma + \delta\vu)$ (see \cite[\S 3]{writhingNumber}). Since $\Link$ is symmetric and invariant by translations, $\Wr(\gamma,\vu) = \Wr(\gamma,-\vu)$.

\begin{lemma}
\label{lemma:directionalWrithingNumber}
Fix $\beta \in \bB$, and let $\gamma$ be a simple closed curve of $\cR$. If $0 < |a|,|b| < 1/N$, where $N$ is the $\vbeta_3$-length of $\cR$, then 
$\Wr(\gamma, \vbeta_3 + a\vbeta_1 + b\vbeta_2) = T^{\beta}_{a,b}(\gamma).$
\end{lemma}
\begin{proof}
We would like to use the fact that the sums of the signs of the crossings of the orthogonal projection of a smooth curve in the direction of a vector $\vu$ equals its directional writhing number (in the direction of $\vu$): this is essentially what we're trying to prove for our curve, except that $\Pi^{\beta}_{a,b}$ is not the orthogonal projection and that $\gamma$ is not a smooth curve. However, these difficulties can be avoided, as the following paragraphs show.

The orthogonality of the projection makes no real difference, because the orthogonal projection in the direction of $(a,b,1)$ has the same kernel as $\Pi^{\beta}_{a,b}$, so the crossings occur in the same positions (and clearly have the same signs). Therefore, by Lemma \ref{lemma:transversalCrossings}, $T^{\beta}_{a,b}(\gamma)$ equals the sums of the signs of the crossings of the aforementioned orthogonal projection.

For the smoothness of the curve, there is a finite number of points where $\gamma$ is not smooth: precisely, the set of $k \in \ZZ$ such that the two segments of $\gamma$ that intersect at $\gamma(k)$ are not parallel. To simplify notation, let $[0,n]$ be the domain of $\gamma$, and for $k=0,1,\ldots,n-1$ let $\ell_k$ be the segment of $\gamma$ such that $\ell_k(0) = \gamma(k)$ (notice that $\ell_k(1) = \gamma(k+1)$). It is also convenient to set $\ell_{-1} := \ell_{n-1}$, so that $\ell_{-1}(1) = \ell_{n-1}(1) = \gamma(n) = \gamma(0)$.

Recall from Lemma \ref{lemma:transversalCrossings} that every crossing in the projections occur in the interiors of the segments: since the number of segments is finite, we can pick $0 < \epsilon < 1/2$ sufficiently small so that 
$\Pi^{\beta}_{a,b}(\gamma(U_{\epsilon}))$ 
contains no crossings, where 
$U_\epsilon = [0,n] \cap \left(\bigcup_{k \in \ZZ} [k - \epsilon, k + \epsilon]\right).$

Let $\phi_1: \RR \to \RR$ be a nondecreasing $C^{\infty}$ function such that $\phi_1(t) = 0$ whenever $t \leq -\epsilon$ and $\phi_1(t) = t$ whenever $t \geq \epsilon$.
%$$\phi_1(t) = \begin{cases}
%0, &t \leq -\epsilon;\\
%t, &t \geq \epsilon,
%\end{cases}$$
Let $\phi_0(t) = t + \epsilon - \phi_1(t)$.
Consider the smooth simple closed curve of $\RR^3$, $\tilde{\gamma}:[0,n] \to \RR^3$, given by
$$ \tilde{\gamma}(s) = 
\begin{cases}
\gamma(k - \epsilon) + \phi_0(s-k)\tv(\ell_{k-1}) + \phi_1(s-k)\tv(\ell_k), &s \in (k-\epsilon, k+ \epsilon);\\
\gamma(s), &s \notin U_{\epsilon}.
\end{cases}$$
To simplify notation, write $\vw = \vbeta_3 + a\vbeta_1 + b\vbeta_2$ and fix $\delta < 1/\sqrt{1 + a^2 + b^2}$, so that $\|\delta\vw\| < 1$. By Lemma \ref{lemma:translationIntersection}, $\gamma$ and $\gamma + s\vu$ are disjoint whenever $s \in (0,\delta]$.
  
Clearly, the sums of the signs of the crossings in the orthogonal projection of $\tilde{\gamma}$ equals that of $\gamma$; moreover, $\Link(\tilde{\gamma},\tilde{\gamma} +  s\vw) = \Link(\gamma,\gamma +  s\vw)$ for sufficiently small $s > 0$. Since $\tilde{\gamma}$ is smooth, $T^{\beta}_{a,b}(\gamma) = \Wr(\tilde{\gamma},\vw) = \Link(\gamma,\gamma +  s\vw) = \Wr(\gamma,\vw)$. 
%
%
%Now, if $\delta_0, \delta_1 > 0$ are such that $\delta_i < 1/\sqrt{1 + a^2 + b^2}$, by Lemma \ref{lemma:translationIntersection}, there is a homotopy of the space $\RR^3 \setminus \gamma$ taking $\gamma + \delta_0\vw$ to $\gamma + \delta_1\vw$, so that $\Link(\gamma,\gamma + \delta_0\vw) = \Link(\gamma,\gamma + \delta_1\vw)$. Therefore, the result holds for any $0 < \delta < 1/\sqrt{1 + a^2 + b^2}$.   
\end{proof}

The following rather technical Lemma will be used in the proof of Lemma \ref{lemma:differenceOfLinkingNumbers}:

\begin{lemma}
\label{lemma:projectionIntersections}
Let $\beta \in \bB$, and let $\ell_0$ and $\ell_1$ be two segments of a region $\cR$ whose $\vbeta_3$-length is $N$. Let $\vu = b \vbeta_2 + c \vbeta_3$ with
$bc \neq 0$ and $b^2 + c^2 < 1$.
Let $0 < \epsilon < \min\left(\frac{|b|}{N + |c|}, \frac{1 - |b|}{N + |c|}\right)$. 

If, for some $s_0, s_1 \in [0,1]$,
$\Pi^{\beta}_{\epsilon,\epsilon}(\ell_0(s_0) - \ell_1(s_1) - \vu) = 0$,
%$\Pi^{\beta}_{\epsilon,\epsilon}(\ell_0) \cap \Pi^{\beta}_{\epsilon,\epsilon}(\ell_1 + \vu) \neq \emptyset$
then $\ell_0$ and $\ell_1$ are not parallel, and $s_0,s_1 \in (0,1)$. 
\end{lemma}
\begin{proof}
Suppose $\Pi^{\beta}_{\epsilon,\epsilon}(\ell_0(s_0) - \ell_1(s_1) - \vu) = 0$. Let $\alpha_i = \vbeta_i \cdot (\ell_0(s_0) - \ell_1(s_1)), i=1,2,3$, so that $\alpha_1 = \epsilon (\alpha_3 - c), \alpha_2 - b = \epsilon (\alpha_3 - c)$.
 
Suppose, by contradiction, that at least one of these things occurs: 
\begin{enumerate}[label=(\roman*), topsep = 0.1px, itemsep = 0.1px]
	\item \label{item:parallel} $\ell_0$ and $\ell_1$ are parallel;
	\item \label{item:borderline} $s_0 \in \{0,1\}$ or $s_1 \in \{0,1\}$.
\end{enumerate}
We claim that at least two of the three $\alpha_i$'s are integers. To see this, suppose first \ref{item:parallel}, so that $\tv(\ell_0), \tv(\ell_1) \parallel \vbeta_i$, so that for $j \neq i$, $\alpha_j = \vbeta_j \cdot (\ell_0(s_0) - \ell_1(s_1)) \in \ZZ$. On the other hand, if \ref{item:borderline} holds, say $s_1 \in \{0,1\}$, then $\ell_1(s_1) \in \plshalf{\ZZ}$ and $\tv(\ell_0) \parallel \vbeta_i$, so that, again, for $j \neq i$, $\alpha_j \in \ZZ$.

We claim that $\alpha_2 \notin \ZZ$. In fact, if $\alpha_2 \in \ZZ$ then we would have
$|\alpha_2| = |b + \epsilon(\alpha_3 - c)| < |b| + \frac{1 - |b|}{N + |c|}(N + |c|) = 1$, so that $\alpha_2 = 0$ and $|b| = \epsilon|\alpha_3 - c| < \frac{|b|}{N + |c|}(N + |c|)$, which is a contradiction.

Therefore, we must have $\alpha_1,\alpha_3 \in \ZZ.$ Then $|\alpha_1| = |\epsilon(\alpha_3 - c)| < |b| < 1$, so $\alpha_1 = 0 = \alpha_3 - c$. Thus $c = \alpha_3 \in \ZZ$ but $|c| \in (0,1)$, which is a contradiction.
%
%
%We claim that $y_0 - y_1 \notin \ZZ$. In fact, if that were the case then we would have
%$|y_0 - y_1| = |b + \epsilon(z_0 - z_1 - c)| < |b| + \frac{1 - |b|}{N + |c|}(N + |c|) = 1$, thus $y_0 - y_1 = 0$ and $|b| = \epsilon|z_0 - z_1 - c| < \frac{|b|}{N + |c|}(N + |c|)$, which is a contradiction.
%
%Therefore, $x_0 - x_1 \in \ZZ$ and $z_0 - z_1 \in \ZZ$. Then $|x_0 - x_1| = |\epsilon(z_0 - z_1 - c)| < |b| < 1$, so $x_0 - x_1 = 0 = z_0 - z_1 - c$. Thus $c = z_0 - z_1 \in \ZZ$ but $|c| \in (0,1)$, which is a contradiction.
\end{proof}

The following definition is specific for Lemmas \ref{lemma:differenceOfLinkingNumbers} and \ref{lemma:writheDifference}. Let $\gamma:[0,n] \to \RR^3$ be a simple closed curve of a region $\cR$ and $\beta \in \bB$. For $k=0,1,\ldots, n-1$, set $\ell_k = \gamma|_{[k,k+1]}$, and set also $\ell_n = \ell_0$. Finally, we define 
$$\eta^{\beta}_{\gamma}(k) = 
\begin{cases}
1, &(\tv(\ell_k), \tv(\ell_{k+1})) = (\vbeta_2,\vbeta_3) \text{ or } (-\vbeta_3,-\vbeta_2);\\
-1, & (\tv(\ell_k), \tv(\ell_{k+1})) = (-\vbeta_2,-\vbeta_3) \text{ or } (\vbeta_3,\vbeta_2);\\
0, &\text{otherwise.}
\end{cases}
$$

\begin{lemma}
\label{lemma:differenceOfLinkingNumbers}
Let $\gamma:[0,n] \to \RR^3$ be a simple closed curve of a region $\cR$. For $k=0,1,\ldots, n-1$, set $\ell_k = \gamma|_{[k,k+1]}$, and set also $\ell_n = \ell_0$; for shortness, write $\tv_k = \tv(\ell_k)$. Then if $\beta \in \bB$ and $a,b,c > 0$, then
$$
\Wr(\gamma, a \vbeta_1 + b \vbeta_2 + c \vbeta_3) - \Wr(\gamma, - a \vbeta_1 + b \vbeta_2 + c \vbeta_3) = \sum_{0 \leq k < n} \eta^{\beta}_{\gamma}(k).
$$
\end{lemma}
\begin{proof}
We may assume that $a^2 + b^2 + c^2 < 1$.
Let $0 < \epsilon < \min\left(\frac{|b|}{N + |c|}, \frac{1 - |b|}{N + |c|}\right),$ and set $\vu(s) = s \vbeta_1 + b \vbeta_2 + c \vbeta_3$.
By Lemma \ref{lemma:translationIntersection}, $\Link(\gamma, \gamma + \vu(a))$ depends only on the signs of $a$, $b$ and $c$. Therefore, we may, without loss of generality, assume that $a > 0$ is sufficiently small such that for every $s \in [-a,a]$ and every $i,j \in \{0,1, \ldots, n\}$,
$$\Pi^{\beta}_{\epsilon,\epsilon}(\ell_i) \cap \Pi^{\beta}_{\epsilon,\epsilon}(\ell_j + \vu(s)) \neq \emptyset \Leftrightarrow \Pi^{\beta}_{\epsilon,\epsilon}(\ell_i) \cap \Pi^{\beta}_{\epsilon,\epsilon}(\ell_j + \vu(0)) \neq \emptyset$$
(this is possible by Lemma \ref{lemma:projectionIntersections}). Therefore, clearly $\Link(\gamma, \gamma + \vu(a)) - \Link(\gamma, \gamma + \vu(-a))$ equals the number of pairs of segments $\ell_i, \ell_j$ such that $\Pi^{\beta}_{\epsilon,\epsilon}(\ell_i) \cap \Pi^{\beta}_{\epsilon,\epsilon}(\ell_j + \vu(s)) \neq \emptyset$ for every $s \in [-a,a]$ and such that the crossing changes its sign as $s$ goes from $-a$ to $a$. Now a crossing may only change its sign if $\ell_i \cap (\ell_j + \vu(s)) \neq \emptyset$ for some $s$: by Lemma \ref{lemma:translationIntersection}, this can only happen if $s = 0$. 

By Lemma \ref{lemma:intersectionOfSegments}, $\ell_i \cap (\ell_j + \vu(0)) \neq \emptyset$ if and only if for some $m_i,m_j \in \{0,1\}$, $\ell_i(m_i) = \ell_j(m_j)$ and
$\vu(0) = b \vbeta_2 + c \vbeta_3 = a_i \tv(\ell_i) + a_j \tv(\ell_j)$, with $(-1)^{m_i} a_i \geq 0, (-1)^{m_j} a_j \leq 0$.
Since $\ell_i$ and $\ell_j$ are segments of the simple curve $\gamma$, they can only be adjacent if, for some $k$, $\{\ell_i, \ell_j\} = \{\ell_k, \ell_{k+1}\}$. Now, $\ell_{k+1}(0) = \ell_k(1)$, so that 
$(0,b,c) = a_0\tv(\ell_{k+1}) - a_1\tv(\ell_k)$ with either $a_0,a_1 \geq 0$ or $a_0,a_1 \leq 0$ (depending on which is $\ell_i$ and which is $\ell_j$). Since $b,c > 0$, this implies that $\{\tv(\ell_k), \tv(\ell_{k+1})\} = \{\vbeta_2,\vbeta_3\} \text{ or } \{-\vbeta_2, -\vbeta_3\}$ and, therefore, $\eta^{\beta}_{\gamma}(k) = \pm 1$.
\begin{figure}[ht]
\centering
\def\svgwidth{0.35\columnwidth}
\subfloat[$\eta^{\beta}_{\gamma}(k) = 1$.]{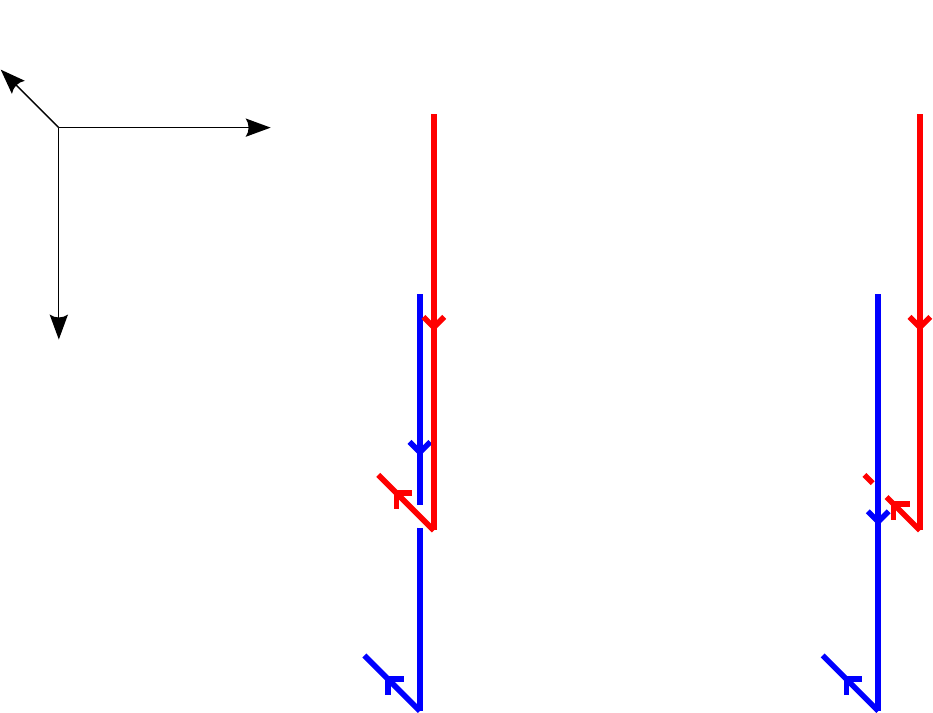\label{subfig:crossings_jk}} \qquad \qquad \qquad \qquad 
\def\svgwidth{0.35\columnwidth}
\subfloat[$\eta^{\beta}_{\gamma}(k) = -1$.]{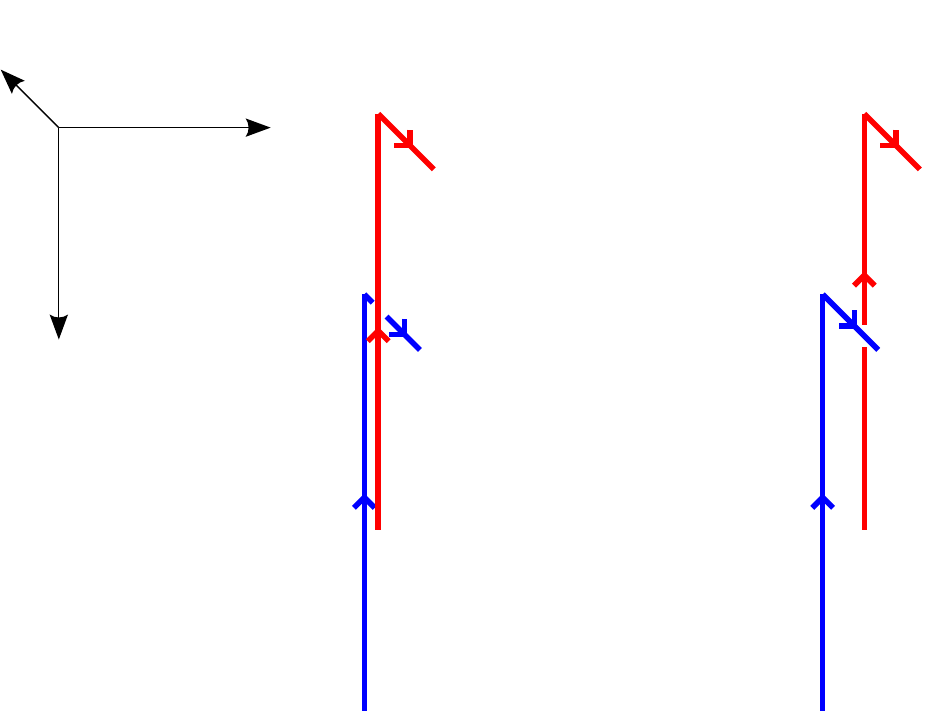\label{subfig:crossings_kj}}
\caption{Illustration of the crossings $\Pi^{\beta}_{\epsilon,\epsilon}(\gamma) \cap \Pi^{\beta}_{\epsilon,\epsilon}(\gamma+s\vbeta_1 + b \vbeta_2 + c \vbeta_3)$ for $s \in [-a,a]$. Notice that simultaneously switching the orientations of both segments does not change the signs of the crossings.}%
\label{fig:cases_YDimers}%
\end{figure}

We now analyze each of the four possible cases for $(\tv(\ell_k), \tv(\ell_{k+1}))$ (as an ordered pair). When $(\tv(\ell_k), \tv(\ell_{k+1})) = (\vbeta_2,\vbeta_3) \text{ or } (-\vbeta_3,-\vbeta_2)$, so that $\eta^{\beta}_{\gamma}(k) = 1,$ we see a situation as illustrated in Figure \ref{subfig:crossings_jk} (perhaps with both orientations reversed): when $s > 0$, we have a positive crossing; when $s < 0$, we have a negative crossing.
Figure \ref{subfig:crossings_kj} illustrates (up to orientation) the case $(\tv(\ell_k), \tv(\ell_{k+1})) = (-\vbeta_2,-\vbeta_3) \text{ or } (\vbeta_3,\vbeta_2)$ ($\eta^{\beta}_{\gamma}(k) = -1)$: negative crossing for $s > 0$, and positive crossing for $s < 0$. These observations yield the result.
\end{proof}

%-------------------------
\section{Writhe formula for the twist}
\label{sec:writheInterpretation}
%{\color[rgb]{0,0,1} Should we comment on the fact that we will do things for the $\ez$-pretwists of $\ez$-multiplexes, which we will ultimately see, via the proof of Proposition \ref{prop:equalTwistsMultiplex}, that we can simply call the twist?}
%{\color[rgb]{0,0,1} Pay attention to the fact that axes of multiplexes are in $\Delta$, not in $\Phi$. Maybe it would be a good idea to call the axis of the multiplex $\vw \in \Delta$ (consistently), as we usually use the letter $\vu \in \Phi$}
 
Now that the groundwork is done, we set out to obtain a new formula for the twist of pseudocylinders of even depth (we work with pseudocylinders because the hypothesis of simple connectivity will not play any role). Pseudocylinders of even depth have the advantage of always admitting a tiling such that all dimers are parallel to its axis: for a $\vw$-pseudocylinder $\cR$ ($\vw \in \Delta$) with even depth, let $\tbase = \tbase(\cR)$ denote the tiling such that every dimer is parallel to $\vw$ (see Figure \ref{fig:auxiliaryLinesExample}). Not only does this tiling trivially satisfy $T^{\vw}(\tbase) = 0$, but also for any segment $\ell$ of $\cR$ and any dimer $\ell_0 \in \tbase$ we have $\tau^{\vw}(\ell_0,\ell) = \tau^{\vw}(\ell, \ell_0) = 0$. This allows for a direct interpretation of the twist via a set of curves, which, in particular, allows us to show that it is an integer.

\begin{figure}[ht]%
\centering
\subfloat[The tiling $t$.]{\includegraphics[width=0.45\columnwidth]{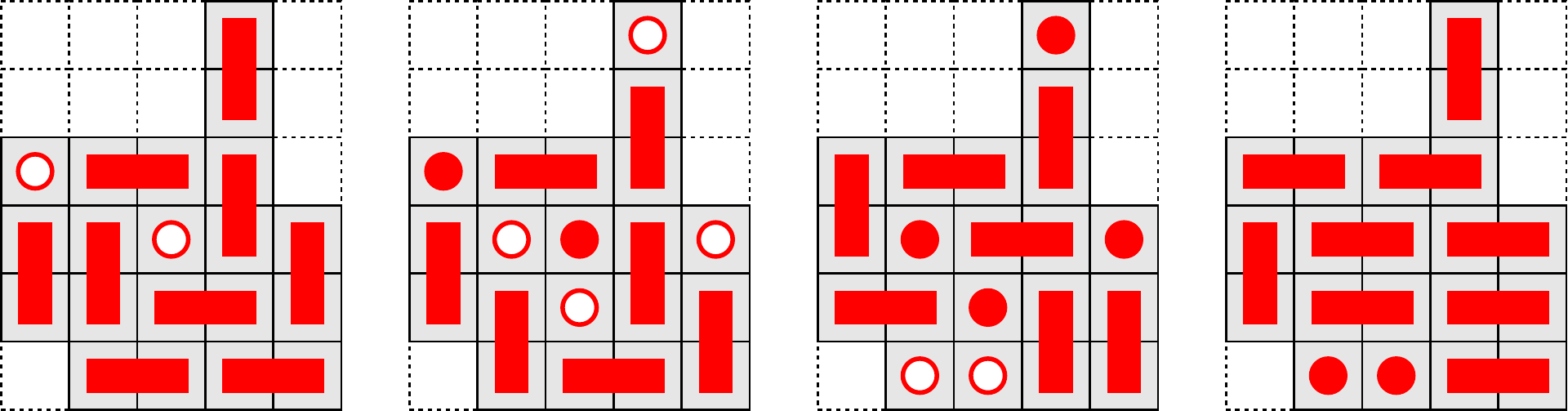}} \qquad
\def\svgwidth{0.47\columnwidth}
\subfloat[The curves in $\Gamma(t,\tbase)$.]{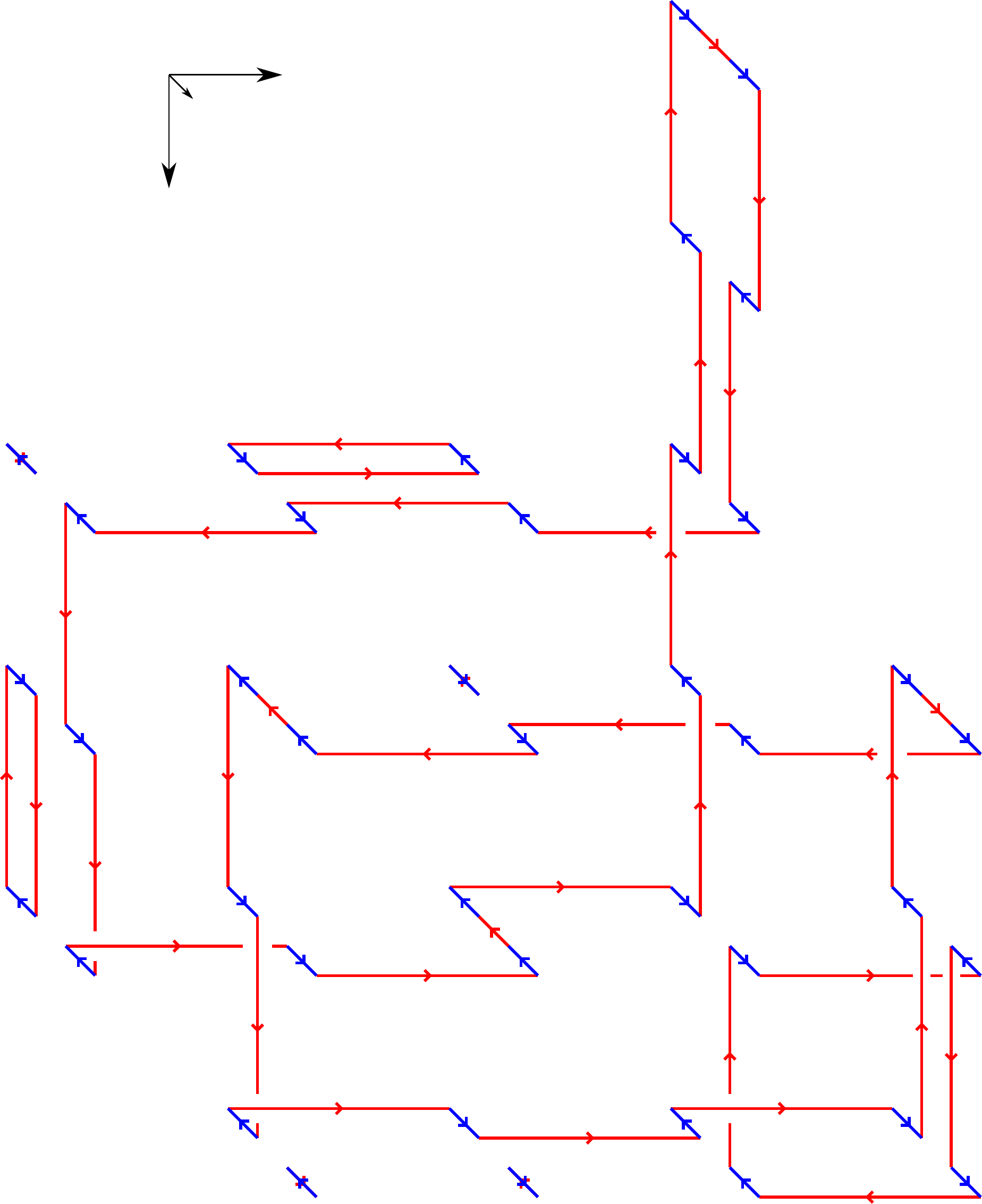\label{subfig:auxiliaryLinesDrawing}}%
\caption{A tiling $t$ of a $\vw$-cylinder with depth $4$, and $\Gamma(t,\tbase)$, where $\tbase$ is the tiling such that every dimer is parallel to $\vw$. The dimers of $t$ are the red segments, and the blue segments are the ones in $(-\tbase)$. We chose a basis $\beta \in \bB$ with $\vw = \vbeta_3$; $\vw$ points ``towards the paper''. $\Gamma(t,\tbase)$ consists of nine curves, four of which are trivial; the five nontrivial curves form $\Gamma^*(t,\tbase)$.}
\label{fig:auxiliaryLinesExample}%
\end{figure}

\begin{lemma}
\label{lemma:partialWritheFormula}
Given $\vw \in \Delta$, let $t$ be a tiling of a $\vw$-pseudocylinder of even depth $\cR$, and let $\tbase = \tbase(\cR)$. If $\Gamma^*(t, \tbase) = \{\gamma_i \,|\, 1 \leq i \leq m\}$, then 
$$T^{\vw}(t) = \sum_{1 \leq i \leq m} T^{\vw}(\gamma_i) + 2\sum_{1 \leq i < j \leq m} \Link(\gamma_i, \gamma_j).$$
\end{lemma}
\begin{proof}
%Since all the segments in $(-\tbase)$ are parallel to $\vw$, it follows that, if $\ell_0 \in \tbase$, $\tau(\ell_0, \ell) = \tau(\ell,\ell_0) = 0$ for every segment $\ell$ of $\cR$. This also holds for the dimers of $t$ that belong to the trivial curves of $\Gamma(t,\tbase)$, since they are also parallel to $\vw$: in other words, %$\Tw(t) = T(t) = T(t \cup (-\tbase))$.
%Then:
Clearly,  
$$
T^{\vw}(t) = T^{\vw}(t \sqcup (-\tbase))  
= \sum_{i,j} T^{\vw}(\gamma_i, \gamma_j)  
= \sum_{i} T^{\vw}(\gamma_i) + 2\sum_{i < j} \Link(\gamma_i, \gamma_j),
$$
the last equality holding by Lemma \ref{lemma:linkingNumber}.
\end{proof}

%Just as with the second term in Equation \eqref{eq:partialWritheFormula}, we can also give a more topological interpretation of the first term. However, we will first need a few topological facts about the curves $\gamma \in \Gamma^*(t,\tbase)$. For what follows, we need the concept of directional writhing number (or writhe) of a curve in the direction of a vector $\tv$ (see \cite[Section 3]{writhingNumber}). 

%For a curve $\gamma$ and a vector $(a,b,c) \in \RR^3$, let $(\gamma + (a,b,c)):[0,1] \to \RR^3$, $(\gamma + (a,b,c))(s) = \gamma(s) + (a,b,c)$ denote the translation of $\gamma$ by $(a,b,c)$; also, if $\ell$ is a segment of $\gamma$, $\ell + (a,b,c)$ will denote the translation of $\ell$ by $(a,b,c)$ (which is not necessarily a segment of $\cR$).

For Lemmas \ref{lemma:gammaNoIntersection} and \ref{lemma:writhes}, assume $\vw \in \Delta$, $t$ is a tiling of a $\vw$-pseudocylinder with even depth $\cR$, and $\tbase = \tbase(\cR)$.

\begin{lemma}
\label{lemma:gammaNoIntersection}
Fix $\beta \in \bB$ such that $\vbeta_3 = \vw$.
If $\gamma$ is a curve of $\Gamma^*(t,\tbase)$ and $a^2 + b^2 + c^2 < 1$ and $ab \neq 0$, then $(\gamma+a \vbeta_1 + b \vbeta_2 + c \vbeta_3) \cap \gamma = \emptyset$.
\end{lemma}
Notice that the case $c \neq 0$ follows from Lemma \ref{lemma:translationIntersection}.
\begin{proof} 
Let $\vu = a \vbeta_1 + b \vbeta_2 + c \vbeta_3.$
Suppose, by contradiction, that $\gamma$ and $\gamma + \vu$ are not disjoint, and let 
$\ell_0$ and $\ell_1$ be two segments of $\gamma$ such that $\ell_0(s_0) = \ell_1(s_1) + \vu$ for some $s_0,s_1 \in [0,1]$. By Lemma \ref{lemma:intersectionOfSegments}, $\ell_0$ and $\ell_1$ must be adjacent, so that at least one of these two segments is in $(-\tbase)$, hence parallel to $\vu$. Lemma \ref{lemma:intersectionOfSegments} also implies that $\vu = a \vbeta_1 + b \vbeta_2 + c \vbeta_3 = a_0 \tv(\ell_0) + a_1 \tv(\ell_1)$. Since at least one of $\tv(\ell_0), \tv(\ell_1)$ is parallel to $\vw = \vbeta_3$, it follows that $a = 0$ or $b = 0$, which contradicts the hypothesis.  
\end{proof}

By Lemma \ref{lemma:gammaNoIntersection}, if $\gamma \in \Gamma^*(t,\tbase)$, $\Wr(\gamma, a\vbeta_1 + b\vbeta_2 + c\vbeta_3)$ is defined whenever $ab \neq 0$. Set
$$\Wr^{+}(\gamma) = \Wr(\gamma, \vbeta_1 + \vbeta_2), \quad \Wr^{-}(\gamma) = \Wr(\gamma, \vbeta_1 - \vbeta_2).$$
Clearly, 
\begin{equation}
\Wr(\gamma, a \vbeta_1 + b \vbeta_2 + c \vbeta_3) = \begin{cases} \Wr^+(\gamma), & ab > 0; \\
\Wr^-(\gamma), & ab < 0.
\end{cases}
\label{eq:writhes}
\end{equation}

\begin{lemma}
\label{lemma:writhes}
If $\gamma$ is a curve of $\Gamma^*(t,\tbase)$, then
$$T^{\vw}(\gamma) = \frac{\Wr^+(\gamma) + \Wr^-(\gamma)}{2}.$$
\end{lemma}
\begin{proof}
Fix $\beta \in \bB$ with $\vbeta_3 = \vw$, and let $N$ denote the $\vw$-length of the pseudocylinder (which is equal to its depth).
By Lemmas \ref{lemma:crossings} and \ref{lemma:directionalWrithingNumber}, given $0 < \epsilon < 1/N,$
$$
T^{\vw}(\gamma) = \frac{1}{4}\sum_{i,j \in \{-1,1\}} T^{\beta}_{i\epsilon,j\epsilon}(\gamma) = \frac{1}{4}\sum_{i,j \in \{-1,1\}} \Wr(\gamma, i\epsilon\vbeta_1 + j \epsilon \vbeta_2 + \vbeta_3);  
$$
Equation \eqref{eq:writhes} completes the proof.
\end{proof}

%It turns out that Proposition \ref{prop:writheFormula} allows us to show that the $\vw$-pretwist is an integer:

\begin{lemma}
\label{lemma:writheDifference}
Let $\vw \in \Delta$, and let $t$ be a tiling of a $\vw$-pseudocylinder with even depth.
If $\gamma$ is a curve of $\Gamma^*(t,\tbase)$, then $(\Wr^+(\gamma) + \Wr^-(\gamma))/2 \in \ZZ$.
%In particular, since $\Wr^+(\gamma) + \Wr^-(\gamma)$ and $\Wr^+(\gamma) - \Wr^-(\gamma)$ have the same parity, $T^{\vw}(t) \in \ZZ$.
\end{lemma}
\begin{proof}
Pick $\beta \in \bB$ with $\vbeta_3 = \vw$. Assume without loss of generality that $\cR = \cD + [0,2N]\vbeta_3$, $\cD \subset \vbeta_3^{\perp}$. 
If $\gamma:[0,n] \to \RR^3$, set $\ell_k = \gamma|_{[k,k+1]}$ for $k=0,1,\ldots, n-1$, and set $\ell_n = \ell_0$. 

By definition and using Lemma \ref{lemma:differenceOfLinkingNumbers},
$
\Wr^+(\gamma) - \Wr^-(\gamma) = \sum_k \eta^{\beta}_{\gamma}(k).
$
We need to look at $k$ such that $\eta^{\beta}_{\gamma}(k) \neq 0$, i.e., $\{\tv(\ell_k),\tv(\ell_{k+1})\} = \{\vbeta_2,\vbeta_3\} \text{ or } \{-\vbeta_2,-\vbeta_3\}$. Since every segment of $-\tbase$ is parallel to $\vbeta_3$, we need to look at every segment of $t$ that is parallel to $\vbeta_2$.

For each segment $\ell_k$ of $t$ with $\tv(\ell_k) = \pm \vbeta_2$, let $\plshalf{z_k} = \vbeta_3 \cdot \ell_k(0)$, so that $z_k \in \ZZ$. If $z_k$ is odd, then, by definition of $\tbase$, $\tv(\ell_{k-1}) = \vbeta_3 = - \tv(\ell_{k+1})$, so that either $(\tv(\ell_{k-1}),\tv(\ell_k)) = (\vbeta_3,\vbeta_2)$ or $(\tv(\ell_k),\tv(\ell_{k+1})) = (-\vbeta_2,-\vbeta_3)$. Making a similar analysis for $z_k$ even, we see that $\eta^{\beta}_{\gamma}(k-1) + \eta^{\beta}_{\gamma}(k) = (-1)^{z_k}$. 
Working with congruence modulo $2$,
\begin{equation*}
\Wr^+(\gamma) + \Wr^-(\gamma) \equiv \Wr^+(\gamma) - \Wr^-(\gamma) = \sum_{\tv(\ell_k) = \pm \vbeta_2}(-1)^{z_k} \equiv \sum_{k} (\tv(\ell_k) \cdot \vbeta_2) = 0, 
\end{equation*}
which completes the proof.
\end{proof}

\begin{prop}
\label{prop:writheFormula}
If $\vw \in \Delta$, $\cR$ is a $\vw$-pseudocylinder with even depth, $t$ is a tiling of $\cR$, $\tbase = \tbase(\cR)$ and $\Gamma^*(t,\tbase) = \{\gamma_i \,|\, 1 \leq i \leq m\}$, then
$$
T^{\vw}(t) = \sum_{1 \leq i \leq m} \frac{\Wr^+(\gamma_i) + \Wr^-(\gamma_i)}{2} + 2\sum_{1 \leq i < j \leq m} \Link(\gamma_i, \gamma_j) \in \ZZ.
$$
\end{prop}
\begin{proof}
Follows directly from Lemmas \ref{lemma:partialWritheFormula}, \ref{lemma:writhes} and \ref{lemma:writheDifference}.
\end{proof}

\section{Different directions of projection}
\label{sec:differentDirections}

Our goal for this section is to prove Proposition \ref{prop:equalTwistsMultiplex}, that is, that all pretwists coincide for a cylinder. %In order to achieve this, we will need technology developed in Sections \ref{sec:topologicalGroundwork} and \ref{sec:writheInterpretation}.

\begin{lemma}
\label{lemma:xEffectsEqualsTwist}
Let $\vw \in \Delta$, and let $\cR$ be a $\vw$-pseudocylinder with even depth.
Let $t$ be a tiling of $\cR$, and let $\tbase$ be the tiling such that every dimer is parallel to $\vw$. If $\vu \in \Phi$, then 
$T^{\vu}(t \sqcup (-\tbase)) = T^{\vw}(t).$ 
%= \sum_{0 \leq i \leq k} \frac{\Wr^+(\gamma_i) - \Wr^-(\gamma_i)}{2} + 2\sum_{0 \leq i < j \leq k} \Link(\gamma_i, \gamma_j).$$
\end{lemma}
%Notice that the right hand-side is the usual definition for the twist (by Equation \eqref{eq:partialWritheFormula}).
\begin{proof}
%$\sum_{ 1 \leq i,j \leq m } T^{\vu}(\gamma_i,\gamma_j) = \sum_{ 1 \leq i,j \leq m } T^{\vw}(\gamma_i,\gamma_j) $ 
%The case $\vu = \vw$ is trivial.
Suppose $\Gamma^*(t,\tbase) = \{\gamma_i \,|\, 1 \leq i \leq m\}$.  
Clearly, 
$$
T^{\vu}(t \sqcup (-\tbase)) = 
\sum_{i,j} T^{\vu}(\gamma_i,\gamma_j)
= \sum_i T^{\vu}(\gamma_i) + 2\sum_{i < j} \Link(\gamma_i, \gamma_j).
$$
Let $L$ be the $\vu$-length of $\cR$ and $0 < \epsilon < 1/L$. Let $\beta \in \bB$ such that $\vbeta_3 = \vu$. Then, by Lemmas \ref{lemma:crossings} and \ref{lemma:directionalWrithingNumber}, 
$$
T^{\vu}(\gamma_i) = \frac{1}{4}\sum_{k,l \in \{-1,1\}}T^{\beta}_{k \epsilon, l \epsilon}(\gamma_i) = \frac{1}{4}\sum_{k,l \in \{-1,1\}}\Wr(\gamma_i, k\epsilon \vbeta_1 + l\epsilon \vbeta_2 + \vbeta_3).
$$
%By Lemma \ref{lemma:posAndNegWrithes},
%$$\Link(\gamma_i, \gamma_i + \delta(\vbeta_3 + k\epsilon \vbeta_1 + l\epsilon \vbeta_2)) = \begin{cases}\Wr^+(\gamma_i), &k = 1;\\ \Wr^-(\gamma_i), &k = -1. \end{cases}$$
By Equation \eqref{eq:writhes} and Proposition \ref{prop:writheFormula}, %$T^{\ex}(\gamma_i) = \frac{\Wr^+(\gamma_i) - \Wr^-(\gamma_i)}{2}$. By Lemma \ref{lemma:writheDifference}, 
$$
T^{\vu}(t \sqcup (-\tbase)) = \sum_i \frac{\Wr^+(\gamma_i) + \Wr^-(\gamma_i)}{2} + 2\sum_{i < j} \Link(\gamma_i, \gamma_j) = T^{\vw}(t).
$$
\end{proof}

\begin{lemma}
\label{lemma:allPretwistsBoxes}
Let $\cB = [0,L] \times [0,M] \times [0,N]$ be a box that has at least one even dimension, and let $t$ be a tiling of $\cB$. Then $T^{\ex}(t) = T^{\ey}(t) = T^{\ez}(t)$.
\end{lemma}

%\begin{lemma}
%\label{lemma:xEffectsCancelation}
%Let $\cB = [0,L] \times [0,M] \times [0,N]$ be a box, with $N$ even, and let $\tbasez$ be the tiling of $\cB$ such that every dimer is parallel to $\ez$.
%Let $t$ be a tiling of $\cB$. If $\Gamma^*(t,\tbasez) = \{\gamma_1, \ldots, \gamma_m\}$ and $\vu \in \Phi$, then  
%$$ T^{\vu}(t) = \sum_{1 \leq i,j \leq m} T^{\vu}(\gamma_i,\gamma_j). 
%$$
%%Therefore, $T^{\ex}(t) = T^{\ey}(t) = T^{\ez}(t)$. 
%\end{lemma}
\begin{proof}
By rotating, we may assume that $N$ is even, so that $\cB$ is a $\ez$-cylinder with even depth; let $\vu \in \Phi, \vu \perp \ez$. We want to show that $T^{\vu}(t) = T^{\ez}(t)$.

By Lemma \ref{lemma:xEffectsEqualsTwist}, 
$T^{\ez}(t) = T^{\vu}(t \sqcup (-\tbasez))$.
%
%The right hand-side is equal to $T^{\vu}(t \sqcup (-\tbasez))$, so we need to show that $T^{\vu}(t) = T^{\vu}(t \sqcup (-\tbasez))$. Since $\cB$ is a $\ez$-cylinder with even depth $N$, the result is obvious if $\vu = \pm \ez$: let us assume that $\vu \perp \ez$.
%
Now, 
$$T^{\vu}(t \sqcup (-\tbasez)) = T^{\vu}(t) + T^{\vu}(-\tbasez) + T^{\vu}(t, -\tbasez) + T^{\vu}(-\tbasez,t).$$ 
$T^{\vu}(-\tbasez) = 0$ because all segments of $(-\tbasez)$ are parallel. 
It remains to show that $T^{\vu}(t, -\tbasez) = T^{\vu}(-\tbasez,t) = 0$, which yields the result.
 
Let $\vw = \vu \times \ez$.
Given $\ell_0 \in t$, we now want to show that $\sum_{\ell \in \tbasez} \tau^{\vu}(\ell,\ell_0) = \sum_{\ell \in \tbasez} \tau^{\vu}(\ell_0,\ell) = 0$. 
This is obvious if $\ell_0$ is not parallel to $\vw$. 
Otherwise, effects cancel out, as illustrated in Figure \ref{fig:example_boxEffects}.
\begin{figure}[ht]%
\centering
\includegraphics[width=0.5\columnwidth]{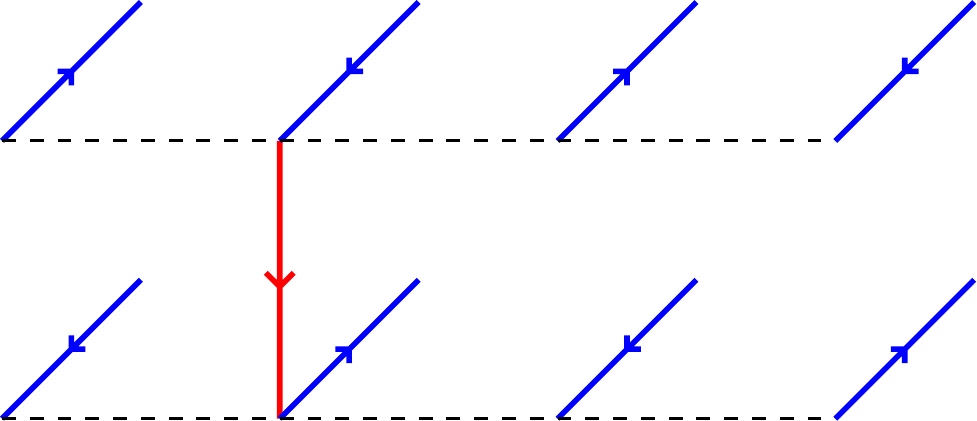}%
\caption{A dimer $\ell_0$ parallel to $\vw$, portrayed in red, and the pairs of segments (blue) of $\tbasez$ affected by it: $\vu$-effects cancel.}%
\label{fig:example_boxEffects}%
\end{figure} 
%Otherwise, suppose $\ell_0$ has endpoints $p$ and $\tilde{p}$ with $\tilde{p} - p = \pm \vw$.
%
%For each $k \in \ZZ$, $p + k\vu \in \cB$ if and only if $\tilde{p} + k \vu \in \cB$ (because $\cB$ is a box). For each such $k$, 
%let $\ell_k, \tilde{\ell}_k$ be the segments of $(-\tbasez)$ such that $\ell_k$ contains $p + k \vu$ and $\tilde{\ell}_k$ contains $\tilde{p} + k \vu$. Clearly $\tv(\ell_k) = - \tv(\tilde{\ell}_k)$, so that $\tau^{\vu}(\ell, \ell_k) = - \tau^{\vu}(\ell,\tilde{\ell}_k)$ and $\tau^{\vu}(\ell_k, \ell) = - \tau^{\vu}(\tilde{\ell}_k,\ell)$. Therefore,
%$$\sum_{\ell \in \tbasez} \tau^{\vu}(\ell,\ell_0) = \sum_{k < 0}  \tau^{\vu}(\ell_k, \ell_0) + \tau^{\vu}(\tilde{\ell}_k, \ell_0) = 0,$$
%and similarly $\sum_{\ell \in \tbasez} \tau^{\vu}(\ell_0,\ell) = 0.$ Therefore, $T^{\vu}(t, -\tbasez) = T^{\vu}(-\tbasez,t) = 0$, which completes the proof.
\end{proof}

If $Q \subset \pi$ is a basic square and $\vw \in \Delta$ is a normal vector for $\pi$, define the color of $Q$ to be the same as the color of the basic cube $Q - [0,1]\vw$; and 
$$\ccol(Q) =
\begin{cases}
1, &\text{if } Q \text{ is black;}\\
-1, &\text{if } Q \text{ is white.}
\end{cases}
$$

Recall the definition of $\vu$-shade from Section \ref{sec:combTwistBoxes}. If $A$ is a set of segments or a set of dominoes, $\vu \in \Phi$ and $Q$ is a basic square with normal $\vw \in \Delta$, we set
$$S(A,\vu,Q,n) = \{\ell \in A \,|\, \ell \cap \cS^{\vu}(Q + [0,n]\vw) \neq \emptyset\}.$$

\begin{lemma}
\label{lemma:alternativeFluxFormula}
Let $\cR$ be a $\vw$-cylinder ($\vw \in \Delta$) with base $\cD \subset \pi$ and even depth $N$. Let $Q \subset \pi$ be a basic square, $Q \not\subset \cD$, let $t$ be a tiling of $\cR$ and let $\vu \in \Phi$.
Then
$$\sum_{d \in S(t, \vu,Q,N)} \det(\tv(d), \vw,\vu) = 0.$$  
\end{lemma}
\begin{figure}[ht]%
\centering
\def\svgwidth{0.6\columnwidth}
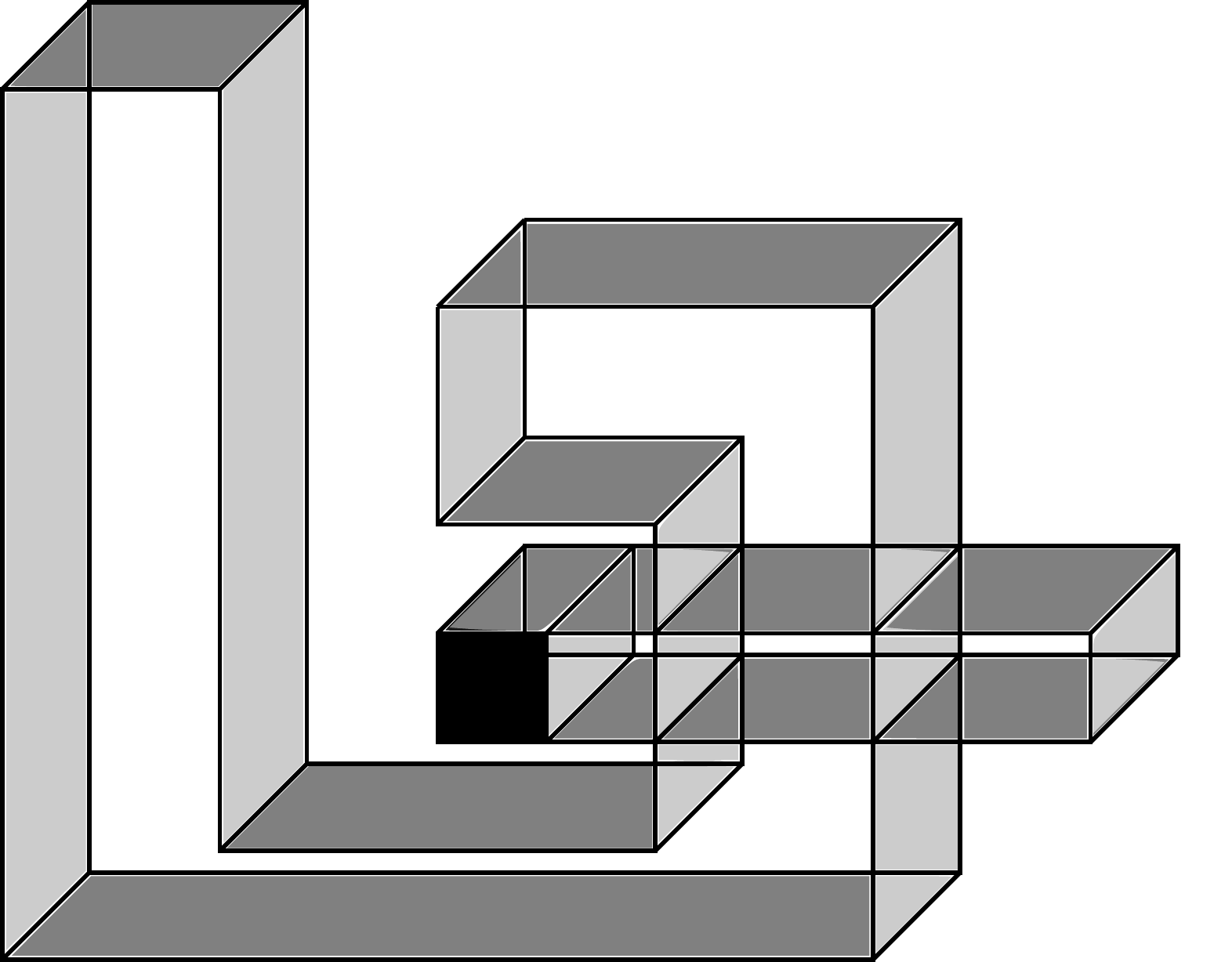
\caption{A cylinder $\cR$ with base $\cD \subset \pi$ and depth $N$, a basic square $Q \subset \pi$, $Q \not\subset \cD$ and the shade $S^{\vu}(Q + [0,N]\vw)$. }%
\label{fig:squareWalls}%
\end{figure}
\begin{proof}
The reader may want to follow by looking at Figure \ref{fig:squareWalls}.
Let $\tbasew = \tbasew(\cR)$, $S_t = S(t,\vu, Q,N)$, and for each $\gamma \in \Gamma^*(t,\tbasew)$, 
let $S_\gamma$ denote $S(\gamma,\vu,Q,N)$.
Clearly,
$$ \sum_{d \in S_t} \det(\tv(d), \vw,\vu) = \sum_{\substack{\gamma \in \Gamma^*(t,\tbasew) \\ \ell \in S_\gamma}} \det(\tv(\ell),\vw,\vu)).$$
Let $p_Q$ be the center of the square $Q$, and let $\Pi$ denote the orthogonal projection on $\pi$. For each $\gamma \in \Gamma^*(t,\tbasew)$, $\Pi \circ \gamma$ is a polygonal curve, so that the winding number of $\gamma$ around $p_Q$ equals (see, e.g., \cite{polygonalWindingNumber} for an algorithmic discussion of winding numbers)
\begin{align*} \wind(\Pi \circ \gamma, p_Q) &= \half\left(\#\{\ell \in S_\gamma \,|\, \tv(\ell) = \vw \times \vu\} - \#\{\ell \in S_\gamma \,|\, \tv(\ell) = -\vw \times \vu\}\right)\\
&= \half \sum_{\ell \in S_\gamma} \det(\tv(\ell), \vw,\vu).
\end{align*}
But $\wind(\Pi \circ\gamma, p_Q) = 0$ ($p_Q \notin \cD$ and $\cD$ is simply connected), so we get the result.
\end{proof}

\begin{rem}
\label{rem:pseudocylinderEvenDepth} 
The fact that $$\sum_{d \in S(t, \vu,Q,N)} \det(\tv(d), \vw,\vu) = \sum_{\gamma \in \Gamma^*(t,\tbasew)}\wind(\Pi \circ \gamma, p_Q)$$ (using the notation in the proof of Lemma \ref{lemma:alternativeFluxFormula}) also holds in pseudocylinders, as the hypothesis of simply connected is only needed in the last step of the proof. In fact, if we define the flux $\phi(t) = \sum_{Q , \gamma} \ccol(Q) \wind(\Pi \circ \gamma, p_Q)$ ($Q$ runs through all the basic squares not contained in $\cD$, $\gamma$ runs through all the curves in $\Gamma^*(t,\tbasew)$), then a modification in the proof of Proposition \ref{prop:equalTwistsIntNEven} yields that $T^{\ex}(t) = T^{\ey}(t) = T^{\ez}(t) + \phi(t)$ (in particular, since $\phi(t) \in \ZZ$, the three pretwists are integers).  
\end{rem}

\begin{prop}
\label{prop:equalTwistsIntNEven}
Let $N \in \NN$ be even, and suppose $\cR$ is a cylinder with depth $N$. If $t$ is a tiling of $\cR$, then $T^{\ex}(t) = T^{\ey}(t) = T^{\ez}(t) \in \ZZ$.
\end{prop}

\begin{proof}
Suppose $\cR = \cD + [0,N]\vw$, where $\cD \subset \pi$ is simply connected and $\vw \in \Delta$ is the axis of the cylinder. 
Let $\cA \subset \pi$ be a rectangle with vertices in $\ZZ^3$ such that $\cD \subset \cA$: this implies that the box
$\cB = \cA + [0,N]\vw \supset \cR$. Let $\vu \in \Phi, \vu \perp \vw$. We want to show that $T^{\vu}(t) = T^{\vw}(t)$.

Let $t$ be a tiling of $\cR$, and let $t_*$ be the tiling of $\cB \setminus \cR$ such that every dimer is parallel to $\vw$. Applying Lemma \ref{lemma:allPretwistsBoxes} to the box $\cB$, we see that $T^{\vu}(t \sqcup t_*) = T^{\vw}(t)$: it remains to show that $T^{\vu}(t \sqcup t_*) - T^{\vu}(t) = 0$.  

Let $\tbasew$ be the tiling of $\cR$ such that every domino is parallel to $\vw$, and let $Q \subset \pi$ be a basic square such that $\interior(Q) \subset \cA \setminus \cD$. Let $t_Q$ be the set of $N/2$ dominoes of $t_*$ contained in $Q + [0,N]\vw$: we have 
$$T^{\vu}(t \sqcup t_*) - T^{\vu}(t) = T^{\vu}(t,t_*) + T^{\vu}(t_*,t) = \sum_{\interior(Q) \subset \cA \setminus \cD} T^{\vu}(t_Q, t) + T^{\vu}(t,t_Q).$$ 
Notice that, for every domino $d \in t_Q$, $\tv(d) = \ccol(Q) \vw$. Moreover, the dominoes in $S_{t,\vu} = S(t,\vu,Q,N)$ are precisely the ones that intersect the $\vu$-shade of at least one domino of $t_Q$, so that
$$T^{\vu}(t_Q, t) = \frac{1}{4}\sum_{d \in S_{t,\vu}} \det(\tv(d),\ccol(Q)\vw,\vu) = \frac{\ccol(Q)}{4}\sum_{d \in S_{t,\vu}} \det(\tv(d),\vw,\vu),$$
which equals $0$ by Lemma \ref{lemma:alternativeFluxFormula}.
Analogously (the first equality below uses Lemma \ref{lemma:twistNegatingDirection}),
$$T^{\vu}(t,t_Q) = T^{-\vu}(t_Q,t) = \frac{\ccol(Q)}{4}\sum_{d \in S(t,-\vu,Q,n)} \det(\tv(d),\vw,-\vu) = 0.$$
%so that
%$$\sum_{\interior(Q) \subset \cA \setminus \cD} T^{\vu}(t_Q, t) = \sum_{\interior(Q) \subset \cA \setminus \cD} T^{\vu}(t,t_Q) = 0.$$
Since $T^{\vw}(t) \in \ZZ$ (by Proposition \ref{prop:writheFormula}), we have completed the proof. 
\end{proof}
\begin{lemma}
\label{lemma:equalTwistsNOdd}
Let $N \in \ZZ$ be odd, and let $\cR$ be a cylinder with depth $N$ that admits a tiling $t$. Then $T^{\ex}(t) = T^{\ey}(t) = T^{\ez}(t) \in \half\ZZ$.
\end{lemma}
In fact, we prove in Proposition \ref{prop:equalTwistsIntNOdd} that $T^{\ex}(t) = T^{\ey}(t) = T^{\ez}(t) \in \ZZ$, but for our proof this first step is needed. Also, it is not clear when a cylinder with odd depth $N$ is tileable: see Lemma \ref{lemma:treeConstruction} for a related result.
\begin{proof}
Suppose $\cR$ has base $\cD$ and axis $\vw \in \Delta$, so that $\cR = \cD + [0,N] \vw$, and let $\vu \in \Phi$, $\vu \perp \vw$. 
Let $t$ be a tiling of $\cR$. We want to show that $T^{\vu}(t) = T^{\vw}(t)$.

Consider $\cR' = \cD + [0,2N]\vw$, and the tiling $\hat{t} = t_0 \sqcup t_1$ of $\cR'$ which consists of two copies $t_0$ and $t_1$ of $t$, where $t_0$ tiles the subregion $\cD + [0,N]\vw$ and $t_1$ tiles the subregion $\cD + [N,2N]\vw$.

%We first note that $\phi(\hat{t}) = 0$. To see this, let $d \in t$, and let $d_0 \in t_0, d_1 \in t_1$ be the copies of $d$ in each copy of $t$. Consider a square $Q$ not contained in $\cD$, and define $S(\vu,Q)$ as in Lemma \ref{lemma:alternativeFluxFormula}. Then $d_0 \in S(\ex,Q) \Leftrightarrow d_1 \in S(\ex,Q)$ and $\tv(d_0) = - \tv(d_1)$, so the sum in Lemma \ref{lemma:alternativeFluxFormula} equals $0$ for $\hat{t}$.  

By Proposition \ref{prop:equalTwistsIntNEven}, $T^{\vu}(\hat{t}) = T^{\vw}(\hat{t}) \in \ZZ$. Now clearly $T^{\vu}(\hat{t}) = 2 T^{\vu}(t)$, because the $\vu$-shades of dimers of $t_0$ do not intersect dimers of $t_1$ (and vice-versa). We need to prove that $T^{\vw}(\hat{t}) = 2 T^{\vw}(t)$. 

Notice that $T^{\vw}(\hat{t}) = T^{\vw}(t_0) + T^{\vw}(t_1) + T^{\vw}(t_0,t_1) = 2T^{\vw}(t) + T^{\vw}(t_0,t_1)$. Let $d_0 \in t_0$, $d_1 \in t_1$ be dominoes, and let $\tilde{d}_0$ and $\tilde{d}_1$ be the dominoes of $t$ that they ``refer to''. If $\tilde{d}_0 \neq \tilde{d}_1$, then clearly
$$d_1 \cap \cS^{\vw}(d_0) \neq \emptyset \Leftrightarrow \tilde{d}_1 \cap (\cS^{\vw}(\tilde{d}_0) \cup \cS^{-\vw}(\tilde{d}_0)) \neq \emptyset $$
and $\tau^{\vw}(d_0,d_1) = \frac{1}{4}\det(\tv(d_1), \tv(d_0), \vw) = - \frac{1}{4}\det(\tv(\tilde{d}_1), \tv(\tilde{d}_0), \vw) = \tau^{-\vw}(\tilde{d}_0,\tilde{d}_1) - \tau^{\vw}(\tilde{d}_0,\tilde{d}_1).$
Therefore,
$T^{\vw}(t_0,t_1) = \sum_{d,d' \in t}\tau^{-\vw}(d,d') - \tau^{\vw}(d,d') = 0$. Consequently, $T^{\vw}(\hat{t}) = 2 T^{\vw}(t)$ and thus $T^{\vu}(t) = T^{\vw}(t)$.

 Moreover, since $T^{\vw}(t) = T^{\vw}(\hat{t})/2$ and $T^{\vw}(\hat{t}) \in \ZZ$, it follows that $T^{\vw}(t) \in \half\ZZ$, which completes the proof.
\end{proof}

\begin{rem}
\label{rem:pseudocylinderOddDepth}
The statement of Lemma \ref{lemma:equalTwistsNOdd} (and also of Proposition \ref{prop:equalTwistsMultiplex}) also holds for pseudocylinders with odd depth. The key observation is that if $\hat{t}$ is the tiling obtained by stacking two copies of $t$, one above the other, then $\phi(\hat{t}) = 0$ (see Remark \ref{rem:pseudocylinderEvenDepth} for the definition of the flux $\phi$). 
\end{rem}

%If $n$ is even and $t$ is a tiling of a cylinder with depth $n$, we already know that $T^{\ex}(t), T^{\ey}(t),T^{\ez}(t) \in \ZZ$ (use Lemma \ref{lemma:writheDifference} and the fact that $\phi(t) \in \ZZ$). We now wish to show that the same holds when $n$ is odd.

\begin{lemma}
\label{lemma:integerTwistDifference}
Let $\cD \subset \pi$ be a planar region, and let $\vw \in \Delta$ be the normal vector for $\pi$. For each $k \in \NN$, write $\cR_k = \cD + [0,2k+1]\vw$. If $k_1, k_2 \in \NN$, then for each $\vu \in \Phi$ and every pair of tilings $t_1$ of $\cR_{k_1}$, $t_2$ of $\cR_{k_2}$, $T^{\vu}(t_1) - T^{\vu}(t_2) \in \ZZ$.
\end{lemma}
Should $\cR_{k_1}$ or $\cR_{k_2}$ not be tileable, the statement is vacuously true.
\begin{proof}
By Lemma \ref{lemma:equalTwistsNOdd}, it suffices to show the result for $\vu \perp \vw$.
Let $t_1$ and $t_2$ be tilings of $\cR_{k_1}$ and $\cR_{k_2}$, respectively. Consider the cylinder with even depth $\cR = \cD + [0,2k_1 + 2k_2 + 2]\vw$, and let $\tilde{t}_2$ denote the tiling of $\cD + [2k_1 + 1,2k_1 + 1 + 2k_2 + 1]\vw$ which is a copy of $t_2$. If $t = t_1 \sqcup \tilde{t}_2$, then $T^{\vu}(t) \in \ZZ$, by Proposition \ref{prop:equalTwistsIntNEven}. Also, since $\vu \perp \vw$, $T^{\vu}(t) = T^{\vu}(t_1) + T^{\vu}(\tilde{t}_2) = T^{\vu}(t_1) + T^{\vu}(t_2)$, so that  
$$ T^{\vu}(t_1) - T^{\vu}(t_2) = T^{\vu}(t_1) + T^{\vu}(t_2) - 2T^{\vu}(t_2) = T^{\vu}(t) - 2T^{\vu}(t_2).$$ 
Since, by Lemma \ref{lemma:equalTwistsNOdd}, $2T^{\vu}(t_2) \in \ZZ$, we're done. 
%
%We first show that the result holds when $k_1 = k_2 = k$:
%by Proposition \ref{prop:equalTwists}, it suffices to show that $T^{\ex}(t_1) - T^{\ex}(t_2) \in \ZZ$. Let $\tilde{t}_0$ be any tiling of $\cR_k = \cD \times [0,2k+1]$, and let $t_0$ denote the tiling of $\cD \times [2k+1,4k+2]$ which is a copy of $\tilde{t}_0$. Clearly, $t_1 \cup t_0$ and $t_2 \cup t_0$ are both tilings of a cylinder with even depth, so that
%$$
%T^{\ex}(t_1) - T^{\ex}(t_2) = T^{\ex}(t_1) + T^{\ex}(t_0) - (T^{\ex}(t_2) + T^{\ex}(t_0)) = T^{\ex}(t_1 \cup t_0) - T^{\ex}(t_2 \cup t_0), 
%$$ 
%which is an integer.
%
%For the general case, assume without loss of generality that $k_1 \leq k_2$.
%Let $t_1$ and $t_2$ be, respectively, tilings of $\cR_{k_1}$ and $\cR_{k_2}$; let $\tilde{t}_2$ be the tiling of $\cR_{k_2}$ obtained from $t_1$ by tiling $\cD \times [2k_1 + 1, 2k_2 + 1]$ with all dominoes parallel to $\ez$. Then, clearly
%$T^{\ez}(t_1) - T^{\ez}(t_2) = T^{\ez}(\tilde{t}_2) - T^{\ez}(t_2) \in \ZZ,$
%by what we proved before ($\tilde{t}_2$ and $t_2$ are both tilings of $\cR_{k_2}$). By Proposition \ref{prop:equalTwists}, we're done.
\end{proof}

%{\color[rgb]{0,0,1} We should somehow introduce the subject, and explain the relation between tilings and matchings.}
\begin{lemma}
\label{lemma:treeConstruction}
Let $\pi$ be a basic plane with normal $\vw \in \Delta$, and let $\cD \subset \pi$ be a planar region with connected interior such that 
$$\# (\text{black squares in } \cD) = \# (\text{white squares in } \cD) = n.$$ 
Then there exists a tiling $t_0$ of $\cD +[0,2n-1]\vw$ such that $T^{\vw}(t_0) \in \ZZ$.
\end{lemma}

Notice that, with Lemma \ref{lemma:treeConstruction}, the proof of Proposition \ref{prop:equalTwistsMultiplex} is complete. However, we need some preparation before we can prove Lemma \ref{lemma:treeConstruction}.

%%----- Graph theory stuff

It is a well-known fact that domino tilings of a region can be seen as perfect matchings of a related graph: in fact, if we consider the graph whose vertices are centers of the cubes (squares in the planar case) of the region, and where two vertices are joined if their Euclidean distance is $1$, then a domino tiling can be directly translated as a perfect matching in this graph. This graph is called the \emph{associated graph} of a region $\cR$ (planar or spatial), and denoted $G(\cR)$.
Since the proof of Lemma \ref{lemma:treeConstruction} will come more naturally in the setting of matchings in associated graphs, we shall revert to this viewpoint for what follows. 

A \emph{bicoloring} of a graph $G$ is a coloring of each vertex of $G$ as black or white, in such a way that no two adjacent vertices have the same color.
Associated graphs for a region $\cR$ are always bicolored: each vertex inherits the color of the cube (or square) it refers to.
For what follows, we shall assume that all graphs are already bicolored. 
Moreover, any subgraph of a bicolored graph $G$ (for instance, the one obtained after deleting a vertex) shall inherit the bicoloring of $G$.  

\begin{lemma}
\label{lemma:leafColors}
Let $T$ be a bicolored tree. If all leaves are white, then the number of white vertices in $T$ is strictly larger than the number of black vertices in $T$.
\end{lemma}
By definition, a tree is connected and, therefore, nonempty.
\begin{proof}
We proceed by induction on the number of vertices. The result is clearly true if $T$ has three or fewer vertices.
Suppose, by induction, that the result holds for balanced trees with $m$ vertices for any $m < n$. Let $T$ be a tree with $n$ vertices such that all leaves are white.

Let $w \in T$ be a (white) leaf, and let $v \in T$ be the only neighbor of $w$. Let $F$ be the forest obtained by deleting $w$ and $v$: $F$ is nonempty, otherwise $v$ would have to be a black leaf, which contradicts the hypothesis. Now for each connected component $T'$ of $F$, $T'$ is a tree with less than $n$ vertices such that all leaves are white: therefore, by induction, $F$ has more white vertices than black vertices. However, the vertices of $T$ are those of $F$ plus one black vertex ($v$) and one white vertex ($w$), so that the number of white vertices in $T$ is greater than that of black ones.
By induction, we get the result.  
\end{proof}

A connected bicolored graph $G$ is \emph{balanced} if the number of white vertices equals the number of black ones. By Lemma \ref{lemma:leafColors}, a balanced tree must have at least one white leaf and one black leaf.

A \emph{perfect matching} of a bipartite graph $G$ is a set of pairwise disjoint edges of $G$, such that every vertex is adjacent to (exactly) one of the edges in the matching. Clearly, a necessary condition for the existence of a perfect matching is that $G$ is balanced.

Let $G = (V,E)$ be a bicolored graph (in this notation, $V$ is the vertex set of $G$, and $E$ is its edge set), and let $I_n = \{0,1, \ldots, n-1\}$. Let $G \times I_n = (V \times I_n,E_n)$, where $E_n$ consists of all edges connecting $(v,j)$ and $(v,j+1)$, for each $v \in V$ and $j \in I_{n-1}$, plus the edges connecting $(v_1,j)$ and $(v_2,j)$ for each $j \in I_n$ whenever the edge $v_1v_2 \in E$. The color of a vertex $(v,j) \in G \times I_n$ equals the color of $v$ if and only if $j$ is even. Naturally, if $\cD \subset \pi$ is a planar region with normal $\vw$, then $G(\cD) \times I_n \approx G(\cD + [0,n]\vw)$.

Let $G$ be a (nonempty) balanced connected bicolored graph with $2n$ vertices. Algorithm \ref{algorithm:perfectMatching} finds a perfect matching $M$ of $G \times I_{2n-1}$.
%Let $T$ be a spanning tree for $G$. Since $T$ is a balanced tree, by Lemma \ref{lemma:leafColors} it contains at least one white leaf and one black leaf.
\begin{algorithm}
\caption{Algorithm for finding a perfect matching $M$ of $G \times I_{2n-1}$.}
\label{algorithm:perfectMatching}
\begin{algorithmic}
\State Pick a spanning tree $T$ for $G$.
\Comment \textit{\footnotesize $T$ is a balanced tree}
\State $M_0 \gets \emptyset$
\State $T_0 \gets T$
\State $k \gets 0$
\While {$T_k \neq \emptyset$}
\State Pick a white leaf $v_w$ and a black leaf $v_b$ of $T_k$ \newline
\Comment \textit{ \footnotesize Lemma \ref{lemma:leafColors} ensures that a balanced tree has at least one white leaf and one black leaf}
\State Pick a path $P_k = v_{k,1}v_{k,2} \ldots v_{k,m_k}$ in $T_k$ from $v_w$ to $v_b$ \newline
\Comment{\textit{ \footnotesize i.e., $v_{k,1} = v_w, v_{k,m_k} = v_b$; notice that $m_k$ is necessarily even}}
\State $D_k \gets \{(v,2k-1)(v,2k) \,|\, v \in T \setminus T_k\} \sqcup \{(v,2k)(v,2k+1)\,|\, v \in T_k \setminus P_k)\}$
\State $E_k \gets \{(v_{k,2i-1}v_{k,2i},2k) \,|\, 1 \leq i \leq \frac{m_k}{2}\} \sqcup \{(v_{k,2i}v_{k,2i+1},2k+1) \,|\, 1 \leq i < \frac{m_k}{2}\}$\newline        
\Comment{\textit{\footnotesize Here $(vw,l)$ means the edge $(v,l)(w,l)$, i.e., the edge between the vertices $(v,l)$ and $(w,l)$}}
\State $M_{k+1} \gets M_k \sqcup D_k \sqcup E_k$
\State $T_{k+1} \gets T_k \setminus \{v_w, v_b\}$ \newline
\Comment{\textit{\footnotesize Notice that $T_{k+1}$ is still a balanced tree (except in the last iteration, when it is empty)}}
\State $k \gets k + 1$
\EndWhile 
\State $M \gets M_k$
\end{algorithmic}
\end{algorithm}

\begin{lemma}
If $G$ is a connected bicolored balanced graph with $2n$ vertices, then the set of edges $M$ generated by running Algorithm \ref{algorithm:perfectMatching} on $G$ is a perfect matching of $G \times I_{2n-1}$.
\end{lemma}
\begin{proof}
%It suffices to prove that $M \subset E(G \times I_{2n-1})$ and that every vertex $(v,j)$ of $G \times I_{2n-1}$ is adjacent to exactly one edge of $M$.
To see that $M \subset E(G \times I_{2n-1})$, notice that any spanning tree has $2n$ vertices, and exactly two vertices are deleted in each iteration, so that the last iteration where $T_k \neq \emptyset$ occurs when $k = n-1$. In all other iterations (i.e., $0 \leq k < n-1$), clearly any edge created is contained in $E(G \times I_{2n-1})$. When $k = n-1$, $T_k$ is a balanced tree with two vertices, so $P_k = v_w v_b$ and only the edges $\{(v,2n-3)(v,2n-2) \,|\, v \in T \setminus T_k\}$ plus the edge $(v_w,2n-2)(v_b,2n-2)$ are created. Now these edges are contained in $E(G \times I_{2n-1})$, so we're done.  

%This proves that $M$ is a matching of $G \times I_{2n-1}$. 
This proves that $M$ is a subset of the edgeset of $G \times I_{2n-1}$. The reader will easily convince himself that it is a perfect matching (i.e., that every vertex $(v,j)$ of $G \times I_{2n-1}$ is adjacent to exactly one edge of $M$).
%Let $(v,j) \in G \times I_{2n-1}$, and let $k_0 = \ceil{j/2}$: clearly an edge involving $(v,j)$ can only appear when $k = k_0$. Now either $v \in T \setminus T_k$, $v \in P_k$ or $v \in T_k \setminus P_k$: in either case, exactly one edge  
\end{proof}

Next we shall prove Lemma \ref{lemma:treeConstruction}. In order to make the explanation clearer, we shall first introduce a few concepts. Let $G$ be a bicolored connected balanced graph, and consider the perfect matching $M$ of $G \times I_{2n-1}$ obtained by running Algorithm \ref{algorithm:perfectMatching} on $G$, as well as the intermediate objects that were created, such as $E_k$ and $P_k$.

Given an edge $e = (vw,j)$ of $E_k$, we say $e$ is adjacent to $v$ and to $w$ (even though it is not an edge of $G$). For $v \in G$ and $0 \leq k \leq 2n-1$, we write $E(v,k) = \{e \in E_k \,|\, e \text{ adjacent to } v\}$. 

Consider the paths $P_k = v_{k,1}v_{k,2} \ldots v_{k,m_k}$ chosen in each step of the algorithm (we shall also use this notation in the proof).
If $j > k$, we say that a path $P_j$ \emph{meets} $P_k$ \emph{at} $v \in G$ if $v = v_{j,i} \in P_k$ for some $i > 1$, but $v_{j,i-1} \notin P_k$. Analogously, $P_j$ \emph{leaves} $P_k$ \emph{at} $v$ if $v = v_{j,i} \in P_k$ for some $i < m_j$, but $v_{j,i+1} \notin P_k$. Notice that a path $P_j$ can meet and leave $P_k$ at the same vertex $v$. Also, notice that $P_j$ can only meet (resp. leave) $P_k$ at most once (i.e., at no more than one vertex).

%Let $G$ be a connected balanced graph with $2n$ vertices such that $G = G(\cD)$ for some planar region $\cD$, and let $M$ be the perfect matching of $G \times I_{2n-1}$ obtained after running Algorithm \ref{algorithm:perfectMatching} on $G$. We can clearly associate each edge $e \in M$ with a unique domino $d$ of $\cR = \cD \times [0,2n-1]$, so that $M$ is associated with a tiling $t$ of $\cR$.

%If $e_0,e_1 \in M$, we shall abuse notation and write $\tau^{\ez}(e_0,e_1) = \tau^{\ez}(d_0,d_1)$, where $d_0$ and $d_1$ are the associated dominoes of $e_0$ and $e_1$, respectively. Therefore, sets of edges can also be seen as sets of dominoes (or sets of dimers), so we can write, for instance, $T^{\ez}(E_k, E_j) = \sum_{e \in E_k, e' \in E_j} \tau^{\ez}(e,e')$.     

\begin{proof}[Proof of Lemma \ref{lemma:treeConstruction}]
%Since the result is clear when $n = 1$, assume, by induction, that the result is true for $n-1$ for some $n \geq 2$.
Consider the graph $G = G(\cD)$ associated with the planar region $\cD$. Clearly $G$ is balanced; since $\cD$ has connected interior, it follows that $G$ is also connected. Let $M$ be the perfect matching obtained after running Algorithm \ref{algorithm:perfectMatching} on $G$, and let $t$ be the tiling of $\cD + [0,2n-1]\vw$ associated with $M$.

If $e_0, e_1 \in M$, we will abuse notation and write $\tau^{\vw}(e_0,e_1) = \tau^{\vw}(d_0,d_1)$, where $d_i \in t$ is the domino associated with $e_i \in M$: we also say that two edges are parallel if their associated dominoes are parallel. %Analogously, sets $A_0, A_1 \subset M$ are also seen as sets of dominoes, and we may write $T^{\ez}(A_0,A_1)$ to mean $\sum_{e_0 \in A_0, e_1 \in A_1} \tau^{\ez}(e_0,e_1)$. 

Notice that the only dominoes that are not parallel to $\vw$ are those associated with the edges of $E_k$ for each $k$: therefore, 
$$
T^{\vw}(t) = \sum_{i \leq j} T^{\vw}(E_i, E_j) = \sum_{\substack{i \leq j \\e \in E_i, e' \in E_j}}\tau^{\vw}(e,e').
$$ 
Fix $0 \leq k \leq n-1$. We want to show that $\sum_{j \geq k} T^{\vw}(E_k, E_j) \in \ZZ$. First, write 
$$\sum_{j \geq k} T^{\vw}(E_{k}, E_j) = \sum_{\substack{1 < i < m_{k}\\j \geq k}}T^{\vw}(E(v_{k,i},k), E(v_{k,i},j));$$
we may ignore $v_{k,1}$ and $v_{k,m_k}$ because they are deleted from the tree in step $k$, so that $E(v_{k,1},j) = E(v_{k,m_k},j) = \emptyset$ for each $j > k$ (for $j = k$, it contains only one edge, so there is also no effect).

\begin{figure}[ht]%
\centering
\subfloat[Some cases where $P_k$ goes straight at $v$: the effects are, respectively, $1$, $1/2$, $1/2$ and $0$.]{\includegraphics[width=0.45\columnwidth]{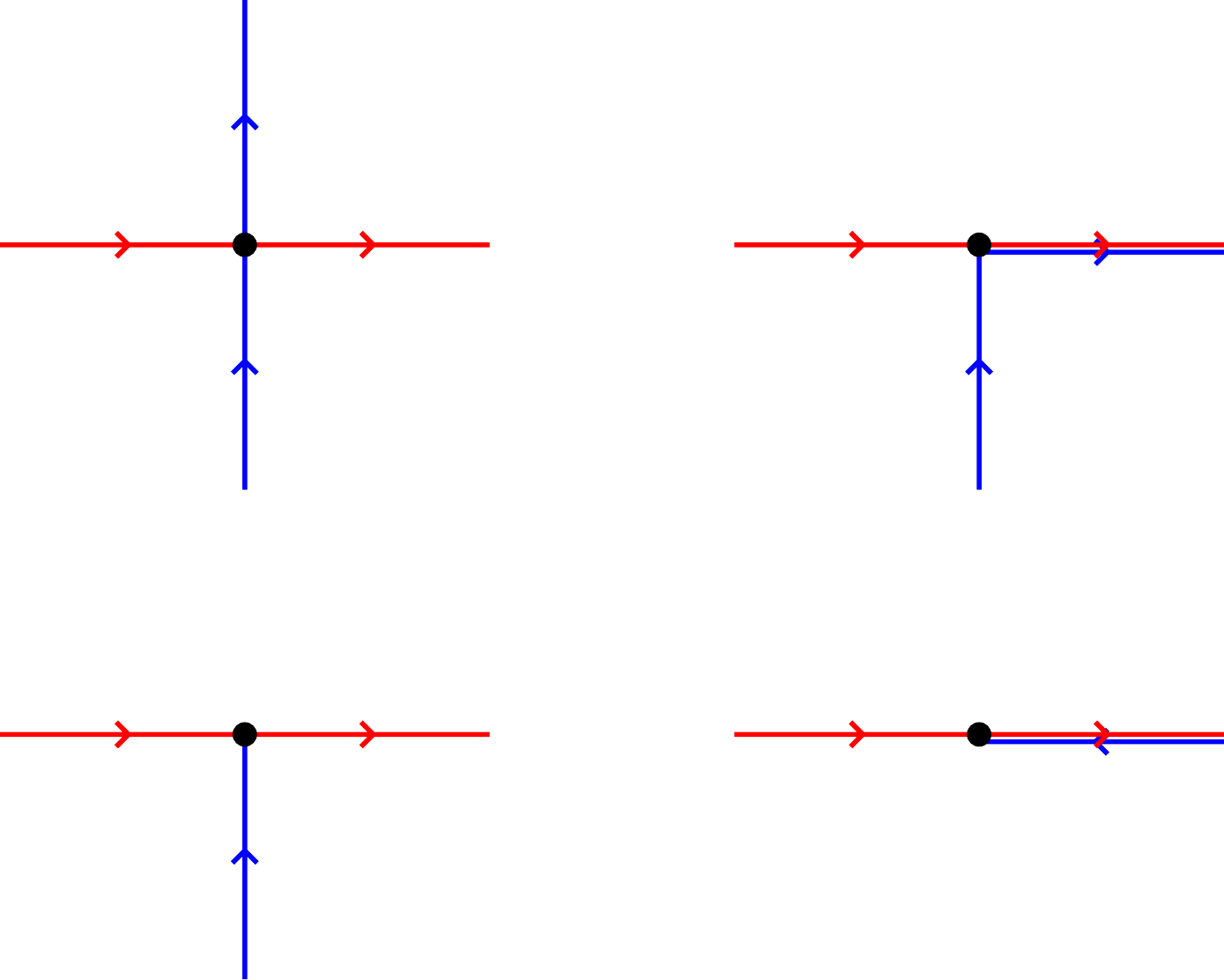}\label{subfig:treeConstructionStraight}} 
\hspace{0.075\columnwidth}
\subfloat[Some cases where $P_k$ makes a left turn at $v$: in this case the red segments have nonzero effect on one another (in this case it is $1/4$): the effects on the blue segments are, respectively, $1/2$, $0$, $1/4$ and $-1/4$.]{\includegraphics[width=0.45\columnwidth]{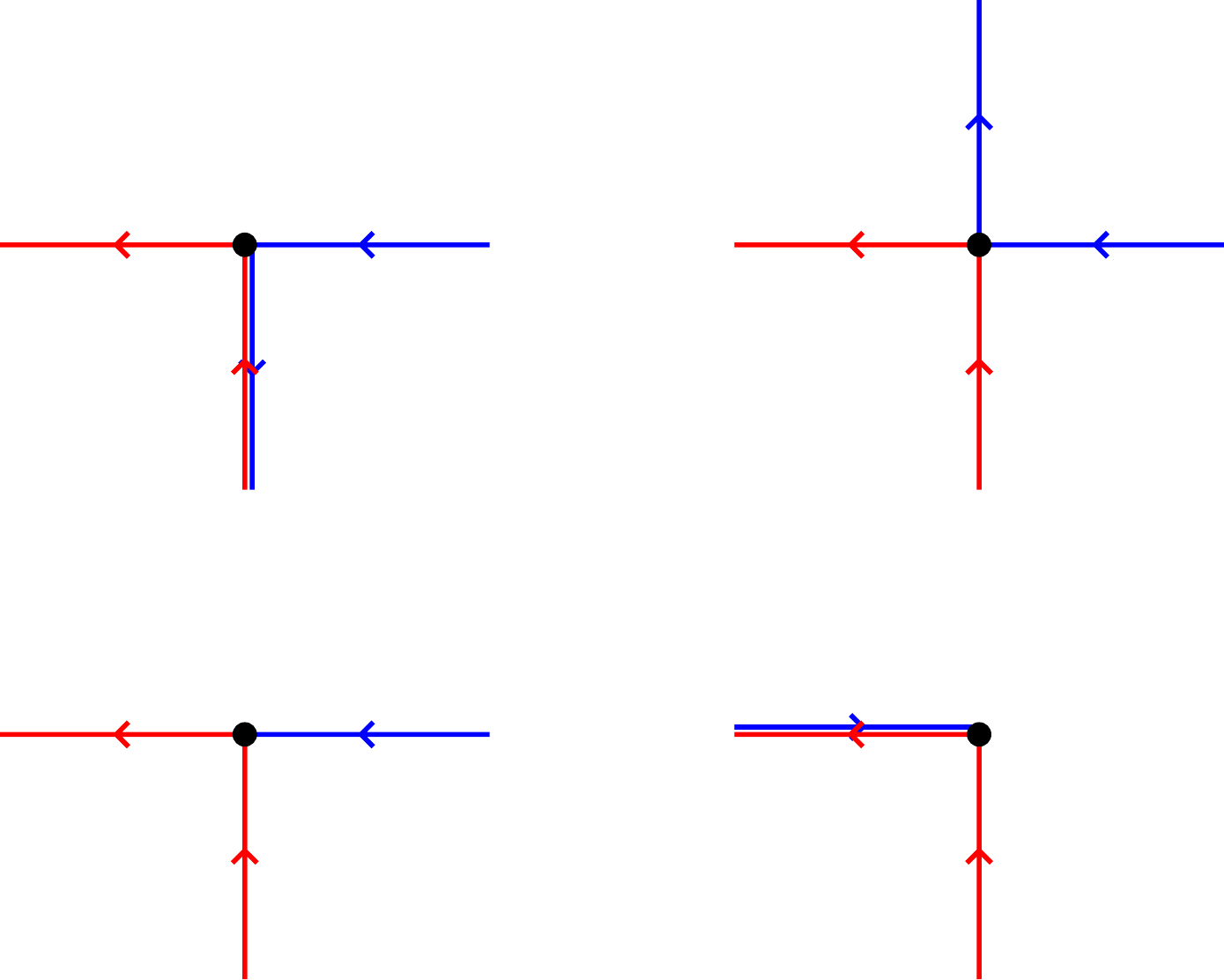}\label{subfig:treeConstructionTurn}}
\caption{Edges of $E_k$ (red) and $E_j$ (blue) for some $j > k$, portrayed as edges of $G$. The edges are oriented as $\tv(d)$, where $d$ is the associated domino. The portrayed vertex is $v$, which we assume here to be black: notice that $v$ is one of the endpoints of $P_j$ in the bottom two cases of each figure.}%
\label{fig:treeConstructionEffects}%
\end{figure}

For $v = v_{k,i}$, $1 < i < m_k$, we claim that, modulo $1$, 
$\sum_{j \geq k} T^{\vw}(E(v,k), E(v,j))$ equals  
$$\half \left(\#\{j > k \,|\, P_j \text{ meets } P_k \text{ at } v\} + \#\{j > k \,|\, P_j \text{ leaves } P_k \text{ at } v\}\right)$$ 
(in other words, their difference is an integer).

If the two edges in $E(v,k)$ are parallel (i.e., $P_k$ goes straight at $v$), then $T^{\vw}(E(v,k), E(v,k)) = 0$. By checking a number of cases (see Figure \ref{subfig:treeConstructionStraight}) we see that the following holds for each $j > k$: 
\begin{equation}
T^{\vw}(E(v,k), E(v,j)) = 
\begin{cases}
\pm 1, &P_j \text{ meets and leaves } P_k \text{ at } v; \\
\pm 1/2, &P_j \text{ either meets or leaves } P_k \text{ at } v; \\ 
0, &\text{otherwise.}
\end{cases}
\label{eq:effectsMeetingAndLeaving}
\end{equation}

If the two edges in $E(v,k)$ are not parallel (i.e., $P_k$ makes a turn at $v$), we proceed as follows: assume that the path $P_k$ makes a left turn and that $v$ is a black vertex (the other cases are analogous). Let $k'$ be the step where $v$ is chosen as the black leaf to be deleted (so that $v = v_{k', m_{k'}}$): again, inspection of a few possible cases (some of which are shown in Figure \ref{subfig:treeConstructionTurn}) shows that \eqref{eq:effectsMeetingAndLeaving} holds for $k < j < k'$ (and for $j > k'$, obviously $T^{\vw}(E(v,k), E(v,j)) = 0$).
Also, $T^{\vw}(E(v,k), E(v,k)) = 1/4$ (because it is a left turn and $v$ is black), and (see the last two examples in Figure \ref{subfig:treeConstructionTurn})
$$
T^{\vw}((E(v,k), E(v,k')) = 
\begin{cases}
1/4, &P_{k'} \text{ meets } P_k \text{ at } v;\\
-1/4, &\text{otherwise;}
\end{cases}
$$
so that $T^{\vw}(E(v,k), E(v,k)) + T^{\vw}((E(v,k), E(v,k')) = 1/2$ if and only if $P_{k'}$ meets $P_j$ at $v$ (and $0$ otherwise), so that we get the result.

Now let $N(v) = \#\{j > k \,|\, P_j \text{ meets } P_k \text{ at } v\} + \#\{j > k \,|\, P_j \text{ leaves } P_k \text{ at } v\}.$ To finish the proof, we need to show that 
$$
N = \sum_{1 < i < m_k}N(v_{k,i}) = \#\{j > k \,|\, P_j \text{ meets } P_k\} + \#\{j > k \,|\, P_j \text{ leaves } P_k\}
$$ is even.
Because all $P_j$'s are paths in a tree $T_k$, it follows that each path meets (or leaves) $P_k$ at most once. Therefore, each $j > k$ may contribute $0$ (if it never meets nor leaves $P_k$), $1$ (if it either meets or leaves $P_k$, but not both) or $2$ (if it meets and leaves $P_k$) to the above sum. This contribution is $0$ if $v_{j,1}, v_{j,m_j} \in P_k$; it is $0$ or $2$ if $v_{j,1}, v_{j,m_j} \notin P_k$. If exactly one of the two is in $P_k$, the contribution is $1$; however, since
$\#\{j > k \,|\, v_{j,1} \in P_k, v_{j,m_j} \notin P_k\} = \#\{j > k \,|\, v_{j,1} \notin P_k, v_{j,m_j} \in P_k\},$  
it follows that $N$ is even, so that
$T^{\vw}(t) = \sum_{j \geq k} T^{\vw}(E_{k}, E_j) \equiv N/2 \pmod{1}$
is an integer. 
\end{proof}
%---------------- End of Graph Theory stuff
We sum up our main results in the following proposition:
\begin{prop}
\label{prop:equalTwistsIntNOdd}
Let $\cD \subset \pi$ be a planar region with normal vector $\vw$ and connected interior such that
$$\# (\text{black squares in } \cD) = \# (\text{white squares in } \cD) = n.$$ 
Then $\cD + [0,2n-1]\vw$ is tileable.
Moreover, for each $k \in \NN$ such that $\cD + [0,2k-1]\vw$ is tileable (in particular, for each $k \geq n$), every tiling $t$ of $\cD + [0,2k-1]\vw$ satisfies $T^{\ex}(t) = T^{\ey}(t) = T^{\ez}(t) \in \ZZ$.
\end{prop}
\begin{proof}
Follows directly from Lemmas \ref{lemma:equalTwistsNOdd}, \ref{lemma:integerTwistDifference} and \ref{lemma:treeConstruction}.
\end{proof}

%----------------- 
Now that we have seen that the twist, as in Definition \ref{def:twist}, is well-defined for cylinders, we may adopt the notation $\Tw(t)$ when $t$ is a tiling of a cylinder. 
\section{Additive properties and proof of Theorem \ref{theo:main}}
\label{sec:additiveProperties}

The goal for this section is to discuss some additive properties of the twist and to complete the proof of Theorem \ref{theo:main}. %These properties allow us to extend the twist, up to a constant, to a much broader class of regions than we did in Section \ref{sec:combTwistBoxes}.

%We say that a region $\cR$ is \emph{tileable} if there exists a tiling $t$ of $\cR$. The empty region $\cR = \emptyset$ is tileable, since it admits the empty tiling. 

\begin{lemma}
\label{lemma:secondOrderTwistDiff}
Let $\cR_0$ and $\cR_1$ be two regions whose interiors are disjoint. Let $t_{\cR_0,0}$ and $t_{\cR_0,1}$ be two tilings of $\cR_0$ and $t_{\cR_1,0}$ and $t_{\cR_1,1}$ be two tilings of $\cR_1$. For each $(i,j) \in \{0,1\}^2$, set $t_{ij} = t_{\cR_0,i} \sqcup t_{\cR_1,j}$, which is a tiling of $\cR = \cR_0 \cup \cR_1$. Let $\Gamma_i^* = \Gamma^*(t_{\cR_i,0}, t_{\cR_i,1})$, $i=0,1$. Then, for each $\vu \in \Phi$,
$$
T^{\vu}(t_{00}) - T^{\vu}(t_{01}) - T^{\vu}(t_{10}) + T^{\vu}(t_{11}) = 2\sum_{\gamma_0 \in \Gamma_0^*,\gamma_1 \in \Gamma_1^*} \Link(\gamma_0,\gamma_1). 
$$
In particular, if $\cR_0$ or $\cR_1$ is simply connected, then $T^{\vu}(t_{00}) - T^{\vu}(t_{01}) - T^{\vu}(t_{10}) + T^{\vu}(t_{11}) = 0$.
\end{lemma}
\begin{proof}
For shortness, given two sets of segments $A_0$ and $A_1$, we shall in this proof write $\ST^{\vu}(A_0,A_1) = T^{\vu}(A_0,A_1) + T^{\vu}(A_1,A_0)$.

For each $(i,j) \in \{0,1\}^2,$ we have 
$$
T^{\vu}(t_{ij}) = T^{\vu}(t_{\cR_0,i} \sqcup t_{\cR_1,j})
= T^{\vu}(t_{\cR_0,i}) + T^{\vu}(t_{\cR_1,j}) + \ST^{\vu}(t_{\cR_0,i},t_{\cR_1,j}).
$$
Notice that the last term is the only one that depends on both $i$ and $j$, so that it is the only one that does not cancel out in the sum $\sum_{i,j \in \{0,1\}} (-1)^{i+j} \Tw(t_{ij})$.
Therefore, we have
\begin{align*}
&\sum_{i,j \in \{0,1\}} (-1)^{i+j} T^{\vu}(t_{ij}) = \sum_{i,j \in \{0,1\}} (-1)^{i+j} \ST^{\vu}(t_{\cR_0,i},t_{\cR_1,j})=\\
&= \ST^{\vu}(t_{\cR_0,0} \sqcup (-t_{\cR_0,1}), t_{\cR_1,0} \sqcup (-t_{\cR_1,1})) 
= \sum_{\gamma_0 \in \Gamma_0^*, \gamma_1 \in \Gamma_1^*} \ST^{\vu}(\gamma_0,\gamma_1).
\end{align*}
%(the last equality is essentially the fact that for two segments $\ell$ and $\tilde{\ell}$, $\tau^{\vu}(\ell,-\tilde{\ell}) = \tau^{\vu}(-\ell,\tilde{\ell}) = -\tau^{\vu}(\ell,\tilde{\ell})).$ 
%notice that $t_{\cR_i,0}$ and $(-t_{\cR_i,1})$ are disjoint, since one is a set of dimers, and the other is a set of minus dimers.
%
%Recall from Section \ref{sec:topologicalGroundwork} the definitions of $\Gamma(t_0,t_1)$ and $\Gamma^*(t_0,t_1)$ for a pair of tilings $t_0$ and $t_1$. 
%If $\gamma \in \Gamma(t_0,t_1)$ is trivial, then $\sum_{\ell \in \gamma}\tau^{\vu}(\ell,\tilde{\ell}) = 0$ for any segment $\tilde{\ell}$: in particular, $T^{\vu}(\gamma,t) = 0$ for any set of segments $t$. 
%Thus, we have
%$$
%\ST^{\vu}(t_{\cR_0,0} \sqcup (-t_{\cR_0,1}), t_{\cR_1,0} \sqcup (-t_{\cR_1,1})) = \sum_{\gamma_0 \in \Gamma_0^*, \gamma_1 \in \Gamma_1^*} \ST^{\vu}(\gamma_0,\gamma_1). 
%$$
Since for each pair $\gamma_0, \gamma_1$ in the sum we have $\gamma_i \subset \interior(\cR_i)$, it follows that $\gamma_0 \cap \gamma_1 = \emptyset$.
Hence, by Lemma \ref{lemma:linkingNumber}, $\ST^{\vu}(\gamma_0,\gamma_1) = 2 \Link(\gamma_0, \gamma_1)$, which yields the result. 
\end{proof}

\begin{coro}
\label{coro:embeddableRegions}
Let $\cR$ be a simply connected region, and suppose that there exists a box $\cB \supset \cR$ such that $\cB \setminus \cR$ is tileable. If $t_0, t_1$ are two tilings of $\cR$ and $t_a, t_b$ are two tilings of $\cB \setminus \cR$, then
$$ 
\Tw(t_0 \sqcup t_a) - \Tw(t_1 \sqcup t_a) = \Tw(t_0 \sqcup t_b) - \Tw(t_1 \sqcup t_b).
$$ 
\end{coro}
\begin{proof}
Use Lemma \ref{lemma:secondOrderTwistDiff} with $\cR_0 = \cR$, $\cR_1 = \cB \setminus \cR$.
\end{proof}

\begin{lemma}
\label{lemma:trivialTilingBoxes}
Suppose $L,M,N$ are even positive integers, and let $\cB = [0,L] \times [0,M] \times [0,N]$. If $\cR \subset \cB$ is a cylinder with even depth, then there exists a tiling $t_*$ of $\cB \setminus \cR$ such that 
$\Tw(t \sqcup t_*) = \Tw(t)$
for each tiling $t$ of $\cR$.
\end{lemma}
Corollary \ref{coro:embeddableRegions} and Lemma \ref{lemma:trivialTilingBoxes} imply that for any tiling $\tilde{t}_*$ of $\cB \setminus \cR$, there exists a constant $K$ such that, for any tiling $t$ of $\cR$, $\Tw(t \sqcup \tilde{t}_*) = \Tw(t) + K$.  
\begin{proof}
%Since all dimensions of the box are even, we may assume without loss of generality (by Proposition \ref{prop:equalTwistsMultiplex}, rotating does not change the twist) that $\cR$ is a $\ez$-cylinder, so that $\cR = \cD \times [e,f]$, with $f - e$ even.
We may without loss of generality assume that the axis of $\cR$ is $\ez$, so that $\cR = \cD + [E,F]\ez$, where $\cD \subset [0,L] \times [0,M] \times \{0\}$ and $F - E$ is even.   

Let $\cB_1 = [0,L] \times [0,M] \times [E,F]$. Clearly there exists a tiling $t_{1,*}$ of $\cB_1 \setminus \cR$ such that every domino is parallel to $\ez$: hence, $\Tw(t \sqcup t_{1,*}) = T^{\ez}(t \sqcup t_{1,*}) = T^{\ez}(t) = \Tw(t)$ for each tiling $t$ of $\cR$. On the other hand, since $L$ is even, there exists a tiling $t_{2,*}$ of $\cB \setminus \cB_1$ such that every dimer is parallel to $\ex$, so that $\Tw(t \sqcup t_{2,*}) = T^{\ex}(t \sqcup t_{2,*}) = T^{\ex}(t) = \Tw(t)$ for each tiling $t$ of $\cB_1$. Setting $t_* = t_{1,*} \sqcup t_{2,*}$ we get the result. 
\end{proof}
\begin{lemma}
\label{lemma:trivialTilingAllMultiplexes}
Let $\cR$ be a tileable cylinder with base $\cD$, axis $\vw \in \Delta$ and depth $n$. Let $\cR' = \cD + [0,2n]\vw$ be a cylinder with even depth formed by two copies of $\cR$; let $\cB \supset \cR'$ be a box with all dimensions even. Then there exist a tiling $t_*$ of $\cB \setminus \cR$ and a constant $K$ such that, for each tiling $t$ of $\cR$, $\Tw(t \sqcup t_*) = \Tw(t) + K.$
\end{lemma}
%Of course, when $\cR$ has even depth, we can find $t_*$ satisfying $K = 0$.
\begin{proof}
By Lemma \ref{lemma:trivialTilingBoxes}, there exists a tiling $\tilde{t}$ of $\cB \setminus \cR'$ such that $\Tw(t \sqcup \tilde{t}) = \Tw(t)$ for each tiling $t$ of $\cR'$.
Fix a tiling $t_0$ of $\cD + [n,2n]\vw$ (which is tileable because $\cR$ is tileable). If we set $t_* = t_0 \sqcup \tilde{t}$ and $K = \Tw(t_0)$, then for every tiling $t$ of $\cR$,
$$ \Tw(t \sqcup t_*) = \Tw(t \sqcup t_0 \sqcup \tilde{t}) = \Tw(t \sqcup t_0) = \Tw(t) + \Tw(t_0);$$
the last equality holding by fixing $\vu \in \Phi$, $\vu \perp \vw$ and writing $\Tw(t \sqcup t_0) = T^{\vu}(t \sqcup t_0) = T^{\vu}(t) + T^{\vu}(t_0)$. %where $\vw \perp \vu \in \Phi$.
\end{proof}

\begin{proof}[Proof of Theorem \ref{theo:main}]
The twist is constructed in Definition \ref{def:twist} and its integrality follows from Proposition \ref{prop:equalTwistsMultiplex}. Lemma \ref{lemma:fullyBalancedMultiplex} and Proposition \ref{prop:flipsAndTrits} yield items \ref{item:flips} and \ref{item:trits} . %We prove item \ref{item:duplexes} in Section \ref{sec:duplexRevisited}. We're left with proving item \ref{item:multiplexUnion}. 
To see item \ref{item:duplexes}, let $\cR$ be a duplex region with axis $\vw$, and consider the tiling $t_{\vw}$ such that all dominoes are parallel to $\vw$: clearly $\Tw(t_{\vw}) = P_{t_{\vw}}'(1) = 0$ (we assume that the reader is familiar with the notation from Chapter \ref{chap:twofloors}). Since the space of domino tilings of $\cR$ is connected by flips and trits (by Theorem \ref{theo:invariantProperties}), Proposition \ref{prop:flipsAndTrits}, together with Theorems \ref{theo:invariantProperties} and \ref{theo:tritProperties}, implies that for each tiling $t$ of $\cR$, $\Tw(t) = P_t'(1)$ (for a more direct proof of item \ref{item:duplexes}, see Section \ref{sec:duplexnotes}).

We're left with proving item \ref{item:multiplexUnion}.
Let $\cR$ be a cylinder, and suppose $\cR = \bigcup_{1 \leq j \leq m} \cR_j$, where each $\cR_j$ is a cylinder (they need not have the same axis) and $\interior(\cR_i) \cap \interior(\cR_j) = \emptyset$ if $i \neq j$. Suppose the bases, axes and depths are respectively, $\cD,\vw,n$ and $\cD_j, \vw_j, n_j$.

Let $t_{j,0}$ and $t_{j,1}$ be two tilings of $\cR_j$. It suffices to show that
$$\Tw\left(\bigsqcup_{1 \leq j \leq m} t_{j,1}\right) - \Tw\left(\bigsqcup_{1 \leq j \leq m} t_{j,0}\right) = \sum_{1 \leq j \leq m} (\Tw(t_{j,1}) - \Tw(t_{j,0})).$$ 

For $0 \leq j \leq m$, let $t_j = \bigsqcup_{1 \leq i \leq j} t_{i,1} \sqcup \bigsqcup_{j < i \leq m} t_{i,0}$. We want to show that $\Tw(t_m) - \Tw(t_0) = \sum_{1 \leq j \leq m} (\Tw(t_{j,1}) - \Tw(t_{j,0}))$.

Let $\cB$ be a box with all dimensions even such that $\cD + [0,2n]\vw \subset \cB$ and $\cD_j + [0,2n_j]\vw_j \subset \cB$ for $j=1,\ldots,m$. By Lemma \ref{lemma:trivialTilingAllMultiplexes}, there exist: a tiling $t_*$ of $\cB \setminus \cR$ and a constant $K$; and for each $j$, a tiling $t_{j,*}$ of $\cB \setminus \cR_j$ and a constant $K_j$ such that $\Tw(t \sqcup t_{*}) = \Tw(t) + K$ for each tiling $t$ of $\cR$, and $\Tw(t \sqcup t_{j,*}) = \Tw(t) + K_j$ for each tiling $t$ of $\cR_j$.

Write $\hat{t}_j = t_* \sqcup \bigsqcup_{1 \leq i < j} t_{i,1} \sqcup \bigsqcup_{j < i \leq m} t_{i,0}$ for each $j$, so that $\hat{t}_j$ is a tiling of $\cB \setminus \cR_j$. Notice that, for $1 \leq j \leq m$, 
$t_j \sqcup t_* = t_{j,1} \sqcup \hat{t}_j$ and $t_{j-1} \sqcup t_* = t_{j,0} \sqcup \hat{t}_j.$  
Therefore, we have
\begin{align*}
&\Tw(t_m) - \Tw(t_0)
= \sum_{1 \leq j \leq m} (\Tw(t_j) - \Tw(t_{j-1}))\\ 
&= \sum_{1 \leq j \leq m} ((\Tw(t_j \sqcup t_*) - K) - (\Tw(t_{j-1} \sqcup t_*) - K))\\
&= \sum_{1 \leq j \leq m} (\Tw(t_{j,1} \sqcup \hat{t}_j) - \Tw(t_{j,0} \sqcup \hat{t}_j)) 
\stackrel{\dagger}{=} \sum_{1 \leq j \leq m} (\Tw(t_{j,1} \sqcup t_{j,*}) - \Tw(t_{j,0} \sqcup t_{j,*}))\\
&= \sum_{1 \leq j \leq m} ((\Tw(t_{j,1}) + K_j) - (\Tw(t_{j,0}) + K_j)) = \sum_{1 \leq j \leq m} (\Tw(t_{j,1}) - \Tw(t_{j,0})).
\end{align*}
Equality $\dagger$ holds by 
Corollary \ref{coro:embeddableRegions}, because $\hat{t}_j$ and $t_{j,*}$ are two tilings of $\cB \setminus \cR_j$.
%
%{\color[rgb]{0,0,1} This last step is incomplete and needs further explanation! Check which of the old lemmas need to be used here (maybe we still need to embed in a larger box, and use the trivial tiling Lemma)}.
\end{proof}

\section{Embeddable regions}
\label{sec:embeddable}
We shall now describe a simple generalization of the twist to a wider class of simply connected regions. An \emph{embedding} of a region $\cR$ is a triple $(\cR, \cB, t_*)$, where $\cB \supset \cR$ is a box and $t_*$ is a tiling of $\cB \setminus \cR$. A region $\cR$ is said to be \emph{embeddable} if it admits an embedding.
\begin{lemma}
\label{lemma:independenceOfBox}
Let $\cR$ be a simply connected region and suppose $(\cR,\cB_0,t_0)$ and $(\cR,\cB_1,t_1)$ are two embeddings of $\cR$. If $t,t'$ are two tilings of $\cR$, then
$$\Tw(t \sqcup t_0) - \Tw(t' \sqcup t_0) = \Tw(t \sqcup t_1) - \Tw(t' \sqcup t_1).$$
\end{lemma}
\begin{proof}
%Suppose $(\cR,\cB_0,t_0)$ and $(\cR,\cB_1,t_1)$ are two embeddings of $\cR$, and let $t,t'$ be two tilings of $\cR$.
Let $\cB$ be a box containing $\cB_0$ and $\cB_1$ such that all its dimensions are even. Since $\cB_0$ and $\cB_1$ are tileable boxes, at least one of their dimensions is even: by Lemma \ref{lemma:trivialTilingBoxes}, there exist tilings $t_{0,*}$ of $\cB \setminus \cB_0$ and $t_{1,*}$ of $\cB \setminus \cB_1$ such that $\Tw(t_{\cB_0} \sqcup t_{0,*}) = \Tw(t_{\cB_0})$ and $\Tw(t_{\cB_1} \sqcup t_{1,*}) = \Tw(t_{\cB_1})$ for every tiling $t_{\cB_0}$ of $\cB_0$ and $t_{\cB_1}$ of $\cB_1$. Therefore,
$$
\Tw(t \sqcup t_0) - \Tw(t' \sqcup t_0) = \Tw(t \sqcup t_0 \sqcup t_{0,*}) - \Tw(t' \sqcup t_0 \sqcup t_{0,*}). 
$$
Notice that $t_0 \sqcup t_{0,*}$ and $t_1 \sqcup t_{1,*}$ are two tilings of $\cB \setminus \cR$. By Corollary \ref{coro:embeddableRegions}, 
$$
\Tw(t \sqcup t_0 \sqcup t_{0,*}) - \Tw(t' \sqcup t_0 \sqcup t_{0,*}) = \Tw(t \sqcup t_1 \sqcup t_{1,*}) - \Tw(t' \sqcup t_1 \sqcup t_{1,*}), 
$$
and hence $\Tw(t \sqcup t_0) - \Tw(t' \sqcup t_0) = \Tw(t \sqcup t_1) - \Tw(t' \sqcup t_1).$
\end{proof}

Given a simply connected embeddable region $\cR$ and two tilings $t_0, t_1$, define the \emph{relative twist} $\TW(t_0, t_1) = \Tw(t_0 \sqcup t_*) - \Tw(t_1 \sqcup t_*)$, where $(\cR, \cB, t_*)$ is any embedding of $\cR$. Moreover, if we fix a base tiling $t_0$, then the function $t \mapsto \TW(t,t_0)$ is an invariant with the same properties as the twist in cylinders (for instance, items \ref{item:flipInvariance}, \ref{item:tritDifference} and \ref{item:multiplexUnion} in Theorem \ref{theo:main}); and different choices of base tiling alter this function by an additive constant.

\begin{example}[Cylinders]
\label{ex:multiplexes}
If $\cR$ is a cylinder, then $\TW(t,t') = \Tw(t) - \Tw(t')$ for any two tilings of $\cR$ (this follows directly from Lemma \ref{lemma:trivialTilingAllMultiplexes}).
\end{example}

A natural question at this point is: ``how large'' is the class of tileable embeddable simply connected regions?
Since the box $\cB \supset \cR$ can be chosen to be as large as needed, it is clear that embeddable regions are, in fact, very common: for instance, if the complement of $\cR$ does not contain any narrow areas (where space is limited), $\cR$ will almost certainly be embeddable. However, there exist simply connected (and even contractible) regions that are not embeddable: Figure \ref{fig:contractible_notgood} shows an example. Notice that the complement of $\cR$ contains a narrow area where cubes of one color happen in much larger number.

\begin{figure}[ht]%
\centering
\includegraphics[width=0.55\columnwidth]{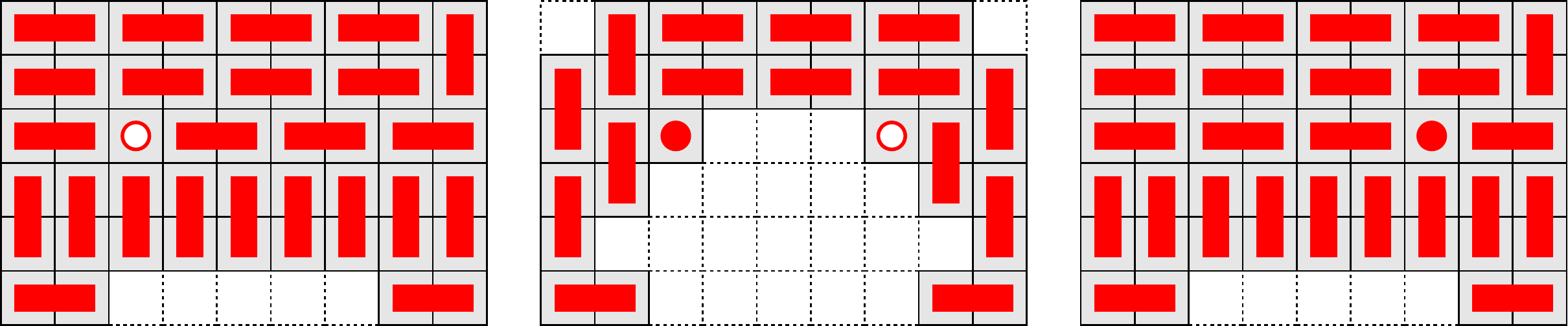}%
\caption{A tileable contractible region that is not embeddable. Notice the narrow area in the middle floor, which contains a large surplus of cubes of one color, while only a few of them are connected to the ``outer space''. Adding more space in the complement of $\cR$ clearly does not help, so $\cR$ is not embeddable.}%
\label{fig:contractible_notgood}%
\end{figure}

%%%--------------- End of topological part
\input{duplexappend}
\section{Examples and counterexamples}
\label{sec:examples}
In this short section, we give a few examples and counterexamples that help motivate the theory and some of the results obtained.

For instance, when looking at Proposition \ref{prop:equalTwistsMultiplex}, one might wonder whether the pretwists are always integers or if they always coincide, at least for, say, simply connected or contractible regions. This turns out not to be the case, as Figure \ref{subfig:twoDominoes} shows: for the tiling $t$ portrayed there, $T^{\ex}(t) = T^{\ey}(t) = 0$ but $T^{\ez}(t) = 1/4$.

\begin{figure}%
\centering
\subfloat[A tiling $t$ satisfying $T^{\ez}(t) = 1/4$.]{\includegraphics[width=0.2\columnwidth]{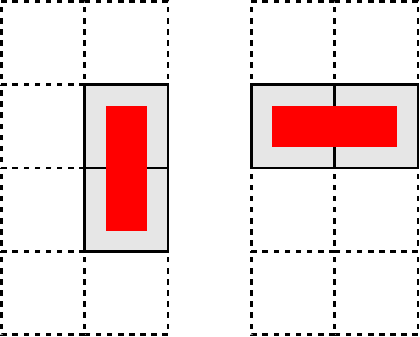}\label{subfig:twoDominoes}}%
\qquad \qquad
\subfloat[A tiling of a pseudocylinder where $T^{\ex}(t) \neq T^{\ez}(t)$.]{\includegraphics[width=0.3\columnwidth]{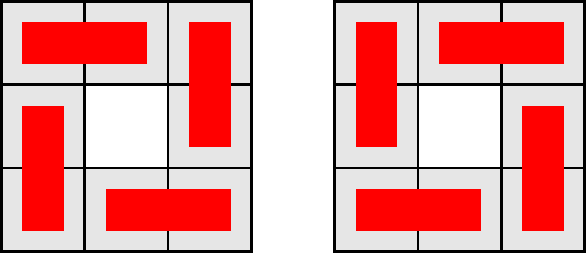}\label{subfig:pseudomultiplex}}
\caption{Two examples of tilings: $\ez$ is chosen to point towards the paper.}%
\label{fig:examples}%
\end{figure}

One might ask whether the pretwists coincide in a pseudocylinder (i.e., if the base is not necessarily simply connected): the tiling $t$ portrayed in Figure \ref{subfig:pseudomultiplex} satisfies $T^{\ex}(t) = T^{\ey}(t) = 0$ and $T^{\ez}(t) = 1$ (see Remarks \ref{rem:pseudocylinderEvenDepth} and \ref{rem:pseudocylinderOddDepth} for an idea of what can be shown for pseudocylinders in this respect).  
%One can prove that they coincide if the pseudocylinder has odd depth (via a modification in the proofs of Propositions \ref{prop:equalTwistsIntNEven} and \ref{prop:equalTwistsIntNOdd}), but we shall not dwell on this.

\begin{example}[The $4 \times 4 \times 4$ box]
\label{example:444box}
%For more examples, we refer the reader to \cite{webpageNicolau}. 
A particularly interesting example is the $4 \times 4 \times 4$ box, which has 5051532105 tilings, divided into 93 flip connected components. The number associated to each component is completely arbitrary, it has to do with the order in which our computer program found each one. Table \ref{table:CC444} shows the components in decreasing order of size, whereas Figure \ref{fig:graphConnectedComps_box444} contains a graph whose vertices are the connected components: it shows which components are connected via trits.

The reader will find more information in \cite{webpageNicolau}. There, he will find details about each connected component (e.g., pictures of a sample tiling in each of the components), as well as other examples and the source code for the \emph{C$^\sharp$} program used to generate them. 
\begin{table}[ht]
\centering
\begin{tabular}{|c|c|c|}
\hline
 \parbox[c]{2cm}{\centering Connected \\Component} & \parbox[c]{2cm}{\centering Number of \\tilings} & $\Tw$\\ \hline
 0 & \num{4412646453}  & $0$  \\ \hline
 1 & \num{310185960} & $1$  \\ \hline
 2 & \num{310185960}  & $-1$  \\ \hline
 8 & \num{8237514}  & $2$  \\ \hline 
 7 & \num{8237514}  & $-2$   \\ \hline
 3 & \num{718308}  & $2$  \\ \hline
 5 & \num{718308}   & $-2$  \\ \hline
 4,6 &  \num{283044} & $0$  \\ \hline
 27,36,58 & \num{2576} & $3$   \\ \hline
 28,35,57 & \num{2576} & $-3$   \\ \hline  
 9,13,22,23,29,32,37,43,49,55,76,80 & \num{618}  & $3$  \\ \hline
 12,16,21,25,30,31,39,45,51,56,75,77 & \num{618}  & $-3$  \\ \hline
 10,15,24,34,38,42,47,52,54,59,78,82 & \num{236}  & $1$  \\ \hline
 11,14,26,33,40,41,44,50,53,60,79,81 & \num{236}  & $-1$  \\ \hline
 48,61,83 & \num{4} & $4$   \\ \hline
 46,62,84 & \num{4} & $-4$   \\ \hline
 17,63,67,72,85,92 & \num{1}  & $4$  \\ \hline
 18,19,64,65,68,70,71,73,86,87,90,91 & \num{1} & $0$  \\ \hline
 20,66,69,74,88,89 & \num{1} & $-4$  \\ \hline

\end{tabular}
\caption{Flip connected components of a $4 \times 4 \times 4$ box}
\label{table:CC444}
\end{table}

\begin{figure}[ht]%
\centering
\includegraphics[width=0.7\columnwidth]{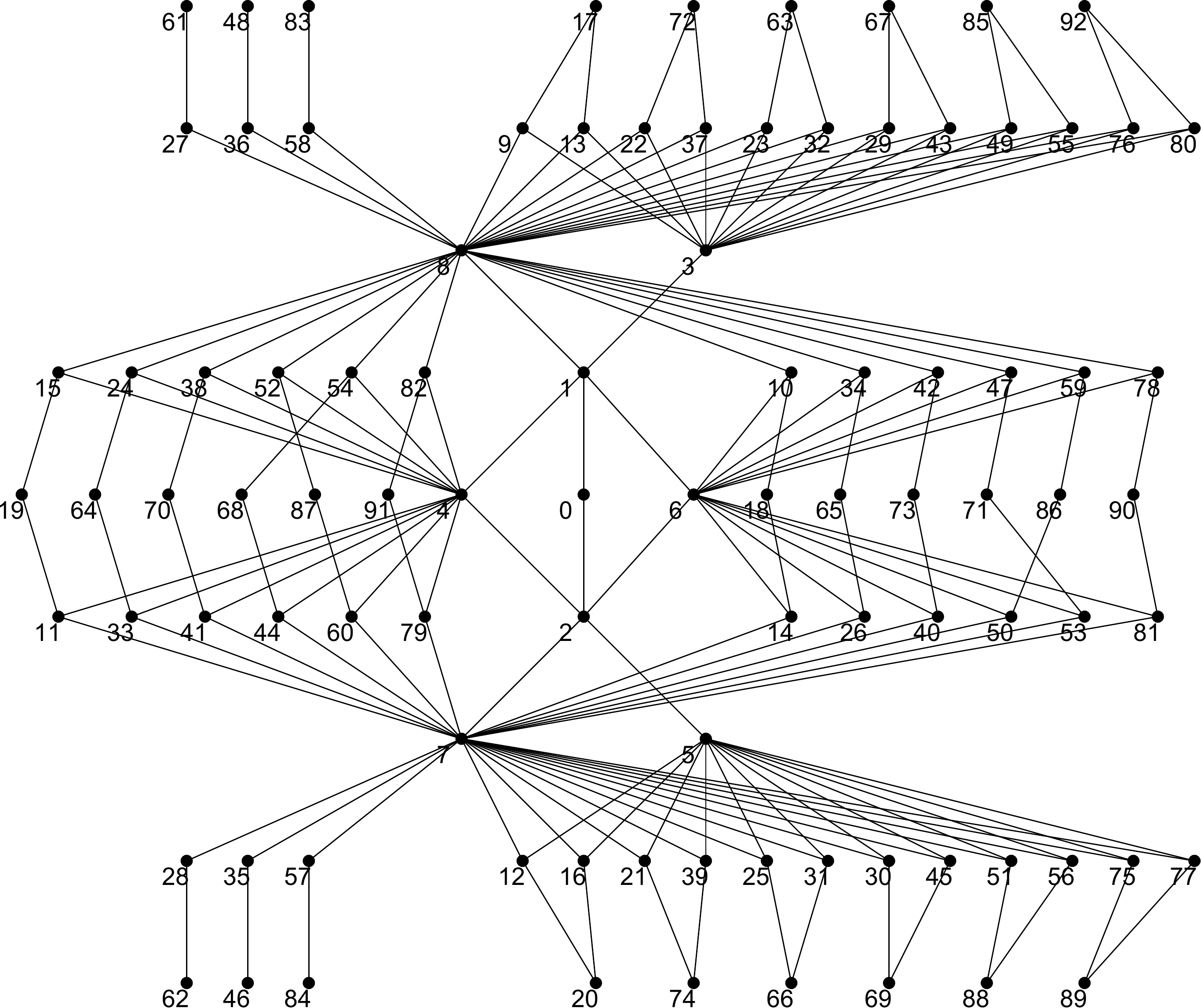}%
\caption{Graph with $93$ vertices, each one representing a flip connected component of the $4 \times 4 \times 4$ box. Two components are joined by an edge if there exists a trit taking a tiling in one component to a tiling of the other. Components that are below have a lower twist: hence, edges going up always refer to positive trits (and edges going down always refer to negative trits).}%
\label{fig:graphConnectedComps_box444}%
\end{figure}
\end{example}
%
%
%
%(see \cite{webpageNicolau} for this and other examples, as well as source codes).
%
%\section{The \texorpdfstring{$4 \times 4 \times 4$}{4x4x4} box}
%\label{sec:444box}
%The $4 \times 4 \times 4$ box is a particularly interesting example for two reasons: first, it is large enough to allow some nontrivial examples (for example, one can have two tilings such that their difference contains a trefoil knot); second, it is small enough not to have a huge number of tilings and connected components, which would have made the computations very hard (if not impossible) and would have made this exposition unnecessarily verbose.
%
%This box has \num{5051532105} tilings, divided into 93 connected components. The number associated to each component is completely arbitrary, it has to do with the order in which our computer program found each one. Table \ref{table:CC444} shows the components in decreasing order of size, whereas Figure \ref{fig:graphConnectedComps_box444} contains a graph whose vertices are the connected components: it shows which components are connected via trits.
%
%The reader may find more information about this in \cite{webpageNicolau}. There, he will find details about each connected component (e.g., pictures of a sample tiling in each of the components), as well as other examples, and the source code for the $C\sharp$ program used to generate them. 
%

%% file: figures/flip_shadow_c2_scheme_tx.pdf_tex
%% Creator: Inkscape 0.48.2, www.inkscape.org
%% Accompanies image file 'flip_shadow_c2_scheme_tx.pdf' (pdf, eps, ps)
\begingroup%
  \makeatletter%
  \providecommand\color[2][]{%
    \errmessage{(Inkscape) Color is used for the text in Inkscape, but the package 'color.sty' is not loaded}%
    \renewcommand\color[2][]{}%
  }%
  \providecommand\transparent[1]{%
    \errmessage{(Inkscape) Transparency is used (non-zero) for the text in Inkscape, but the package 'transparent.sty' is not loaded}%
    \renewcommand\transparent[1]{}%
  }%
  \providecommand\rotatebox[2]{#2}%
  \ifx\svgwidth\undefined%
    \setlength{\unitlength}{360.8bp}%
    \ifx\svgscale\undefined%
      \relax%
    \else%
      \setlength{\unitlength}{\unitlength * \real{\svgscale}}%
    \fi%
  \else%
    \setlength{\unitlength}{\svgwidth}%
  \fi%
  \global\let\svgwidth\undefined%
  \global\let\svgscale\undefined%
  \makeatother%
  \begin{picture}(1,0.26829268)%
    \put(0,0){\includegraphics[width=\unitlength]{flip_shadow_c2_scheme_tx.pdf}}%
    \put(-0.018000,0.12114238){\scalebox{0.55}{\color[rgb]{0,0,1}\makebox(0,0)[lb]{\smash{$-v(d_0)$}}}}%
    \put(0.08954893,0.04071222){\scalebox{0.55}{\color[rgb]{1,0.4,0}\makebox(0,0)[lb]{\smash{$v(d_1)$}}}}%
	\put(0.171972633,0.12114238){\scalebox{0.55}{\color[rgb]{0,0,1}\makebox(0,0)[lb]{\smash{$-v(\tilde{d_0})$}}}}%
	\put(0.08954893,0.20471222){\scalebox{0.55}{\color[rgb]{1,0.4,0}\makebox(0,0)[lb]{\smash{$v(\tilde{d_1})$}}}}%
  \end{picture}%
\endgroup%

%% file: figures/postrit_diff_mod.pdf_tex
%% Creator: Inkscape 0.48.2, www.inkscape.org
%% Accompanies image file 'postrit_diff_mod.pdf' (pdf, eps, ps)
\begingroup%
  \makeatletter%
  \providecommand\color[2][]{%
    \errmessage{(Inkscape) Color is used for the text in Inkscape, but the package 'color.sty' is not loaded}%
    \renewcommand\color[2][]{}%
  }%
  \providecommand\transparent[1]{%
    \errmessage{(Inkscape) Transparency is used (non-zero) for the text in Inkscape, but the package 'transparent.sty' is not loaded}%
    \renewcommand\transparent[1]{}%
  }%
  \providecommand\rotatebox[2]{#2}%
  \ifx\svgwidth\undefined%
    \setlength{\unitlength}{301.83552551bp}%
    \ifx\svgscale\undefined%
      \relax%
    \else%
      \setlength{\unitlength}{\unitlength * \real{\svgscale}}%
    \fi%
  \else%
    \setlength{\unitlength}{\svgwidth}%
  \fi%
  \global\let\svgwidth\undefined%
  \global\let\svgscale\undefined%
  \makeatother%
  \begin{picture}(1,0.2367908)%
    \put(0,0){\includegraphics[width=\unitlength]{postrit_diff_mod.pdf}}%
    \put(0.02,0.06758648){\color[rgb]{0,0,0}\scalebox{0.5}{\makebox(0,0)[lb]{\smash{$-v(d_0)$}}}}%
    \put(0.16485588,-0.03){\color[rgb]{0,0,0}\scalebox{0.5}{\makebox(0,0)[lb]{\smash{$v(d_1)$}}}}%
    \put(0.38437215,0.17540988){\color[rgb]{0,0,0}\scalebox{0.5}{\makebox(0,0)[lb]{\smash{$-v(\tilde{d_0})$}}}}%
    \put(0.52711369,0.06758648){\color[rgb]{0,0,0}\scalebox{0.5}{\makebox(0,0)[lb]{\smash{$v(\tilde{d_1})$}}}}%
  \end{picture}%
\endgroup%

%% file: figures/cases_crossings_jk.pdf_tex
%% Creator: Inkscape 0.48.2, www.inkscape.org
%% Accompanies image file 'cases_crossings_jk.pdf' (pdf, eps, ps)
\begingroup%
  \makeatletter%
  \providecommand\color[2][]{%
    \errmessage{(Inkscape) Color is used for the text in Inkscape, but the package 'color.sty' is not loaded}%
    \renewcommand\color[2][]{}%
  }%
  \providecommand\transparent[1]{%
    \errmessage{(Inkscape) Transparency is used (non-zero) for the text in Inkscape, but the package 'transparent.sty' is not loaded}%
    \renewcommand\transparent[1]{}%
  }%
  \providecommand\rotatebox[2]{#2}%
  \ifx\svgwidth\undefined%
    \setlength{\unitlength}{268.5077602bp}%
    \ifx\svgscale\undefined%
      \relax%
    \else%
      \setlength{\unitlength}{\unitlength * \real{\svgscale}}%
    \fi%
  \else%
    \setlength{\unitlength}{\svgwidth}%
  \fi%
  \global\let\svgwidth\undefined%
  \global\let\svgscale\undefined%
  \makeatother%
  \begin{picture}(1,0.76475618)%
    \put(0,0){\includegraphics[width=\unitlength]{cases_crossings_jk.pdf}}%
    \put(0.37590258,0.73210175){\color[rgb]{0,0,0}\scalebox{0.8}{\makebox(0,0)[lb]{\smash{$s > 0$}}}}%
    \put(0.89730278,0.73210175){\color[rgb]{0,0,0}\scalebox{0.8}{\makebox(0,0)[lb]{\smash{$s < 0$}}}}%
    \put(0.27396826,0.56313026){\color[rgb]{0,0,0}\scalebox{0.6}{\makebox(0,0)[lb]{\smash{$\vbeta_1$}}}}%
    \put(0,0.37712161){\color[rgb]{0,0,0}\scalebox{0.6}{\makebox(0,0)[lb]{\smash{$\vbeta_2$}}}}%
    \put(0.03326816,0.7172046){\color[rgb]{0,0,0}\scalebox{0.6}{\makebox(0,0)[lb]{\smash{$\vbeta_3$}}}}%
  \end{picture}%
\endgroup%

%% file: figures/cases_crossings_kj.pdf_tex
%% Creator: Inkscape 0.48.2, www.inkscape.org
%% Accompanies image file 'cases_crossings_kj.pdf' (pdf, eps, ps)
\begingroup%
  \makeatletter%
  \providecommand\color[2][]{%
    \errmessage{(Inkscape) Color is used for the text in Inkscape, but the package 'color.sty' is not loaded}%
    \renewcommand\color[2][]{}%
  }%
  \providecommand\transparent[1]{%
    \errmessage{(Inkscape) Transparency is used (non-zero) for the text in Inkscape, but the package 'transparent.sty' is not loaded}%
    \renewcommand\transparent[1]{}%
  }%
  \providecommand\rotatebox[2]{#2}%
  \ifx\svgwidth\undefined%
    \setlength{\unitlength}{268.46166645bp}%
    \ifx\svgscale\undefined%
      \relax%
    \else%
      \setlength{\unitlength}{\unitlength * \real{\svgscale}}%
    \fi%
  \else%
    \setlength{\unitlength}{\svgwidth}%
  \fi%
  \global\let\svgwidth\undefined%
  \global\let\svgscale\undefined%
  \makeatother%
  \begin{picture}(1,0.76274565)%
    \put(0,0){\includegraphics[width=\unitlength]{cases_crossings_kj.pdf}}%
    \put(0.37596712,0.73008561){\color[rgb]{0,0,0}\scalebox{0.8}{\makebox(0,0)[lb]{\smash{$s > 0$}}}}%
    \put(0.89745685,0.73008561){\color[rgb]{0,0,0}\scalebox{0.8}{\makebox(0,0)[lb]{\smash{$s < 0$}}}}%
    \put(0.27396826,0.56313026){\color[rgb]{0,0,0}\scalebox{0.6}{\makebox(0,0)[lb]{\smash{$\vbeta_1$}}}}%
    \put(0,0.37712161){\color[rgb]{0,0,0}\scalebox{0.6}{\makebox(0,0)[lb]{\smash{$\vbeta_2$}}}}%
    \put(0.03326816,0.7172046){\color[rgb]{0,0,0}\scalebox{0.6}{\makebox(0,0)[lb]{\smash{$\vbeta_3$}}}}%
  \end{picture}%
\endgroup%

%% file: figures/auxiliaryLinesExample.pdf_tex
%% Creator: Inkscape 0.48.2, www.inkscape.org
%% Accompanies image file 'auxiliaryLinesExample.pdf' (pdf, eps, ps)
\begingroup%
  \makeatletter%
  \providecommand\color[2][]{%
    \errmessage{(Inkscape) Color is used for the text in Inkscape, but the package 'color.sty' is not loaded}%
    \renewcommand\color[2][]{}%
  }%
  \providecommand\transparent[1]{%
    \errmessage{(Inkscape) Transparency is used (non-zero) for the text in Inkscape, but the package 'transparent.sty' is not loaded}%
    \renewcommand\transparent[1]{}%
  }%
  \providecommand\rotatebox[2]{#2}%
  \ifx\svgwidth\undefined%
    \setlength{\unitlength}{532.15bp}%
    \ifx\svgscale\undefined%
      \relax%
    \else%
      \setlength{\unitlength}{\unitlength * \real{\svgscale}}%
    \fi%
  \else%
    \setlength{\unitlength}{\svgwidth}%
  \fi%
  \global\let\svgwidth\undefined%
  \global\let\svgscale\undefined%
  \makeatother%
  \begin{picture}(1,1.22550033)%
    \put(0,0){\includegraphics[width=\unitlength]{auxiliaryLinesExample.pdf}}%
    \put(0.25476839,1.17180306){\color[rgb]{0,0,0}\scalebox{0.6}{\makebox(0,0)[lb]{\smash{$\vbeta_1$}}}}%
    \put(0.11450155,1.0590529){\color[rgb]{0,0,0}\scalebox{0.6}{\makebox(0,0)[lb]{\smash{$\vbeta_2$}}}}%
    \put(0.20715563,1.10510259){\color[rgb]{0,0,0}\scalebox{0.6}{\makebox(0,0)[lb]{\smash{$\vbeta_3$}}}}%
  \end{picture}%
\endgroup%

%% file: figures/squareWalls.pdf_tex
%% Creator: Inkscape 0.48.5, www.inkscape.org
%% Accompanies image file 'squareWalls.pdf' (pdf, eps, ps)
\begingroup%
  \makeatletter%
  \providecommand\color[2][]{%
    \errmessage{(Inkscape) Color is used for the text in Inkscape, but the package 'color.sty' is not loaded}%
    \renewcommand\color[2][]{}%
  }%
  \providecommand\transparent[1]{%
    \errmessage{(Inkscape) Transparency is used (non-zero) for the text in Inkscape, but the package 'transparent.sty' is not loaded}%
    \renewcommand\transparent[1]{}%
  }%
  \providecommand\rotatebox[2]{#2}%
  \ifx\svgwidth\undefined%
    \setlength{\unitlength}{452.72714844bp}%
    \ifx\svgscale\undefined%
      \relax%
    \else%
      \setlength{\unitlength}{\unitlength * \real{\svgscale}}%
    \fi%
  \else%
    \setlength{\unitlength}{\svgwidth}%
  \fi%
  \global\let\svgwidth\undefined%
  \global\let\svgscale\undefined%
  \makeatother%
  \begin{picture}(1,0.78104439)%
    \put(0,0){\includegraphics[width=\unitlength]{squareWalls.pdf}}%
    \put(0.30427169,0.20592637){\color[rgb]{0,0,0}\makebox(0,0)[lb]{\smash{$Q$}}}%
    \put(0.56757053,0.63145812){\color[rgb]{0,0,0}\makebox(0,0)[lb]{\smash{$\cR$}}}%
    \put(0.85077303,0.35155689){\color[rgb]{0,0,0}\makebox(0,0)[lb]{\smash{$S^{\vu}$}}}%
  \end{picture}%
\endgroup%

%% file: duplexappend.tex
\section{The twist for duplex regions}
\label{sec:duplexnotes}

In this section, we provide an alternative proof for the following fact:

\begin{prop}
\label{prop:duplexes}
If $\cR$ is a duplex region, then, for any tiling $t$ of $\cR$,
$$P_t'(1) = T^{\ex}(t) = T^{\ey}(t) = T^{\ez}(t).$$ 
\end{prop}
The equality $T^{\ex}(t) = T^{\ey}(t) = T^{\ez}(t)$ above is a special case of Proposition \ref{prop:equalTwistsMultiplex}. We shall now give an independent proof of this equality in the particular case of duplex regions. 
Let $\cR$ be a duplex region. 
Let $t$ be a tiling of $\cR$, and let $s$ be its corresponding sock in $G$. For $p \in \RR^2$ and a cycle $\gamma$ of $s$, let $\wind(\gamma,p)$ be the winding number of $\gamma$, thought of as a curve in $\RR^2$, around $p$. Clearly we can write our invariant $P_t(q)$ as
\begin{equation}
P_t(q) = \sum_{v \in G} \ccol(v) q^{\sum_{\gamma, v \notin \gamma} \wind(\gamma,v)},
\label{eq:definitionOfPt}
\end{equation}
where the sum in the exponent of $q$ is taken over all the cycles in $s$ that do not contain $v$. Notice that $$P_t'(1) = \sum_{v \in G, \gamma \subsumtext{ cycle}} \ccol(v) \wind(\gamma,v) = \sum_{v \in \ZZ^2, \gamma \subsumtext{ cycle}} \ccol(v) \wind(\gamma,v).$$  

\begin{lemma}
\label{lemma:pretwistsNotAxis}
If $\cR$ is a $\ez$-duplex region with associated graph $G$, $\vu \in \{\pm \ex, \pm \ey\}$ and $t$ is a tiling of $\cR$ with corresponding sock $s$, then
$T^{\vu}(t) = P_t'(1)$.
\end{lemma} 
\begin{proof}
Two dominoes that are not parallel to $\ez$ have no effect along $\vu$ on one another. Therefore, we only consider pairs of dominoes where one is parallel to $\ez$, that is, refers to a jewel of $s$.

If $\gamma$ is a cycle of $s$ and $v$ is a jewel, one way of computing $\wind(\gamma,v)$ is to count (with signs) the intersections of $\gamma$ with the half-line $v +  [0,\infty) \vu$. Thus, if $d_v$ denotes the domino containing $v$ and $d \in \gamma$ means that $d$ refers to an edge of $\gamma$, then
$\ccol(v)\wind(\gamma,v) = 2 \sum_{d \in \gamma} \tau^{\vu}(d,d_v) = 2 \sum_{d \in \gamma} \tau^{\vu}(d_v,d).$
Thus, $$P_t'(1) = \sum_{\gamma, v} \ccol(v)\wind(\gamma,v) = \sum_{\substack{\gamma,v\\d \in \gamma}} (\tau^{\vu}(d,d_v)+ \tau^{\vu}(d_v,d)) = T^{\vu}(t),$$
completing the proof.
\end{proof}
%\section{Charges and weights}
We now consider $T^{\ez}$. Again, let $t$ be a tiling of a duplex region with corresponding sock $s$.
Let the \emph{charge enclosed} by a cycle $\gamma$ of $s$ be $$\charge_{\interior}(\gamma) = \sum_{v \notin \gamma} \ccol(v)\wind(\gamma,v),$$
so that  
$P_t'(1) = \sum_{\gamma \scalebox{0.7}{\mbox{ cycle of }} s} \charge_{\interior}(\gamma).$
Charges can be looked at from a point of view that is more interesting for our purposes.  Given $v \in \RR^2$, consider the set of four points $\neighbor_v = \{v + (\frac{k}{2}, \frac{l}{2}) | k,l \in \{-1, 1\}\},$ i.e., the set of points of the form $v + (\pm \half, \pm \half)$. The \emph{metric weight} of a vertex $v \in \ZZ^2$ with respect to a cycle $\gamma$ of $s$ is given by 
$$ \metricw{\gamma}{v} = \frac{1}{4} \sum_{u \in \neighbor_v} \wind(\gamma,u),$$
while the \emph{topological weight} $\topw{\gamma}{v}$ of $v$ is the (arithmetic) average of the set $\wind(\gamma,\neighbor_v) = \{\wind(\gamma,u) | u \in \neighbor_v\}$ (see Figure \ref{fig:topAndMetricWeights}). %These objects, which will gain further significance in later sections, are very simple in the setting of duplex regions.

\begin{lemma}
\label{lemma:interiorCharge}
$$\charge_{\interior}(\gamma) = \sum_{v \in \ZZ^2} \ccol(v)\topw{\gamma}{v}.$$
\end{lemma}
\begin{proof}
Notice that
$$ \topw{\gamma}{v} = \begin{cases} \wind(\gamma,v), &\text{if } v \notin \gamma,\\  
 \half, &\text{if } v \in \gamma \text{ and $\gamma$ is counterclockwise oriented,}\\
-\half, &\text{if } v \in \gamma \text{ and $\gamma$ is clockwise oriented.}\end{cases}$$
In particular, $\sum_{v \in \gamma}\topw{\gamma}{v} \ccol(v) = \pm \half \sum_{v \in \gamma}\ccol(v) = 0.$ Hence, 
$$\charge_{\interior}(\gamma) = \sum_{v \in \ZZ^2} \ccol(v)\topw{\gamma}{v}.$$
\end{proof}

\begin{figure}%
\centering
\def\svgwidth{0.7\columnwidth}
\def\myScaleVar{0.8}
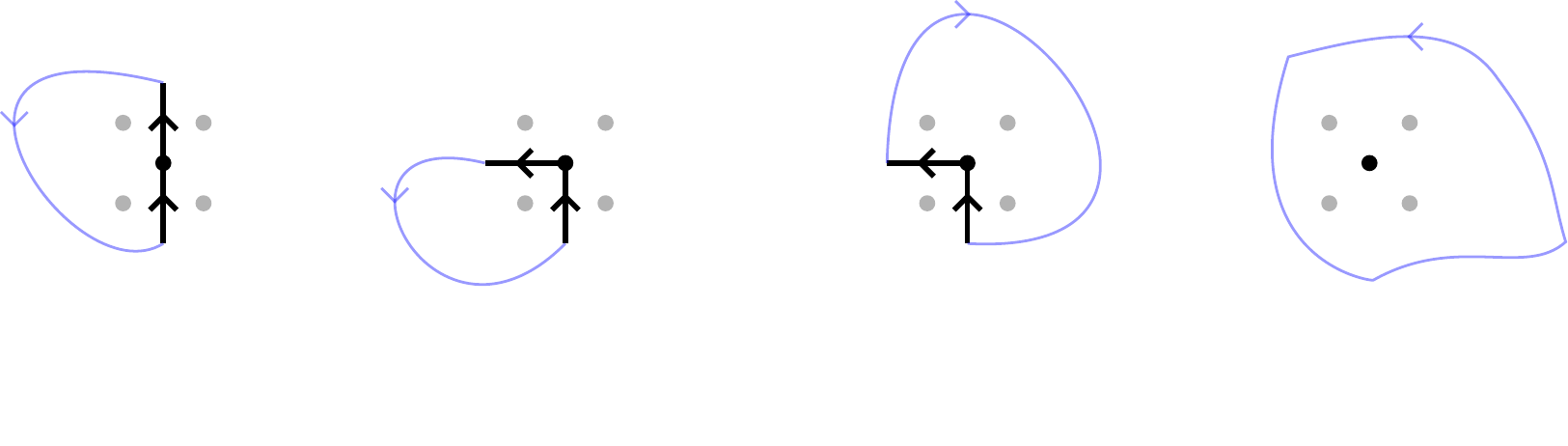%
\caption{Illustration of topological and metric weights. The points in $\neighbor_v$ are in grey.}%
\label{fig:topAndMetricWeights}%
\def\myScaleVar{1}
\end{figure}

%We now wish to relate the above to the twist $\Tw(t)$. Let $\vw$ be the axis of $\cR$: we shall calculate the twist via $T^{\vw}(t)$, as follows.
If $s$ is the corresponding sock of a tiling $t$ and $\gamma$ is a cycle of $s$, then the angle $\angle(\gamma,v)$ of a vertex $v \in \gamma$ is the difference between the angle of the edge of $\gamma$ leaving $v$ and the angle of the edge of $\gamma$ entering $v$, counted in counterclockwise laps. In other words, a vertex $v$ where the curve goes straight has angle $0$, whereas a vertex where a left (resp. right) turn occurs has angle $1/4$ (resp. $-1/4$), as shown in Figure \ref{fig:angles}.
\begin{figure}[ht]%
\centering
\def\svgwidth{0.5\columnwidth}
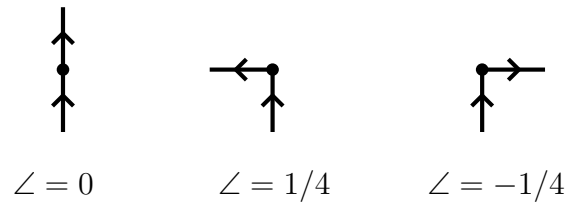
\caption{Illustration of the angle of a vertex.}%
\label{fig:angles}%
\end{figure}

The \emph{boundary charge} of a curve $\gamma$ is $\charge_{\partial}(\gamma) = \sum_{v \in \gamma} \angle(\gamma,v) \ccol(v)$. Notice that 
\begin{equation}
T^{\ez}(t) = \sum_{\gamma \scalebox{0.7}{\mbox{ cycle of }} s} \charge_{\partial}(\gamma).
\label{eq:ezpretwist}
\end{equation} 
We now set out to prove that, for each $\gamma$, $\charge_{\partial}(\gamma) = \charge_{\interior}(\gamma)$, which will complete the proof of Proposition \ref{prop:duplexes}.
\begin{lemma}
\label{lemma:metricColorZero}
For each cycle $\gamma$ of $s$, 
$$\sum_{v \in \ZZ^2} \metricw{\gamma}{v} \ccol(v) = 0.$$
\end{lemma}
\begin{proof}
$$\sum_{v \in \ZZ^2} \metricw{\gamma}{v} \ccol(v) = \sum_{u \in \ZZ^2 + (\half,\half)} \left(\frac{1}{4} \wind(\gamma,u) \sum_{v \in \neighbor_u} \ccol(v)\right) = 0.$$  
\end{proof}

\begin{lemma}
\label{lemma:topMinusMetric}
If $v \in \ZZ^2$, 
$$\topw{\gamma}{v} - \metricw{\gamma}{v} = \begin{cases} \angle(\gamma,v), &\text{if } v \in \gamma,\\
0, &\text{otherwise.} \end{cases} $$
\end{lemma}
\begin{proof}
Figure \ref{fig:topAndMetricWeights} illustrates most of the elements needed in this proof. If $v \notin \gamma$, then $\wind(\gamma,u) = \wind(\gamma,v)$ for each $u \in \neighbor_v$, so $\metricw{\gamma}{v} = \topw{\gamma}{v} = \wind(\gamma,v).$ 

If $v \in \gamma$ and the curve goes straight at $v$, then, for some $k$, two points in $\neighbor_v$ have winding number $k$ and the other two have winding number $k+1$, so $\metricw{\gamma}{v} = \frac{1}{4}(k + k + (k+1) + (k+1)) = \half(k + (k+1)) = \topw{\gamma}{v}.$

If $v \in \gamma$ and the curves turns left at $v$, then the winding numbers are $k,k,k,k+1$, so $\topw{\gamma}{v} - \metricw{\gamma}{v} = 1/4 = \angle(\gamma,v)$. Analogously, $\topw{\gamma}{v} - \metricw{\gamma}{v} = - 1/4 = \angle(\gamma,v)$ if $\gamma$ turns right at $v$.
\end{proof}

\begin{lemma}
\label{lemma:ptprimeEqualsTwist}
For each cycle $\gamma$ of $s$, $\charge_{\partial}(\gamma) = \charge_{\interior}(\gamma)$.
\end{lemma}
\begin{proof}
By Lemma \ref{lemma:topMinusMetric}, the boundary charge of $\gamma$ can also be written as 
\begin{align*}
\charge_{\partial}(\gamma) &= \sum_{v \in \gamma} \angle(\gamma,v) \ccol(v) \\
													 &= \sum_{v \in \ZZ^2} (\topw{\gamma}{v} - \metricw{\gamma}{v})\ccol(v)\\
                           &=\sum_{v \in \ZZ^2} \topw{\gamma}{v}\ccol(v) - \sum_{v \in \ZZ^2}\metricw{\gamma}{v}\ccol(v)\\
													&= \sum_{v \in \ZZ^2} \topw{\gamma}{v}\ccol(v) = \charge_{\interior}(\gamma),
\end{align*}
the fourth equality holding because of Lemma \ref{lemma:metricColorZero}; the last equality is Lemma \ref{lemma:interiorCharge}.
\end{proof}

\begin{proof}[Proof of Proposition \ref{prop:duplexes}]

Lemma \ref{lemma:ptprimeEqualsTwist} and Equation \eqref{eq:ezpretwist} imply that $P_t'(1) = T^{\ez}(t)$. Together with Lemma \ref{lemma:pretwistsNotAxis}, this proves the result.
\end{proof}

%\begin{rem}
%Proposition \ref{prop:duplexes} and Proposition 3.6 in \cite{segundoartigo} establish Remark 10.1 in %\cite{primeiroartigo}. 
%\end{rem}

%% file: figures/topAndMetricWeights.pdf_tex
%% Creator: Inkscape 0.48.2, www.inkscape.org
%% Accompanies image file 'topAndMetricWeights.pdf' (pdf, eps, ps)
\begingroup%
  \makeatletter%
  \providecommand\color[2][]{%
    \errmessage{(Inkscape) Color is used for the text in Inkscape, but the package 'color.sty' is not loaded}%
    \renewcommand\color[2][]{}%
  }%
  \providecommand\transparent[1]{%
    \errmessage{(Inkscape) Transparency is used (non-zero) for the text in Inkscape, but the package 'transparent.sty' is not loaded}%
    \renewcommand\transparent[1]{}%
  }%
  \providecommand\rotatebox[2]{#2}%
  \ifx\svgwidth\undefined%
    \setlength{\unitlength}{467.775bp}%
    \ifx\svgscale\undefined%
      \relax%
    \else%
      \setlength{\unitlength}{\unitlength * \real{\svgscale}}%
    \fi%
  \else%
    \setlength{\unitlength}{\svgwidth}%
  \fi%
  \global\let\svgwidth\undefined%
  \global\let\svgscale\undefined%
  \makeatother%
  \begin{picture}(1,0.28417212)%
    \put(0,0){\includegraphics[width=\unitlength]{topAndMetricWeights.pdf}}%
    \put(0.00146117,0.06045345){\color[rgb]{0,0,0}\scalebox{\myScaleVar}{\makebox(0,0)[lb]{\smash{$\operatorname{w}_{\operatorname{metric}} = \half$}}}}%
    \put(0.00146117,0.00914673){\color[rgb]{0,0,0}\scalebox{\myScaleVar}{\makebox(0,0)[lb]{\smash{$\operatorname{w}_{\operatorname{top}} = \half$}}}}%
    \put(0.25799476,0.06045345){\color[rgb]{0,0,0}\scalebox{\myScaleVar}{\makebox(0,0)[lb]{\smash{$\operatorname{w}_{\operatorname{metric}} = \frac{1}{4}$}}}}%
    \put(0.25799476,0.00914673){\color[rgb]{0,0,0}\scalebox{\myScaleVar}{\makebox(0,0)[lb]{\smash{$\operatorname{w}_{\operatorname{top}} = \half$}}}}%
    \put(0.54452835,0.06045345){\color[rgb]{0,0,0}\scalebox{\myScaleVar}{\makebox(0,0)[lb]{\smash{$\operatorname{w}_{\operatorname{metric}} = - \frac{3}{4}$}}}}%
    \put(0.54452835,0.00914673){\color[rgb]{0,0,0}\scalebox{\myScaleVar}{\makebox(0,0)[lb]{\smash{$\operatorname{w}_{\operatorname{top}} = - \half$}}}}%
    \put(0.82106194,0.06045345){\color[rgb]{0,0,0}\scalebox{\myScaleVar}{\makebox(0,0)[lb]{\smash{$\operatorname{w}_{\operatorname{metric}} = 1$}}}}%
    \put(0.82106194,0.00914673){\color[rgb]{0,0,0}\scalebox{\myScaleVar}{\makebox(0,0)[lb]{\smash{$\operatorname{w}_{\operatorname{top}} = 1$}}}}%
  \end{picture}%
\endgroup%

%% file: figures/angles.pdf_tex
%% Creator: Inkscape 0.48.2, www.inkscape.org
%% Accompanies image file 'angles.pdf' (pdf, eps, ps)
\begingroup%
  \makeatletter%
  \providecommand\color[2][]{%
    \errmessage{(Inkscape) Color is used for the text in Inkscape, but the package 'color.sty' is not loaded}%
    \renewcommand\color[2][]{}%
  }%
  \providecommand\transparent[1]{%
    \errmessage{(Inkscape) Transparency is used (non-zero) for the text in Inkscape, but the package 'transparent.sty' is not loaded}%
    \renewcommand\transparent[1]{}%
  }%
  \providecommand\rotatebox[2]{#2}%
  \ifx\svgwidth\undefined%
    \setlength{\unitlength}{203.08193359bp}%
    \ifx\svgscale\undefined%
      \relax%
    \else%
      \setlength{\unitlength}{\unitlength * \real{\svgscale}}%
    \fi%
  \else%
    \setlength{\unitlength}{\svgwidth}%
  \fi%
  \global\let\svgwidth\undefined%
  \global\let\svgscale\undefined%
  \makeatother%
  \begin{picture}(1,0.37560512)%
    \put(0,0){\includegraphics[width=\unitlength]{angles.pdf}}%
    \put(-0.00052067,0.02106841){\color[rgb]{0,0,0}\makebox(0,0)[lb]{\smash{$\angle = 0$}}}%
    \put(0.385409,0.02106841){\color[rgb]{0,0,0}\makebox(0,0)[lb]{\smash{$\angle = 1/4$}}}%
    \put(0.77933868,0.02106841){\color[rgb]{0,0,0}\makebox(0,0)[lb]{\smash{$\angle = -1/4$}}}%
  \end{picture}%
\endgroup%

%% file: possiblevalschap.tex
\chapter{Possible values of the twist}
\label{chap:possiblevals}
%\section{The possible values for the twist of a box}
%\label{sec:possibleValuesTwist}

Now that we have established (in Proposition \ref{prop:writheFormula}) that the twist of a cylinder is always an integer, we wish to take a closer look at the set of values that can be the twist of some tiling of a given box. 
If $\cB$ is a box, define $ \Tw(\cB) = \{\Tw(t) | t \text{ is a tiling of } \cB\} \subset \ZZ$. An easy observation, that follows directly from Lemma \ref{lemma:twistReflection}, is that $\Tw(\cB)$ is symmetric with respect to zero, i.e., $k \in \Tw(\cB) \Leftrightarrow -k \in \Tw(\cB)$. In this chapter, our main concern is to provide bounds for the value $\maxtw(L,M,N) = \max \Tw\left([0,L] \times [0,M] \times [0,N]\right)$.

\begin{lemma}
\label{lemma:maxTwist}
Let $L,M,N$ be integers, at least one of them even, with $N \leq M \leq L$. Then 
$$
\maxtw(L,M,N) \leq \frac{LM\ceil{N/2}\floor{N/2}}{4}\leq \frac{LMN^2}{16}.
$$
%Moreover, $C_1 = 1/16$ is the smallest constant that works for all values of $L$,$M$ and $N$.
\end{lemma}
\begin{proof}
Suppose, without loss of generality, that $N \leq L,M$, and let $t$ be a tiling of $\cB$. 
For each $(x,y) \in \ZZ^2$, $0 \leq x \leq L-1$, $0 \leq y \leq M-1$, let, for $\vu \in \{\ex, \ey\}$,
$d_{\vu}(x,y) = \{d \in t | \tv(d) = \pm \vu, C(\plshalf{x},\plshalf{y},\plshalf{z}) \subset d \text{ for some } z \in \ZZ\}$, and let $n_{\vu}(x,y)$ be the number of elements of $d_{\vu}(x,y)$ 
Notice that there is a one-to-one correspondence between pairs $(d_0, d_1)$ such that $d_0 \in d_{\ex}(x,y)$, $d_1 \in d_{\ey}(x,y)$, and pairs $(d_0,d_1)$ such that $\tau(d_0,d_1) + \tau(d_1,d_0) \neq 0$; also, the latter condition implies that $\tau(d_0,d_1) + \tau(d_1,d_0) = \pm 1/4$. For fixed $x,y$, the number of such pairs is $n_{\ex}(x,y) n_{\ey}(x,y)$. Therefore,   
$$
|\Tw(t)| \leq \frac{1}{4}\sum_{x = 0}^{L-1} \sum_{y = 0}^{M-1} n_{\ex}(x,y) n_{\ey}(x,y).
$$
Since $n_{\ex}(x,y) + n_{\ey}(x,y) \leq N$, it follows that $n_{\ex}(x,y) n_{\ey}(x,y) \leq \ceil{N/2}\floor{N/2}$, so that
$$
|\Tw(t)| \leq \frac{1}{4}\sum_{x = 0}^{L-1} \sum_{y = 0}^{M-1} {\ceil{N/2}\floor{N/2}} = \frac{LM\ceil{N/2}\floor{N/2}}{4}.
$$ 
\end{proof}

In Lemma \ref{lemma:techOrderOptima} below, we use a similar (but more refined) analysis to find a slightly better upper bound.
%A slightly better upper bound is obtained by refining this analysis (see Lemma \ref{lemma:techOrderOptima}).

\section{Asymptotic lower bounds}

We will now describe some constructions that yield asymptotic lower bounds for $\maxtw(L,M,N)$. We begin with an exact description for the asymptotic behavior in boxes with two floors (Lemma \ref{lemma:possibleTwoFloors}).
First, however, we need an auxiliary construction.

Recall from Section \ref{sec:twoFloorsMoreSpace} the definition of a sock. The \emph{eel} of \emph{thickness} $m$ and \emph{length} $n$ is a sock in $\ZZ^2$ that consists of $m$ cycles, with the same orientation, such that each cycle surrounds a set of $n$ consecutive jewels in a diagonal (and those are the only jewels each cycle surrounds): Figures \ref{fig:eel} and \ref{fig:triangleEel} show examples of eels of different thicknesses and lengths. Notice that an eel of length $1$ is a boxed jewel.
If $e$ is an eel of thickness $m$ and length $n$ such that the $n$ jewels are black and the cycles spin counterclockwise, then $P_e(q) = nq^m$, so that $P_e'(1) = mn$.  
 
%
%consists of $n$ consecutive jewels in a diagonal plus $m$ cycles, with the same orientation, surrounding all $n$ jewels (and no other jewel).

\begin{lemma}
\label{lemma:fish}
For $n \geq 0$, let $f(n) = \maxtw(n,n,2)$. Then
$$\lim_{n \to \infty} \frac{f(n)}{4n^2} = \frac{1}{16}.$$
\end{lemma}
\begin{proof}
%For $n \geq 0$, let $f(n) = \maxtw(n,n,2)$ and 
Let $\cB_n = [0,n] \times [0,n] \times [0,2]$. By Lemma \ref{lemma:maxTwist}, we only need to show that $\lim_{n \to \infty} f(n)/(4n^2) \geq 1/16$. Also, let $\cP_n = \cD_n \times [0,2]$, where $\cD_n \subset \RR^2$ consists of all the basic squares completely contained in the set $\{(x,y)\,|\,x,y \geq 0, x + y \leq n+1\}$ (see Figure \ref{fig:triangleEel}). Let, for $n \geq 0$, $g(n) = \max \Tw(\cP_n).$  We claim that $f$ and $g$ satisfy the following inequalities:
\begin{align}
 &f(n) \geq \max_{0 \leq k \leq (n-1)/2} k(n-2k) + 2 g(n - 2k - 1), \label{eq:recurrenceF}\\
 &g(n) \geq \max_{0 \leq k \leq (n-1)/4} k(n-4k) + g(n-4k-1) + 2g(2k),
\label{eq:recurrenceG}  
\end{align}
with $f(0) = g(0) = 0$. 

\begin{figure}[ht]%
\centering
\includegraphics[width=0.6\columnwidth]{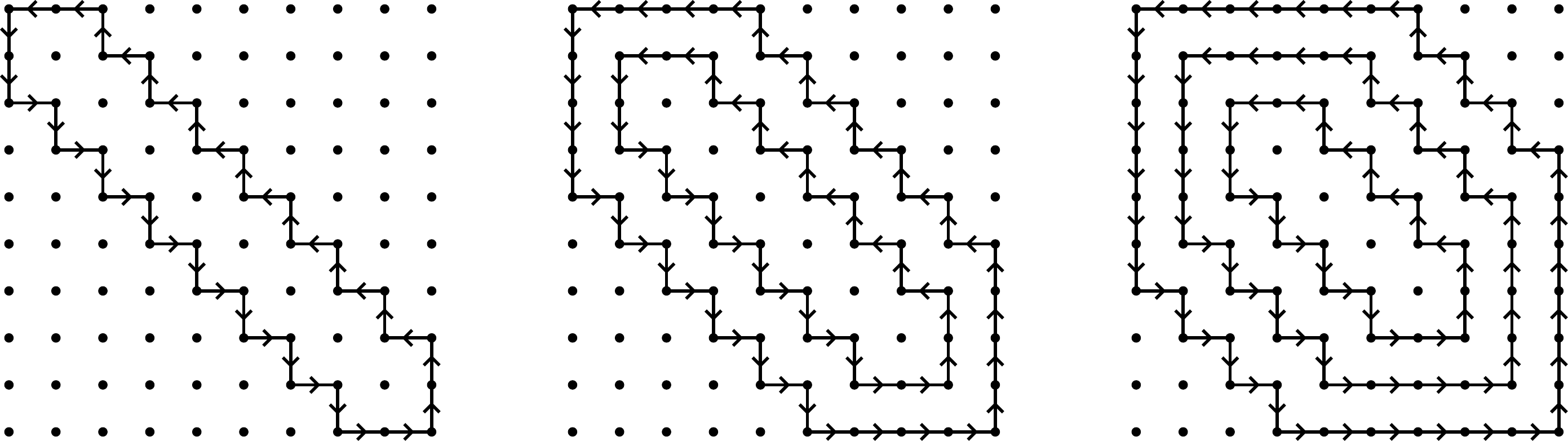}%
\caption{Socks for the eels of thickness $1$, $2$ and $3$ in $\cB_{10}$.}%
\label{fig:eel}%
\end{figure} 

Figure \ref{fig:eel} illustrates that we can fit an eel $e_k$ of thickness $k$ and length $n-2k$ in $\cB_n$ whenever $0 \leq k \leq (n-1)/2$. Notice that the jewels not in the eel form two copies of $\cP_{n-2k-1}$. Since $\Tw(t) = P_t'(1)$ for tilings of duplex regions, it follows that $\max \Tw(\cB_n)  \geq P_{e_k}'(1) + 2 \max \Tw(\cP_{n-2k-1})$, which implies \eqref{eq:recurrenceF}. %or $f(n) \geq k(n-2k) + 2g(n-2k-1)$ for each $0 \leq k \leq (n-1)/2$. 

%Notice that eels have many jewels that have no contribution to the twist. In fact, outside the outer cycle of $t_{k,n}$ there are two (rotated) copies of $\cP_{n-2k-1}$, and we can replace these jewels by any tiling of these two copies: we conclude that, whenever $1 \leq k \leq (n-1)/2$, $f(n) \geq k(n-k) + 2 g(n - 2k - 1)$. 

For Equation \eqref{eq:recurrenceG}, notice that, whenever $0 \leq k \leq (n-1)/4$, we can fit an eel of thickness $k$ and length $n-4k$ in $\cP_n$, as illustrated in Figure \ref{fig:triangleEel}. The remaining jewels form three regions: two copies of $\cP_{2k}$ and one copy of $\cP_{n - 4k - 1}$, thus $g(n) \geq k(n-4k) + g(n-4k-1) + 2g(2k)$.
\begin{figure}[ht]%
\centering
\includegraphics[width=0.65\columnwidth]{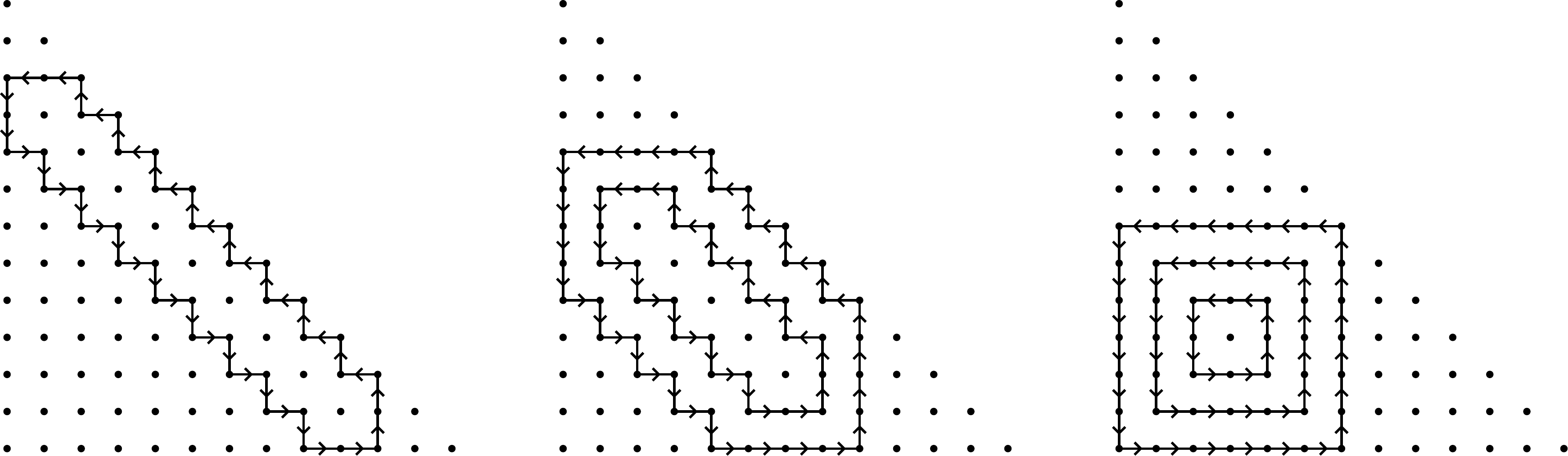}%
\caption{Socks for the constructions of eels of thickness $1$, $2$ and $3$ in $\cP_{13}$.}%
\label{fig:triangleEel}%
\end{figure}

%Therefore, we have the inequalities:
%\begin{align}
 %&\text{For each } 0 \leq k \leq \frac{n-1}{2}, f(n) \geq k(n-2k) + 2 g(n - 2k - 1), \label{eq:recurrenceF}\\
 %&\text{For each } 0 \leq k \leq \frac{n-1}{4}, g(n) \geq k(n-4k) + g(n-4k-1) + 2g(2k),
%\label{eq:recurrenceG}  
%\end{align}
%with $f(0) = g(0) = 0$.
Now let $g^*(n) = n(n+1)/2 - 4g(n)$. We claim that, for all $n \in \NN$, $g^*(n) < 2 n^{3/2}$.
By \eqref{eq:recurrenceG}, 
\begin{equation}
g^*(n) \leq \min_{0 \leq k \leq (n-1)/4} 4k^2 + n - 4k + g^*(n-4k-1) + 2 g^*(2k).
\label{eq:recurrenceGst}
\end{equation} 
Now for $0 \leq n \leq 100$, we check (with a computer) that $g^*(n) \leq 2n^{3/2}$ by recursively calculating upper bounds for $g^*(n)$ via \eqref{eq:recurrenceGst}. %inequality %$g^*(n) \leq \min_{k} 4k^2 + n - 4k + g^*(n-4k-1) + 2 g^*(2k)$.

Let $n > 100$, and assume, by induction, that $g^*(n_0) \leq 2 n_0^{2/3}$ whenever $n_0 < n$. Let $k = \floor{\sqrt{n}/2} \geq 5$. Using the induction hypothesis on $n - 4k - 1$ and on $2k$, we get
%Let $k = k(n) = \floor{\sqrt{n}/2}$. The induction hypothesis yields $g^*(n-4k-1) < 2(n - 4k - 1)^{3/2}$ and $g^*(2k) < 2(2k)^{3/2}$, so that 
$$2n^{3/2} - g^*(n) \geq 2(n^{3/2} - (n-4k-1)^{3/2}) - (4k^2 + 4(2k)^{3/2} - 4k + n).$$
Since $h(x) = x^{3/2}$ is convex and $h'(x) = (3/2)x^{1/2}$, we have %$$n^{3/2} - (n-4k-1)^{3/2} > (3/2)(n-4k-1)^{1/2}(4k+1).$$ 
$$2n^{3/2} - g^*(n) \geq 3(n-4k-1)^{1/2}(4k+1) - 4k^2 - 4(2k)^{3/2} + 4k - n.$$
Since $\sqrt{n}/2 - 1 \leq k \leq \sqrt{n}/2$, we have $4k^2 \leq n \leq 4k^2 + 8k + 4$. Thus,
$$
2n^{3/2} - g^*(n) %&\geq 3(n-4k-1)^{1/2}(4k+1) - (4k^2 + 4(2k)^{3/2} - 4k + n)\\
\geq 3(4k^2 - 4k - 1)^{1/2}(4k+1) - 8k^2 - 4(2k)^{3/2} - 4k - 4.
$$
Consider the function $u(x) = 3(4x^2 - 4x - 1)^{1/2}(4x+1) - 8x^2 - 4(2x)^{3/2} - 4x - 4$, defined on $[2,\infty)$. It is a simple exercise to prove that $u'(x) > 0$ for $x \geq 2$. Also, $u(5) \approx 209.465  > 0$ and hence $2n^{3/2} - g^*(n) \geq u(k) \geq u(5) > 0$, establishing the claim.

By \eqref{eq:recurrenceF}, $0 \leq n^2 - 4f(n) \leq 4k^2 + n - 2k + 2 g^*(n - 2k - 1)$ for each $0 \leq k \leq (n-1)/2$, and thus $0 \leq n^2 - 4f(n) \leq 2n^{3/2} + n$. Multiplying by $1/(16n^2)$ and taking the limit, we complete the proof.%we get that $\lim_{n \to \infty} \frac{f(n)}{4n^2} = 1/16$.
\end{proof}

%Actually, a stronger result follows:
\begin{lemma}
\label{lemma:possibleTwoFloors}
The following holds:
$$\lim_{L,M \to \infty} \frac{\maxtw(L,M,2)}{4LM} = \frac{1}{16}.$$
\end{lemma}
\begin{proof}
%Let $\cB_k = [0,k] \times [0,k] \times [0,2]$. %Fix $\epsilon > 0$: by Lemma \ref{lemma:fish}, we can find $k$ such that $\frac{\max\Tw(\cB_k)}{4k^2} \geq 1/16 - \epsilon$.
%Let $f(M,N) = \maxtw(M,N,2)$, and $f(N) = f(N,N)$. Also, 
Let $B_{L,M} = [0,L] \times [0,M] \times [0,2]$, and let $f(n) = \maxtw(n,n,2)$. 
Fix $k \in \NN$. Notice that we can place $\floor{L/k}\floor{M/k}$ copies of $\cB_{k,k}$ inside $\cB_{L,M}$, so that $\maxtw(L,M,2) \geq \floor{L/k}\floor{M/k} f(k) \geq (L/k - 1)(M/k-1)f(k)$. Therefore,
$$\liminf_{L,M \to \infty} \frac{\maxtw(L,M,2)}{4LM} \geq \frac{f(k)}{4k^2}.$$
Since this inequality holds for any fixed $k$, we may take the limit as $k \to \infty$ to get, by Lemma \ref{lemma:fish}, that
$$ \frac{1}{16} \leq  \liminf_{L,M \to \infty} \frac{\maxtw(L,M,2)}{4LM} \leq \limsup_{m,n \to \infty} \frac{\maxtw(L,M,2)}{4LM} \leq \frac{1}{16},$$
yielding the result.
\end{proof}
\begin{figure}[ht]%
\centering
\includegraphics[width=0.75\columnwidth]{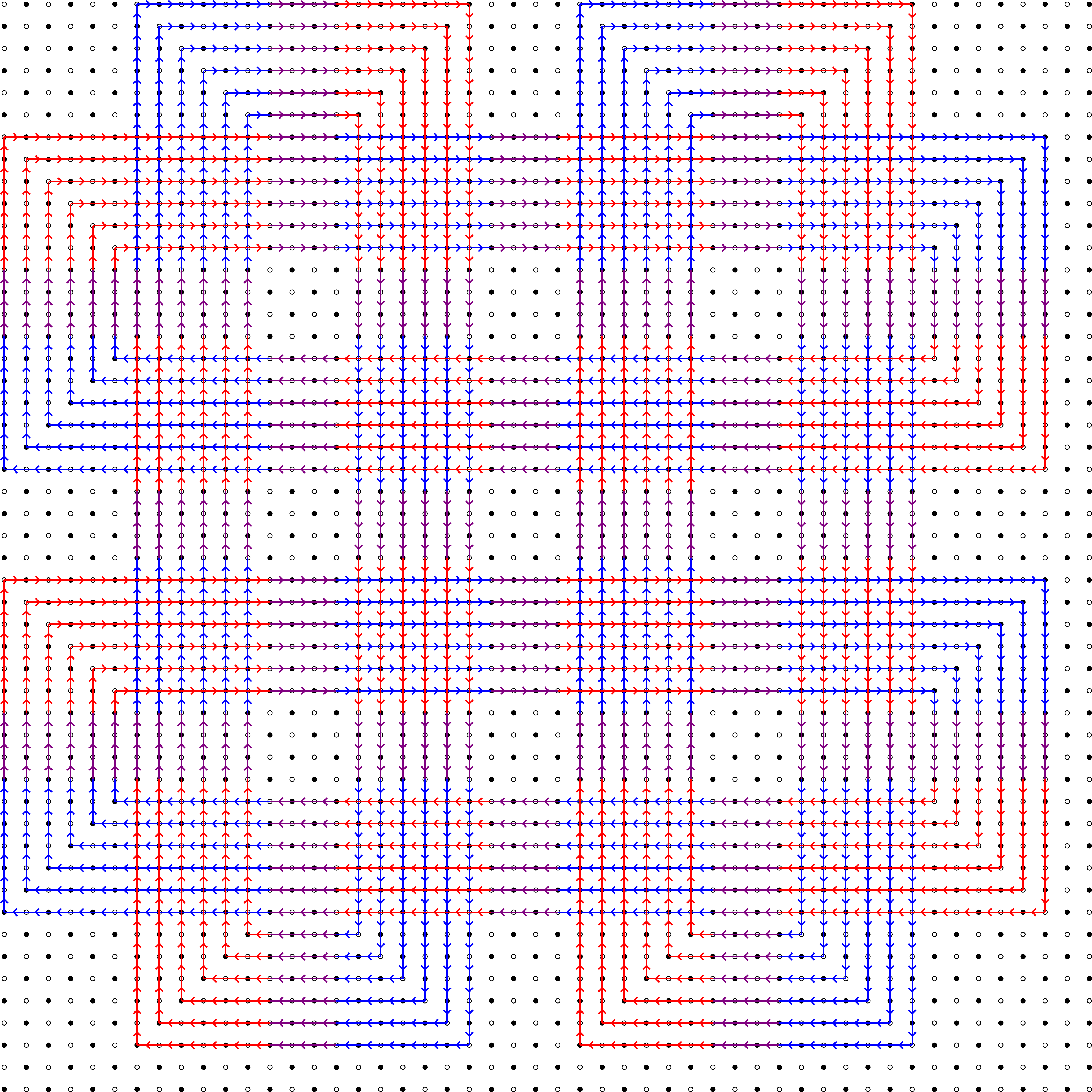}%
\caption{Schematic representation of a tiling of $\cB = [0,50] \times [0,50] \times [0,8]$, illustrating the construction with $K = 6$, so that $q_1 = q_2 = q = 4$ and $n_1 = n_2 = 4$. This representation works more or less like a sock, but the red arrows indicate that in the top four (in general $q_1$ or $q_2$, depending on the cycle) floors, all the dominoes are in the given direction and with the given orientation (taking cube colors into account);
blue arrows are similar, but with respect to the bottom four (again, $q_1$ or $q_2$) floors; finally, the purple arrows indicates that dominoes that are parallel to the arrow point in the given direction (but some dominoes are not parallel to it: see Figures \ref{fig:detailConstructionEvenOdd} and \ref{fig:detailConstructionOddEven}).}%
\label{fig:sockConstruction}%
\end{figure}
\begin{lemma}
\label{lemma:asymptoticLowerBound}
The following holds:
\begin{enumerate}[label=\upshape(\roman*)]
	\item \label{item:constFixedN} If $N \geq 2$ is fixed, then $$\lim_{L,M \to \infty} \frac{\maxtw(L,M,N)}{LMN^2} = \frac{\ceil{N/2}\floor{N/2}}{4N^2}.$$
	\item \label{item:constVariableN} $$\lim_{\substack{N,\mu,\lambda \to \infty \\ L = \floor{\lambda N}, M = \floor{\mu N}}} \frac{\maxtw(L,M,N)}{LMN^2} = \frac{1}{16},$$ 
	where the limit is taken over values $N \in \ZZ$ and $\mu,\lambda \in \RR$ such that at least one of $L,M,N$ is even.  
\end{enumerate}
\end{lemma}
\begin{proof}
Given $L,M,N,$ take any even positive integer $K$ such that $L,M \geq 2K$. We shall describe a tiling $t_{K,L,M,N}$ of the box $\cB = [0,L] \times [0,M] \times [0,N]$, which is hinted at in Figure \ref{fig:sockConstruction}.
Let $q_1 = \ceil{N/2}$, $q_2 = \floor{N/2}$, $q = 2 \ceil{N/4} = 2 \ceil{q_1/2}$, 
$$n_1 = 2\floor{\frac{L - 2K + q}{2(K+q)}}, n_2 = 2\floor{\frac{M - 2K + q}{2(K+q)}}.$$
Divide the box $\cB$ into the following parts (see Figure \ref{fig:sockConstruction}):
\begin{itemize}
	\item The \emph{exterior padding area}, which is the area in the complement of $[K, 2K + (n_1-1)(K+q)] \times [K, 2K + (n_2-1)(K+q)] \times [0,N]$ where the curves make turns.
	\item The \emph{crossing zones}, formed by the subregions $[K + i(K+q), 2K + i(K + q)] \times [K + j(K + q), 2K + j(K + q)] \times [0,N]$, $0 \leq i < n_1$, $0 \leq j < n_2$, where red and blue arrows cross. %As far as effects go, only dominoes intersecting these areas count.
	\item The \emph{transition zones}, formed by the areas where purple arrows occur: each of these areas form a $q \times K \times N$ or $K \times q \times N$ subregion. %Figure \ref{fig:detailConstruction} illustrates that a length of at least $N/2$ is needed to transition from red arrow to blue arrows in Figure \ref{fig:sockConstruction} and vice versa.
	\item The \emph{filler}, which is the remainder of the box.
\end{itemize}
The idea for the tiling is as follows: each subregion $[0,L] \times [j,j+1] \times [0,N]$, where $ 0 \leq j - K - n(K+q) < K$ for some $n$, is tiled in a similar fashion as Figure \ref{fig:detailConstructionEvenOdd}: the part that ``flows horizontally'' is formed by $q_1$ simple paths that move up and down together such that all horizontal dominoes $d$ have $\tv(d)$ parallel to the corresponding arrow in Figure \ref{fig:sockConstruction} (which are also horizontal in that figure). Similarly, subregions of the form $[i,i+1] \times [0,M] \times [0,N]$, where $ 0 \leq i - K - n(K+q) < K$ for some $n$, are tiled similarly to Figure \ref{fig:detailConstructionOddEven}: one clearly sees $q_2$ paths flowing horizontally, which correspond to vertical arrows in Figure \ref{fig:sockConstruction}.
\begin{figure}[ht]%
\centering
\includegraphics[width=0.8\columnwidth]{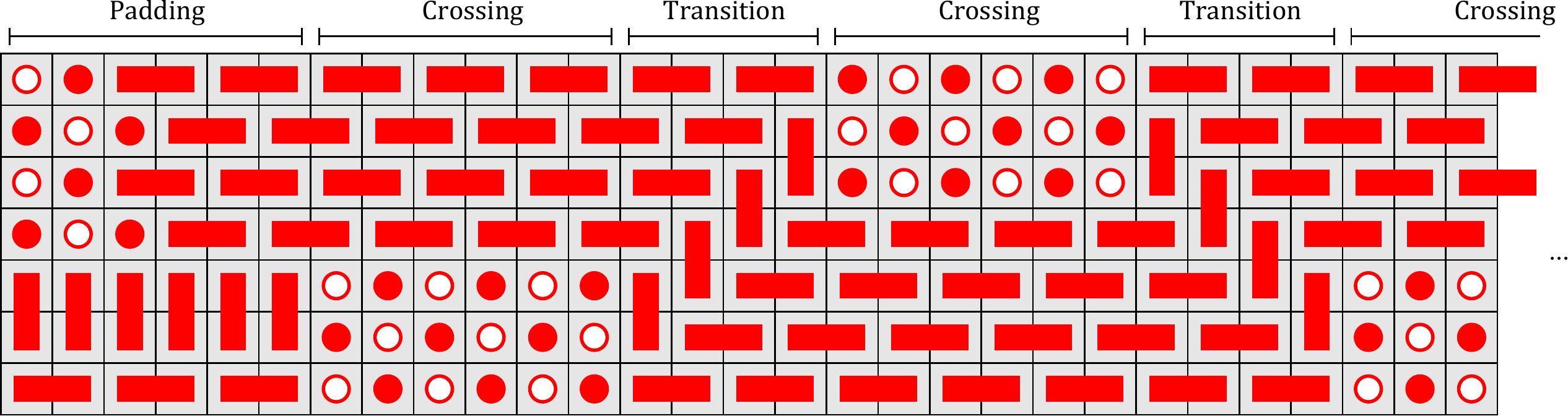}%
\caption{Lateral view on the construction, with $K = 6$ and $N = 7$, so that $q = q_1 = 4$ and $q_2 = 3$: this represents part of the subregion $[0,L] \times [K+2,K+3] \times [0,N]$ (see also Figure \ref{fig:sockConstruction}).}%
\label{fig:detailConstructionEvenOdd}%
\end{figure}
\begin{figure}[ht]%
\centering
\includegraphics[width=0.8\columnwidth]{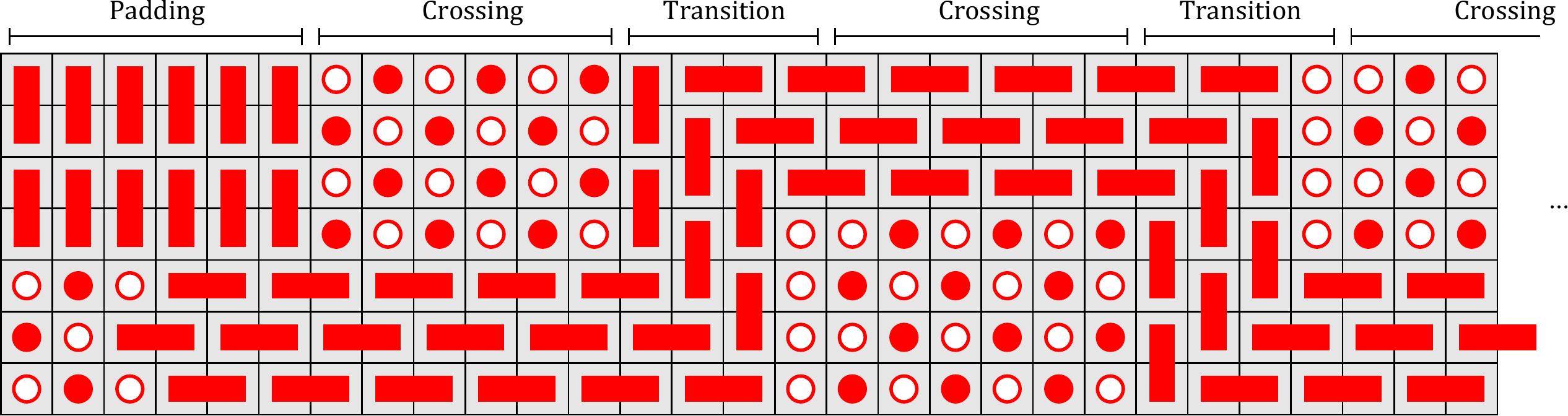}%
\caption{Lateral view on the construction, with $K = 6$ and $N = 7$, so that $q = q_1 = 4$ and $q_2 = 3$: this represents part of the subregion $[K+2,K+3] \times [0,M] \times [0,N]$ (see also Figure \ref{fig:sockConstruction}).}%
\label{fig:detailConstructionOddEven}%
\end{figure}

Figures \ref{fig:detailConstructionEvenOdd} and \ref{fig:detailConstructionOddEven} also make it clear that the paths need a length of $q_1$ (for the first kind) and $q_2$ (second kind) to transition from top to bottom and vice-versa. Since $q \geq \max(q_1,q_2)$, the paths have enough space to transition: any remaining space is filled with dominoes that are orthogonal to the paths, as shown in Figure \ref{fig:detailConstructionOddEven} (this is possible because $K$ is even). 

Finally, notice that the external padding area has enough space for the curves to make turns, as indicated in Figure \ref{fig:sockConstruction}. It is also easy to see that the remainder of the box can be filled in a way that doesn't create any $\ez$-effects (for instance, since $q$ is even, the $q \times q \times N$ filler subregions can be tiled in a trivial way).

Let $t_{K,L,M,N}$ denote the tiling obtained from this construction: we claim that, in essence, only $\ez$-effects in the contact zones count towards the twist. To see this, notice that the only areas outside contact zones where there is nonzero $\ez$-effects between dominoes are the external padding area and (possibly) the transition zones (for instance, if $q > q_1$ and $q_1$ is odd, there might be nonzero effects between the dominoes in the path and the dominoes that are used to fill the remaining space). For the transitions zones, each cycle has an even number of transitions: since $K$ and $q$ are even, it is easy to see that the effects from two consecutive transitions cancel out. As for the external padding zones, the effects in the four corners of the cycles cancel out if $N$ is even; if $N$ is odd, however, and, say, $q_i$ is even (exactly one of $q_1, q_2$ is even), then some effects that do not cancel, adding up to $K n_i q_i/2 \leq K \max(n_1, n_2) q/2 \leq (L+M)KN$.  

By construction, the effects in a single $K \times K$ contact zone equal $(K^2q_1q_2)/4$. Thus, letting $E = (L+M)KN$, we may write
\begin{align}
&\Tw(t_{K,L,M,N}) \geq \frac{n_1n_2K^2q_1q_2}{4} - E \nonumber\\
&\geq \floor{\frac{L - 2K + 2\ceil{N/4}}{2(K+2\ceil{N/4})}}\floor{\frac{M - 2K + 2\ceil{N/4}}{2(K+2\ceil{N/4})}}K^2\ceil{N/2}\floor{N/2} - E, %\nonumber\\
%& \geq \frac{LMN^2}{16} \cdot \frac{(1 - 4K/L - N/(2L))(1-4K/M - N/(2M))}{(1 + N/(2K))^2}.
\label{eq:constBound}
\end{align}
To see \ref{item:constFixedN}, let $N \geq 2$ be fixed.
If, given $L,M$, we let let $K = K(L,M) = 2 \floor{\sqrt{\min(L,M)}}$, \eqref{eq:constBound} yields that
$$\lim_{L,M \to \infty} \frac{\Tw(t_{K,L,M,N})}{LMN^2} = \lim_{K \to \infty} \frac{K^2 \ceil{N/2}\floor{N/2}}{4(K+2\ceil{N/4})^2N^2} = \frac{\ceil{N/2}\floor{N/2}}{4N^2},$$
which, together with Lemma \ref{lemma:maxTwist}, implies \ref{item:constFixedN}.

For \ref{item:constVariableN}, let $L = \floor{\lambda N}$, $M = \floor{\mu N}$ and $K = K(\lambda, \mu, N) = 2N \floor{\sqrt{\min(\lambda,\mu)}}$.
Then, as $\lambda, \mu, N \to \infty$ we have $K/N \to \infty$ and $\frac{\Tw(t_{K,L,M,N})}{LMN^2} \approx \frac{\ceil{N/2}\floor{N/2}}{4N^2} \to 1/16$. 
%$\lim_{\lambda, \mu, N \to \infty} \frac{\Tw(t_{K,L,M,N})}{LMN^2} = \lim_{\kappa \to \infty} \frac{K^2 \ceil{N/2}\floor{N/2}}{4(K+2\ceil{N/4})^2N^2} = \frac{\ceil{N/2}\floor{N/2}}{4N^2} 
\end{proof}

\section{General bounds for boxes}

\begin{lemma}
\label{lemma:techUpperBound}
Let $L,M,N$ be integers, at least one of them even. Then
$$\maxtw(L,M,N) \leq \frac{LMN}{4}\max_{\substack{\alpha,\beta,\gamma \geq 0\\ \alpha + \beta + \gamma = 1}} \min\left(\frac{N\alpha\beta}{\alpha+\beta}, \frac{M\alpha\gamma}{\alpha+\gamma}, \frac{L\beta\gamma}{\beta+\gamma}\right).$$
\end{lemma}
\begin{proof}
Let $t$ be a tiling of an $L \times M \times N$ box, and let $a,b,c$ denote the number of dominoes parallel to $\ex, \ey, \ez$, respectively: notice that $2a + 2b + 2c = LMN$.

For each pair of integers $0 \leq x < L, 0 \leq y < M$, let $a(x,y)$ (resp. $b(x,y)$) denote the number of dominoes parallel to $\ex$ (resp. $\ey$) which contain a cube of the form $C(\plshalf{x}, \plshalf{y} \plshalf{z}$) for some $z$. Then $0 \leq a(x,y) + b(x,y) \leq N$ and $\sum_{x,y} a(x,y) = 2a$, $\sum_{x,y} b(x,y) = 2b$.

Notice that $\Tw(t) \leq \frac{1}{4} \sum_{x,y} a(x,y) b(x,y)$. We claim that $\frac{1}{4} \sum_{x,y} a(x,y) b(x,y) \leq \frac{Nab}{2(a+b)}$. To see this, we will consider the following relaxed optimization problem: we want to find the maximum value of $\frac{1}{4}\sum_{x,y} \bar{a}(x,y)\bar{b}(x,y)$, subject to the conditions that $\bar{a}(x,y)$ and $\bar{b}(x,y)$ are nonnegative real numbers, $\bar{a}(x,y) + \bar{b}(x,y) \leq N$, $\sum_{x,y} \bar{a}(x,y) = 2a$, $\sum_{x,y} \bar{b}(x,y) = 2b$.

%Let $a^*(x,y), b^*(x,y) \in \RR$ be positive real numbers with the property that 
Let $(\bar{a}(x,y), \bar{b}(x,y))$ be such that $\sum_{x,y} \bar{a}(x,y)\bar{b}(x,y)$ is maximum. We claim that there is at most one pair $(x,y)$ such that $0 < \bar{a}(x,y) + \bar{b}(x,y) < N$. To see this, suppose that $(x_1,y_1)$ and $(x_2,y_2)$ are two such pairs, and suppose $\bar{a}(x_1,y_1) + \bar{b}(x_1,y_1) \geq \bar{a}(x_2,y_2) + \bar{b}(x_2,y_2)$. Then for any $\epsilon > 0$ we have $(\bar{a}(x_1,y_1) + \epsilon)(\bar{b}(x_1, y_1) + \epsilon) + (\bar{a}(x_2,y_2) - \epsilon)(\bar{b}(x_2, y_2) - \epsilon) > \bar{a}(x_1,y_1)\bar{b}(x_1, y_1) + \bar{a}(x_2,y_2)\bar{b}(x_2, y_2)$, which contradicts the maximality of $\bar{a}, \bar{b}$.  

The latter claim implies that there are exactly $\ceil{(2a+2b)/N}$ pairs $(x,y)$ such that $\bar{a}(x,y) + \bar{b}(x,y) > 0$, and exactly $\floor{(2a+2b)/N} = K$ pairs $(x,y)$ where $\bar{a}(x,y) + \bar{b}(x,y) = N$; if we consider two of these $K$ positions $(x_1,y_1)$ and $(x_2,y_2)$, we see that we must have $\bar{a}(x_1,y_1) = \bar{a}(x_2,y_2)$ and $\bar{b}(x_1,y_1) = \bar{b}(x_2,y_2)$ (if that's not the case, it's easy to see that it cannot be a maximum), so that, in each of these $K$ pairs, $\bar{a}(x,y) = Na / (a+b)$ and $\bar{b}(x,y) = Nb / (a+b)$. If the fractional part $s$ of $(2a+2b)/N$ is not zero, then there is exactly one position where $\bar{a}(x,y) = Nas/(a+b)$ and $\bar{b}(x,y) = Nbs/(a+b)$. Consequently, 
$$\frac{1}{4} \sum_{x,y} \bar{a}(x,y) \bar{b}(x,y) = \frac{1}{4}(K + s^2)\frac{N^2ab}{(a+b)^2} \leq \frac{2a + 2b}{4N} \cdot \frac{N^2ab}{(a+b)^2} = \frac{Nab}{2(a+b)},   $$
establishing the first claim.

Finally, by symmetry (by looking at effects along other directions), we see that $\Tw(t) \leq \half\min\left(\frac{Nab}{a+b}, \frac{Mac}{a+c}, \frac{Lbc}{b+c}\right)$, and
\begin{align*}
 \maxtw(L,M,N) &\leq \half\max_{\substack{a,b,c \geq 0\\ a + b + c = LMN/2}} \min\left(\frac{Nab}{a+b}, \frac{Mac}{a+c}, \frac{Lbc}{b+c}\right)\\
&\leq \frac{LMN}{4}\max_{\substack{\alpha,\beta,\gamma \geq 0\\ \alpha + \beta + \gamma = 1}} \min\left(\frac{N\alpha\beta}{\alpha+\beta}, \frac{M\alpha\gamma}{\alpha+\gamma}, \frac{L\beta\gamma}{\beta+\gamma}\right),
\end{align*}
completing the proof.
\end{proof}

\begin{lemma}
\label{lemma:techOrderOptima}
Suppose $N \leq M \leq L$, and suppose $(\bar{\alpha}, \bar{\beta}, \bar{\gamma})$ realize the maximum for the problem
\begin{equation} 
\max_{\substack{\alpha,\beta,\gamma \geq 0\\ \alpha + \beta + \gamma = 1}} \min\left(\frac{N\alpha\beta}{\alpha+\beta}, \frac{M\alpha\gamma}{\alpha+\gamma}, \frac{L\beta\gamma}{\beta+\gamma}\right).
\label{eq:optimalLMN}
\end{equation}
Then 
$$\frac{N\baralpha\barbeta}{\baralpha+\barbeta} = \frac{M\baralpha\bargamma}{\baralpha+\bargamma} \leq \frac{L\barbeta\bargamma}{\barbeta+\bargamma}.$$
\end{lemma}
\begin{proof}
First, observe that $\baralpha\barbeta\bargamma > 0$.
Suppose, by contradiction, that $\frac{N\baralpha\barbeta}{\baralpha+\barbeta} > \frac{M\baralpha\bargamma}{\baralpha+\bargamma}$. Then $\barbeta > \bargamma$, and we can pick a sufficiently small $\epsilon > 0$ such that the triple $(\baralpha, \barbeta - \epsilon, \bargamma + \epsilon)$ satisfies 
$\frac{N\baralpha\barbeta}{\baralpha+\barbeta} > \frac{N\baralpha(\barbeta-\epsilon)}{\baralpha+\barbeta-\epsilon} > \frac{M\baralpha(\bargamma+\epsilon)}{\baralpha+\bargamma + \epsilon} > \frac{M\baralpha\bargamma}{\baralpha+\bargamma}$ 
and
$
\frac{L(\barbeta-\epsilon)(\bargamma+\epsilon)}{\barbeta+\bargamma} = \frac{L(\barbeta\bargamma + \epsilon(\barbeta - \bargamma) - \epsilon^2)}{\barbeta+\bargamma} > \frac{L\barbeta\bargamma}{\barbeta+\bargamma},$ which is a contradiction. A similar analysis shows that we cannot have $\frac{L\barbeta\bargamma}{\barbeta+\bargamma} > \frac{M\baralpha\bargamma}{\baralpha+\bargamma}$, so that we necessarily have $\frac{N\baralpha\barbeta}{\baralpha+\barbeta} \leq \frac{M\baralpha\bargamma}{\baralpha+\bargamma} \leq \frac{L\barbeta\bargamma}{\barbeta+\bargamma}$.

Finally, if $\frac{N\baralpha\barbeta}{\baralpha+\barbeta} < \frac{M\baralpha\bargamma}{\baralpha+\bargamma} \leq \frac{L\barbeta\bargamma}{\barbeta+\bargamma}$, then clearly the triple $(\baralpha+\epsilon, \barbeta, \bargamma-\epsilon)$ would have a strictly larger objective function for sufficiently small $\epsilon > 0$, and thus we must have $\frac{N\baralpha\barbeta}{\baralpha+\barbeta} = \frac{M\baralpha\bargamma}{\baralpha+\bargamma}$.
\end{proof}
\begin{coro}
Suppose $N \leq M \leq L$, at least one of them even. Then
$$\maxtw(L,M,N) < \frac{1}{16}\left(1 - \frac{N}{4M}\right)LMN^2. $$
\end{coro}
\begin{proof}
By Lemma \ref{lemma:techOrderOptima}, the maximum value of \eqref{eq:optimalLMN} can be bounded above by the problem
\begin{equation}
\max_{\substack{\alpha,\beta,\gamma \geq 0\\ \alpha + \beta + \gamma = 1}} \frac{N\alpha\beta}{\alpha + \beta}, \quad \text{ subject to } \frac{N\alpha\beta}{\alpha+\beta} = \frac{M\alpha\gamma}{\alpha + \gamma}.
\label{eq:partialOptimalLMN}
\end{equation}
Let $\mu = N/M$, so that $0 < \mu \leq 1$, and denote by $A(\mu)$ the maximum value of
$\frac{\alpha\beta}{\alpha + \beta}$ subject to the condition that $\frac{\mu\alpha\beta}{\alpha+\beta} = \frac{\alpha\gamma}{\alpha + \gamma}$: we claim that $A(\mu) < \frac{1}{4}(1 - \mu/4)$. To see this, suppose $\frac{\alpha\beta}{\alpha + \beta} \geq \frac{1}{4}(1 - \mu/4)$ for some $\alpha, \beta$: a straightforward calculation shows that $\alpha + \beta \geq 1 - \mu/4$, so that $\gamma = 1 - \alpha - \beta \leq \mu/4$ and so
$$\frac{\alpha \gamma}{\alpha + \gamma} < \gamma(1-\gamma) < \frac{\mu}{4}(1 - \mu/4) \leq  \frac{\mu\alpha\beta}{\alpha + \beta},$$ 
hence the triple $\alpha,\beta,\gamma$ cannot satisfy the condition $\frac{N\alpha\beta}{\alpha+\beta} = \frac{M\alpha\gamma}{\alpha + \gamma}.$
Therefore, we have that the maximum of \eqref{eq:partialOptimalLMN} is less than $\frac{N}{4}(1 - N/(4M))$. By Lemma \ref{lemma:techUpperBound}, 
$$\maxtw(L,M,N) < \frac{LMN}{4} \cdot \frac{N}{4}\left(1 - \frac{N}{4M}\right),$$
completing the proof.
\end{proof}

\begin{coro}[Case $L = M$]
\label{coro:caseLEqualsM}
Suppose $N \leq M$, at least one of them even, and let $\mu = N/M$. Then
$$\maxtw(M,M,N) \leq \frac{2 - \mu}{8(4 - \mu)} M^2N^2.$$
\end{coro}
In particular, $\maxtw(N,N,N) \leq N^4/24$.
\begin{proof}
Suppose $(\baralpha, \barbeta, \bargamma)$ realize the maximum of \eqref{eq:optimalLMN} for $L = M$. A very similar analysis to the one done in the proof of Lemma \ref{lemma:techOrderOptima} shows that 
$\frac{N\baralpha\barbeta}{\baralpha+\barbeta} = \frac{M\baralpha\bargamma}{\baralpha+\bargamma} = \frac{M\barbeta\bargamma}{\barbeta+\bargamma},$ so that $\baralpha = \barbeta$ and $\frac{N}{2} = \frac{M(1-2\baralpha)}{1-\baralpha}$. Writing $\mu = N/M$, we see that $\baralpha = \frac{2-\mu}{4-\mu}$ and so $\frac{\baralpha\barbeta}{\baralpha+\barbeta} = \frac{\baralpha}{2} = \frac{2-\mu}{2(4-\mu)}$. By Lemma \ref{lemma:techUpperBound}, we're done.
\end{proof}

\begin{coro}[Case $M = N$]
Suppose $N \leq L$, at least one of them even, and let $\lambda = N/L$. Then $\maxtw(L,N,N) \leq C(\lambda) LN^3,$ where
$$C(\lambda) = 
\begin{cases} 
\frac{3 - 2\sqrt{2}}{4}, &\lambda \leq \frac{2+\sqrt{2}}{4};\\
\frac{2\lambda - 1}{8\lambda(4\lambda - 1)}, &  \frac{2+\sqrt{2}}{4} < \lambda \leq 1.
\end{cases}
$$ 
\end{coro}
\begin{proof}
Let $(\baralpha, \barbeta, \bargamma)$ realize the maximum of \eqref{eq:optimalLMN} for $M = N$. By Lemma \ref{lemma:techOrderOptima}, $\frac{\baralpha\barbeta}{\baralpha+\barbeta} = \frac{\baralpha\bargamma}{\baralpha+\bargamma} \leq \frac{\barbeta\bargamma}{\lambda(\barbeta+\bargamma)},$
so that $\barbeta = \bargamma$. Therefore, the value
$
\max_{0 \leq \gamma \leq 1}\min\left(\frac{(1-2\gamma)\gamma}{1-\gamma}, \frac{\gamma}{2\lambda}\right)
$,
multiplied by $N$, equals the maximum in \eqref{eq:optimalLMN}. Let $f(\gamma) = \frac{(1-2\gamma)\gamma}{1-\gamma}$ and $g(\gamma) = \frac{\gamma}{2\lambda}$.  

Notice that $f$ is strictly concave in $(0,1)$. Since $f'(1 - \sqrt{2}/2) = 0$, it follows that this is the only maximum point of $f$.
Let $\gamma^* \in [0,1]$ be the largest value for which $f(\gamma^*) = g(\gamma^*)$ ($f(0) = g(0)$). Since for $\gamma \geq 1/2$ we have $f(\gamma) \leq 0 < g(\gamma)$, it follows that $\gamma^* < 1/2$ and $f(\gamma) < g(\gamma)$ whenever $\gamma > \gamma^*$. Therefore, if $\gamma^* \leq 1 - \sqrt{2}/2$, then $f(\gamma^*) < f(1 - \sqrt{2}/2) \leq g(1-\sqrt{2}/2)$, thus the maximum value is given by $f(1-\sqrt{2}/2)$. Otherwise, since $g$ is always increasing and $f$ is decreasing for $\gamma > 1 - \sqrt{2}/2$ (because it is concave), it follows that the maximum is given by $g(\gamma^*)$.

Solving the equation $f(\gamma) = g(\gamma)$ yields $\gamma^* = \max(\frac{2\lambda - 1}{4\lambda - 1},0)$, and $\frac{2\lambda - 1}{4\lambda - 1} = 1 - \sqrt{2}/2$ for $\lambda = \frac{2+\sqrt{2}}{4}$. Hence, the maximum is
$$\begin{cases}
f(1-\sqrt{2}/2) = 3 - 2\sqrt{2}, &\lambda \leq \frac{2+\sqrt{2}}{4};\\
g\left(\frac{2\lambda-1}{4\lambda-1}\right) = \frac{2\lambda - 1}{2\lambda(4\lambda - 1)}, & \frac{2+\sqrt{2}}{4} < \lambda \leq 1.
\end{cases}
$$
By Lemma \ref{lemma:techUpperBound}, we have
$$\maxtw(L,N,N) \leq\frac{LN^3}{4} \max_{0 \leq \gamma \leq 1}\min\left(\frac{(1-2\gamma)\gamma}{1-\gamma}, \frac{\gamma}{2\lambda}\right) \leq C(\lambda)LN^3,$$
completing the proof. 
\end{proof}

We shall now finish the discussion with a construction that yields a lower bound for the maximum twist of boxes. Let $a,b,c,d \in \NN$. The \emph{solenoid} with parameters $(a,b,c,d)$ is a tiling of the box $\cB = [0,2b+2c] \times [0,2b+d] \times [0,a+2c]$ whose twist is $\ceil{abcd/2}$. 

A \emph{wire loop} is a region which, after translation by vectors of integer coordinates and permutations of the axes, equals $([0,m] \times [0,n] \times [0,1]) \setminus ((1,m-1) \times(1,n-1) \times [0,1])$ as illustrated by Figure \ref{fig:onionRings}. Notice that a wire loop has exactly two tilings. 

\begin{figure}[ht]%
\centering
\subfloat[]{\includegraphics[width=0.3\columnwidth]{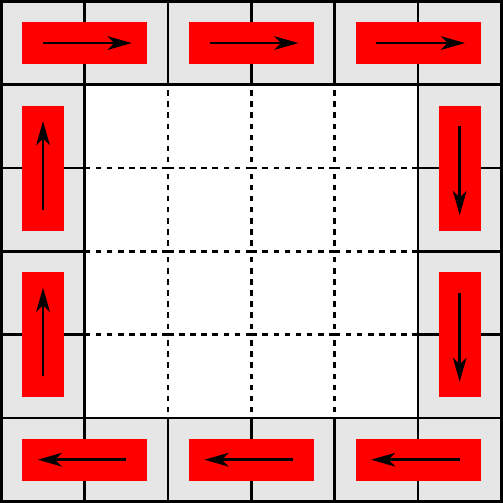}} \qquad \qquad \qquad \qquad
\subfloat[]{\includegraphics[width=0.3\columnwidth]{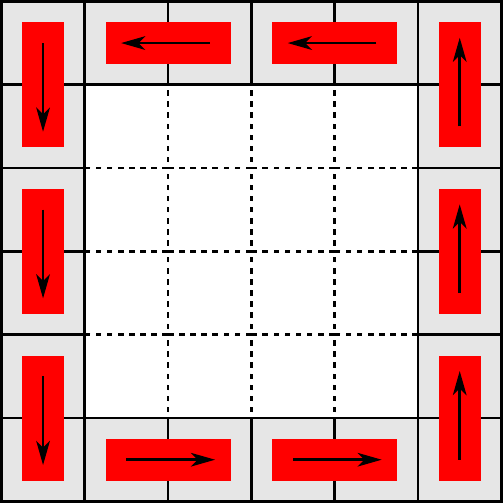}}
\caption{The two tilings of a wire loop. The vector $\tv(d)$ is shown for each domino $d$.}%
\label{fig:onionRings}%
\end{figure}

Figure \ref{fig:solenoid} shows the basic idea of the solenoid construction. We consider the \emph{cable loops}
\begin{align*}
L_0 &= ([0,2b+c] \times [0,2b+d] \times [c,a+c]) \setminus ((b,b+c) \times (b,b+d) \times [c,a+c]), \\
L_1 &= ([b,2b+2c] \times [b,b+d] \times [0,a+2c])\\ &\setminus ((b+c,2b+c) \times [b,b+d] \times (c,a+c)),
\end{align*} 
illustrated in Figure \ref{fig:solenoid}. 
\begin{figure}[ht]%
\centering
\def\svgwidth{0.4\columnwidth}
\subfloat[Cables drawn separately.]{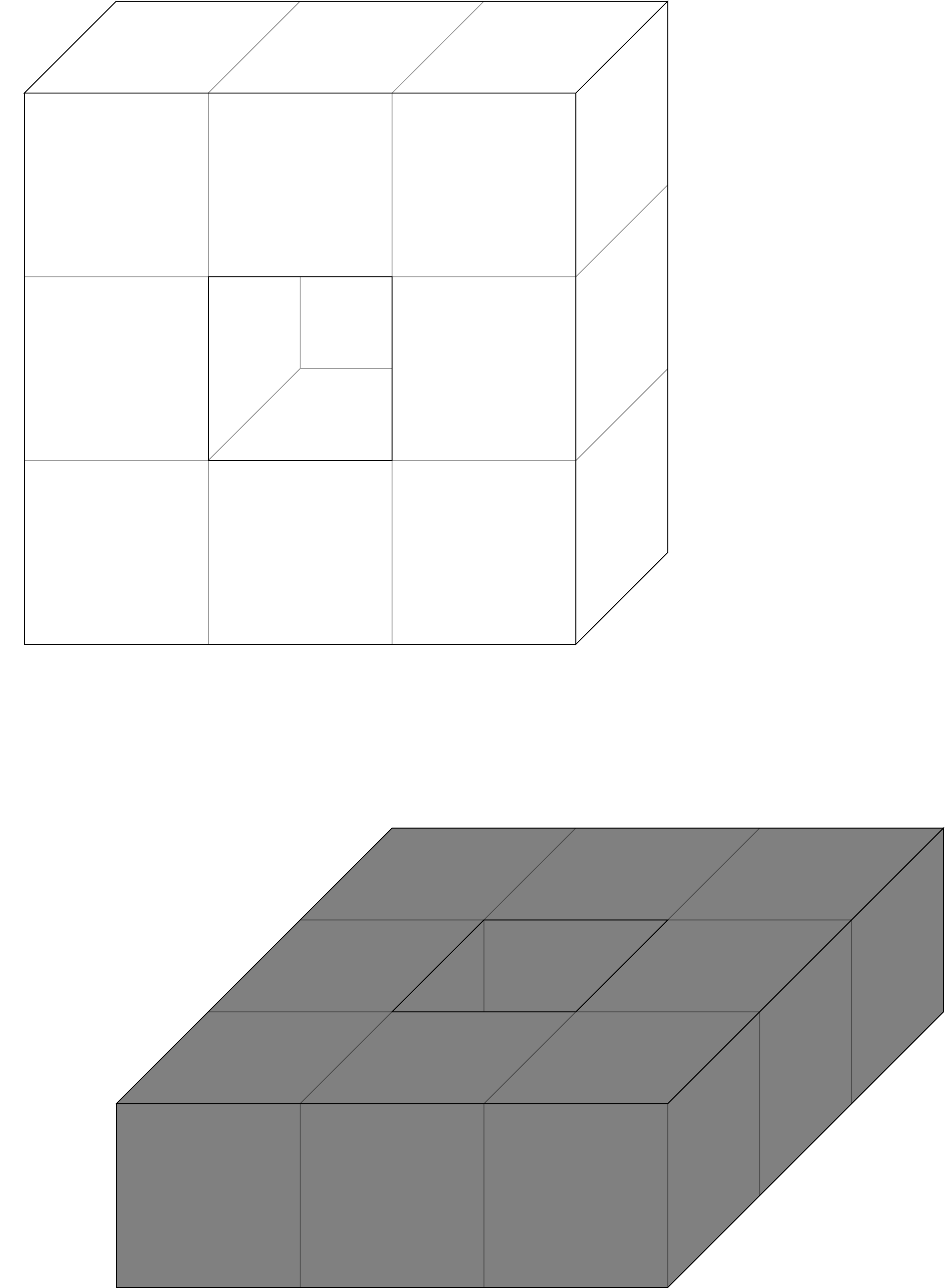}
\qquad \qquad \qquad
\def\svgwidth{0.4\columnwidth}
\subfloat[Cables in the same figure.]{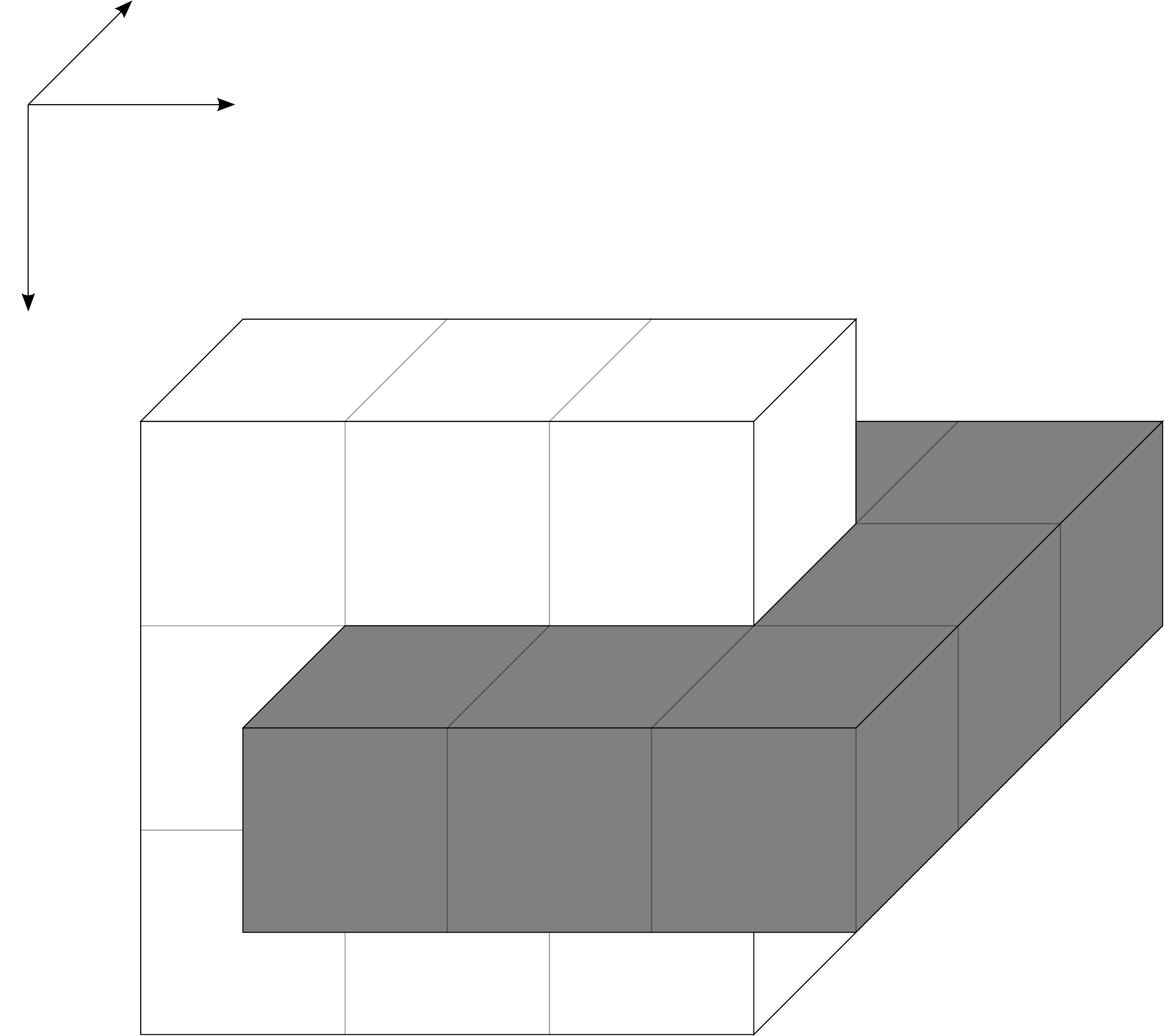}
\caption{The cable loops $L_0$ (white) and $L_1$ (grey). }%
\label{fig:solenoid}%
\end{figure}
Notice that $L_0$ is the disjoint union of $ab$ wire loops, whereas $L_1$ is the disjoint union of $cd$ wire loops. Each of these wire loops can be tiled in such a way that if $W_0$ is a wire loop in $L_0$ with tiling $t_{W_0}$ and $W_1$ is a wire loop in $t_1$ with tiling $t_{W_1}$, then $T^{\ex}(t_{W_0},t_{W_1}) + T^{\ex}(t_{W_1},t_{W_0}) = 1/2$. If we consider the tilings $t_0$ of $L_0$ and $t_1$ of $L_1$ formed by individual tilings of each wire loop, then $T^{\ex}(t_0,t_1) + T^{\ex}(t_1,t_0) = (abcd)/2$. Figure \ref{fig:solenoid_example} shows an example of this construction.
\begin{figure}[ht]%
\centering
\includegraphics[width=0.8\columnwidth]{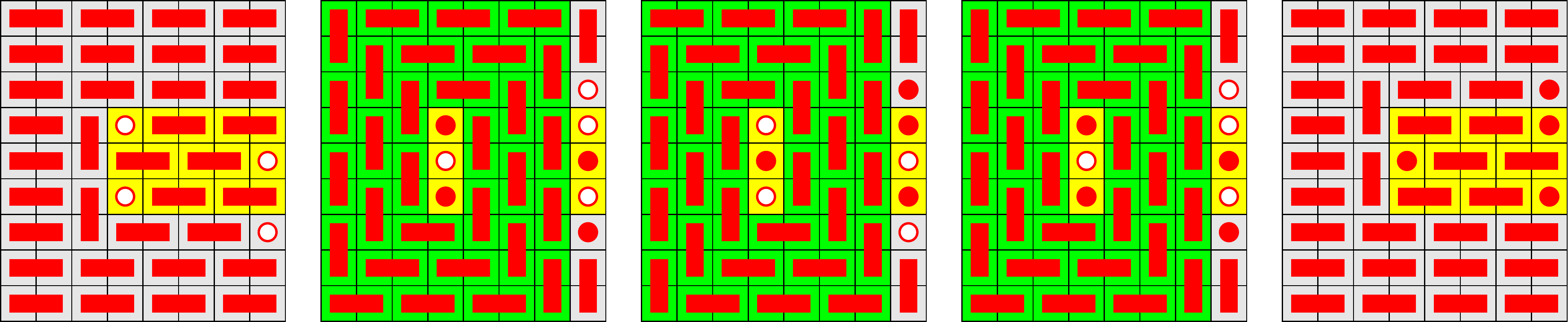}%
\caption{A solenoid with parameters $(3,3,1,3)$. $L_0$ is highlighted in green, while $L_1$ is highlighted in yellow. Notice that the twist of this tiling is $14 = \ceil{3\cdot 3 \cdot 1 \cdot 3/2}$.}%
\label{fig:solenoid_example}%
\end{figure}

A reasonably straightforward case-by-case analysis shows that, if at least one of the parameters $a,b,c,d$ is even, the remainder of the box can be tiled in such a way that the twist of the solenoid is exactly $abcd/2$ (most dominoes will be parallel to $\ex$, and we need to check that the effects between others cancel out). If all dimensions are odd (as in Figure \ref{fig:solenoid_example}), then everything cancels out with the exception of two pairs of dominoes with positive sign in the twist, hence yielding $\ceil{abcd/2}$.

We shall now use this construction to find a lower bound for all boxes. First, an auxiliary Lemma:
\begin{lemma}
\label{lemma:auxLowerBound}
Let $\cB = [0,L] \times [0,M] \times [0,N]$ be a box, with $N$ even. Then $\max \Tw(\cB) \geq \floor{L/3}\floor{M/3} N/2$.
\end{lemma}
\begin{proof}
Consider the boxes $\cB_{i,j,k} = [3i, 3(i+1)] \times [3j, 3(j+1)] \times [2k, 2(k+1)]$ for $0 \leq i < \floor{L/3}$, $0 \leq j < \floor{M/3}$, $0 \leq k < N/2$. Since the complement $\cB \setminus (\bigcup_{i,j,k} B_{i,j,k})$ is tileable with every domino parallel to $\ez$, we can clearly construct a tiling $t$ whose twist is $\sum_{i,j,k}\maxtw(3,3,2) = \floor{L/3} \floor{M/3} (N/2)$ (see Example \ref{example:332box}).%, where $t_{i,j,k}$ is a tiling of $\cB_{i,j,k}$. Hence, $\max \Tw(\cB) \geq \floor{L/3} \floor{M/3} (N/2) \cdot \max \Tw(\cB_{0,0,0}) =  \floor{L/3} \floor{M/3} N/2.$   
\end{proof}
\begin{rem}
\label{rem:flipConnectedBoxes}
Lemma \ref{lemma:auxLowerBound} also implies that the only tileable flip connected boxes are those of dimension $L \times M \times 1$ and $L \times 2 \times 2$ (and rotations of those boxes). It is easy to show that these boxes are flip connected: $L \times M \times 1$ boxes are essentially plane regions, and one can show by induction that $L \times 2 \times 2$ boxes are flip connected. For any other tileable $L \times M \times N$ box, we can after a rotation assume that $L, M \geq 3$, $N \geq 2$ and that $N$ is even, so that $k = \floor{L/3}\floor{M/3}N > 0$. Since all the elements in the sequence differ by positive trits, it follows that the box has at least $k+1> 1$ flip connected components. 
\end{rem} 
\begin{lemma}
\label{lemma:solenoid}
If $L,M,N \geq 2$, at least two of them greater than or equal to $3$ (and at least one of them even), then $\maxtw(L,M,N) \geq (1/162)LMN \min(L,M,N)$.
\end{lemma}
\begin{proof}
Assume that $N \leq M \leq L$. If $N < 4$, then Lemma \ref{lemma:auxLowerBound} suffices: if $N = 2$, $\floor{L/3}\floor{M/3}N/2 \geq LMN/50 \geq LMN^2/100$; if $N = 3$, and, say, $L$ is even, then $\floor{N/3}\floor{M/3}L/2 \geq LMN/30 \geq LMN^2/90$. Therefore, we may assume that $N \geq 4$. 

In order to construct the most efficient solenoid, we rotate the box and consider $\cB = [0,N] \times [0,M] \times [0,L]$.
Let $b = \floor{N/4}$, $c = \floor{(N+2)/4}$, $a = L - 2c$ and $d = M - 2b$. Let $t$ be the tiling of $\cB$ obtained from the solenoid construction with parameters $a,b,c,d$ (and $[N - 2b - 2c,N] \times [0,M] \times [0,L]$ is tiled with all dominoes parallel to each other).

Now
$$
\frac{\Tw(t)}{LMN^2} = \frac{\ceil{\frac{abcd}{2}}}{LMN^2} \geq \frac{bc}{2N^2} \cdot \frac{M-2b}{M} \cdot \frac{L-2c}{L}
\geq \frac{bc}{2N^2} \cdot \left(1 - \frac{2b}{N}\right)\left(1 - \frac{2c}{N}\right).
$$ 
Let $g(N)$ denote the last expression: we see that $g(4) = 1/128, g(5) = 9/1400, g(6) = 1/162, g(7) = 15/2401$, which are all greater than or equal to $1/162$. 
We now want to see that $g(N+4) \geq g(N)$.
Clearly $g(4k) = 1/128$ for any $k > 0$. For the other cases, we see that for $i = 1,2,3$, we have $g(4k+i+4) - g(4k+i) = p_i(k)/(2(4k+i+4)^4(4k+i)^4),$ where $p_1(k) = (4k+3)(64k^4+192k^3+180k^2+54k+3)$, $p_2(k) = 256(k+1)^3(4k^2+8k+1)$, $p_3(k) = 2(k+1)(2k+3)(4k+5)(80k^2+200k+81)$. Since clearly $p_i(k) > 0$ for $k > 0$, we see that $g(N+4) \geq g(N)$ and thus
$\Tw(t)/(LMN^2) \geq 1/162$.
\end{proof}

\begin{proof}[Proof of Theorem \ref{theo:possible}]
The general lower and upper bounds follow directly from Lemmas \ref{lemma:maxTwist} and \ref{lemma:solenoid}. Lemmas \ref{lemma:asymptoticLowerBound} and Corollary \ref{coro:caseLEqualsM} (using $L = M = N$) imply that, if $h(L,M,N) = \maxtw(L,M,N)/(LMN\min(L,M,N))$,
then $$\limsup_{L,M,N \to \infty} h(L,M,N) = \frac{1}{16} \text{ and } \liminf_{L,M,N \to \infty} h(L,M,N) \leq \frac{1}{24},$$
completing the proof.
\end{proof}

%% file: figures/cables_individual_noaxis_markers.pdf_tex
%% Creator: Inkscape 0.48.2, www.inkscape.org
%% Accompanies image file 'cables_individual_noaxis_markers.pdf' (pdf, eps, ps)
\begingroup%
  \makeatletter%
  \providecommand\color[2][]{%
    \errmessage{(Inkscape) Color is used for the text in Inkscape, but the package 'color.sty' is not loaded}%
    \renewcommand\color[2][]{}%
  }%
  \providecommand\transparent[1]{%
    \errmessage{(Inkscape) Transparency is used (non-zero) for the text in Inkscape, but the package 'transparent.sty' is not loaded}%
    \renewcommand\transparent[1]{}%
  }%
  \providecommand\rotatebox[2]{#2}%
  \ifx\svgwidth\undefined%
    \setlength{\unitlength}{822.24774933bp}%
    \ifx\svgscale\undefined%
      \relax%
    \else%
      \setlength{\unitlength}{\unitlength * \real{\svgscale}}%
    \fi%
  \else%
    \setlength{\unitlength}{\svgwidth}%
  \fi%
  \global\let\svgwidth\undefined%
  \global\let\svgscale\undefined%
  \makeatother%
  \begin{picture}(1,1.36379212)%
    \put(0,0){\includegraphics[width=\unitlength]{cables_individual_noaxis_markers.pdf}}%
    \put(0.11915613,1.28862297){\makebox(0,0)[lb]{\smash{$b$}}}%
    \put(0.32321634,1.28862297){\makebox(0,0)[lb]{\smash{$c$}}}%
    \put(0.52622499,1.28862297){\makebox(0,0)[lb]{\smash{$b$}}}%
    \put(-0.0536485,1.130305){\makebox(0,0)[lb]{\smash{$b$}}}%
    \put(-0.0536485,0.94670235){\makebox(0,0)[lb]{\smash{$d$}}}%
    \put(0.24908312,0.96317816){\makebox(0,0)[lb]{\smash{$a$}}}%
    \put(-0.0536485,0.76626943){\makebox(0,0)[lb]{\smash{$b$}}}%
    \put(0.22596498,0.35411579){\makebox(0,0)[lb]{\smash{$a$}}}%
    \put(0.1230759,0.24354318){\makebox(0,0)[lb]{\smash{$c$}}}%
    \put(0.32362237,0.44468829){\makebox(0,0)[lb]{\smash{$c$}}}%
    \put(0.50687051,0.50299116){\makebox(0,0)[lb]{\smash{$c$}}}%
    \put(0.6876096,0.50882145){\makebox(0,0)[lb]{\smash{$b$}}}%
    \put(0.88146678,0.5117366){\makebox(0,0)[lb]{\smash{$c$}}}%
    \put(0.04875011,0.07049793){\makebox(0,0)[lb]{\smash{$d$}}}%
  \end{picture}%
\endgroup%

%% file: figures/cables.pdf_tex
%% Creator: Inkscape 0.48.2, www.inkscape.org
%% Accompanies image file 'tubes.pdf' (pdf, eps, ps)
\begingroup%
  \makeatletter%
  \providecommand\color[2][]{%
    \errmessage{(Inkscape) Color is used for the text in Inkscape, but the package 'color.sty' is not loaded}%
    \renewcommand\color[2][]{}%
  }%
  \providecommand\transparent[1]{%
    \errmessage{(Inkscape) Transparency is used (non-zero) for the text in Inkscape, but the package 'transparent.sty' is not loaded}%
    \renewcommand\transparent[1]{}%
  }%
  \providecommand\rotatebox[2]{#2}%
  \ifx\svgwidth\undefined%
    \setlength{\unitlength}{911.11071777bp}%
    \ifx\svgscale\undefined%
      \relax%
    \else%
      \setlength{\unitlength}{\unitlength * \real{\svgscale}}%
    \fi%
  \else%
    \setlength{\unitlength}{\svgwidth}%
  \fi%
  \global\let\svgwidth\undefined%
  \global\let\svgscale\undefined%
  \makeatother%
  \begin{picture}(1,0.88903714)%
    \put(0,0){\includegraphics[width=\unitlength]{cables.pdf}}%
    \put(0.19985921,0.73466454){\color[rgb]{0,0,0}\scalebox{0.7}{\makebox(0,0)[lb]{\smash{$\ex$}}}}%
    \put(-0.0201372,0.60549967){\color[rgb]{0,0,0}\scalebox{0.7}{\makebox(0,0)[lb]{\smash{$\ey$}}}}%
    \put(0.12877319,0.8804208){\color[rgb]{0,0,0}\scalebox{0.7}{\makebox(0,0)[lb]{\smash{$\ez$}}}}%
  \end{picture}%
\endgroup%

%% file: glossary_new.tex
\chapter*{Glossary}
\begin{description}

\item[{$\bB$}]
{The set of all positively oriented bases formed by vectors of $\Phi$.}

\item[{$\ccol(v)$}]
{The color of some object (a cube, a square, a vertex, etc): $1$ if it is black, and $-1$ if it is white.}

\item[{cylinder}]
{Region of the form $\cD \times [0,N]$ (possibly rotated), where $\cD \subset \RR^2$ is a connected and simply connected planar region. Same as multiplex.}

\item[{$\Delta$}]
{The set $\{\ex, \ey, \ez\}$.}

\item[{$A_0 \sqcup A_1$}]
{The union of two disjoint sets $A_0$ and $A_1$.}

\item[{duplex region}]
{Region of the form $\cD \times [0,2]$, where $\cD \subset \RR^2$ is a connected and simply connected planar region.}

\item[{embedding}]
{In general, an embedding of a tiling $t$ of $\cR$ in a larger region $\cR' \supset \cR$ is a tiling of $\cR'$ which coincides with $t$ in $\cR$. For specific uses, see Sections \ref{sec:twoFloorsMoreSpace}, \ref{sec:embedFourFloors} and \ref{sec:embeddable}.}

\item[{ghost curve}]
{Curve connecting a source to a sink in a two-story region. See page \pageref{def:ghostCurve}.}

\item[{jewel}]
{Projection of a domino that is parallel to the axis of a duplex region. See page \pageref{fig:tiling742_invariant}.}

\item[{$\Link(\gamma_0, \gamma_1)$}]
{The linking number between two curves $\gamma_0$ and $\gamma_1$.}

\item[{multiplex}]
{Region of the form $\cD \times [0,N]$ (possibly rotated), where $\cD \subset \RR^2$ is a connected and simply connected planar region. Same as cylinder.}

\item[{$\Phi$}]
{The set $\{\pm \ex, \pm \ey, \pm \ez\}$.}

\item[{($\vu$-)pretwist}]
{See $T^{\vu}(t)$.}

\item[{$P_t(q)$}]
{Polynomial invariant for regions with two floors. See pages \pageref{def:ptduplex} and \pageref{def:ptTwoStory}.}

\item[{$\plshalf{n}$}]
{Equals $n + 1/2$.}

\item[{sock}]
{System of cycles. See page \pageref{fig:tilingTwoFloors_withSOC}.}

\item[{$\tau^{\beta}_{a,b}(\ell_0, \ell_1)$}]
{The slanted effect of the segment $\ell_0$ on the segment $\ell_1$ with parameters $\beta \in \bB$ and $a,b \in \RR$. See page \pageref{def:slantedEffect}.}

\item[{$\tau^{\vu}(d_0,d_1)$}]
{The effect of the domino (or segment) $d_0$ on the domino (or segment) $d_1$ along $\vu$: see pages \pageref{def:effect} and \pageref{def:effectSegments}.}

\item[{$T^{\beta}_{a,b}(A_0,A_1)$}]
{Sums of slanted effects $\tau^{\beta}_{a,b}(\ell_0,\ell_1)$ for all segments $\ell_0 \in A_0, \ell_1 \in A_1$. See page \pageref{def:effectsTwoParams}.}

\item[{$T^{\beta}_{a,b}(A)$}]
{Equals $T^{\beta}_{a,b}(A,A)$: see page \pageref{def:effectsTwoParams}.}

\item[{$T^{\vu}(t), T^{\vu}(A)$}]
{The $\vu$-pretwist of a tiling $t$ (see page \pageref{def:pretwist}); also, if $A$ is an arbitrary set of segments, this equals $T^{\vu}(A,A)$: see page \pageref{def:effectsTwoParams}.}

\item[{$T^{\vu}(A_0,A_1)$}]
{The sum of effects of segments of $A_0$ on segments of $A_1$ along $\vu$. See page \pageref{def:effectsTwoParams}.}

\item[{$\maxtw(L,M,N)$}]
{The maximum value of the twist on the set of tilings of an $L \times M \times N$ box.}

\item[{$\Tw(t)$}]
{Integer that is invariant by flips in cylinders. See Definition \ref{def:twist} on page \pageref{def:twist}.}

\item[{$\wind(\gamma,p)$}]
{The winding number of a planar curve $\gamma$ around a point $p \in \RR^2$.}

\item[{$\Wr^{\pm}(\gamma)$}]
{A particular directional writhing number. See page \pageref{eq:writhes}.}

\item[{$\Wr(\gamma,\vu)$}]
{The directional writhing number of a curve $\gamma$ in the direction of a vector $\vu$. See page \pageref{def:directionalWrithe}.}

\end{description}